\documentclass[12pt,a4paper]{amsbook}

\usepackage[utf8]{inputenc}
\usepackage[T1]{fontenc} 
\usepackage{amscd}
\usepackage{amsmath}
\usepackage{amsthm}
\usepackage{amsfonts}
\usepackage{amssymb}
\usepackage[dvips]{graphicx}
\usepackage[usenames,dvipsnames,svgnames,table]{xcolor}
\usepackage[active]{srcltx}
\usepackage{hyperref}
\usepackage{color}
\usepackage[section]{placeins}
\usepackage{float}
\usepackage[all]{xy}

\usepackage[english]{babel}

\renewcommand{\thesection}{\thechapter .\arabic{section}}

\renewcommand{\thefigure}{\arabic{chapter}.\arabic{section}.\arabic{figure}}

\newtheorem{theorem}{Theorem}[section]
\newtheorem{prop}[theorem]{Proposition}
\newtheorem{cor}[theorem]{Corollary}
\newtheorem{rmk}[theorem]{Remark}

\newtheorem{lem}[theorem]{Lemma}
\newtheorem{dfn}[theorem]{Definition}
\newtheorem*{dfn*}{Definition}

\newtheorem{conjecture}[theorem]{Conjecture}
\numberwithin{equation}{section}

\def\RR{\mathbb{R}}

\def\Rn{\mathbb{R}^{n}}

\def\Rnuno{\mathbb{R}^{n-1}}
\def\Hnuno{\mathcal{H}^{n-1}}
\def\Hndos{\mathcal{H}^{n-2}}
\def\Hntres{\mathcal{H}^{n-3}}

\DeclareMathOperator{\Img}{Image}

\definecolor{grey}{rgb}{0.75,0.75,0.75}
\definecolor{orange}{rgb}{1.0,0.5,0.5}
\definecolor{brown}{rgb}{0.5,0.25,0.0}
\definecolor{pink}{rgb}{1.0,0.5,0.5}
\definecolor{green}{rgb}{0.0,0.533,0.133}
\definecolor{darkred}{rgb}{0.3,0.0,0.0}
\newcommand{\strong}[1]{\textbf{#1}}

\newcommand{\nin}{\not\in}
\newcommand{\nobracket}{}
\newcommand{\nocomma}{}
\newcommand{\tmdummy}{$\mbox{}$}
\newcommand{\tmop}[1]{\ensuremath{\operatorname{#1}}}
\newcommand{\diff}[1]{ #1}
\newenvironment{itemizedot}{\begin{itemize}

	}
	{ \end{itemize}}
\newenvironment{itemizeminus}{\begin{itemize}

	}
	{\end{itemize}}

\newenvironment{remark}{\medskip\noindent{\strong{Remark.}}}
{}

\newcommand{\Tone}{T_{p_1}M_1 }
\newcommand{\Ttwo}{T_{p_2}M_2 }
\newcommand{\T}{T_p M}
\newcommand{\e}{\exp_p}
\newcommand{\cleaveq}{q=e_1(x_1)=e_2(x_2)}
\def\NC{\mathcal{NC}}
\newcommand{\SAtwo}{\mathcal{SA}_2}

\def\author{Pablo Angulo Ardoy}
\def\title{Cut and conjugate points of the exponential map, with
applications}
\def\subtitle{A study on the singularities of the exponential maps of 
Riemann and Finsler manifolds, with applications to Hamilton-Jacobi equations,
and the Ambrose conjecture}
\date{Fecha}
\def\advisor{Luis Guijarro Santamaría}

\makeatletter
\renewcommand{\tocsection}[3]
	{{\indentlabel{\@ifnotempty{#2}
	{\ignorespaces\llap{#1 #2.\,\,}}}}#3}
 \renewcommand{\tocchapter}[3]
 	{{}\hskip-14pt{\indentlabel{\@ifnotempty{#2}
 	{\ignorespaces{#1 #2.\,\,}}}}#3}
\makeatother

\usepackage{fancyhdr}
\begin{document}

 \begin{titlepage}
   \begin{tabular}{c}
     \vspace{25mm}
   \end{tabular}
   \begin{center}
   \makeatletter
        \includegraphics[width=45mm]{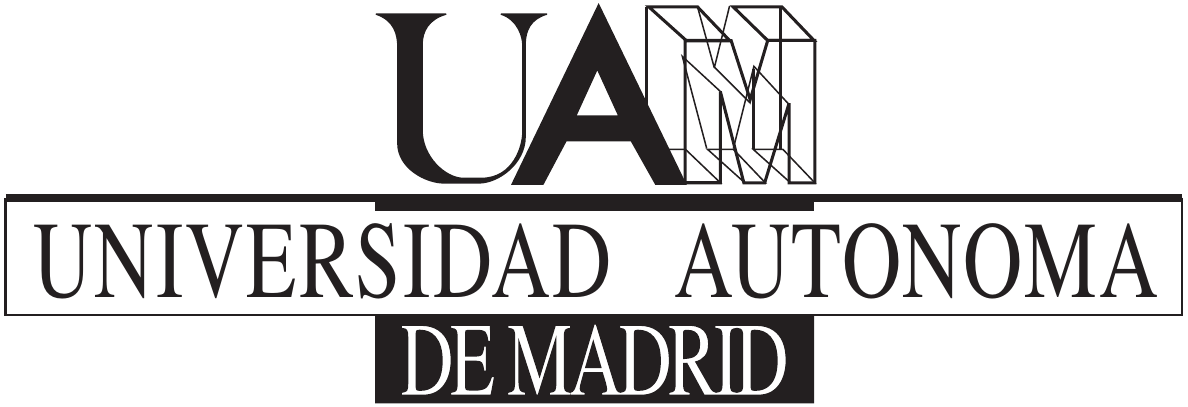}\\
        {Facultad de Ciencias \\ Departamento de Matem\'{a}ticas}\\
     \vspace{25mm}
     \LARGE{\sc{\title}}\\
     \ifx\@empty\subtitle
     \else
       \par\smallskip
       \vspace{4mm}
       \normalsize{\sc{\subtitle}}\\
     \fi
     \vspace{12mm}
     \Large{\sc{\author}}
     \ifx\@empty\@dedicatory
     \else
       \par\smallskip
       {\footnotesize\itshape\@dedicatory}
     \fi
     \mbox{}\vspace{40mm}\\
     \ifx\@empty\advisor
     \else
       \large{Tesis doctoral dirigida por\\
       \advisor}
     \fi
     \emph{\small
       \addresses\\
     }
\makeatother
   \end{center}
 \end{titlepage}
 
\setcounter{page}{0} 

\frontmatter 

\pagestyle{fancy} 

\tableofcontents

\chapter*{Introducción}

\markboth{\sc INTRODUCCIÓN}{\sc INTRODUCCIÓN}

En esta tesis estudiamos las singularidades de la aplicación exponencial en
variedades Riemannianas y Finslerianas, y el objeto conocido en inglés como
\emph{cut locus}, \emph{ridge}, \emph{medial axis} o \emph{skeleton}, de los
cuales sólo el último término suele traducirse al castellano.
En primer lugar mejoramos los resultados existentes sobre las singularidades de
la aplicación exponencial y la estructura del cut locus, y después aplicamos
estos resultados a los problemas de frontera para ecuaciones de Hamilton-Jacobi
y a la conjectura de Ambrose.

El cut locus es un objeto de interés para muchas disciplinas: geometría
diferencial, teoría de control óptimo, teoría de transporte óptimo,
procesamiento de imágenes, estadística y una herramienta útil en algunas
demostraciones de resultados en otras disciplinas en las que el cut locus en sí
no es un objeto de interés directo.

Durante la primera fase recogimos resultados sobre la estructura del cut locus
provenientes de muchas de estas disciplinas, encontrando resultados
duplicados, y mucho desconocimiento en cada área del trabajo que sobre este
objeto se hacía desde las otras disciplinas.
Cuando aportamos nuestros propios resultados, tuvimos que elegir una notación
que no podía ser compatible con toda la literatura existente.

Nuestros resultados sobre estructura en los capítulos \ref{Chapter: structure} y
\ref{Chapter: balanced} generalizan resultados bien conocidos y demostrados
muchas veces de forma independiente, que describen la estructura del cut locus
excepto por un conjunto de codimensión $2$, lo que es útil para muchas
aplicaciones, pero no para todas, aumentando el conocimiento del cut locus hasta
codimensión $3$.

Estos resultados de estructura son esenciales para nuestras aportaciones a la
teoría de Problemas de Frontera para Ecuaciones de Hamilton-Jacobi en el
capítulo \ref{Chapter: balanced}, donde conectamos la noción de solución de
viscosidad con la solución clásica por características caracterizando el lugar
singular de la primera como un cut locus, o como un \emph{balanced split locus},
noción que identificamos en este trabajo aunque estaba implícito en trabajos
previos.

Creemos que los resultados sobre las singularidades de la aplicación exponencial
del capítulo \ref{Chapter: structure} podrían ser útiles para extender la
demostración de la conjetura de Ambrose que aportamos a todas las métricas
riemannianas.
En~el capítulo \ref{chapter: ambrose}, damos una demostración nueva de la
conjetura de Ambrose que cubre un conjunto genérico de variedades riemannianas,
pero en el capítulo \ref{chapter: beyond}, pergeñamos una estrategia que podría
servir para dar una demostración más general que usa de forma esencial los
resultados de estructura mencionados.

El resto de esta disertación doctoral será en inglés para ser útil a un público
más amplio, esperamos que este hecho no suponga un impedimento al lector
interesado. 
\let\elchapter\thechapter 
\let\laseccion\thesection 
\let\lafigura\thefigure 
\renewcommand{\thechapter}{} 
\renewcommand{\thesection}{\Roman{section}} 
\renewcommand{\thefigure}{\thesection.\arabic{figure}} 

\chapter*{Acknowledgements}

\markboth{\sc ACKNOWLEDGEMENTS}{\sc ACKNOWLEDGEMENTS}

Yanyan Li introduced me to the Ambrose conjecture and sparked my interest in the conjecture during my stay at Rutgers University.
He is a great teacher, mathematician, and human being.

During the preparation of the papers \cite{Nosotros} and \cite{Nosotros2}, I had many helpful conversations with young and senior mathematicians. The list includes Biao Yin, Luc Nguyen, Juan Carlos \'Alvarez Paiva, Ireneo Peral, Yanyan Li and Marco Fontelos.
Finally, the referee of \cite{Nosotros2} was very helpful. Whoever that was, please receive my warmest regards.

I have also talked with many people about the Ambrose conjecture. The list includes Juan Carlos \'Alvarez Paiva, Paolo Piccione, Herman Gluck.

Other resources were more impersonal but equally useful.
Neil Strickland and Ben Wieland answered a question I posted in the algebraic topology list ALGTOP-L.
The Wikipedia helped save a lot of time by providing quick answers to many simple questions.
The site mathoverflow already contained answers to a few questions before we could even word them correctly.
Some anonymous mathematicians scanned, uploaded and shared a big mathematical library. They helped a lot, and they have my respect and my full support.
I hold even more respect for all the mathematicians that released their work directly to the public. The wonderful book of Allen Hatcher was particularly helpful. You even can find one picture from his book in this thesis (with permission of the author).
A special place goes to the Sage community, an open source mathematical software that I've used mainly for teaching, but also to do computations and explore some hypothesis related to this thesis. I'm proud to have been part of that community for several years.

Luis Guijarro was always helpful and respectful as a thesis advisor.
Our friendship has grown during these years and will survive this work.
His patience knows no bounds.

During most of the years I worked on this thesis, I was lucky to work on the Mathematics Department of the Universidad Autónoma de Madrid.
The different universities and research centers at Madrid make it a great place to stay tuned with the latests advances in mathematics, but it is the warmth and fellowship of the people in the department that made those years so pleasant.
In particular, I want to recall my office mates Pedro Caro and Carlos Vinuesa, with whom I had many laughs and interesting mathematical conversations, and Daniel Ortega, who taught me by example that is more satisfying to be useful to the department than to improve the cv. Daniel is also the latex guru of the department, and is responsible for fixing all the badboxes and other stylistic errors in this thesis.

Besides my positions in the department, I was partially supported during the preparation of this work by grants MTM2007-61982 and MTM2008-02686 of the MEC and the MCINN respectively. At the end of this work, I was supported by the Instituto Nacional de Empleo.

\paragraph{Agradecimientos}

Let me switch now to spanish: the rest is more personal.

La casualidad quiso que mi amigo Daniel estuviera en la UAM preparando su tesis cuando entré como profesor ayudante. Dani me abrió muchas puertas, y durante estos años, al igual que siempre, fue una persona dispuesta a escuchar cualquier problema en cualquier área y a proponer soluciones creativas.

María y Clara han sido mis compañeras durante toda mi vida matemática. Con ellas he compartido matemáticas y muchas otras cosas. Creo que me han hecho crecer como persona de un modo tan profundo que listar sus contribuciones a esta tesis en particular sería frívolo. De no haberlas conocido, otra persona habría escrito este trabajo.

Clara y mis dos hijos C\'esar y H\'ector son mi pasión y fuerza vital. Subido a hombros de estos gigantes, afronto el futuro con alegría, incluso con optimismo, aún intuyendo algunos de los momentos difíciles que nos esperan, confiado de que podrán superar todas las dificultades.

Pero quiero dedicar esta tesis a mis padres. Desde niño me he sentido siempre orgulloso de mi madre. Incluso ahora, siento que tengo mucho que aprender de ella para poder ser el padre que mis hijos merecen. Mi padre es ahora el abuelo de mis hijos, y es un orgullo haber contribuido a elevarle a este status. En un hombre con muchas virtudes, su estilo y saber hacer como abuelo se elevan sobre todas las demás. De todo el legado de mis padres, sin duda su impronta sobre sus hijos y sus nietos será la más importante y duradera. Su gran humanidad me acompaña y me da fuerza en los momentos fáciles y difíciles.

\chapter*{Introduction}
\markboth{\sc INTRODUCTION}{\rightmark}

The goal of this thesis is to study the \strong{singularities of the exponential
map} of \emph{Riemannian and Finsler manifolds} (a concept related to
\emph{caustics} and \emph{catastrophes}), and the object known as the
\strong{cut locus} (aka \emph{ridge}, \emph{medial axis} or \emph{skeleton},
with applications to differential geometry, control theory, statistics, image
processing...), to improve existing results about its structure, to look at it
in new ways, and to derive applications to the \strong{Ambrose conjecture} and
the \strong{Hamilton-Jacobi equations}.

\section{Relation between Hamilton-Jacobi equations and Finsler Geometry}

Boundary Value Problems of Hamilton-Jacobi (HJBVP) are intimately relationed to Finsler Geometry. 
In such problems, we look for an unknown function $u:M\rightarrow\RR$ satisfying the following equations:
\begin{eqnarray*}
  H(p,du(p))\;=&1\quad&p\in M\notag\\
  u(p)\;=&g(p)&p\in \partial M\notag
\end{eqnarray*}
where the first equation is a \emph{non-linear first order partial differential equation} and the second equation prescribes the \emph{boundary values} for $u$.

We ask for the following conditions:
\begin{itemize}
 \item $M$ is a smooth compact manifold of dimension $n$ with boundary
 \item $H:T^\ast M\rightarrow RR$ is a smooth function defined on the \emph{cotangent space} to $M$, $H^{-1}(1)\cap T^{\ast}_{p} M$ strictly convex for every $p$
 \item $g:\partial M\rightarrow \RR$ smooth
\end{itemize}

Furthermore, the \emph{boundary data} $g$ and the \emph{equation coefficients} $H$ must satisfy a \strong{compatibility condition}:

\begin{equation*}
 \vert g(y) - g(z)\vert < d(y,z)\qquad \forall y,z \in\partial M
\end{equation*}

where $d$ is the distance on $M$ induced by the following \emph{Finsler metric}:

\begin{equation*}
 \varphi_{p}(v)=\sup\left\lbrace 
\left\langle v,\alpha\right\rangle_{p}\, :\,
\alpha\in T^{\ast}_{p} M, \,H(p,\alpha)=1
\right\rbrace 
\end{equation*} 

The above definition gives a norm in every tangent space $T_p M$.
Indeed, $H$ can be redefined so that $H$ is positively homogeneous of order $1$: $H(p,\lambda \alpha)=\lambda H(p, \alpha)$ for $\lambda>0 $, and the HJBVP is the same.
Then, $H$ is a norm at every cotangent space and $\varphi$ is the dual norm in the tangent space.

A \strong{classical solution} to these equations has been known for a long time, and it admits a geometrical interpretation in terms of Finsler geometry.

First, using the definition of dual form in Finsler geometry (see 
\ref{dual one form}), we define the \emph{characteristic vector field} at points
$p\in \partial M $:
\begin{eqnarray*}
 &\varphi_p(X_p)=1\notag\\
 &\widehat{X_p}\vert_{T(\partial M)}=d g\notag\\
 &X_p\text{ points inwards}
\end{eqnarray*}

The (projected) characteristic curves are the geodesics $M$ with initial point $z\in \partial  M$ and initial speed given by the characteristic vector field.

A local smooth solution $u$ to the HJBVP can be computed near $\partial  M$ following characteristic curves:

\begin{dfn*}
Let $U$ be a neighborhood of $\partial M$ such that every point $q\in U$ belongs to a unique (projected) characteristic contained in $U$ and starting at a point $p\in \partial M$ (the point $p$ is often called the \emph{footpoint} of $q$).

The \strong{solution by characteristics} $u:U\rightarrow \RR$ is defined as follows: if $\gamma:[0,t]\rightarrow  M$ is the unique (projected) characteristic from a point $p\in\partial M $ to $q=\gamma(t)$ that does not intersect $Sing$, then
$$u(q)=g(p)+t$$
\end{dfn*}

In this way, the classical solution can be defined in a neighborhood of $\partial M$, but not in all of $M$.

A different notion of solution appeared later (see \cite{Lions}).
The solution (in the \emph{viscosity} sense) to the above HJBVP is given by the Lax-Oleinik formula:
\begin{equation*}
u(p)=\inf_{q\in \partial M}
\left\lbrace 
   d(p,q)+ g(q)
\right\rbrace
\end{equation*} 
where $d$ is again the distance function induced from the Finsler metric. We defer the definition to \ref{viscosity solution of HJBVP}, because the actual definition of the viscosity solution plays no role in this thesis.
All we need to know is that the viscosity solution is given by the above formula.

Thus, when $g=0$, the solution to the equations is the distance to the boundary.

In theorem \ref{reduce HJ to boundary value 0}, we prove that when $g\neq 0$, the viscosity solution is also a distance function, but to the boundary of a larger manifold $\tilde{M}\supset M $.

\section{Singularities of the exponential map}

The \emph{exponential map} from a point or submanifold in a Finsler manifold is defined in the same way as that of Riemannian manifolds and has similar properties.

Let $M$ be a smooth Finsler manifold, $p$ a point in $M$ and $v\in T_p M$ a tangent vector to $M$ at $p$.
Then the exponential of $v$ is the point $\exp_p(v)=\gamma(1)\in M$, for the unique geodesic $\gamma$ that starts at $p$ and has initial speed vector $v$.
The exponential map from $p$ is a diffeomorphism from a small ball near the origin, but it can develop singularities as we move far away from the origin.

The exponential map from a submanifold $L\subset M$ is defined for (some) vectors of the normal bundle of $L$ in $M$: let $p$ be a point in $L$ and $v\in T_p M$ a vector orthogonal to the subspace $T_p L$, the exponential of $v$ is the point $\exp_L(v)=\gamma(1)\in M$, for the unique geodesic $\gamma$ that starts at $p$ and has initial speed vector $v$.
If the submanifold is the boundary $\partial M$ of a closed manifold $M$, the definition is the same, but the exponential is only defined in the \emph{inner normal bundle}.

This time, the exponential map is a diffeomorphism from a tubular neighborhood of the zero section of the normal bundle of $L$ into a tubular neighborhood of $L$ in $M$ and again, it can develop singularities and self-intersections if we consider larger vectors.

The singularities, however, are not those of an arbitrary smooth map between $n$-dimensional manifolds.
Let us restrict for a moment to the exponential map from a point $p$ in a manifold $M$ without boundary.
For a point $x\in T_p M$ where $\exp_p$ is singular, the \emph{order of conjugacy} of $x$ is the corank of the linear map $d_x\exp_p $.
Along a radial line in $T_p M$, the singularities cannot cluster: if we add the orders of all the singularities along a radial line, in a small neighborhood of a point $x\in T_p M$ of order $k$, the number is always $k$.

It makes sense, thus, to talk about \emph{the $k$-th conjugate point in the direction $x$}, for a point $x$ in the unit ball in $T_p M$.
This is the point $t_0\cdot x $, for $t_0>0$, such that $t_0\cdot x$ is a conjugate point of order $j$ and such that the sum of the orders of all the singularities along the radial segment $t\rightarrow t\cdot x $, for $t<t_0$ is an integer between $k-j$ and $k-1$.

This allows us to define the function $\lambda_k$ that maps $x\in B_1(T_p M) $ to the parameter $t=\lambda_k(x)$ such that $t\cdot x $ is the the $k$-th conjugate point in the directon $x$.
It follows from the above that $\lambda_k$ is continuous.

In \cite{Itoh Tanaka 98}, J. I. Itoh and M. Tanaka proved that for Riemannian manifolds, all the functions $\lambda_k$ are locally Lipschitz continuous.
In theorem \ref{landa es Lipschitz}, we prove that all $\lambda_k$ are locally Lipschitz continuous in Finsler manifolds.
M. Castelpietra and L. Rifford were working on this result simultaneously, and shortly after our proof appeared, they gave a proof that $\lambda_1$ is locally semiconcave, a stronger property that Lipschitz. This property does not hold for $\lambda_k$, for $k>1$.
The three proofs are different.

\section{Structure of the cut locus}\label{section: introduction about structure of the cut locus}

Let $M$ be a Riemannian or Finsler manifold, and let $S$ be a smooth submanifold of any dimension.
$S$ can also be a point, and we also consider $S=\partial M$. It is easy to see that this latter case contains the others, substracting a tubular neighborhood of the submanifold $S$.
The cut locus $Cut_{S}$ of $S$ in $M$ can be defined in several equivalent ways:

\begin{figure}[ht]
 \centering
 \begin{tabular}{cc}
 \includegraphics[width=0.4\textwidth]{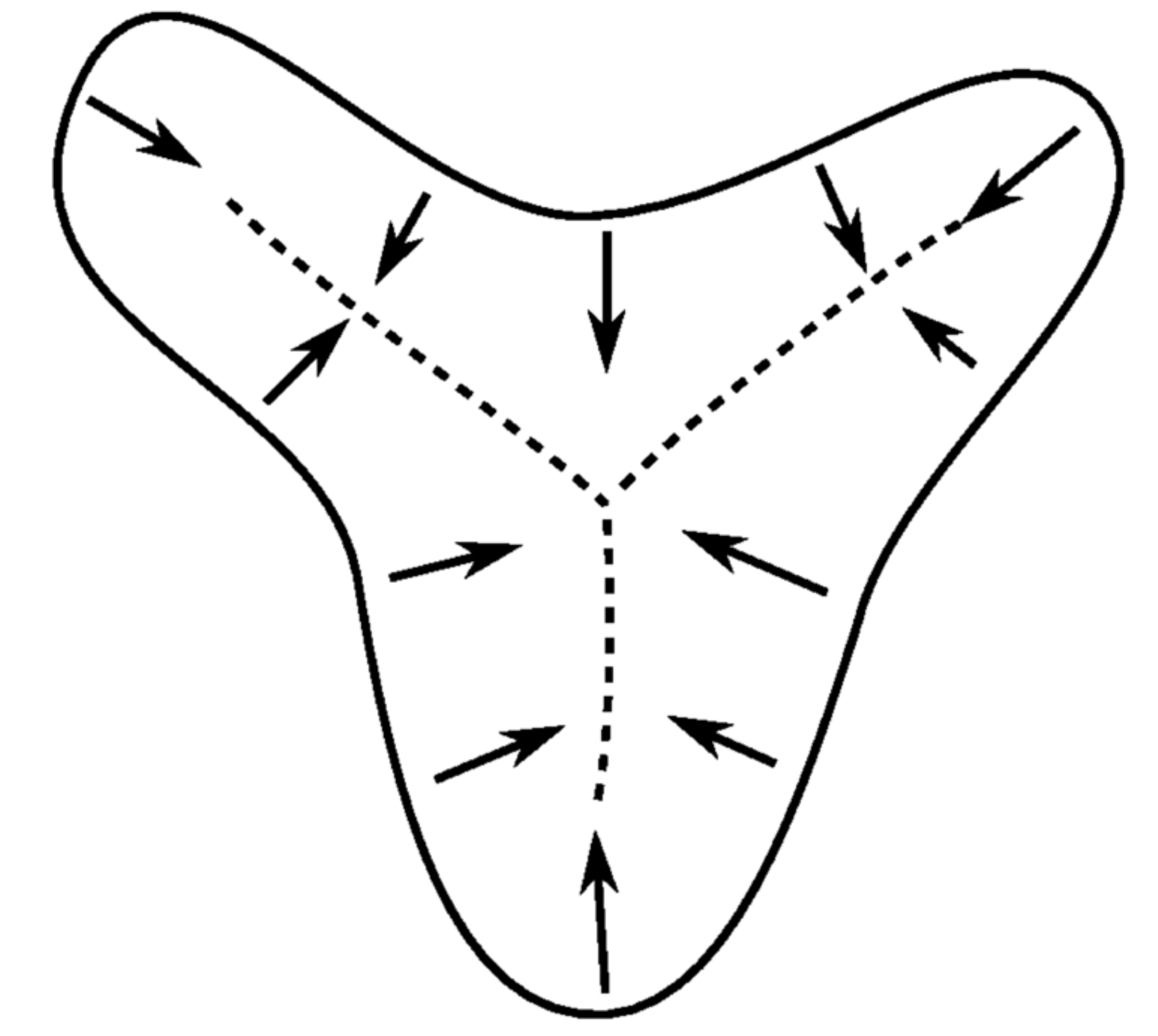}
 \includegraphics[width=0.4\textwidth]{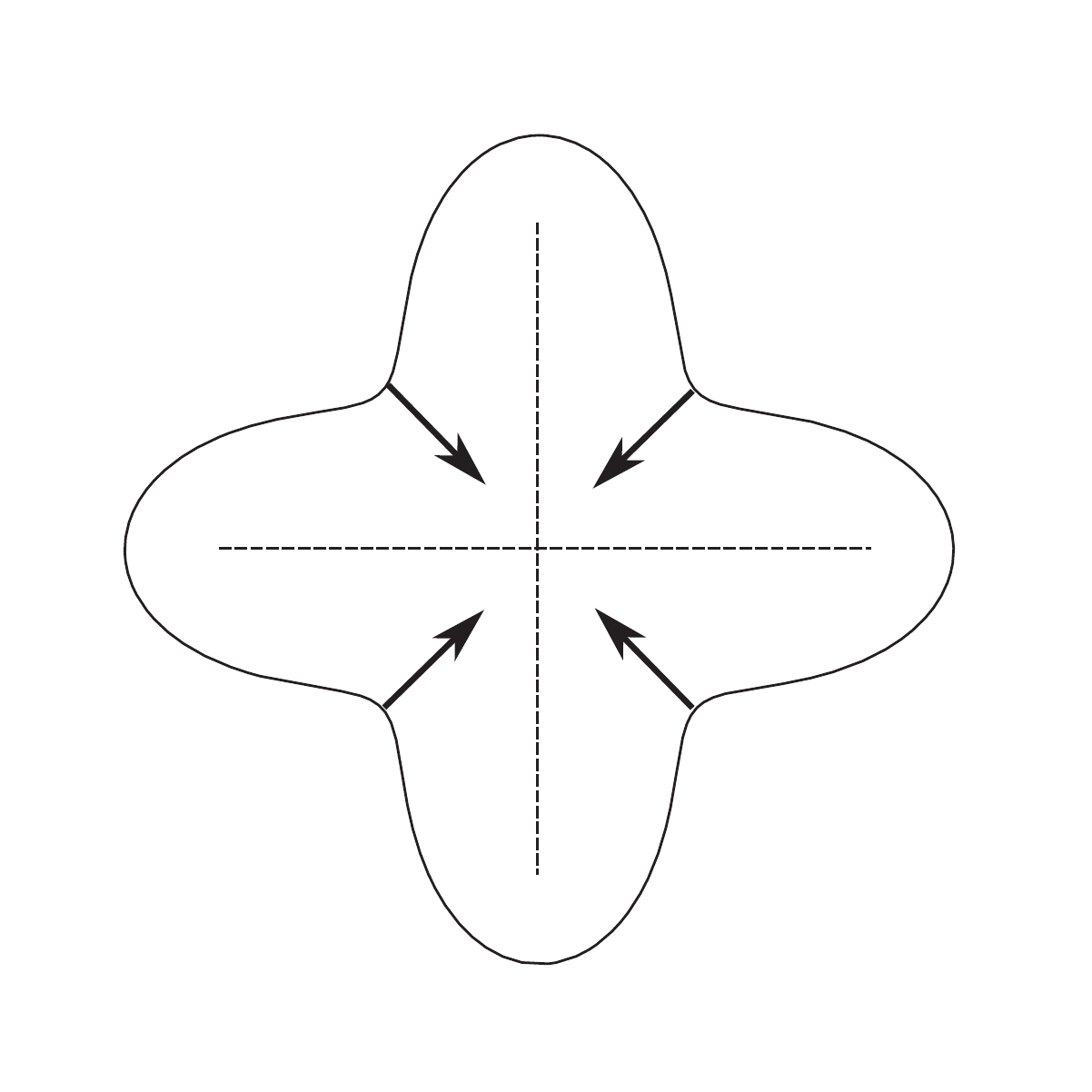}  
 \end{tabular} 
 \caption{Two open sets in $\RR^2$ with their respective cut loci (with respect to the boundary). The second set has analytic boundary, and its cut locus is a subanalytic set.}
 \label{figure: good cut loci}
\end{figure}

\begin{itemize}
 \item Every unit speed geodesic $\gamma$ with initial point in $S$ and initial speed orthogonal to $S$ will minimize the distance between $\gamma(0)$ and $\gamma(t)$, for small $t$. Define $t_{cut}=\sup\{t:d(\gamma(0),\gamma(t))=t \}$. The cut locus is the set of all points $\gamma(t_{cut}(\gamma))$ for all geodesics $\gamma$ starting at $S$ with initial speed orthogonal to $S$.
 \item For a point $p\in M$, let $Q_p$ be the set of points $q\in S$ such that $d(p,q)=d(p,S)$. Then $Cut_{S}$ is the closure of the set of points such that $Q_p$ has more than one element.
 \item The distance function to $S$ is singular exactly when $Q_p$ has more than one element, so we can also define $Cut_{S}$ as the closure of the set where the distance function to $S$ is singular.
 \item $Cut_{S}$ is also the set of points such that either $Q_p$ has more than one element, or $Q_p=\{q\}$, and $q$ is conjugate to $p$ along one geodesic that minimizes the distance between them.
\end{itemize}

These sets have been studied by mathematicians from many different fields:

  \begin{itemize}
    \item $Cut_{\partial M}$ is a deformation retract of $ M$ (and $Cut_S$ is a deformation retract of $ M\setminus S$ ) (obvious).

    \item $Cut_S$ is the union of a ($n - 1$)-dimensional smooth manifold
    consisting of points with two minimizing non-conjugate geodesics and a set of Hausdorff
    dimension at most $n - 2$ (Hebda83,
    Itoh-Tanaka98, Barden-Le97,
    Mantegazza-Menucci03 for the riemannian case).

    \item $Cut_S$ is stratified by the dimension of the
    subdifferential of the function \emph{distance to $S$}: $\partial d_S$ (
    Alberti-Ambrosio-Cannarsa-Etcetera92-94).

    \item It has finite Hausdorff measure $\mathcal{H}^{n - 1}$ (Itoh-Tanaka00 for the riemannian case, Li-Nirenberg05 and Castelpietra-Rifford10 for Finsler manifolds).

    \item If all the data is analytic, $Cut_S$ is a stratified 	\emph{subanalytic set} ( Buchner77).

    \item If we add a generic perturbation to $H$ or $ M$, $Cut_S$ becomes a stratified smooth manifold (Buchner78).
  \end{itemize}

Despite these facts, cut loci can be non-triangulable, even for surfaces of revolution (see \cite{Gluck Singer}).
Also, their combinatorial topology can be complicated: even though $Cut_S$ is
homotopic to $ M\setminus S$, there are metrics in the $3$-dimensional sphere
whose cut locus is a simplicial complex equivalent to the house with two rooms
(see figure \ref{figure: house_with_two_rooms}).

\begin{figure}[ht]
 \centering
 \includegraphics[width=0.8\textwidth]{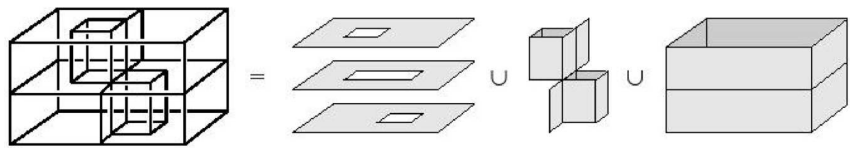}  
 \caption{The house with two rooms}
 \label{figure: house_with_two_rooms}
\end{figure}

We have improved the previous knowledge about \emph{cut loci} and the \emph{singular set} of solutions to static Hamilton-Jacobi equations in the following ways:

\begin{figure}[ht]
 \centering
 \includegraphics[width=0.8\textwidth]{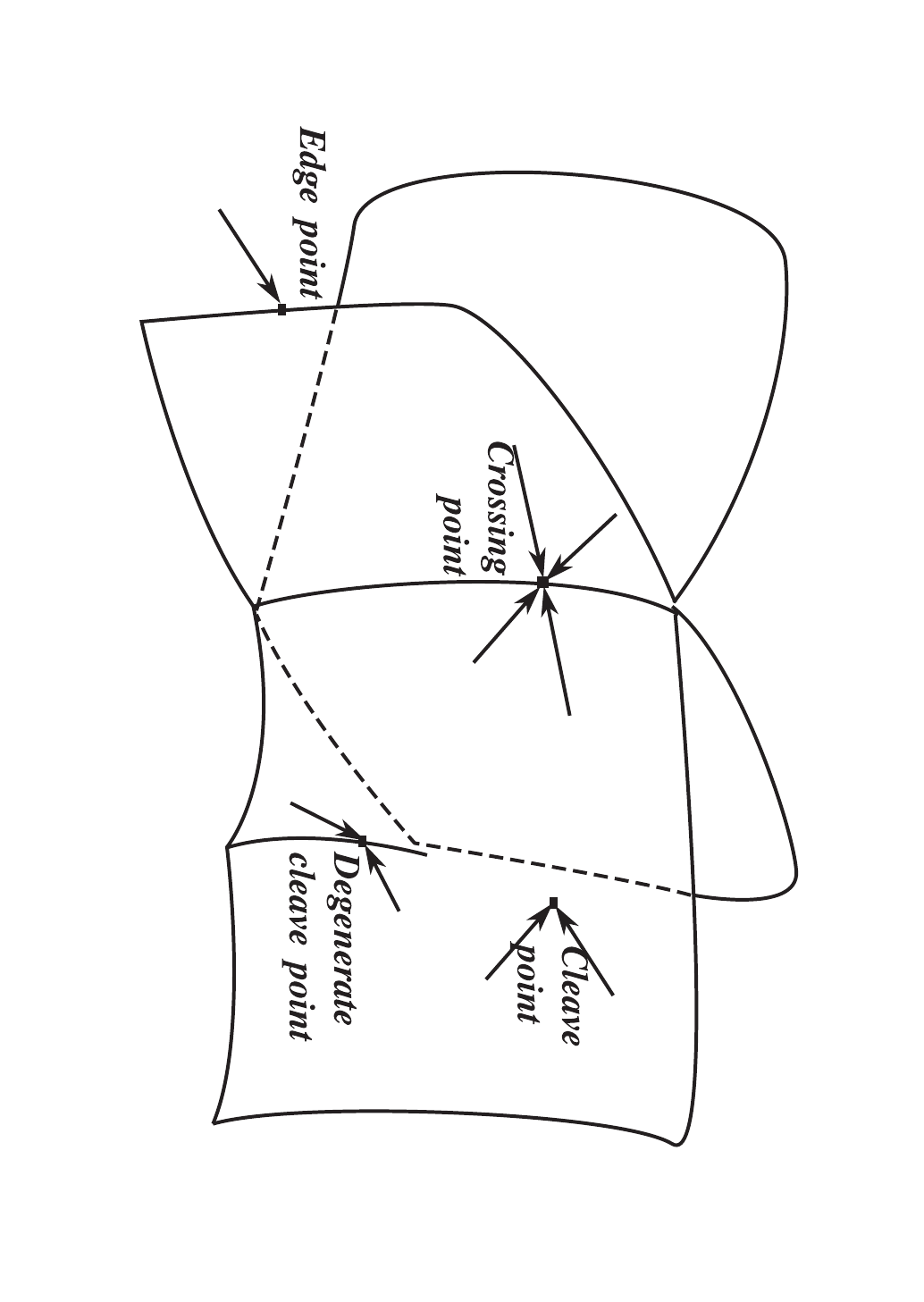}
 
 \caption{Different types of points in the cut locus, except for a set of small Haussdorff dimension}
 \label{figure: complete description of cut locus}
\end{figure}

\begin{itemize}
 \item In theorem \ref{regularity of mu for nontrivial g}, we show that  the \emph{singular set} of a solution to a Hamilton-Jacobi BVP is the cut locus of a Finsler manifold.
 This was only known when the boundary data was identically zero.
 \item In the same chapter, in subsection \ref{section: Balanced}, we also show that cut loci are what we call \emph{balanced split loci}. In general, there are many balanced split loci besides the cut locus, but there is only one such set on a simply connected manifold with connected boundary.
 \item In theorem \ref{complete description}, we prove a local structure theorem for balanced split loci. The theorem states that the cut locus consists only of \emph{cleave points}, \emph{edge points}, \emph{degenerate cleave points}, \emph{crossing points}, and a remainder, of Hausdorff codimension at least 3 (see figure \ref{figure: complete description of cut locus}).
 \item In chapter \ref{Chapter: balanced}, we provide more detailed descriptions of a balanced split loci near the different types of points listed above (see \ref{uniqueness near order 1 points}, \ref{structure of cleave points}, \ref{structure of crossing points1}, \ref{structure of crossing points2} and \ref{structure of crossing points3}).
\end{itemize}

We believe that our description of the cut locus can be useful in other contexts. For instance, brownian motion on manifolds is often studied on the complement of the cut locus from a point, and then the results have to be adapted to take care of the situation when the brownian motion hits the cut locus. As brownian motion almost never hits a set with null $\Hndos$ measure (but will almost surely hit any set with positive $\Hndos$ measure), we think our result can be useful in that field.

\section{Characterization of the cut locus and the balanced split loci}

As we mentioned in the previous section all cut loci, and thus the singular set of HJBVP, are \emph{balanced split loci}. In chapter \ref{Chapter: balanced}, we study and classify all possible balanced split loci. The following list is a summary of theorems \ref{maintheorem0},  \ref{maintheorem1} and  \ref{maintheorem2}:

\noindent\begin{tabular}{l@{\,\,}l@{\,\,}l}
  $M$ is simply connected & $\rightarrow$ & The singular set is the
  unique\\
  {\emph{and}} $\partial M$ connected &  & balanced split locus\\
  &  & \\
  $M$ is simply connected, & $\rightarrow$ & We can add a different
  constant \\
  $\partial M$ is {\emph{{\strong{{\emph{not}}}}}} connected &  & to
  $g$ at each component of $\partial M$ and get\\
  &  & different balanced split loci (see fig \ref{figure: cut locus of two
spheres})\\
  &  & \\
  General case & $\rightarrow$ & Balanced split loci are parametrized by a\\
  &  & neighborhood of $0$ in $H_{n - 1} (M, \mathbb{R})$ (see fig \ref{figure:
cut locus of torus})
\end{tabular}

\begin{figure}[ht]
 \centering
 \includegraphics[width=0.6\textwidth]{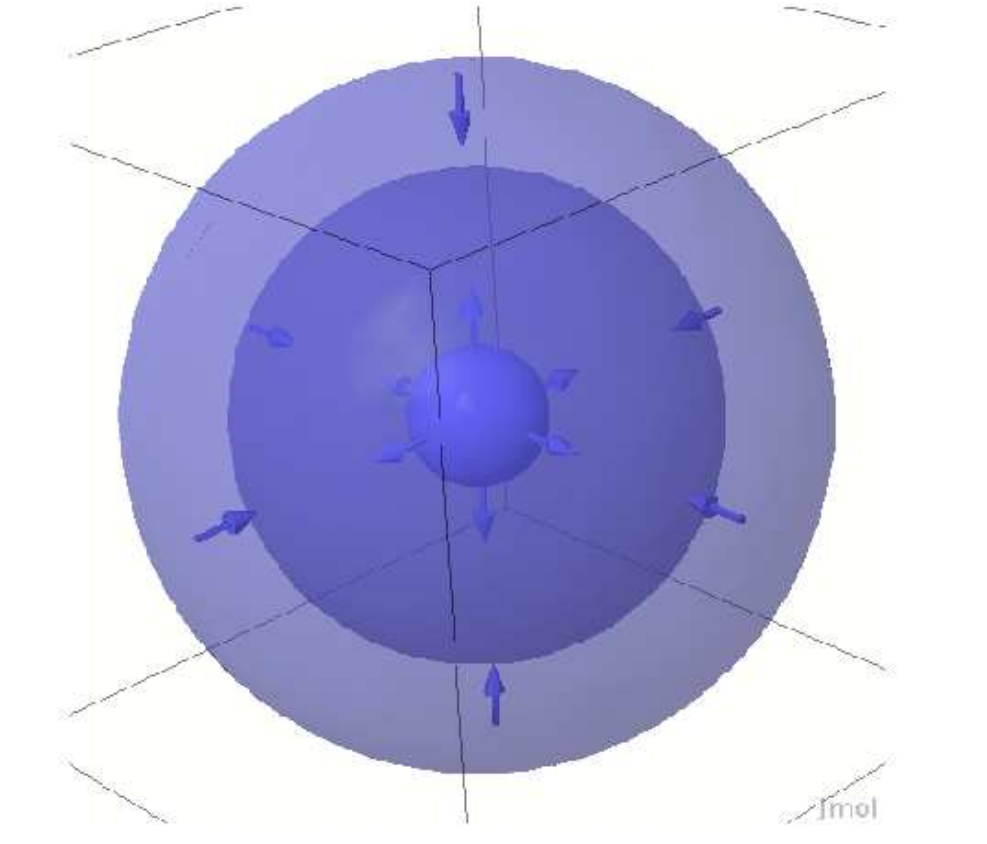}
 \caption{A balanced split locus when $\partial M$ consists of two concentric spheres}
 \label{figure: cut locus of two spheres}
\end{figure}

\begin{figure*}[ht]
 \centering
 \includegraphics[bb= 0 0 175 173,width=0.4\textwidth]{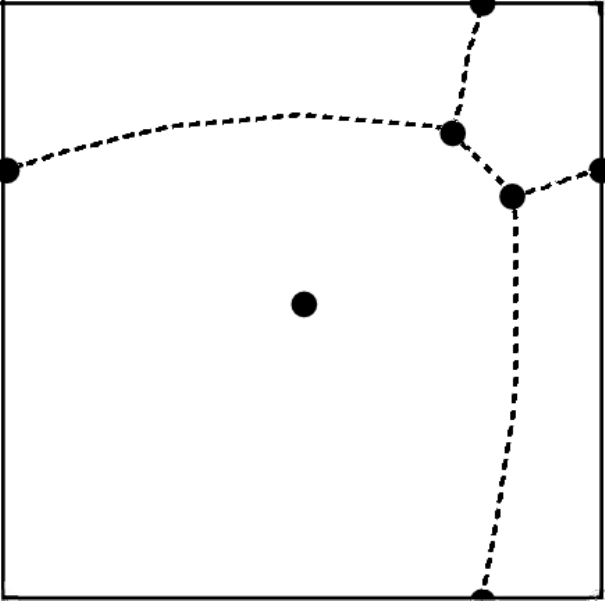}
 \caption{A balanced split locus of a non-simply-connected manifold (a torus with a disc removed)}
 \label{figure: cut locus of torus}
\end{figure*}

\section{The Ambrose conjecture}

Let ($M_{1} ,g_{1}$) and ($M_{2} ,g_{2}$) be two Riemannian manifolds
{\emph{of the same dimension}}, with selected points $p_{1} \in M_{1}$ and
$p_{2} \in M_{2}$. We'll speak about the pointed manifolds ($M_{1} ,p_{1}$)
and ($M_{2} ,p_{2}$). Any linear map $L:T_{p_{1}} M_{1} \rightarrow T_{p_{2}}
M_{2}$ induces the map $\varphi = \exp_{2} \circ L \circ ( \exp_{1} |_{O_{1}}
)^{-1}$, defined in any domain $O_{1} \subset T_{p_{1}} M_{1}$ such that
$e_1|_{O_{1}}$ is injective (for example, if $O_{1}$ is a normal neighborhood of
$p_{1}$).

\begin{figure}[H]
 \centering
 \includegraphics[width=0.8\textwidth]{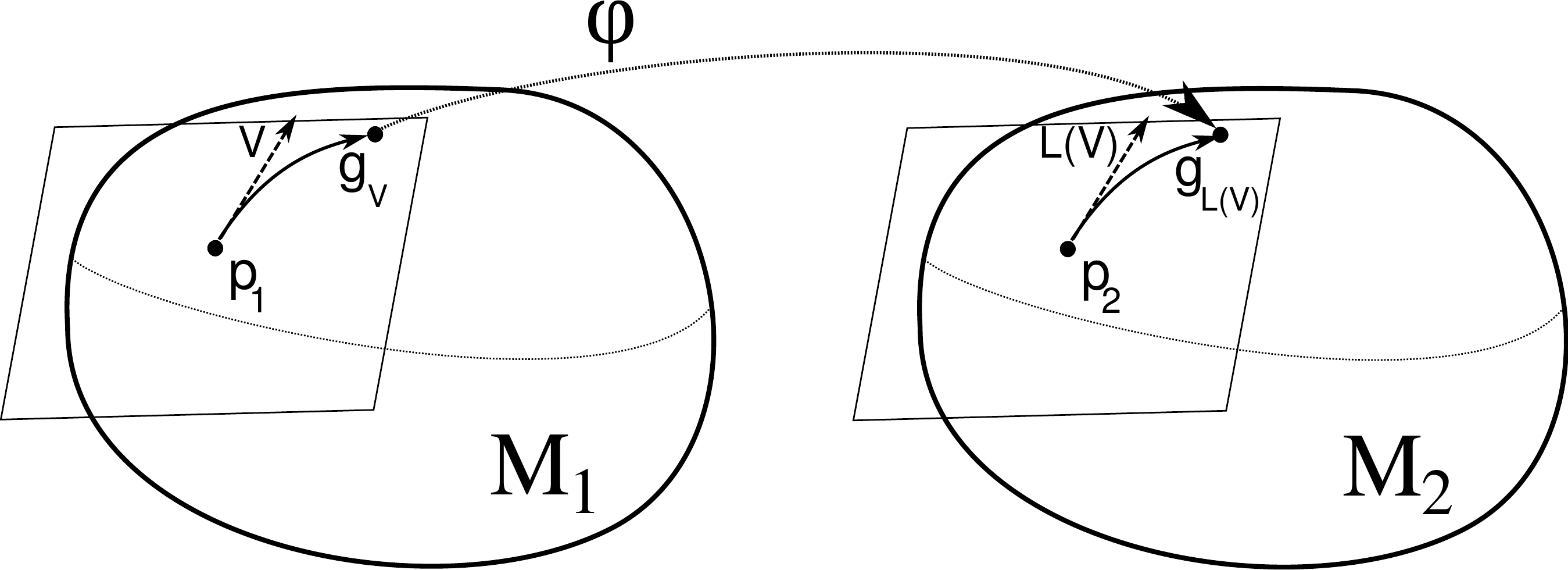}
 
 \caption{The map $\varphi$ induced by the linear map $L$}
 \label{figure: statement of ambrose}
\end{figure}

This idea was introduced by E. Cartan {\cite{Cartan51}}, who proved that under some (strong) hypothesis on the curvature of $M_1$ and $M_2$, this map is a local or even a global isometry. We prefer to rephrase it in the following terms:

\begin{dfn*}
  Let $(M_{1}, p_{1} )$ and $(M_{2}, p_{2} )$ be complete Riemannian manifolds
  of the same dimension with base points, and $L:T_{p_{1}} M_{1} \rightarrow
  T_{p_{2}} M_{2}$ a linear isometry.
  
  Let $\gamma_{1}$ and $\gamma_{2}$ be the geodesics defined in the interval
  $[0,1]$, with $\gamma_{1}$ starting at $p_{1}$ with initial speed vector $x
  \in T_{p_{1}} M_{1}$ and $\gamma_{2}$ starting at $p_{2}$ with initial speed
  $L(x)$.
  
  For any three vectors $v_{1} ,v_{2} , v_{3}$ in $T_{p_{1}} M_{1}$, define:
  \begin{itemizedot}
    \item $R_{1} (x, v_{1} ,v_{2} , v_{3} )$ is the vector of $\Tone$ obtained by
    performing parallel transport of $v_{1} ,v_{2} , v_{3}$ along
    $\gamma_{1}$, computing the Riemann curvature tensor at the point
    $\gamma_{1} (1) \in M_{1}$ acting on those vectors, and then performing
    parallel transport backwards up to the point $p_{1}$.
    
    \item $R_{2} (x, v_{1} ,v_{2} , v_{3} )$ is the vector of $\Tone$ obtained by
    performing parallel transport of $L(v_{1} ),L(v_{2} ), L(v_{3} )$ along
    $\gamma_{2}$, computing the Riemann curvature tensor at the point
    $\gamma_{2} (1) \in M_{2}$ acting on those vectors, then performing
    parallel transport backwards up to the point $p_{2}$, and finally applying
    $L^{-1}$ to get a vector in $T_{p_{1}} M_{1}$.
  \end{itemizedot}
  If $R_{1} (x, v_{1} ,v_{2} , v_{3} )=R_{2} (x, v_{1} ,v_{2} , v_{3} )      \forall
  v_{1} ,v_{2} , v_{3} \in T_{p_{1}} M_{1}$ for any two geodesics $\gamma_{1}$
  and $\gamma_{2}$ as above, we say that \emph{the curvature tensors of $(M_{1}
  ,p_{1} )$ and $(M_{2} ,p_{2} )$ are \strong{$L$-related}}.
\end{dfn*}

The usual way to express that the curvature tensors of $(M_{1}, p_{1} )$ and $(M_{2}, p_{2} )$ are $L$-related is to say that the {\emph{parallel traslation of curvature along radial geodesics}} of $(M_{1}, p_{1} )$ and $(M_{2}, p_{2} )$ coincides.

We say $(M_{1} ,p_{1} )$ and $(M_{2} ,p_{2} )$ are
\strong{$L$-related} iff they have the same dimension and, whenever $\exp_{1}
|_{O_{1}}$ is injective for some domain $O_{1} \subset T_{p_{1}} M_{1}$,
then the map $\varphi = \exp_{2} \circ L \circ ( \exp_{1} |_{O_{1}} )^{-1}$
is an isometric inmersion.

Cartan's theorem states that if the curvature tensors of $(\!M_{1},p_{1} )$
and $(M_{2} ,p_{2} )$ are $L$-related, then $(\!M_{1} ,p_{1}\!)$ and $(\!M_{2}
,p_{2}\!)$ are $L$-related. \mbox{(see lemma 1.35 of~\cite{Cheeger Ebin}).}

In 1956 (see \cite{Ambrose}), W. Ambrose proved a global version of the above
theorem, but with stronger hypothesis.
A \emph{broken geodesic} is the concatenation of a finite amount of geodesic segments.

The theorem of Ambrose states that if the parallel traslation of curvature along \emph{broken geodesics} on $M_{1}$ and $M_{2}$ coincide,
then there is a global isometry $\varphi:M_1\rightarrow M_2$ whose differential at $p_1$ is $L$.
It is simple to prove that $\varphi$ can be constructed as above.
It is enough if the hypothesis holds for broken geodesics with only one ``{\emph{elbow}}'' (the reader can find more details in \cite{Cheeger Ebin}).

However, he conjectured that the same hypothesis of the Cartan's lemma should suffice, except for the obvious counterexamples of covering spaces:

The {\strong{Ambrose conjecture}} states that if the curvature tensors of $(M_{1} ,p_{1} )$ and $(M_{2} ,p_{2} )$ are $L$-related, and $M_{1}$ and $M_{2}$ are \emph{simply connected}, there is a \strong{global isometry} $\psi :M_{1} \rightarrow M_{2}$ such that $\psi \circ \exp_{1} = \exp_{2} \circ L$.

Ambrose himself was able to prove the conjecture if all the data is analytic.
In \cite{Hicks}, in 1959, the conjecture was generalized to parallel transport for affine connections, and in \cite{Blumenthal Hebda}, in 1987, to Cartan connections.
Also in 1987, in the paper \cite{Hebda}, James Hebda proved that the conjecture was true for surfaces that satisfy a certain regularity hypothesis, that he was able to prove true in 1994 in \cite{Hebda94}. J.I. Itoh also proved the regularity hypothesis independently in \cite{Itoh96}.
The latest advance came in 2010, after we had started our research on the Ambrose conjecture, when James Hebda proved in \cite{Hebda10} that the conjecture holds if $M_1$ is a \emph{heterogeneous manifold}. Such manifolds are generic.

In \ref{main theorem ambrose} we provide a new proof that works for surfaces and for a generic class of manifolds in dimension $3$. James Hebda's proof in \cite{Hebda10} is shorter than ours and works for any dimension.
However, his proof does not extend to arbitrary metrics and we think that our proof might, even though we have been unable to complete all the details to this day. Indeed, the proof presented here extends to some manifolds that are not covered by the result of J. Hebda, as this is truly a different approach. 
In chapter \ref{chapter: beyond} we also provide some hints on how our proof might extend to a $3$-dimensional manifold with an arbitrary metric.

We remark that in both our proof and James Hebda's, it is only used that $(M_{1} ,p_{1} )$ and $(M_{2} ,p_{2} )$ are $L$-related, and there is no need to use the original hypothesis that the curvarture tensors are $L$-related.

\section{Summary of results}

\addtolength{\leftmargini}{-12pt} 
\begin{itemize}
 \item 
 Reduction of a HJBVP with $g\neq 0$ to a HJBVP with
$g=0$: theorem \ref{reduce HJ to boundary value 0}, published in
\cite{Nosotros}.
 \item Proof that all $\lambda_k$ are locally Lipschitz continuous in Finsler manifolds: theorem \ref{landa es Lipschitz}, published in \cite{Nosotros}.
 \item Proof that  the \emph{singular set} of a solution to a Hamilton-Jacobi BVP is the cut locus of a different manifold: theorem \ref{regularity of mu for nontrivial g}, published in \cite{Nosotros}.
 \item Proof that cut loci in Finsler manifolds are \emph{balanced split loci}: section \ref{section: Balanced}, published in \cite{Nosotros}.
 \item Local structure theorem for balanced split loci: theorem \ref{complete description}, published in \cite{Nosotros}.
 \item Detailed descriptions of a balanced split loci near the different types of points, except for a set of Hausdorff codimension 3:  \ref{uniqueness near order 1 points}, \ref{structure of cleave points}, \ref{structure of crossing points1}, \ref{structure of crossing points2} and \ref{structure of crossing points3}, published in \cite{Nosotros2}.
 \item Characterization of the cut locus and the balanced split loci: theorems \ref{maintheorem0},  \ref{maintheorem1} and  \ref{maintheorem2}, published in \cite{Nosotros2}.
 \item Proof of the Ambrose conjecture for a metric in $\mathcal{G}_{M}$: theorem \ref{main theorem ambrose}, publication pending.

\item A whole selection of fresh conjectures for the new generations in
chapter~\ref{chapter: beyond}...
\end{itemize}
\addtolength{\leftmargini}{12pt}

\mainmatter

\def\thechapter{\elchapter} 
\def\thesection{\laseccion} 
\def\thefigure{\lafigura} 
\chapter{Preliminaries}\label{Chapter: preliminaries}

\subsection*{Notation} 
We fix the following notation for the rest of the thesis:

\begin{itemize}
 \item A $C^{\infty}$ manifold $M$ with boundary $\partial M$.
 \item A Finsler metric $\varphi$ on $M$.
 \item The distance $d$ induced on $M$ by $\varphi$, and the ``distance to the boundary'' $d_{\partial M}(p)=\inf_{q\in\partial M} d(q,p)$.
 \item The geodesic vector field $\rho$ in $TM$.
 \item The time-$t$ flow of $\rho$: $\Phi_t:TM\rightarrow TM$. For a non-complete manifold, such as a manifold with boundary, these maps are not defined in all of $TM$.
 \item A smooth map $\Gamma:\partial M \rightarrow TM$  that is a section of the projection map $\pi:TM\rightarrow M$ of the tangent to $M$, and such that $\Gamma(p)$ points to the inside of $M$ for every $p\in \partial M$.
\end{itemize}

\section{A little background}
\subsection{Approximate Tangent Cone}
\begin{dfn}\label{vector de x a y}
For a pair of points $p, q\in  M $ such that $q$ belongs to a convex neighborhood of $p$, we define, following \cite{Itoh Tanaka 00},
\begin{equation}
 v_{p}(q)=\dot{\gamma}(0)
\end{equation}
as the speed at $0$ of the unique unit speed minimizing geodesic $\gamma$ from $p$ to $q$.
\end{dfn}

\begin{dfn}\label{approximate tangent cone}
The \emph{approximate tangent cone} to a subset $E\subset M$ at $p$ is:
$$T(E,p)=\left\lbrace
rv:\;v=\lim \frac{v_{p}(p_{n})}{|v_{p}(p_{n})|},
\exists \lbrace p_{n}\rbrace\subset E, p_{n}\rightarrow p, r>0
\right\rbrace$$
and the \emph{approximate tangent space} $Tan(E,p)$ to $E$ at $p$ is the vector space generated by $T(E,p)$.
\end{dfn}

We remark that the definition is independent of the Finsler metric, despite its apparent dependence on the vectors $v_{p}(p_{n}) $.

\subsection{Subdifferentials of semiconcave functions}\label{subsection: semiconcave functions}
Concave (or convex) functions $u$ may not be differentiable, but they are differentiable almost everywhere.
This allows for a simple definition of the subdifferential of a convex or semiconvex function (see \cite{Cannarsa Sinestrari} for different definitions):

\begin{dfn}
The \emph{subdifferential} $\partial u(p)$ of a concave function $u$ at $p$ can be defined as the convex hull of all the one forms that are limits of differentials of $u$ at points $p_n$ where $u$ is differentiable.
\end{dfn}

\begin{dfn}
A function $u:S\rightarrow \RR$ is \strong{semiconcave} if there exists a
nondecreasing upper semicontinuous function $\omega: \RR_+ \rightarrow \RR+$
such that $\lim\limits_{\rho\rightarrow 0} \omega(\rho)=0$ and, for any $x,y\in
\RR$ and $\lambda\in[0,1]$:
$$
\lambda u(x) + (1-\lambda)u(y) - u(\lambda u + (1-\lambda) y) \leq \lambda(1-\lambda)|x-y|\omega(|x-y|)
$$

The function $\omega$ is called the \emph{modulus of semiconcavity}.
\end{dfn}

The concave functions are those for which $\omega$ is zero.
The functions with a \emph{linear modulus} of semiconcavity are exactly those that can be written as the sum of a concave and a smooth function.

It turns out that \emph{viscosity} solutions to HJBVP (to be defined later), and distance functions to the boundary in Finsler geometry, are semiconcave functions.
Those functions share many of the regularity results of concave functions. For example, they are differentiable almost everywhere. Indeed, this statement can be refined: let $u:S\rightarrow \RR$ be a semiconcave function and define the sets
$$\Sigma_k=\{x\in S: \dim(\partial u(p))\geq k\}$$
Then $\Sigma_k$ is countably $\mathcal{H}^{n-k}$ rectifiable: it is contained in the union of countably many $C^1$ hypersurfaces of dimension $n-k$, plus a $\mathcal{H}^{n-k}$-negligible set.

Check \cite{Cannarsa Sinestrari} and \cite{Alberti Ambrosio Cannarsa} for more details.

\subsection{Duality in Finsler Geometry}

\begin{dfn}\label{Finsler orthogonal hyperplane to a vector}
The \strong{orthogonal hyperplane} to a vector $v\in T_{p}M$ is the hyperplane tangent at $v$ to the level set
$$
\left\lbrace v' \in T_{p}M: \varphi(v') = \varphi(v) \right\rbrace
$$

The orthogonal distribution to a vector field is defined pointwise.
\end{dfn}

\begin{remark}
There are two unit vectors with a given hyperplane as orthogonal hyperplane. The first need not to be the opposite of the second unless $H$ is symmetric ($H(-v)=H(v)$).
We thus define two unit normal vectors to a hypersurface (the \emph{inner} normal and \emph{outer} normal).
\end{remark}

\begin{dfn}\label{dual one form}
The \strong{dual one form} to a vector $v\in T_{p}M$ with respect to a Finsler metric $\varphi $ is the unique one form $\omega \in T^{\ast}_{p}M$ such that $\omega(v)=\varphi(v)^{2}$ and $\omega\vert_{H}=0$, where $H$ is the orthogonal hyperplane to $v$.

For a vector field $X$, the dual differential one-form is obtained by applying the above construction at every point. We will often use the notation $\widehat{X}$ for the dual one-form to the vector field $X$.
\end{dfn}

\begin{remark}
In Riemannian geometry, a different scaling is often used: $\omega(v)=1$ instead of $\omega(v)=\varphi(v)^{2}$. We have chosen this definition because it makes the duality map $v\rightarrow \omega $ continuous.
\end{remark}

\begin{remark}
In coordinates, the dual one form $w$ to the vector $v$ is given by:

$$
w_{j}=\varphi(v)\frac{\partial\varphi}{\partial v^{j}}(p,v)
$$

Actually $\varphi$ is 1-homogeneous, so Euler's identity yields:
$$
w_{j}v^{j}
 =\varphi(v)\frac{\partial\varphi}{\partial v^{j}}(p,v)v^{j}
 =\varphi(v)^2
$$
and, for a curve $\gamma (-\varepsilon, \varepsilon)\rightarrow T_{p}M$ such that $\gamma(0)=v$, $\varphi(\gamma(t))=\varphi(v)$ and $\gamma'(0)=z$,
$$
\varphi(v)w_{j}z^{j}
 =\varphi(v)^2\frac{\partial}{\partial t}\vert_{t=0}\varphi (\gamma(t))=0
$$
\end{remark}

\begin{remark}
The hypothesis on $\varphi$ imply that the orthogonal form to a vector is unique, and the correspondence between a vector and its dual one form is one to one, but it is only linear for riemannian metrics.
\end{remark}

\section{Exponential maps of Finsler Manifolds}\label{section: intro to exponential maps}

\begin{dfn}\label{various definitions related to the exponential map}
Let $D(\Phi_t)$ be the domain of the time-$t$ flow of the geodesic vector field in $T M$.
We introduce the sets $V$ and $W$:
\begin{equation}\label{V is ...}
 V=\left\lbrace (t,p): t\geq 0, p\in \partial M, \Gamma(p)\in D(\Phi_t) \right\rbrace \subset \RR\times\partial M
\end{equation}
\begin{equation}\label{W is ...}
 W=\left\lbrace \Phi_t(\Gamma(p)): t\geq 0, p\in \partial M, \Gamma(p)\in D(\Phi_t) \right\rbrace \subset TM
\end{equation}

$V$ and $W$ are diffeomorphic through the map $G(t,p) = \Phi_t(\Gamma(p))$.
We define the \strong{exponential map} $F$ associated to $(M,\Gamma)$ as $F = \pi\circ G:V\rightarrow M$.

The interior of $W$ is locally invariant under $\Phi_{t}$. This is equivalent to saying that $\rho$ is tangent to $W$. The \strong{radial vector} $r=\frac{\partial}{\partial t}$ is mapped to $\rho$ by $G_{\ast}$.
\end{dfn}

\begin{rmk}
The map $G$ is injective in its domain; its inverse can be computed walking a geodesic backwards until we hit the boundary for the first time.
In other words, $W$ is a smooth manifold and $G$ is a diffeomorphism from $V\subset\RR\times\partial M $.
\end{rmk}

\begin{rmk}
The map $G$ can also be written $G(t,p)=(F(p),dF_p(r))$, as follows from the geodesic equations.
\end{rmk}

In the particular case when $\Gamma(p)$ is the inner unit normal vector to $\partial M$, this is the standard ``exponential map from $\partial M$''.
This, in turn, includes the exponential map from a point $p$ in a manifold without boundary, in the following way:

Let $p$ be a point in a manifold without boundary, and remove a ball $B$ of small radius around $p$. The result is a manifold with boundary $\partial M = \partial B$, and the exponential map from $\partial B\subset M\setminus B$ coincides with the exponential map from $p\in M$. 

The same trick works for the exponential from a submanifold of any codimension inside a manifold without boundary, removing a tubular neighborhood around the submanifold.

The motivation for working in the above setting is that it also works for the Hamilton-Jacobi BVP with non-trivial boundary data, and it does not make the proofs more complicated.

\begin{dfn}\label{conjugate}
Let $x=(t,z)\in V$.

We say $x$ is \emph{conjugate} iff $F$ is not a local diffeomorphism at $x$. The \emph{order} of conjugacy is the dimension of the kernel of $dF$.

We say $x$ is \emph{a first conjugate vector} iff no point $(s,z)$, for $s<t$, is conjugate.

We also call the image of the radial vector $d_xF(r_x)$ a \emph{conjugate vector} (for $F$) whenever $x$ is conjugate.
\end{dfn}

In differential geometry, it is more usual to use the term \emph{focal} instead of \emph{conjugate}, when studying the distance function from a hypersurface, but other authors do otherwise (see for instance \cite{Li Nirenberg}). We have decided to use this term because our results about the Ambrose conjecture, that will appeal the most to differential geometers, deal only with the exponential map from a single point in a Riemannian manifold, while the other results appeal more to people working in PDEs, which prefer the term \emph{conjugate}.

\subsection{Regular exponential map}\label{subsection: regular exponential map}
The following proposition states some properties of a Finsler exponential map that correspond approximately to the definition of \emph{regular exponential map} introduced in \cite{Warner}. The second property is the only one that is not standard, but we need it to prove the existence of the special coordinates (and only for that).

\begin{prop}\label{regular exponential map}
The exponential map $F$ has the following properties:
\begin{itemize}
 \item $dF_{x}(r_{x})$ is a non zero vector in $T_{F(x)}M$.
 \item
at every point $x\in V$ there is a basis
$$
\mathcal{B}=\left\lbrace v_{1},..,v_{n} \right\rbrace
$$
of $T_{x}V$ where $v_{1}=r$ and $v_{n-k+1},..,v_{n}$ span $\ker dF_{x}$, and such that:
$$\mathcal{B}'=
\left\lbrace
dF(v_{1}), \dots, dF(v_{n-k}),
\widetilde{d^{2}F(r\sharp v_{n-k+1})},\dots
\widetilde{d^{2}F(r\sharp v_{n})}
\right\rbrace
$$
is a basis of $T_{F(x)}M$, where $\widetilde{d^{2}F(r\sharp v_{j})}$ is a representative of $d^{2}F(r\sharp v_{j})\in T_{F(x)}M/dF(TV_{x}) $, for $n-k+1\leq j\leq n $.

 \item Any point $x\in V$ has a neighborhood $U$ such that for any ray $\gamma $ (an integral curve of $r$), the sum of the dimensions of the kernels of $dF$ at points in $\gamma \cap U$ is constant.
 \item For any two points $x_{1}\neq x_{2}$ in $V$ with $F(x_{1})=F(x_{2})$, $dF_{x_{1}}(r_{x_{1}})\neq dF_{x_{2}}(r_{x_{2}})$
\end{itemize}
\end{prop}
\paragraph{Proof.} The first three properties follow from the work of Warner \cite[Theorem 4.5]{Warner} for a Finsler exponential map. We emphasize that they are local properties. The last one follows from the uniqueness property for second order ODEs. We remark that the second property implies the last one locally.

Indeed, properties 1 and 3 are found in standard textbooks such as \cite{Milnor}. Let us recall some of the notation in \cite{Warner} and \cite{Ambrose Palais Singer} and show the equivalence of the second property with his condition (R2) on page 577:

\begin{itemize}
 \item Second order tangent vectors at a point in an $n$-dimensional manifold are written in coordinates in the following way ($a_{ij}$ is symmetric):
$$
\sigma(f)=\sum_{i,j}a_{ij}\frac{\partial^{2} f}{\partial x_{i}\partial x_{j}} + \sum_i b_i \frac{\partial f}{\partial x_{i}}
$$
 \item $T^{2}_{x}V$ is the set of second order tangent vectors at the point $x$ in the manifold $V$.
 \item The second order differential of $F:V\rightarrow M$ at $x$ is the map $d^{2}_{x}F:T^{2}_{x}V \rightarrow T^{2}_{x}M $ defined by:
$$
d^{2}_{x}F(\sigma)f=\sigma(f\circ F)
$$
 \item The symmetric product $v\sharp w$ of $v\in T_{x}V$ and $w\in T_{x}V$ is a well defined element of $T^{2}_{x}V/T_{x}V$
 with a representative given by the formula:
$$
(v\sharp w)f=\frac{1}{2}(v( w(f))+w (v(f)))
$$
for arbitrary extensions of $v$ to $w$ to vector fields near $x$.
\item The map $d^{2}_{x}F $ induces the map $d^{2}F:T^{2}V_{x}/TV_{x} \rightarrow T^{2}_{F(x)}M/dF(TV_{x})$ by the standard procedure in linear algebra.
\item For $x\in V$, $v\in T_{x}V$ and $w\in \ker dF_{x}$, $d^{2}F(v\sharp w)$ makes sense as a vector in the space $T_{F(x)}M/dF(TV_{x}) $. For any extension of $v$ and $w$, the vector $d^{2}F(v\sharp w)$ is a first order vector.
\end{itemize}

Thus, our condition is equivalent to property (R2) of Warner:
\begin{quotation}  
At any point $x$ where $\ker dF_{x}\neq 0$, the map 
$$ 
d^{2}F:T^{2}V_{x}/TV_{x}
\rightarrow T^{2}M_{F(x)}/dF(TV_{x})
$$ 
sends $\langle r_{x}\rangle \sharp \ker
dF_{x}$ isomorphically onto  $T_{F(x)}M/dF(TV_{x})$.
\end{quotation}

We recall that $d_{x}F(v)$ can be computed as a Jacobi field when we consider
only the exponential map from a point, but this interpretation is somewhat
diluted when we work with the exponential map from the boundary.

\begin{remark}
 Warner defines a \emph{regular exponential map} as any map from $T_p M$ into $M$ thast satisfies the properties of proposition \ref{regular exponential map}. We do not need to work in that generality, as it does not include any new application.
\end{remark}

\subsection{Special coordinates}\label{subsection: special coordinates}

In order to study the map $F$ more comfortably, we define the \emph{special coordinates}, a pair of coordinates near a conjugate point $z$ of order $k$ and its image $F(z)$ that make $F$ specially simple. They can be defined for any \emph{regular exponential map}.

Let $\mathcal{B}=\lbrace v_{1},\dots,v_{n}\rbrace$ be the basis of $T_{z}V$ indicated in the second part of Proposition \ref{regular exponential map}, and $\mathcal{B}'_{F(z)}$ the corresponding basis at $F(z)\in M$ formed by vectors $d_{z}F(v_{i}), i\leq n-k$, and
$\widetilde{d^{2}_{z} F(v_{1}\sharp v_{n-k+1})},\dots,
\widetilde{d^{2}_{z} F(v_{1}\sharp v_{n})}$.

Make a linear change of coordinates in a neighborhood of $F(z)$ taking $\mathcal{B}'_{F(z)}$ to the canonical basis.
The coordinate functions $F^{i}(x)-F^{i}(z)$ of $F$ for $i\leq n-k$ can be extended to a coordinate system near $z$ with the help of $k$ functions having $v_{n-k+1},\dots,v_{n}$ as their respective gradients at $z$. In these coordinates $F$ looks:
\begin{equation}\label{F near order k in special coords}
 F(x_{1},\dots,x_{n})=(x_{1},\dots,x_{n-k}, F_{z}^{n-k+1}(x),\dots,F_{z}^{n}(x))
\end{equation}
and
\begin{itemize}
\item $\frac{\partial}{\partial x_i} F_j(x^0)$ is $0$ for any $i$ and $j\geq n-k+1$,
\item $\frac{\partial}{\partial x_i} \frac{\partial}{\partial x_1} F_j(x^0)$ is $\delta^i_j$, for $i,j\geq n-k+1$.
\item $\frac{\partial}{\partial x_1} (x^0) = r_{x^0} $
\end{itemize}

\section{The Cut Locus}\label{section: The Cut Locus}

Let $M$ be a Finsler manifold with boundary $\partial M$. We mentioned earlier that the study of the exponential map from a point or submanifold can be reduced to a exponential map of a manifold with boundary $\partial M$.
The same principle applies to the cut locus, so we will consider only the cut locus from the boundary $Cut = Cut_{\partial M}$. It can be defined in several equivalent ways:
\begin{dfn}\label{definition: the cut locus}
Let $M$ be a Finsler manifold with boundary $\partial M$:
\begin{itemize}
 \item For any $p\in \partial M$, let $\gamma_p$ be the unit speed geodesic $\gamma$ with initial point in $\partial M$ and initial speed orthogonal to $\partial M$ (and inner-pointing). $\gamma_p$ minimizes the distance between $\gamma(0)$ and $\gamma(t)$, for small $t$. Define
 $$t_{cut}(p)=\sup\{t:d(\gamma_p(0),\gamma_p(t))=t \}$$
 Then
 $$
 Cut = \{\gamma_p(t_{cut}(p)):p\in\partial M\}
 $$
 \item For a point $p\in M$, let $Q_p$ be the set of points $q\in \partial M$ such that $d(q,p)=d_{\partial M}(p)$. Then $Cut$ is the closure of the set of points such that $Q_p$ has more than one element.
 \item The function $d_{\partial M}$ is singular exactly when $Q_p$ has more than one element, so we can also define $Cut$ as the closure of the set where the distance function to $\partial M$ is singular.
 \item $Cut$ is also the set of points such that either $Q_p$ has more than one element, or $Q_p=\{q\}$, and $(d(q,p),q)\in V$ is conjugate.
\end{itemize}
\end{dfn}

The reader can find the proof of those facts for Riemann manifolds in standard textbooks in Riemannian geometry (see for example chapter 13 in \cite{do Carmo}). The proof for Finsler manifolds can be found in \cite{Li Nirenberg}, for example.
For basic information about the distance function, such as its differentiabilty, the reader can use \cite{Cannarsa Sinestrari}.

Much is known about the set $Cut$:

  \begin{itemize}
    \item $Cut$ is a deformation retract of $ M$ (obvious).

    \item It is the union of a ($n - 1$)-dimensional smooth manifold consisting of points with two minimizing geodesics and a set of Hausdorff
    dimension at most $n - 2$. This easy but important lemma appears to have been proven at least in \cite{Hebda83},  \cite{Itoh Tanaka 98}, \cite{Barden Le} and \cite{Mantegazza Mennucci}, always for the Riemannian case. We give a proof of this result in lemma \ref{theorem: conjugate of order 1} that is also true for Finsler manifolds.

    \item It is stratified by the dimension of the subdifferential of the \emph{distance to the boundary} $\partial d_{\partial M}$. This follows from the properties of semiconcave functions mentioned in \ref{subsection: semiconcave functions}, as $d_{\partial M}$ is semiconcave.
    This result can be found in \cite{Alberti Ambrosio Cannarsa}, and their proof works verbatim for balanced split locus in our \ref{main theorem 4}.

    \item The local homology of the cut locus of $p\in M$ at a point $q$ is related to the set of minimizing geodesics from $p$ to $q$ (see \cite{Ozols} and \cite{Hebda83}).

    \item The cut locus has finite Hausdorff measure $\mathcal{H}^{n - 1}$. This result can be found in \cite{Itoh Tanaka 00} for Riemannian manifolds, and in \cite{Li Nirenberg} for Finsler manifolds. We provide a new proof of this result for Finsler manifolds in \ref{section:rho is Lipschitz}. M. Castelpietra and L. Rifford also gave a proof of this result that appeared shortly after the one we present here.

    \item If all the data is analytic, $Cut$ is a stratified analytic manifold (see \cite{Buchner analytic}). We will not use this result.

    \item If we add a generic perturbation to $H$ or $ M$, \strong{Sing} becomes a stratified smooth manifold. Furthermore, for dimension up to $6$, the cut locus is generically stable, in the sense that adding a small perturbation to the metric, the new cut locus would still be diffeomorphic to the original one. This is a very deep result from M. A. Buchner, a student of J. Matter, and it is beyond the scope of this work to include a proof of his results (see \cite{Buchner} and \cite{Buchner Stability}), but we will make use of them in chapter \ref{chapter: ambrose}.
  \end{itemize}

On the other hand, H. Gluck and D. Singer proved that there are non-triangulable cut loci in \cite{Gluck Singer} and, in \cite{Gluck Singer II}, that there are surfaces of revolution such that the cut locus from any point is non-triangulable.
Another difficulty is mentioned by J. Hebda in \cite{Hebda}: even though the homotopy of $Cut$ is known, and even if the cut locus is a simplicial complex, that simplicial complex may not descend simplicially to one point, and this was a major obstacle in extending his proof of the Ambrose conjecture for surfaces to manifolds of higher dimension.

\section{Hamilton-Jacobi equations and Finsler geometry}

Here we review the relationship between Hamilton-Jacobi equations and Finsler geometry.
The reader can find more details in \cite{Li Nirenberg}, \cite{Lions} and \cite{Cannarsa Sinestrari}.

Here $M$ is a manifold with possibly non-compact boundary.
We are interested on solutions to the system (which we will refer to as a \strong{Hamilton-Jacobi Boundary Value Problem} or \strong{HJBVP} for short):
\begin{eqnarray}
H(p,du(p))\;=&1\quad&p\in  M\label{HJequation}\\
u(p)\;=&g(p)&p\in \partial M\label{HJboundarydata}
\end{eqnarray}

\noindent where  $H:T^{\ast}M\rightarrow \RR$ is a smooth function that is $1$-homogeneous and subadditive for linear combinations of covectors lying over the same point $p$, and $g:\partial M\rightarrow \RR$ is a smooth function that satisfies the compatibility condition:
\begin{equation}\label{compatibility condition}
\left\vert g(p)-g(q)\right\vert < kd(p,q)\quad \forall p, q\in \partial  M
\end{equation}
for some $k<1$.
Here $d$ is the distance induced by the Finsler metric $\varphi$ that is the pointwise dual of the metric in $T^{\ast}M$ given by $H$:
\begin{equation}\label{phi is the dual of H}
 \varphi_{p}(v)=\sup\left\lbrace
\left\langle v,\alpha\right\rangle_{p}\, :\,
\alpha\in T^{\ast}_{p}M, \,H(p,\alpha)=1
\right\rbrace
\end{equation}

\begin{remark}
 As mentioned in the introduction, we can ask that $H^{-1}(1)\cap T^{\ast}_{p} M$ is strictly convex for every $p$ instead of asking that
 $H$ is $1$-homogeneous and subadditive for linear combinations of covectors lying over the same point $p$. The properties are not equivalent for a function $H$, but the equations that we consider are the same, because the only thing we use about $H$ is the $1$-level set. If the sets $H^{-1}(1)\cap T^{\ast}_{p} M$ are convex for every $p$, we can replace $H$ with a new one that is $1$-homogeneous, subadditive for linear combinations of covectors lying over the same point $p$, and has the same $1$-level set.
\end{remark}

\subsection{Characteristics of the HJBVP}\label{subsection: characteristics of the HJBVP}

Using the definition \ref{dual one form} of dual form in Finsler geometry, we can restate the usual equations for the \emph{characteristic vector field} at points $p\in \partial M $:
\begin{eqnarray}\label{equation for the characteristic vector field}
 &\varphi_p(X_p)=1\notag\\
 &\widehat{X_p}\vert_{T(\partial M)}=d g\notag\\
 &X_p\text{ points inwards}
\end{eqnarray}

We define the characteristic vector field as a map $\Gamma:\partial  M \rightarrow T M$, by the formula $\Gamma(p)=X_p$.
The characteristic curves are the integral curves of the geodesic vector field in $T M$ with initial point $\Gamma(z)$ for $z\in \partial  M$.
The projected characteristics are the projection to $ M$ of the characteristics.

A local \emph{classical} solution $u$ to the HJBVP can be computed near $\partial  M$ following characteristic curves:

\begin{dfn}\label{solution by characteristics}
Let $U$ be a neighborhood of $\partial M$ such that every point $q\in U$ belongs to a unique (projected) characteristic contained in $U$ and starting at a point $p\in \partial M$ (the point $p$ is often called the \strong{footpoint} of $q$).

The \strong{solution by characteristics} $u:U\rightarrow \RR$ is defined as follows: if $\gamma:[0,t]\rightarrow  M$ is the unique (projected) characteristic from a point $p\in\partial M $ to $q=\gamma(t)$ that does not intersect $Sing$, then
$$u(q)=g(p)+t$$
\end{dfn}

\subsection{Viscosity solutions of Hamilton-Jacobi equations}\label{section: Singular locus of  Hamilton-Jacobi}

The solution found above using characteristic curves is only defined in a neighborhood of $\partial M$.

There is a different notion of solution to these equations. 
The inspiration came from the following observation: if we add a small viscosity term like $-\varepsilon \Delta$ to \eqref{HJequation}, that equation becomes semilinear elliptic, and admits a global solution. So the idea appeared to add that viscosity term, and later let $\varepsilon$ converge to $0$.
Even though this was the inspiration, it later became clear that it was better to use a different definition, using comparison functions.

\begin{dfn}\label{viscosity solution of HJBVP}
 A function $u:M\rightarrow\RR$ is a \strong{viscosity subsolution} (resp. supersolution) to the HJBVP given by \eqref{HJequation} and \eqref{HJboundarydata} iff for any $\phi\in C^1(M)$ such that $u-\phi$ has a local maximum at $p$ (resp., a local minimum), we have:
 \begin{equation}
 H(p,D\phi(p)) \leq 0 \qquad\text{(resp. } H(p,D\phi(p)) \geq 0 \text{)}
 \end{equation}
 
  It is a \strong{viscosity solution} to the HJBVP iff it is both a subsolution and a supersolution.
\end{dfn}

The reader can find more details in \cite{Lions} or \cite{Cannarsa Sinestrari}.

However, in this thesis we will not be concerned neither with the inspiration that gave them the name ``viscosity solutions'', nor with the actual definition. We only need to know that the \emph{unique} \strong{viscosity solution} is given by the Lax-Oleinik formula (see theorem 5.2 in \cite{Lions}):
\begin{equation} \label{Lax-Oleinik}
u(p)=\inf_{q\in \partial M}
\left\lbrace
   d(p,q)+ g(q)
\right\rbrace
\end{equation}

The viscosity solution can be thought of as a way to extend the classical solution by characteristics to the whole $M$.
When $g=0$, the solution \eqref{Lax-Oleinik} is the \emph{distance to the
boundary}. 

As we mentioned earlier, the viscosity solution to a HJBVP is \emph{a semiconcave function}. It is interesting to remark that a semiconcave function that satifies the equation \ref{HJequation} at the points at which it is differentiable is the viscosity solution to the HJBVP.

\chapter{A new way to look at Cut and Singular Loci}\label{Chapter: IntroBalanced}

\section{The relation between Finsler geometry and Hamilton-Jacobi BVPs}

Let us consider the HJBVP given by \eqref{HJequation} and \eqref{HJboundarydata} when $g=0$.
On the one hand, $\Gamma(p)$ is the inner pointing unit normal to $\partial M$ at $p$.
On the other hand, the viscosity solution $u$ given by \eqref{Lax-Oleinik} is the distance to the boundary.
This has nice consequences: for example, the singular set of $u$ is a cut locus, and we can apply the various structure results about the cut locus mentioned in section \ref{section: The Cut Locus}.

Our intention in this section is to adapt this result to the case $g>0$. If $\partial M$ is compact, a global constant can be added to an arbitrary $g$ so that this is satisfied and $S$ is unchanged.
We still require that $g$ satisfies the compatibility condition \ref{compatibility condition}.

Subject to these conditions, our goal is to show that the Finsler manifold $(M,\varphi)$ can be embedded in a new manifold with boundary $(N,\tilde{\varphi})$ such that $u$ is the restriction of the unique solution $\tilde{u}$ to the problem
\begin{eqnarray}
\tilde{H}(p,d\tilde{u}(p))=1&\quad&p\in N\notag\\
\tilde{u}(p)=0 &&p\in \partial N\notag
\end{eqnarray}
thus reducing to the original problem ($\tilde{H} $ and $\tilde{\varphi}$ are
dual to one another as in \ref{phi is the dual of H}). This allows us to
characterize the singular set of  \eqref{Lax-Oleinik} as a cut locus, which
automatically implies that all the structure results about the cut locus in
section \ref{section: The Cut Locus} apply to the more general case.

\begin{dfn}\label{indicatrix}
The \emph{indicatrix} of a Finsler metric $\varphi$ at the point $p$ is the set
$$I_{p}=\left\lbrace v\in T_{p}M\,: \, \varphi(p,v)=1\right\rbrace  $$
\end{dfn}

\begin{lem}\label{cambia metrica respeta campo}
Let $\varphi_{0}$ and $\varphi_{1}$ be two Finsler metrics in an open set $U$, and let $X$ be a vector field in $U$ such that:
\begin{itemize}
\item The integral curves of $X$ are geodesics for $\varphi_{0}$.
\item $\varphi_{0}(p,X_{p})=\varphi_{1}(p,X_{p})=1$
\item At every $p\in U$, the tangent hyperplanes to the indicatrices of $\varphi_{0}$ and $\varphi_{1}$ in $T_{p}U$ coincide.
\end{itemize}
Then the integral curves of $X$ are also geodesics for $\varphi_{1}$
\end{lem}

 \begin{proof}
Let $p$ be a point in $U$. Take bundle coordinates of $T_{p}U$ around $p$ such that $X$ is one of the vertical coordinate vectors. An integral curve $\alpha$ of $X$ satisfies:
$$(\varphi_{0})_{p}(\alpha(t),\alpha'(t))=
  (\varphi_{1})_{p}(\alpha(t),\alpha'(t))=1$$
because of the second hypothesis. The third hypothesis imply:
$$ (\varphi_{0})_{v}(\alpha(t),\alpha'(t)) =
  (\varphi_{1})_{v}(\alpha(t),\alpha'(t))$$
So inspection of the geodesic equation:
\begin{equation}
\varphi_{p}(\alpha(t),\alpha'(t))=
\frac{d}{dt} \left( \varphi_{v}(\alpha(t),\alpha'(t))\right)
\end{equation}
shows that $\alpha$ is a geodesic for $\varphi_{1}$.
\end{proof}

\begin{cor}\label{una Riemann con las geodesicas de una Fisnler}
Let $\varphi$ be a Finsler metric and $X$ a vector field whose integral curves are geodesics. Then there is a Riemannian metric for which those curves are also geodesics.
\end{cor}
\begin{proof}
The Riemannian metric $g_{ij}(p)=\frac{\partial}{\partial v_{i}v_{j}}\varphi (p,X)$ is related to $\varphi $ as in the preceeding lemma.
\end{proof}

\begin{lem}\label{characterization of Finsler geodesics}
Let $X$ be a non-zero norm-$1$ geodesic vector field in a Finsler manifold and $\omega$ its dual differential one-form. Then the integral curves of $X$ are geodesics if and only if the Lie derivative of $\omega$ in the direction of $X$ vanishes.
\end{lem}
\begin{proof}
The integral curves of $X$ are geodesics for $\varphi$ iff they are geodesics for the Riemannian metric $g_{ij}(p)=\frac{\partial}{\partial v_{i}v_{j}}\varphi (p,X)$, but the dual one-form to $X$ with respect to both metrics is the same one-form $\omega$, and the vanishing of $\mathcal{L}_X\omega$ has nothing to do with the metric.

We have thus reduced the problem to a Riemannian metric, when this result is standard:
$$
\begin{array}{lll}
  \mathcal{L}_{X} ( Y ) & = & X ( w ( Y ) ) - \omega ( [ X,Y ] )\\
  & = & X ( \langle X,Y \rangle ) - \langle X,D_{X} Y-D_{Y} X \rangle\\
  & = & \langle D_{X} X,Y \rangle + \langle X,D_{X} Y \rangle - \langle
  X,D_{X} Y \rangle + \langle X,D_{Y} X \rangle\\
  & = & \langle D_{X} X,Y \rangle + \frac{1}{2} Y \langle X,X \rangle
\end{array}
$$
\end{proof}

\begin{prop}
Let $ M $ be an open  manifold with smooth boundary and a Finsler metric $\varphi$. Let $X$ be a smooth transversal vector field in
$\partial M $ pointing inwards (resp. outwards).
Then $ M $ is contained in a larger open manifold admitting a smooth extension $\tilde{\varphi}$ of $\varphi$
to this open set such that the geodesics starting at points $p\in \partial M $ with initial vectors $X_{p}$
can be continued indefinitely backward (resp. forward) without intersecting each other.
\end{prop}

\begin{proof}
We will only complete the proof for a compact manifold with boundary $ M $ and inward pointing vector $X$,
as the other cases require only minor modifications.

We start with an arbitrary smooth extension $\varphi'$ of $\varphi$ to a larger open set $ M _{2} \supset  M $. The geodesics with initial speed $X$ can be continued backwards to $ M _{2}$, and
there is a small $\varepsilon$ for which they do not intersect each other for negative values of time before the parameter reaches $-\varepsilon$.

Define

$$
P:\partial  M \times (-\varepsilon,0]\to M_2, \qquad P(q,t):= \alpha_{q}(t)
$$

\noindent where $\alpha_{q}:(-\varepsilon,0]\to M_2$ is the geodesic of $\varphi'$ starting at the point $q\in \partial  M $ with initial vector $X_{q}$.
When $p\in U_{\varepsilon}:=\Img(P)$ there is a unique  value of $t$ such that $p=P(q,t)$ for some $q\in \partial M$. We will denote such $t$ by $d(p)$.
Extend also  the vector $X$ to $U_{\varepsilon}$ as
$X_{p}= \dot{\alpha_{q}}(t)$
where $p=P(q,t)$.

Let $c: (-\varepsilon,0]\rightarrow [0,1]$
be a smooth function such that

\begin{itemize}
\item $c$ is non-decreasing
\item $c(t)=1 \text{ for } -\varepsilon/3 \leq t$
\item $c(t)=0 \text{ for } t\leq -2\varepsilon/3$
\end{itemize}
and finally define
$$\tilde{X}_{p}=c(d(p))X_{p}$$
in the set $U_{\varepsilon}$.

Let $\omega_{0}$ be the dual one form of $\tilde{X}$ with respect to $\varphi$ for points in $\partial M $, and let $\omega$ be the one form in $U_{\varepsilon}$ whose Lie derivative in the direction $\tilde{X}$ is zero and which coincides with $\omega_{0}$ in $\partial M $.
Then we take any metric $\varphi''$ in $U_{\varepsilon}$ (which can be chosen Riemannian) such that $\tilde{X}$ has unit norm and the kernel of $\omega$ is tangent to the indicatrix at $\tilde{X}$.

By lemma \ref{characterization of Finsler geodesics}, the integral curves of $\tilde{X}$ are geodesics for $\varphi''$. Now let $\rho$ be a smooth function in $U_{\varepsilon}\cup  M $ such that $\rho\vert_{ M }=1$, $\rho\vert_{U_{\varepsilon} \setminus U_{\varepsilon/3}}=0$ and $0\leq \rho \leq 1$, and define the metric:
$$\tilde{\varphi}= \rho(p)\varphi(p,v)+ (1-\rho(p))\varphi''(p,v)$$

This metric extends $\varphi$ to the open set $U_{\varepsilon}$ and makes the integral curves of $\tilde{X}$ geodesics.
As the integral curves of $X$ do not intersect for small $t$,
the integral curves of $\tilde{X}$ reach infinite length before they approach $\partial U_{\varepsilon}$ and the last part of the statement follows.
\end{proof}

Application of this proposition to $ M $ and the characteristic, inwards-pointing vector field $v$ yields a new manifold $N$ containing $ M $, and a metric for $N$ that extends $\varphi$ (so we keep the same letter) such that the geodesics departing from $\partial M $ which correspond to the characteristic curves continue indefinitely backwards without intersecting.

This allows the definition, for small $\delta$ of
$$
\tilde{P}:\partial  M \times (-\infty,\delta]\to N, \qquad P(q,t):= \tilde{\alpha}_{q}(t)
$$
where $\tilde{\alpha} $ are the geodesics with initial condition $X$, continued backwards if $t$ is negative. Finally, define $\tilde{u}:U \rightarrow \RR$ by:
\begin{equation}\label{def of utilde}
\tilde{u}(p)=\begin{cases}
    g(q)+t& p=\tilde{P}(q,t), \quad p\in N\setminus  M \\
    u(p)& p\in  M
     \end{cases}
\end{equation}
We notice that both definitions agree in an \emph{inner} neighborhood of $\partial M $, so the function $\tilde{u}$ is a smooth extension of $u$ to $N$.

\begin{theorem}\label{reduce HJ to boundary value 0}
Let $\Lambda=\tilde{u}^{-1}(0)$.
Then the following identity holds in $\left\lbrace \tilde{u}\geq 0 \right\rbrace $:
\begin{equation}
\tilde{u}(p)= d(\Lambda,p)
\end{equation}
\end{theorem}

\begin{proof}

$\Lambda$ is smooth because it is contained in $N\setminus M$, where $\tilde{u}$ is smooth and has non-vanishing gradient.

In order to show that $\tilde{u} $ and $d_{\Lambda} $ agree in $U$, we use the uniqueness properties of viscosity solutions.
Let $ N$ be the open set where $\tilde{u}>0$.
The distance function to $\Lambda$ is characterized as the unique viscosity solution to:

\begin{itemize}
\item $\tilde{u}=0$ in $\Lambda$
\item $H(p, d\tilde{u}(p))=1$ in $N$, in the viscosity sense
\end{itemize}

Clearly $\tilde{u}$ satisfies the first condition. It also satisfies the second for points in the set $ M $ because it coincides with $u$, and for points in $ N \setminus  M $ because $\tilde{u}$ is smooth and $H(p, d\tilde{u}(p))=1$ holds in the classical sense there.
\end{proof}

The following fact is well known but we provide a geometric proof.
\begin{cor}\label{du is dual of tangent to characteristics}
 The differential $du$ of the solution by characteristics is Finsler dual to the tangent to the (projected) characteristics.
\end{cor}
\begin{proof}
By the above, we can assume that $g=0$.
Let $U$ be a neighborhood of $\partial M$ where a solution by characteristics $u$ is defined, as in \ref{solution by characteristics}.
Let $q\in U$, $X$ be the tangent to the characteristic that goes through $q$, with footpoint $p$.

The claim can be checked easily if $q\in\partial M$, because both $\tilde{X}$ and $du$ are linear forms and they agree on the hyperplane $T_q\partial M$ and $\Gamma(q)$, by the definition \ref{equation for the characteristic vector field} of the characteristic vector field.

For the rest of points in $U$, we notice that the level curves of $u$ are Lie parallel with respect to $X $, and so $L_X(du)=0$. But \ref{characterization of Finsler geodesics} says that $L_X(\tilde{X})=0$, and thus we have two $1$-forms that agree on $\partial M$ and are parallel with respect to $X$, so they agree everywhere.
\end{proof}

The following theorem is an extension of Theorem 1.1 of \cite{Li Nirenberg}. In this result $\partial M$ may not be compact.

\begin{theorem}\label{regularity of mu for nontrivial g}
Let $S$ be the closure of the singular set of the viscosity solution to the following HJBVP:
\begin{eqnarray*}
H(p,du(p))=1&\quad&p\in M\\
u(p)=g(p)&&p\in \partial M
\end{eqnarray*}
where $g:\partial M\rightarrow \RR$ satisfies the usual compatibility condition \ref{compatibility condition}.

If $\mu$ is the function whose value at $p\in \partial M $ is the distance to $S$ along the unique characteristic departing from $q$, then
\begin{enumerate}
\item $\mu$ is Lipschitz.
\item If in addition $\partial M $ is compact,
then the $(n-1)$-dimensional Hausdorff measure of $S\cap K$ is finite for any compact $K$.
\item In general, $S$ is a Finsler cut locus from the boundary of some Finsler manifold, so all the regularity results for cut loci apply to $S$ (see section \ref{section: The Cut Locus}).
\end{enumerate}
\end{theorem}

\begin{proof}
The first part follows immediately from Theorem \ref{reduce HJ to boundary value 0} and Theorem 1.1  in \cite{Li Nirenberg}.
The second is an easy consequence of the first, while the last is contained in the results of this section.
\end{proof}

\begin{remark}
The regularity hypothesis on $M$ can be softened.
In order to apply the results in \cite{Li Nirenberg}, it is enough that $M$ is $C^{2,1}$, which implies that $\Lambda$ is $C^{2,1}$.
\end{remark}

\section{Split locus and balanced split locus} \label{section: definitions, split and balanced}

We study a Hamilton-Jacobi equation given by \eqref{HJequation} and
\eqref{HJboundarydata} in a $C^{\infty}$ compact manifold with boundary $ M$,
with the hypothesis stated there.

Let $Sing$ be the closure of the singular set of the viscosity solution $u$ to aHamilton-Jacobi BVP. $Sing$ has a key property: any point in $M\setminus Sing$ can be joined to $\partial M $ by a unique characteristic curve that does not intersect $Sing$.
A set with this property is said to \emph{split $M$ along characteristics of the HJBVP} or simply to \emph{split} $M$ for short.
Once characteristic curves are known, if we replace $Sing$ by any set $S$ that splits $M$, we can use the value of $u$ that the characteristics carry along with them, to obtain a function, defined in $M\setminus S$ with some resemblance to the viscosity solution (see definition \ref{u associated to S}).

Looking at the cut locus from this new perspective, we wonder what distinguishes the cut locus from all the other sets that split $M$.

\begin{dfn}\label{splits}
For a set $S\subset  M$, let $A(S)\subset V$ be the set of all $x=(t,z)\in V$ such that $F(s,z)\notin S ,\, \forall \, 0\leq s<t$.
We say that a set $S\subset  M$ \emph{splits} $ M$ iff $F$ restricts to a bijection between $A(S)$ and $ M\setminus S$.
\end{dfn}

Whenever $S$ splits $ M$, we can define a vector field $R_{p}$ in $ M\setminus S$ to be $dF_{x}(r_{x})$ for the unique $x$ in $A(S)$ such that $F(x)=p$.
\begin{dfn}\label{the set R_p}

For a point $a\in S$, we define the \emph{limit set} $R_{a}$ as the set of vectors in $T_{a} M$ that are limits of sequences of the vectors $R_{p}$ defined above at points $p\in  M\setminus S$.

\end{dfn}

\begin{rmk}
If $S$ is a cut locus, the set $R_{p}$ is the set of all vectors tangent to the minimizing geodesics from $p$ to $\partial M$.
\end{rmk}

\begin{dfn}\label{the set Q_p}
If $S$ splits $ M$, we also define a set $Q_{p}\subset V$ for $p\in  M$ by
$$
Q_{p}=\left( F\vert_{\overline{A(S)}}\right)^{-1}(p)
$$
\end{dfn}
The following relation holds between the sets $R_{p}$ and $Q_{p}$:
$$
R_{p}=\left\lbrace dF_{x}(r_{x}):x\in Q_{p}
\right\rbrace
$$

\begin{dfn}\label{u associated to S}
If $S$ splits $ M$, we can define a real-valued function $h$ in $ M\setminus S$ by setting:
$$
h(p)=g(z)+t
$$
where $(t,z)$ is the unique point in $A(S)$ with $F(t,z)=p$.
\end{dfn}

If we start with the viscosity solution $u$ to the Hamilton-Jacobi equations, and let $S=Sing$ be the closure of the set where $u$ is not $C^1$, then $S$ splits $ M$. If we follow the above definition involving $A(S)$ to get a new function $h$, then we find $h=u$.

\begin{dfn}\label{split locus}
 A set $S$ that splits $ M$ is a split locus iff
$$
S =\overline{
\left\lbrace
p\in S: \quad \sharp R_{p}\geq 2
\right\rbrace }
$$
\end{dfn}

The role of this condition is to restrict $S$ to its \emph{essential} part. A set that merely splits $ M$ could be too big: actually $ M$ itself splits $ M$.
The following lemma may clarify this condition.

\begin{lem}\label{characterization of split locus}
A set $S$ that splits $ M$ is a split locus if and only if $S$ is closed and it has no proper closed subsets that split $ M$.
\end{lem}
\begin{proof}
The ``if'' part is trivial, so we will only prove the other implication.
Assume $S$ is a split locus and let $S'\subset S$ be a closed set splitting  $ M$. 
Let $q\in S\setminus S'$ be a point with $\sharp R_q\geq 2$.
Since $S'$ is closed, there is a neighborhood of $q$ away from $S'$; so, if $\gamma_1$ is a segment of a geodesic in $ M\setminus S'$ joining $\partial M$ with $q$, there is a point $q_1$ in $\gamma_1$ lying beyond $q$.
Furthermore, we can choose the point $q_1$ not lying in $S$, so there is a second geodesic $\gamma_2$ contained in $ M\setminus S\subset  M\setminus S'$ from $\partial M$ to $q_1$.
As $q\in S$, we see $\gamma_2$ is necessarily different from $\gamma_1$, which is a contradiction if $S'$ splits $M$.
Therefore we learn $S'\supset\left\lbrace p\in S: \quad \sharp R_{p}\geq 2\right\rbrace$, so $S=\overline{\left\lbrace p\in S: \quad \sharp R_{p}\geq 2\right\rbrace } \subset S'$.

\end{proof}

Finally, we introduce the following more restrictive condition (see \ref{vector de x a y} for the definition of $v_p(q)$, the \emph{vector from $p$ to $q$}, and \ref{dual one form} for the Finsler dual of a vector).

\begin{dfn}\label{balanced}
We say a split locus $S\subset  M$ is \emph{balanced} for given $ M$, $H$ and $g$ (or simply that it is balanced if there is no risk of confusion) iff  for all $p\in S$,
all sequences $p_{i}\to p$ with $v_{p_{i}}(p)\to v\in T_{p} M$, and any sequence of vectors $X_{i}\in R_{p_{i}}\to X_{\infty}\in R_{p}$,
then
$$
w_{\infty}(v)=\max\left\lbrace w(v),\text{ $w$ is dual to some $R\in R_{p}$} \right\rbrace
$$
where $w_{\infty}$ is the dual of $X_{\infty}$.

\end{dfn}

\begin{figure}[H]
 \centering
 \includegraphics{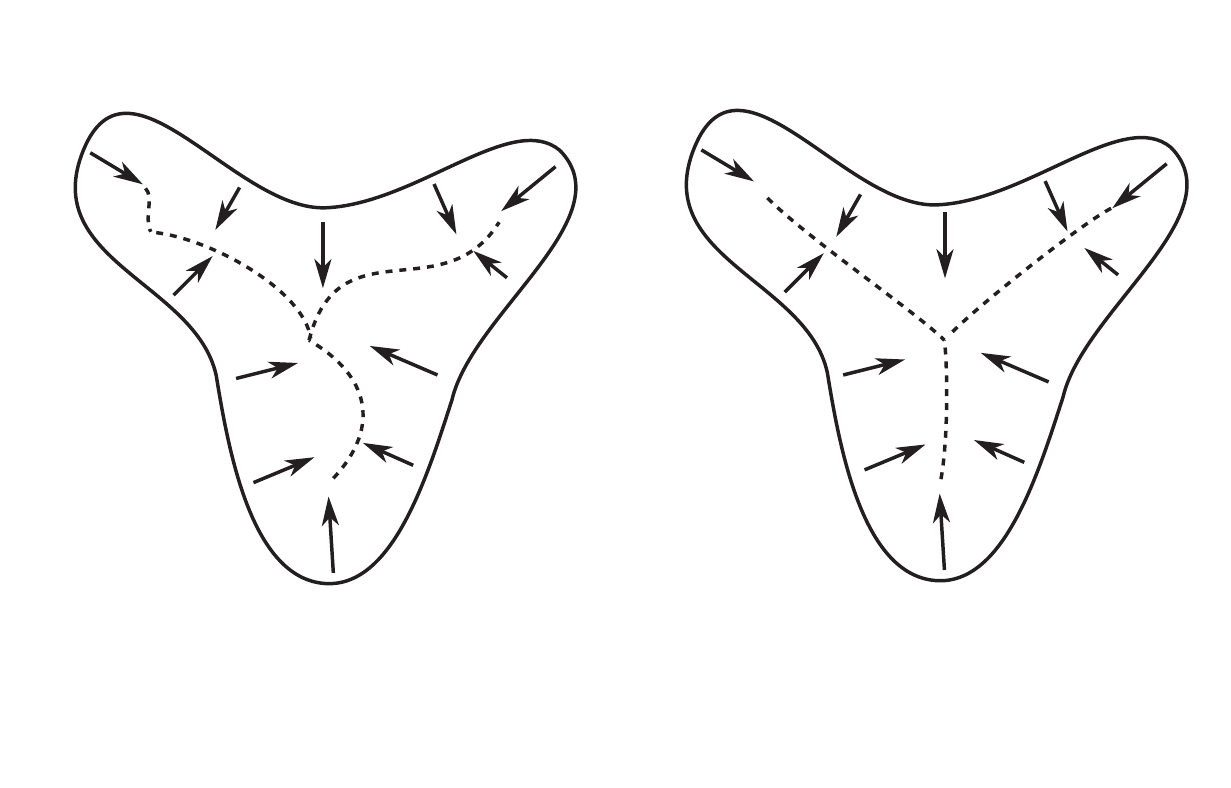}
 
 \caption{An arbitrary split locus and a balanced split locus}
 \label{fig: One split and one balanced}
\end{figure}

\section{Balanced property of the Finsler cut locus}\label{section: Balanced}
In this section we show that the cut locus of a Finsler exponential map is a balanced set.
We provide two proofs, none of which is original.
Although the hypothesis look different, our result \ref{reduce HJ to boundary value 0} show that they are equivalent.

The first proof is the same as in lemma 2.1 in \cite{Itoh Tanaka 00}, but we adapt it so that it also works for Finsler manifolds, where angles are not defined.

\begin{prop}\label{cut locus is balanced}
The cut locus of a Finsler manifold $M$ with boundary
is a balanced split locus.
Moreover, for $p$, $p_{n}$, $v$ and $X_{\infty}$ as in the definition of a balanced split locus, we have
$$
\lim_{n\rightarrow \infty}\dfrac{d(\partial M,p)-d(\partial M,p_{n})}{d(p,p_{n})}=
w_{\infty}(v)
$$
\end{prop}
\begin{proof}
The cut locus $S$ splits $M$, as follows from the first definition of cut locus in \ref{definition: the cut locus}.

It is also a split locus, as follows from the second definition of cut locus.

Next we show that $S$ is balanced. Take any $Y\in R_{p}$, and let $\gamma$ be the minimizing geodesic segment joining $\partial M$ to $p$ with speed $Y$ at $p$. Take any point $q\in\gamma$ that lies in a convex neighborhood of $p$ and use the triangle inequality to get:
$$
d(\partial M, p)-d(\partial M, p_{n}) \geq d(q, p) - d(q,p_{n})
$$
The first variation formula yields, for a constant $C$:
$$
d(q, p) - d(q,p_{n})\geq w(v_{p_{n}}(p))d(p_{n},p)-Cd(p,p_{n})^{2}
$$
and we get:
$$
\liminf_{n\rightarrow \infty}\dfrac{d(\partial M,p)-d(\partial M,p_{n})}{d(p,p_{n})}\geq
w(X)
$$
for any $w$ that is dual to a vector in $R_{p}$.

Then consider $X_{\infty}$, let $\gamma$ be the minimizing geodesic segment joining $\partial M$ to $p$ with speed $X_{\infty}$ at $p$, and let $\gamma_{n}$ be the minimizing geodesic segment joining $\partial M$ to $p_{n}$ with speed $X_{n}$ at $p_n$. Take
points $q_{n}$ in $\gamma_{n}$ that lie in a fix convex neighborhood of $p$. Again:
$$
d(\partial M, p)-d(\partial M, p_{n}) \leq d(q_{n}, p) - d(q_{n},p_{n})
$$
while the first variation formula yields, for a constant $C$:
$$
d(q_{n}, p) - d(q_{n},p_{n})\leq w(v_{p_{n}}(p))d(p_{n},p) + Cd(p,p_{n})^{2}
$$
and thus:
$$
\limsup_{n\rightarrow \infty}\dfrac{d(\partial M,p)-d(\partial M,p_{n})}{d(p,p_{n})}\leq
w_{\infty}(X)
$$

This proves the claim that $S$ is balanced.
\end{proof}

We give now another proof that relates the balanced condition to the notion of semiconcave functions, which is now common in the study of Hamilton-Jacobi equations.
More precisely, we simply translate theorem 3.3.15 in the book \cite{Cannarsa Sinestrari} to our language to get the following lemma:

\begin{lem}
The closure of the singular set of the viscosity solution to \eqref{HJequation}
and \eqref{HJboundarydata} is a balanced split locus.
\end{lem}
\begin{proof}
Let $u$ be the viscosity solution to \eqref{HJequation} and
\eqref{HJboundarydata},
and let $Sing$ be the closure of its singular set.
We leave to the reader the proof that $Sing$ is a split locus (otherwise, recall it is a cut locus).

It is well known that $u$ is semiconcave (see for example \cite[5.3.7]{Cannarsa Sinestrari}).
The superdifferential $ D^+ u(p) $ of $u$ at $p$ is the convex hull of the set of limits of differentials of $u$ at points where $u$ is $C^1$ (see \cite[3.3.6]{Cannarsa Sinestrari}).
At a point where $u$ is $C^1$, the dual of the speed vector of a characteristic is the differential of $u$.
Thus, the superdifferential at $p$ is the convex hull of the duals to the vectors in $R_p$.
We deduce:
$$
\max\left\lbrace w(v),\text{ $w$ is dual to some $R\in R_{p}$} \right\rbrace =
\max\left\lbrace w(v), w\in D^+ u(p) \right\rbrace
$$
Given $p\in  M$, and $v\in T_p M$, the \emph{exposed face} of $D^+ u(p) $ in the direction $v$ is given by:
$$
D^+ (p,v) = \{ \tilde{w}\in D^+ u(p) : \tilde{w}(v)\leq w(v) \;\forall w \in D^+ u(p) \}
$$
The balanced condition can be rephrased in these terms as:

\begin{quote}
 Let $p_i\rightarrow p\in S$ be a sequence with $v_{p_{i}}(p)\to v\in T_{p} M$, and let $w_i\in D^+ u(p_{i})$ be a sequence converging to $w\in D^+ u(p)$.

 Then $w \in D^+ u(p,-v) $
\end{quote}
which is exactly the statement of theorem \cite[3.3.15]{Cannarsa Sinestrari}, with two minor remarks:
\begin{enumerate}
 \item The condition is restricted to points $p\in S$. At points in $ M\setminus S$, the balanced condition is trivial.
 \item In the balanced condition, we use the vectors $v_{p_i}(p)$ from $p_i$ to $p$, contrary to the reference \cite{Cannarsa Sinestrari}. Thus the minus sign in the statement.
\end{enumerate}

\end{proof}

In the light of this new proof, we can regard the balanced condition as a differential version of the semiconcavity condition. A semiconcave function that is a solution to \eqref{HJequation} is also a viscosity solution (see \cite[5.3.1]{Cannarsa Sinestrari}). We will later study if the solution of \eqref{HJequation} built by characteristics using a balanced split loci is also a viscosity solution.

\chapter{Local structure of cut and singular loci up to codimension 3}
\label{Chapter: structure}

Our main result in this chapter is a local description of the cut locus around any point of the cut locus \emph{except for a set of Hausdorff dimension $n-3$} (see Theorem \ref{complete description}).

Actually, our structure results hold for the more general \emph{balanced split loci} (recall that in \ref{cut locus is balanced} and \ref{regularity of mu for nontrivial g}
we showed that cut loci, and singular sets of solutions to HJ equations, are balanced split loci). Working in this generality complicates some proofs and, in particular, we have to prove some results for \emph{balanced split loci} that are long known to be true for cut loci.
However, we need to actually prove the structure results for balanced split loci for the applications to chapter \ref{Chapter: balanced}, and all the new proofs of old facts are either short or interesting for their own sake.

\section{Statements of results}\label{Section: results}

For the results of this chapter, $M$ is a  $C^{\infty}$ Finsler manifold with \emph{compact} boundary $\partial M$, but $M$ itself need not be compact. $S\subset M$ is a balanced split locus. Recall \ref{the set R_p} for the definition of $R_p$.

\subsection{Results}

Our main result asserts that we can avoid conjugate points of order $2$ and above if we neglect a set of Hausdorff dimension $n-3$:

\begin{theorem}[Conjugate points of order $2$]\label{main theorem 3}
There is a set $N\subset S$ of Hausdorff dimension at most $n-3$ such that for any $p\in S\setminus N$  and $x\in V$ such that $F(x)=p$ and $d_x F(r_{x})\in R_{p}$:
$$
\dim (\ker d_{x}F)\leq 1
$$
\end{theorem}

Combining this new result with previous ones in the literature, we are able to provide the following description of a cut locus. All the extra results required for the proof of this result will be proved in this chapter.

\begin{theorem}[The cut locus up to $\mathcal{H}$-codimension 3] \label{complete description}
Let $S$ be either the cut locus of a point or submanifold in a Finsler manifold or the closure of the singular locus of a solution of \ref{HJequation} and \ref{HJboundarydata}.
Then $S$ consists of the following types of points :

 \begin{itemize}
 \item\strong{Cleave points:} Points at which $R_{p}$ consists of two non-conjugate vectors. The set of cleave points is a smooth hypersurface;
 \item \strong{Edge points:} Points at which $R_{p}$ consists of exactly one conjugate vector of order 1. This is a set of Hausdorff dimension at most $n-2$;
 \item\strong{Degenerate cleave points:} Points at which $R_{p}$ consists of two vectors, such that one of them is conjugate of order 1, and the other may be non-conjugate or conjugate of order 1. This is a set of Hausdorff dimension at most $n-2$;
 \item\strong{Crossing points:} Points at which $R_{p}$ consists of non-conjugate and conjugate vectors of order 1, and $R^{\ast}_{p} $ spans an affine subspace of dimension $2$. This is a rectifiable set of dimension at most $n-2$;
 \item \strong{Remainder:} A set of Hausdorff dimension at most $n-3$;
\end{itemize}

\end{theorem}

In the next chapter, we will provide more detailed descriptions of how does a balanced split loci looks near each of these different points (see \ref{uniqueness near order 1 points}, \ref{structure of cleave points}, \ref{structure of crossing points1}, \ref{structure of crossing points2} and \ref{structure of crossing points3}).

In the next chapter and also in section \ref{section: proof for a 3-manifold with arbitrary metric} we show applications of this result, but we believe it can also be useful in other contexts. For instance, stochastic processes on manifolds is often studied on the complement of the cut locus from a point, and then the results have to be adapted to take care of the situation when the process hits the cut locus (see \cite{Barden Le}). Brownian motion, for example, almost never hits a set with null $\Hndos$ measure, but will almost surely hit any set with positive $\Hndos$ measure, so we think our result can be useful in that field.

\subsection{Examples}

We provide examples of Riemannian manifolds and exponential maps which illustrate our results.

First, consider a solid ellipsoid with two equal semiaxis and a third larger one. This is a 3D manifold with boundary, and the geodesics  starting at the two points that lie further away from the center have a first conjugate vector of order $2$ while remaining  minimizing up to that point.
This example shows that our bound on the Hausdorff dimension of the points in the cut locus with a minimizing geodesic of order $2$ cannot be improved.

Second, consider the surface of an ellipsoid with three different semiaxis (or any generic surface as in \cite{Buchner}, with metric close to the standard sphere) and an arbitrary point on it. It is known that in the tangent space the set of first conjugate points is a closed curve $C$ bounding the origin, and at most of these points the kernel of the exponential map is transversal to the curve $C$. More explicitely, the set $C^{\ast}$ of points of $C$ where it is not transversal is finite. Consider then the product $M$ of two such ellipsoids. The exponential map onto $M$ has a conjugate point of order $2$ at any point in $(C\setminus C^{\ast})\times (C\setminus C^{\ast})$, and the kernel of the exponential map is transversal to the tangent to $C\times C$. Thus the image of the set of conjugate points of order $2 $ is a smooth manifold of codimension $2$.

This example shows that theorem \ref{main theorem 3} does not hold for the image of all the conjugate points of order $2$, and only holds for the \emph{minimizing} conjugate points.

Finally, recall the construction in \cite{Gluck Singer}, where the authors build a riemannian surface whose cut locus is not triangulable. Their example shows that the set of points with a conjugate minimizing geodesic can have infinite $\Hndos$ measure. A similar construction replacing the circle in their construction with a 3d ball shows that the set of points with a minimizing geodesic conjugate of order $2$ can have infinite $\Hntres$ measure.

\subsection{Relation to previous results in the literature}\label{subsection: Relation to previous results in the literature}

Our structure theorem generalizes a standard result that has been proven several times by mathematicians from different fields (see for example \cite{Barden Le}\footnote{Although there is a mistake in their proof}, \cite{Hebda}, \cite{Mantegazza Mennucci} and \cite{Itoh Tanaka 98}):

\begin{quote}
A cut locus in a Riemannian manifold is the union of a smooth $(n-1)$-dimensional manifold $\mathcal{C}$ and a set of zero $(n-1)$-dimensional Hausdorff measure (actually, a set of Hausdorff dimension at most $n-2$). 
The set $\mathcal{C}$ consists of cleave points, which are joined to the origin or initial submanifold by exactly two minimizing geodesics, both of which are non-conjugate.
\end{quote}

This result follows from \ref{complete description}, since the union of edge, degenerate cleave, and crossing points is a set of Hausdorff dimension at most $n-2$.

The statement quoted above follows from lemmas \ref{theorem: conjugate of order 1}, \ref{cleave points are a manifold} and \ref{main theorem 4} only.
Theorem \ref{main theorem 3} is not necessary if a description is needed only up to codimension $2$.
The proof of the three lemmas is simple and has many features in common with earlier results on the cut locus.

In a previous paper, A. C. Mennucci studied the singular set of solutions to the HJ equations with only $C^k$ regularity. 
Under this hypothesis, the set $S\setminus \mathcal{C}$ may have Hausdorff dimension strictly between $n-1$ and $n-2$ (see \cite{Mennucci}).
We work only in a $C^{\infty}$ setting, and under this stronger condition, the set $S\setminus \mathcal{C}$ has always Haussdorf dimension at most $n-2$.

Our result \ref{complete description} uses the theory of singularities of semi-concave functions that can be found for example in \cite{Alberti Ambrosio Cannarsa}.
Though their result can be applied to a Finsler manifold, we had to give a new proof that applies to balanced split loci instead of just the cut locus.

\section{Conjugate points in a balanced split locus}\label{section: First}

In this section we prove Theorem \ref{main theorem 3}.
Throughout this section, $M$, $r$, $V$ and $F$ are as in section \ref{section: intro to exponential maps}
and $S$ is a balanced split locus as defined in \ref{balanced}.

\begin{dfn}\label{A2}
A singular point $x\in V$ of the map $F$ is an \emph{A2} point if $ker(dF_x)$ has dimension $1$ and is transversal to the tangent to the set of conjugate vectors.
\end{dfn}

\begin{remark} 
Warner shows in \cite{Warner} that the set of conjugate points of order $1$ is a smooth (open) hypersurface inside $V$, and that  for adequate coordinate functions in $V$ and $M$, the exponential has the following normal form around any A2 point,
\end{remark}

\begin{eqnarray}\label{normal form for A2} 
(x_{1},x_{2},\dots , x_{m})\longrightarrow
(x_{1}^{2},x_{2}, \dots, x_{m})
\end{eqnarray}

\begin{prop}\label{no A2 in S}
 For any $p\in M$ and $X\in R_{p}$, the vector $X$ is not of the form $dF_x(r)$ for any A2 point $x$.

\end{prop}
\begin{proof}
The proof is by contradiction. Let $p\in S$ be such that $R_{p}$ contains an A2 vector $Z=dF_{c}(r_{c})$, for $c\in Q_p$.
By the normal form \eqref{normal form for A2}, we see there is a neighborhood
$U$ of $c$ such that no other point in $U$ maps to $p$. 
Furthermore, in a neighborhood $B$ of $p$ the image of the conjugate vectors is a hypersurface $H$ such that all points at one side (call it $B_{1}$) have two preimages of $F\vert_{U}$, all points at the other side $B_{2}$ of $H$ have no preimages, and points at $H$ have one preimage, whose corresponding vector is A2-conjugate.
It follows that $Z$ is isolated in $R_{p}$.

We notice there is a sequence of points $p_{n}\rightarrow p$ in $B_{2}$ with vectors $Y_{n}\in R_{p_{n}}$ such that  $Y_{n}\rightarrow Y\neq X$. Thus $R_{a}$ does not reduce to $Z$.

The vector $Z$ is tangent to $H$, so we can find a sequence of points $p_{n}\in B_{2}$ approaching $p$ such that
$$
\lim_{n\rightarrow \infty}v_{p_{n}}(p)=Z
$$
We can find a subsequence $p_{n_{k}}$ of the $p_{n}$ and vectors $X_{k}\in R_{p_{n_{k}}}$ such that $X_{k}$ converges to some $X_{\infty}\in R_{p}$. By the above, $X_{\infty}$ is different from $Z$, but $\hat{Z}(X)<1= \hat{Z}(Z)$ (where $\hat{Z}$ is the dual form to $Z$), so the balanced property is violated.
\end{proof}

The following is the analogous to Theorem \ref{main theorem 3} for conjugate points of order $1$.

\begin{prop}[Conjugate points of order $1$]\label{theorem: conjugate of order 1}
There is a set $N\subset S$ of Hausdorff dimension $n-2$ such that for all $p\in S\setminus N$, $R_p$ does not contain conjugate vectors.
\end{prop}
\begin{proof}
The proof is identical to the proof of lemma 2 in \cite{Itoh Tanaka 98} for a cut locus, but we include it here for completeness.
First of all, at the set of conjugate vectors of order $k\geq 2$ we can apply directly the Morse-Sard-Federer theorem (see \cite{Federer}) to show that the image of the set of conjugate cut vectors of order $k\geq 2$ has Hausdorff dimension at most $n-2$.

Let $Q$ be the set of conjugate vectors of order $1$ (recall it is a smooth hypersurface in $V$).
Let $G$ be the set of conjugate vectors such that the kernel of $dF$ is tangent to the conjugate locus. Apply the Morse-Sard-Federer theorem again to the map $F\vert_{Q}$ to show that the image of $G$ has Hausdorff dimension at most $n-2$. Finally, the previous result takes cares of the $A2$ points.
\end{proof}

We now turn to the main result of this paper: we state and prove Theorem \ref{theorem: conjugate of order 2} which has \ref{main theorem 3} as a direct consequence.

\begin{theorem} \label{theorem: conjugate of order 2}
Let $M$, $V$, $F$ and $r$ be as in section \ref{section: intro to exponential
maps}, and let $S$ be a balanced split locus 
\eqref{balanced}.
The set of conjugate points of order $2$ in $V$ decomposes as the union of two subsets $Q_{2}^{1}$ and $Q_{2}^{2}$ such that:
\begin{itemize}
\item No point in $Q_{2}^{1}$ maps under $dF$ to a vector in any of the $R_{a}$ (in set notation: $d F(Q_2^1)\cap \left(\cup R_p\right)=\emptyset$)

\item The image under $F$ of $Q_{2}^{2}$ has Hausdorff dimension at most $n-3$.
\end{itemize}
\end{theorem}

\begin{proof} 

Let $z$ be a conjugate point of order $2$ and take special coordinates at $U_{z}$ near $z$.
In the special coordinates near $z$ (see \ref{subsection: special coordinates}), $F$ is written:
\begin{equation}
 \label{exp in local coords} 
F(x_{1},\dots,x_{n})=(x_{1},\dots,x_{n-2},F_z^{n-1}(x), F_z^n(x))
\end{equation}
for some functions $F_z^{n-1}$ and $F_z^n$, and $x=(x_{1},\dots,x_{n})$ in a neighborhood $U_{z}$ of $z$ with $F(0,\dots,0)=(0,\dots,0)$.

The Jacobian of $F$ is:
$$JF = \left[ \begin{array}{ccccc}
  1 & \ldots & 0 & \ast & \ast\\
  \vdots & \ddots & \vdots & \vdots & \vdots\\
  0 & \ldots & 1 & \ast & \ast\\
  0 & \ldots & 0 & \frac{\partial F_{z}^{n-1}}{\partial x_{n-1}} &
  \frac{\partial F_{z}^{n}}{\partial x_{n-1}}\\
  0 & \ldots & 0 &  \frac{\partial F_{z}^{n-1}}{\partial x_{n}} &
  \frac{\partial F_{z}^{n}}{\partial x_{n}}
\end{array} \right]  $$

A point $x$ is of order $2$ if and only if the $2\times 2$ submatrix for the $x_{n-1}$ and $x_n$ coordinates (and the corresponding coordinates in $F(U_z)$: $y_{n-1}$ and $y_n$) vanish:
\begin{equation}\label{minor of vars n-1 and n}
\left[ \begin{array}{cc}
  \frac{\partial F_{z}^{n-1}}{\partial x_{n-1}} & \frac{\partial
  F_{z}^{n}}{\partial x_{n-1}}\\
  \frac{\partial F_{z}^{n-1}}{\partial x_{n}} & \frac{\partial
  F_{z}^{n}}{\partial x_{n}}
\end{array} \right]
=0
\end{equation} 
Near a point of order $2$, we write:
$$
F_{z}^{n-1}(x)=x_{1}x_{n-1}+q(x_{n-1},x_{n})+T^{n-1}(x)
$$
$$
F_{z}^{n}(x)=x_{1}x_{n}+r(x_{n-1},x_{n})+T^{n}(x)
$$
where
$ q(x_{n-1},x_{n})$ and $r(x_{n-1},x_{n}) $ are the quadratic terms in $x_{n-1}$ and $x_{n}$ in a Taylor expansion, and
$T$ consists of terms of order $\geq 3$ in $x_{n-1}$ and $x_{n}$, and terms of order $\geq 2$ with at least one $x_{i}, i\leq n-2$.

The nature of the polynomials $q$ and $r$ in the special coordinates at $z$ will determine whether $z$ is in $Q_{2}^{1}$ or in $Q_{2}^{2} $. We have the following possibilities:

\begin{enumerate}
 \item either $q$ or $r$ is a sum of squares of homogeneous linear functions in $x_{n-1}$ and $x_{n}$ (possibly with a global minus sign).
 \item both $q$ and $r$ are products of distinct linear functionals (equivalently, they are difference of squares). Later on, we will split this class further into three types: 2a, 2b and 2c.
 \item one of $q$ and $r$ is zero, the other is not.
 \item both $q$ and $r$ are zero.
\end{enumerate}

We set $Q_{2}^{1}$ to be the points of type 1 and 2c, and $Q_{2}^{2}$ to be the points of type 2a, 3 and 4. Points of type 2b do not appear under the hypothesis of this theorem.

\paragraph{Type 1.}
The proof is similar to Proposition \ref{no A2 in S}. Assume $z=(0,\dots,0)$ is of type 1. If, say, $q$ is a sum of squares, then in the
set $\{x_{1}=a,x_{2}=\dots=x_{n-2}=0\}\cap F(U_z)$, $x_{n-1}$ will reach a minimum value that will be greater than $-Ca^{2}$ for some $C>0$.
We learn there is a sequence $p^{k}=(t^{k},0,\dots,x_{n-2},-(C+1) (t^{k})^{2},0)$, for $t^{k}\nearrow 0$, approaching $(0,\dots,0)$  with incoming speed $(1,0,\dots,0)$ and staying in the interior of the complement of $F(U_z)$ for $k$ large enough.
Pick up any vectors $V_{k}\in R_{p^{k}}$ converging to some $V_{0}$ (passing to a subsequence if necessary). Then $V_{0}$ is different from $(1,0,\dots,0)\in R_{0}$, and 
$$
\widehat{V_{0}}\left( (1,\dots,0)\right) <\widehat{(1,\dots,0)}\left( (1,\dots,0)\right)=1
$$
violating the balanced condition.

\paragraph{Type 2 and 3.}
If a point is of type 2 or 3, we can assume $q\neq 0$. Before we proceed, we change coordinates to simplify the expression of $F$ further. Consider a linear change of coordinates near $x$ that mixes only the $x_{n-1}$ and $x_{n}$ coordinates.
$$
\left( 
\begin{array}{c}
 x'_{n-1}\\
 x'_{n}
\end{array}
\right) 
=
A\cdot 
\left( 
\begin{array}{c}
 x_{n-1}\\
 x_{n}
\end{array}
\right)
$$
followed by the linear change of coordinates near $p$ that mixes only the $y_{n-1}$ and $y_{n}$ coordinates with the inverse of the matrix above:
$$
\left( 
\begin{array}{c}
 y'_{n-1}\\
 y'_{n}
\end{array}
\right) 
=
A^{-1}\cdot 
\left( 
\begin{array}{c}
 y_{n-1}\\
 y_{n}
\end{array}
\right)
$$
Straightforward but tedious calculations show that there is a matrix $A$ such that the map $F$ has the following expression in the coordinates above:
$$
F(x_{1},\dots,x_{n})=(x_{1},\dots,x_{n-2},x_{1}x_{n-1}+(x_{n-1}^{2}-x^{2}_{n}),x_{1}x_{n}+r(x_{n-1},x_{n}))+T
$$
In other words, we can assume $q(x_{n-1},x_{n})=(x_{n-1}^{2}-x^{2}_{n})$.

Fix small values for all $x_{i}$ for $2\leq i\leq n-2$. At the origin, $JF$ is a diagonal matrix with zeros in the positions $(n-1,n-1)$ and $(n,n)$.
We recall that $z$ is conjugate of order $2$ iff the submatrix 
\eqref{minor of vars n-1 and n} vanishes. This submatrix is the sum of
\begin{eqnarray}\label{jacobian of the interesting part}
\begin{bmatrix}
x_{1}+2x_{n-1} & r_{x_{n-1}}\\
 -2x_{n} & x_{1}+r_{x_{n}}
\end{bmatrix}
\end{eqnarray} 
and some terms that either have as a factor one of the $x_{i}$ for $2\leq i\leq n-2$, or are quadratic  in $x_{n-1}$ and $x_{n}$.

We want to show that, near points of type 3 and some points of type 2, all conjugate points of order $2$ are contained in a submanifold of codimension $3$.
The claim will follow if we show that the gradients of the four entries span a $3$-dimensional space at points in $U$. For convenience, write $r(x_{n-1},x_{n})=\alpha x_{n-1}^{2}+\beta x_{n-1}x_{n}+\gamma x_{n}^{2}$. It is sufficient that the matrix with the partial derivatives with respect to $x_{i}$ for $i\in\{1,n-1,n\}$ of the four entries have rank $3$:
$$
A=
\begin{bmatrix}
1 & 2 & 0\\
0 & 0 & -2\\
0 & 2\alpha & \beta\\
1 & \beta & 2\gamma
\end{bmatrix}
$$
The claim holds if all $x_{i}$ are small, for $i\not\in\{1,n-1,n\}$, unless $\alpha=0$ and $\beta=2$. This covers points of type $3$. We say a point of type 2 has type 2a if the rank of the above matrix is $3$. Otherwise, the polynomial $r$ looks:
$$
r(x_{n-1},x_{n})=2 x_{n-1}x_{n}+\gamma x_{n}^{2}=2x_{n}(x_{n-1}+\frac{\gamma}{2}x_{n})
$$
We say a point of type 2 has type 2b if $r$ has the above form and $-1<\frac{\gamma}{2} <1$. We will show that there are integral curves of $r$ arbitrarily close to the one through $z$ without conjugate points near $z$, which contradicts property 3 of exponential maps in Proposition \ref{regular exponential map}.

Take a ray $t\rightarrow \zeta_{x_{n}}(t)$ passing through a point $(0,\dots,0,x_{n}) $.
The determinant of \ref{minor of vars n-1 and n} along the ray is:

$$
\begin{array}{rl}
d(t)&=
 \frac{\partial F_{z}^{n-1}}{\partial x_{n-1}} (\zeta(t))
 \frac{\partial F_{z}^{n}}{\partial x_{n}}  (\zeta(t)) -
 \frac{\partial F_{z}^{n}} {\partial x_{n-1}} (\zeta(t))
 \frac{\partial F_{z}^{n-1}} {\partial x_{n}} (\zeta(t))\\
&=t^{2}+t(4x_{n-1}+2\gamma x_{n})+(4x_{n-1}^{2}+4\gamma x_{n-1}x_{n}+4x_{n}^{2})
+R_{3}(x_{n},t)\\
&= (t+2x_{n-1}+\gamma x_{n})^{2}+(4-\gamma^{2})x_{n}^{2}
+R_{3}(x_{n},t)\\
&\geq c(t^{2}+x_{n}^{2})+R_{3}(x_{n},t)
\end{array}
$$
for a remainder $R_{3}$ of order $3$. Thus there is a $\delta >0$ such that for any $x_{n}\neq 0$ and $\vert t\vert<\delta$, $\vert x_{n}\vert<\delta$, $\zeta_{x_{n}}(t)$ is not a conjugate point.

We have already dealt with points of type 3, 2a and 2b.
Now we turn to the rest of points of type 2 (type 2c). We have either $\frac{\gamma}{2} \geq 1$ or $\frac{\gamma}{2} \leq -1$.
We notice that $x_{n-1}^{2}- x_{n}^{2}\leq 0$ iff $\vert x_{n-1}\vert\leq \vert x_{n}\vert $, but whenever $\vert x_{n-1}\vert\leq \vert x_{n}\vert $, the sign of $r(x_{n-1},x_{n})$ is the sign of $\gamma$. Thus the second order part of $F$ maps $U$ into the complement of points with negative second coordinate and whose third coordinate has the opposite sign of $\gamma$.

A similar argument as the one for type 1 points yields a contradiction with the balanced condition. If, for example, $\gamma \geq 2$, none of the following points 
$$
x^{k}=(t^{k},-(C+1) (t^{k})^{2},-(C+1) (t^{k})^{2},..0,)
$$ 
is in $F(U)$, for $t^{k}\rightarrow 0$. But then we can carry a vector other than $(1,0,\dots,0)$ as we approach $F(x_{0})$.

\paragraph{Type 4.}
Let $z$ be a conjugate point of order $2$.
We show that the image of the points of type 4 inside $U_{z}$ has Hausdorff dimension at most $n-3$. $U_z$ is an open set around an arbitrary point $z$ of order $2$, and thus the result follows.

First, we find that for any point $x$ of type 4, we have $d^2_x F(v\sharp w)=0$ for all $v,w\in \ker d_x F$, making the computation in the special coordinates at $x\in U_z$ (see section \ref{subsection: regular exponential map} for the definition of $d^2 F$).

Then we switch to the special coordinates around $z$. 
In these coordinates, the kernel of $dF$ at $x$ is generated by $\frac{\partial }{\partial x_{n-1}} $ and $\frac{\partial }{\partial x_{n}} $.
Thus $\frac{\partial^2 F_{z}^{n-1}}{\partial x_{i}x_{j}} =0$ for $i,j\geq n-1$ at any point $x\in U_z$ of type 4.

The set of conjugate points of order $2$ is contained in the set
$H\!=\!\lbrace\frac{\partial F_{z}^{n-1}}{\partial x_{n-1}} (x)\!=\!0 \rbrace$.
This set is a smooth hypersurface: the second property in \ref{regular
exponential map} implies that $\frac{\partial^2 F_{z}^{n-1}}{\partial
x_{1}x_{n-1}} \neq 0 $ at points of $H$.
At every conjugate point of type 4, the kernel of $dF$ is contained in the
tangent to $H$. Thus conjugate points of type 4 are conjugate points of the
restriction of $F$ to $H$.
The Morse-Sard-Federer theorem applies, and the image of the set of points of
type 4 has Hausdorff dimension $n-3$.\end{proof}

\begin{proof}[Proof of Theorem \ref{main theorem 3}]
Follows immediately from the above, setting \linebreak[4] $N=F(Q^{2}_{2})$.
\end{proof}

\section{Structure up to codimension 3}\label{section: Second}

This section contains the proof of \ref{complete description}, splitted into several lemmas. All of them are known for cut loci in riemannian manifolds, but we repeat the proof so that it applies to balanced split loci in Finsler manifolds.

\begin{dfn}\label{cleave point}
We say $p\in S$ is a \emph{cleave} point iff $R_{p}$ has two elements $X^{1}$ and $X^{2}$, with $(p,X^{1})=(F(y_{1}),dF_{y_{1}}(r_{y_{1}}))$ and $(p,X^{2})=(F(y_{2}),dF_{y_{2}}(r_{y_{2}}))$, and both $dF_{y_{1}}$ and $dF_{y_{2}}$ are non-singular.

In other words, $p\in S$ is a cleave point iff $R_p$ consists of two non-conjugate vectors.
\end{dfn}

\begin{prop}\label{cleave points are a manifold}
$\mathcal{C}$ is a $(n-1)$-dimensional manifold.
\end{prop}

\begin{proof}
Let $p=F(y_{1})=F(y_{2})$ be a cleave point, with $R_{p}\!=\!\{dF_{y_{1}}(r),
dF_{y_{2}}(r)\}$.
We can find a small neighborhood $U$ of $p$ so that the following conditions are satisfied:
\begin{enumerate}
 \item $U$ is the diffeomorphic image of neighborhoods $U_{1}$ and $U_{2}$ of the points $y_{1}$ and $y_{2}$.
Thus, the two smooth vector fields $X^{1}_{q}=dF\vert_{U_{1}}(r)$ and $X^{2}_{q}=dF\vert_{U_{2}}(r)$ are defined in points $q\in U$.
 \item At all points $q\in U$, $R_{q}\subset \lbrace X^{1}_{q},X^{2}_{q}\rbrace$. Other vectors must be images of the vector $r$ at points not in $U_{1}$ or $U_{2}$, and if they accumulate near $p$ we could find a subsequence converging to a vector that is neither $X_{1}$ nor $X_{2}$.
 We reduce $U$ if necessary to achieve the property.
 \item Let $H$ be an hypersurface in $U_{1}$ passing through $y_{1}$ and transversal to $X_{1}$, and let $\tilde{H} =F(H)$. We define local coordinates $p=(z, t)$ in $U$, where $z\in \tilde{H}$ and $t\in \RR$ are the unique values for which $p$ is obtained by following the integral curve of $X^{1}$ that starts at $x$ for time $t$.
$U$ is a cube in these coordinates.

\end{enumerate}

We will show that $S$ is a  graph in the coordinates $(z,t)$.
Let $A_{i}$ be the set of points $q$ for which $R_{q}$ contains $X^{i}_{q}$, for $i=1,2$. 
By the hypothesis, $S=A_{1}\cap A_{2}$.

Every tangent vector $v$ to $S$ at $q\in S$ (in the sense of \ref{approximate tangent cone}), satisfies the following property (where $\hat{X}$ is the dual covector to a vector $X\in TM$.):
$$
\hat{X}^{i}(v)=
\max_{Y\in R_{p}}
\hat{Y}(v)
$$
which in this case amounts to $\hat{X}^{1}(v)=\hat{X}^{2}(v)$, or
$$
v\in \ker(\hat{X}^{1}-\hat{X}^{2})
$$

We can define in $U$ the smooth distribution $D=\ker(\hat{X}^{1}-\hat{X}^{2})$. 
$S$ is a closed set whose approximate tangent space is contained in $D$.

We first claim that for all $z$, there is at most one time $t_{0}$ such that $(z,t_{0})$ is in $S$. 
If $(z,t)$ is in $A_{1}$, $R_{(z,t)}$ contains $X^{1}$ and, unless $(z,s)$ is contained in $A_{1}$ for $s$ in an interval $(t-\varepsilon,t)$, we can find a sequence $(z_{n},t_{n})$ converging to $(z,t)$ with $t_{n}\nearrow t$ and carrying vectors $X^{2}$. The incoming vector is $X^{1}$, but
$$
\tilde{X^2}(X^1)<\tilde{X^1}(X^1)=1
$$
which contradicts the balanced property.
Analogously, if $R_{(z,t)}$ contains $X^{2}$ there is an interval $(t,t+\varepsilon)$ such that $(z,s)$ is contained in $A_{2}$ for all $s$ in the interval. Otherwise there is a sequence $(z_{n},t_{n})$ converging to $(z,t)$ with $t_{n}\searrow t$ and carrying vectors $X^{1}$. The incoming vector is $-X^{1}$, but
$$
-1=\tilde{X^1}(-X^1)<\tilde{X^2}(-X^1)
$$
which is again a contradiction. The claim follows easily.

We show next that the set of $p$ for which there is a $t$ with $(z,t)\in S$ is open and closed in $\Gamma$, and thus $S$ is the graph of a function $h$ over $\Gamma$.
Take $(z,t)\in U\cap S$ 
and choose a cone $D_{\varepsilon}$ around $D_{p}$. We can assume the cone intersects $\partial U$ only in the $z$ boundary. There must be a point in $S$ of the form $(z',t')$ inside the cone for all $z'$ sufficiently close to $z$: otherwise there is either a sequence $(z_n,t_n)$ approaching $(z,t)$ with $t_n>h_+(z)$ ($h$ being the upper graph of the cone $D_{\varepsilon}$) and carrying vectors $X^{1}$ or a similar sequence with $t_n<h_-(z)$ and carrying vectors $X^{2}$.
Both options violate the balanced condition. 
Closedness follows trivially from the definition of $S$.

Define $t=h(z)$ whenever $(z,t)\in S$. The tangent to the graph of $h$ is given by $D$ at every point, thus $S$ is smooth and indeed an integral maximal submanifold of $D$.
\end{proof}

\begin{remark}
It follows from the proof above that there cannot be any balanced split locus unless $D$ is integrable. This is not strange, as the sister notion of cut locus does not make sense if $D$ is not integrable. 
\end{remark}

We recall that the orthogonal distribution to a geodesic vector field is parallel for that vector field, so the distribution is integrable at one point of the geodesic if and only if it is integrable at any other point. In particular, if the vector field leaves a hypersurface orthogonally (which is the case for a cut locus) the distribution $D$ (which is the difference of the orthogonal distributions to two geodesic vector fields) is integrable. It also follows from \ref{regularity of mu for nontrivial g} that the characteristic vector field in a Hamilton-Jacobi problem has an integrable orthogonal distribution.

\begin{remark}
In the next chapter we study whether a balanced split locus is actually a cut locus. The proof of the above lemma showed there is a unique sheet of cleave points near a given point in a balanced split loci.
\end{remark}

\begin{prop}\label{main theorem 4}
The set of points $p\in S$ where $co\, (R^{\ast}_{p})$ has dimension $k$ is $(n-k)$-rectifiable.
\end{prop}
\begin{proof}
Throughout the proof, let $\hat{X}$ be the dual covector to the vector $X\in TM$.

Let $p_{n}$ be a sequence of points such that $co\,(R^{\ast}_{p_{n}})$ contains a $k$-dimensional ball of radius greater than $\delta$. Suppose they converge to a point $p$ and $v_{p_{n}}(p)$ converges to a vector $\eta$.

We take a neighborhood $U$ of $p$ and fix product coordinates in $\pi^{-1}(U)$ of the form $U\times \Rn$.
Then, we extract a subsequence of $p_{n}$ and vectors $X_{n}^{1}\in R_{p_{n}}$ such that $X_{n}^{1}$ converge to a vector $X^{1}$ in $R_{p}$. Outside a ball of radius $c\delta$ at $\hat{X}_{n}^{1}$, where $c$ is a fixed constant and $n>>0$, there must be vectors in $R_{p_{n}}$, and we can extract a subsequence of $p_{n}$ and vectors $X_{n}^{2}$ converging to a vector $X^{2}$ such that $\hat{X}^{2}$ is at a distance at least $c\delta$ of $\hat{X}^{1}$. Iteration of this process yields a converging sequence $p_{n}$ and $k$ vectors 
$$
X_{n}^{1}, .. ,X_{n}^{k}\in R_{p_{n}}
$$
converging to vectors
$$
X^{1}, .. ,X^{k}\in R_{p}
$$
such that the distance between $\hat{X}^{k}$ and the linear span of $\hat{X}^{1},..\hat{X}^{k-1}$ is at least $c\delta$, so that $coV^{\ast}_{p}$ contains a $k$-dimensional ball of radius at least $c'\delta$.

The balanced property implies that the $\hat{X}^{j}$ evaluate to the same value at $\eta$, which is also the maximum value of the $\hat{Z}(\eta)$ for a vector $Z$ in $R_{p}$. In other words, the convex hull of the $\hat{X}^{j}$ belong to the face of $R^{\ast}_{p}$ that is exposed by $\eta$. If $co\,R_{p}^{\ast}$ is $k$-dimensional, $\eta$ belongs to 
$$
\begin{array}{rl}
 \left( co\,R_{p}^{\ast}\right)^{\perp} =&
\left\lbrace 
    v\in T_{p} M \::\quad  \langle w,v\rangle \text{ is constant for }w\in co\,R_{p}^{\ast}
\right\rbrace \\[2ex]
=&
\left\lbrace 
    v\in T_{p} M \::\quad  \langle \hat{X},v\rangle \text{ is constant for }X\in R_{p}
\right\rbrace
\end{array} 
$$
which is a $n-k$ dimensional subspace.

Let $\Sigma^{k}_{\delta}$ be the set of points 
$p\in S$ for which $co\,R_{p}^{\ast}$ is $k$-dimensional and contains a $k$-dimensional ball of radius greater than or equal to $\delta$.
We have shown that all tangent directions to $\Sigma^{k}_{\delta}$ at a point $p$ are contained in a $n-k$ dimensional subspace. We can apply theorem 3.1 in \cite{Alberti Ambrosio Cannarsa} to deduce $\Sigma^{k}_{\delta}$ is $n-k$ rectifiable, so their union for all $\delta>0$ is rectifiable too.

\end{proof}

\chapter{Balanced split sets and Hamilton-Jacobi equations}
\label{Chapter: balanced}

\section{Introduction}

In this chapter we consider the Hamilton-Jacobi boundary value problem 
\eqref{HJequation} and \eqref{HJboundarydata} in a compact set $ M$.

A local \emph{classical} solution can be computed near $\partial  M$ following \emph{characteristic} curves as in section \ref{subsection: characteristics of the HJBVP}.

A unique \emph{viscosity solution} is given by the Lax-Oleinik formula 
\eqref{Lax-Oleinik}.

The viscosity solution can be thought of as a way to extend the classical solution to the whole $ M$.

Recall from section \ref{section: Balanced} that the singular set $Sing$ is a \emph{balanced split locus}.
This notion was inspired originally by the paper \cite{Itoh Tanaka 00}, but is also related to the notion of \emph{semiconcave} functions that is now common in the study of Hamilton-Jacobi equations (see section \ref{section: Balanced}). 
Our goal in this chapter is to determine whether there is a unique balanced split locus. In the cases when this is not true, we also give an interpretation of the multiple balanced split loci.

\subsection{Outline} 

In section \ref{section: balanced, statement} we state our results, give examples, and comment on possible extensions. Section \ref{section: balanced preliminaries} gathers some of the results from the literature we will need, and includes a few new lemmas that we use later. Section \ref{section:rho is Lipschitz} contains our proof that the \emph{distance to a balanced split locus} and \emph{distance to the $k$-th conjugate point} are Lipschitz. Section \ref{section:proof of the main theorems} contains the proof of the main theorems, modulo a result that is proved in section \ref{section: proof that partial T=0}. This last section also features detailed descriptions of a balanced split set at each of the points in the classification in theorem \ref{complete description}.

\section{Statement of results.}\label{section: balanced, statement}

\subsection{Results}

For \emph{fixed} $ M$, $H$ and $g$ satisfying the conditions stated earlier,
there is always at least one balanced split locus, namely the singular set of
the solution of 
\eqref{HJequation} and \eqref{HJboundarydata}.
In general, there might be more than one balanced split loci, depending on the topology of $ M$. 

Our first theorem covers a situation where there is uniqueness.

\begin{theorem}\label{maintheorem0}
Assume $ M$ is simply connected and $\partial  M$ is connected.

Then there is a unique balanced split locus, which is the singular locus of the
solution of 
\eqref{HJequation} and \eqref{HJboundarydata}.
\end{theorem}

The next theorem removes the assumption that $\partial  M$ is connected, and uniqueness goes away:

\begin{theorem}\label{maintheorem1}
Assume $ M$ is simply connected and $\partial  M$ has several connected components.
Let $S\subset  M$ be a balanced split locus.

Then $S$ is the singular locus of the solution of 
\eqref{HJequation} and \eqref{HJboundarydata} with boundary data $g+a$ where
the function $a$ is constant at each connected component of $\partial M$.
\end{theorem}

The above theorem describes precisely all the balanced split loci in a situation where there is non-uniqueness. If $ M$ is not simply connected, the balanced split loci are more complicated to describe. We provide a somewhat involved procedure using the universal cover of the manifold. However, the final answer is very natural in the light of the examples.

\begin{theorem}
There exists a bijection between balanced split loci for given $ M $, $H$ and $g$ and an open subset of the homology space $H^1( M, \partial M)$ containing zero.
\end{theorem}

In fact, this theorem follows immediately from the next, where we construct such bijection:

\begin{theorem}\label{maintheorem2}
Let $\widetilde{ M}$ be the universal cover of $ M$, and lift both $H$ and $g$ to $\widetilde{ M} $.

Let $a:[\partial \widetilde{ M}]\rightarrow \RR $ be an assignment of a constant
to each connected component of $\partial \widetilde{ M} $ that is equivariant
for the action of the automorphism group of the covering and such that
$\widetilde{g}(z)+a(z)$ satisfies the compatibility condition 
\eqref{compatibility condition} in $\widetilde{ M} $. Then 
the singular locus $\widetilde{S}$ of the solution $\widetilde{u} $ to:

$$
\widetilde{H}(x,d\widetilde{u}(x))=1 \quad x\in \widetilde{ M}
$$
$$
\widetilde{u}(x)=\widetilde{g}(x)+a(z)  \quad x\in \partial\widetilde{ M}
$$
\noindent is invariant by the automorphism group of the covering, and its quotient is a set $S$ that is a balanced split locus for $ M $, $H$ and $g$. Furthermore:

\begin{enumerate}
 \item The procedure above yields a bijection between balanced split loci for given $ M $, $H$ and $g$ and \emph{equivariant compatible} functions 
$a:[\partial \widetilde{ M}]\rightarrow \RR $.
 \item Among the set of equivariant functions $a:[\partial \widetilde{ M}]\rightarrow \RR $ (that can be identified naturally with $H^1( M, \partial M)$), those compatible correspond to an open subset of $H^1( M, \partial M)$ that contains $0$.
\end{enumerate}

\end{theorem}

\paragraph{Remark.}
The space $H^1( M, \partial M)$ is dual to $ H_{n-1}( M)$ by Lefschetz theorem. The proof of the above theorems rely on the construction from $S$ of a $(n-1)$-dimensional current $T_S$ that is shown to be closed and thus represents a cohomology class in $H_{n-1}( M)$. The proof of the above theorem also shows that the map sending $S$ to the homology class of $T_S$ is a bijection from the set of balanced split loci onto a subset of $H_{n-1}( M)$.

In order to prove these theorems we will make heavy use of some \emph{structure results} for balanced split loci. To begin with, we start with the results from the previous chapter, specifically theorem \ref{complete description}. In the last section, we prove \emph{new structure results} in order to improve the description of balanced split loci near each of these types of points (see \ref{uniqueness near order 1 points}, \ref{structure of cleave points}, \ref{structure of crossing points1}, \ref{structure of crossing points2} and \ref{structure of crossing points3}).

We also study some very important functions for the study of the cut locus. Recall the global coordinates in $V$ given by $z\in \partial  M $ and $t\in \RR$.
Let $\lambda_j(z)$ be the value of $t$ at which the geodesic $s\rightarrow \Phi (s,z)$ has its $j$-th conjugate point (counting multiplicities), or $\infty$ if there is no such point.
Let $\rho_S:\partial  M \rightarrow \RR$ be the minimum $t$ such that $F(t,z)\in S$.

\begin{lem}\label{landa es Lipschitz} 
All functions $\lambda_j: \partial M \rightarrow \RR $ are Lipschitz continuous.
\end{lem}

\begin{lem}\label{rho is lipschitz}
The function $\rho_{S}:\partial  M \rightarrow \RR$ is Lipschitz continuous if $S$ is balanced.
\end{lem}

Both results were proven in \cite{Itoh Tanaka 00} for Riemannian manifolds, and the second one was given in \cite{Li Nirenberg}. 
Thus, our results are not new for a cut locus, but the proof is different from the previous ones and may be of interest.
We have recently known of another proof that $\rho$ and $\lambda_1$ are Lipschitz (\cite{Castelpietra Rifford}).

\subsection{Examples}

Take as $ M$ any ring in a euclidean $n$-space bounded by two concentric spheres. Solve the Hamilton-Jacobi equations with $H(x,p)=\vert p\vert$ and $g=0$. The solution is the distance to the spheres, and the cut locus is the sphere concentric to the other two and equidistant from each of them. However, any sphere concentric to the other two and lying between them is a balanced split set, so there is a one parameter family of split balanced sets. When $n>2$, this situation is a typical application of \ref{maintheorem1}. In the $n=2$ case, there is also only one free parameter, which is in accord with \ref{maintheorem2}, as the rank of the $H_1$ homology space of the ring is one.

For a more interesting example, we study balanced split sets with respect to a point in a euclidean torus.
We take as a model the unit square in the euclidean plane, centered at the origin, with its borders identified.
It is equivalent to study the distance with respect to a point in this euclidean torus, or the solution to Hamilton-Jacobi equations with respect to a small disc centered at the origin with the Hamiltonian $H(p)=\vert p\vert$ and $g=0$.

\begin{figure}[H]
 \centering
 \includegraphics[width=0.5\textwidth]{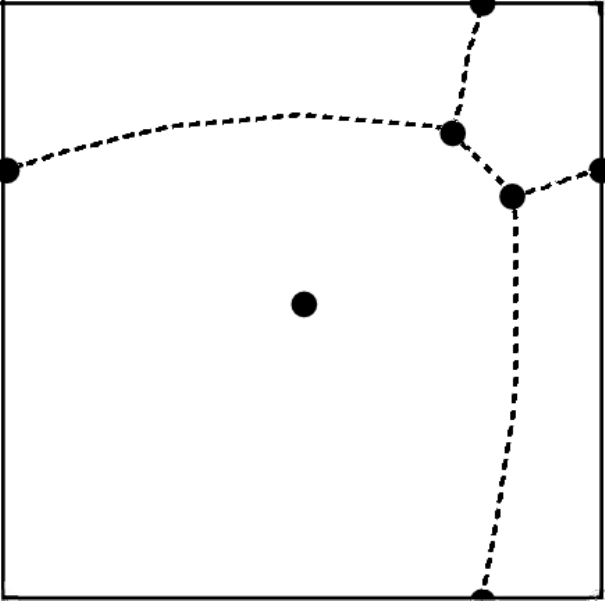}
 \caption{Balanced split set in a torus}
 \label{fig: Balanced split set in a euclidean torus}
\end{figure}

A branch of cleave points (see \ref{complete description}) must keep constant the difference of the distances from either sides (recall the proof of prop \ref{cleave points are a manifold}, or read the beginning of section \ref{section:proof of the main theorems}). Moving to the covering plane of the torus, we see they must be segments of hyperbolas. A balanced split locus is the union of the cleave segments and a few triple or quadruple points. The set of all balanced split loci is a $2$-parameter family, as predicted by our theorem \ref{maintheorem2}.

\section{Preliminaries}\label{section: balanced preliminaries}

\subsection{Lagrangian submanifolds of $T^\ast  M$}\label{subsection: lagrangian}
\begin{dfn}\label{dfn: Lagrangian submanifold}
 The \strong{canonical symplectic form} in $T^\ast  M$ is given in canonical coordinates by
 $$
 \sum_i d p_i\wedge d q_i
 $$
 A submanifold $L\subset T^\ast  M$ is \strong{Lagrangian} iff the restriction of the canonical symplectic form to $L$ vanishes.
\end{dfn}

Let $D$ be the duality homeomorphism between $T  M$ and $T^\ast  M$ induced by the Finsler metric as in definition \ref{dual one form} ($D$ is actually a $C^{\infty}$ diffeomorphism away from the zero section). We define a map:
\begin{equation}\label{definition of the embedding of V into TastOmega} 
\Delta(t,z) = D(\Phi_t(\Gamma(z)))
\end{equation} 
and a subset of $T^\ast M$:
\begin{equation} \label{definition of Theta} 
\Theta = \Delta(V) = D(W)
\end{equation} 
where $\Phi_t$ is the geodesic flow in $T M$. This is a smooth $n$-submanifold of $T^\ast  M$ with boundary.

It is a standard fact that, for a smooth function $u: M\rightarrow \RR$, the graph of its differential $du$ is a Lagrangian submanifold of $T^\ast  M$, for the canonical symplectic structure in $T^\ast  M$.
The subset of $\Theta$ corresponding to small $t$  is the graph of the differential of the solution by characteristics $u$ to the HJ equations.
Indeed, all of $\Theta$ is a lagrangian submanifold of $T^\ast  M$ when $\Gamma$  comes from an exponential map. As we have seen, this covers HJBVPs as well.

We can also carry over the geodesic vector field from $T  M$ into $T^\ast  M$ (outside the zero sections). This vector field in $T^\ast  M$ is tangent to $\Theta$. Then, as we follow an integral curve $\gamma(t)$ within $\Theta$, the tangent space to $\Theta$ describes a curve $\lambda(t)$ in the bundle $G$ of lagrangian subspaces of $T^\ast  M$.  It is a standard fact that the vector subspace $\lambda(t)\subset T^\ast_{\gamma(t)}  M $ intersects the vertical subspace of $T^\ast_{\gamma(t)} M$ in a non-trivial subspace for a discrete set of times. We will review this fact, in elementary terms, and prove a lemma that will be important for the proof of lemma \ref{rho is lipschitz}.

Let  $ \eta(t)$ be an integral curve of $r$ with $x_0=\eta(0)$ a conjugate point of order $k$. In special coordinates near $x_0$, for $t$ close to $0$, the differential of $F$ along $\eta$ has the form:
$$
dF(\eta(t))=
\begin{pmatrix}
I_{n-k} & 0 \\ \ast & \ast
\end{pmatrix}
=\begin{pmatrix}
I_{n-k} & 0 \\ 0 & 0
\end{pmatrix}
+t\begin{pmatrix}
0 & 0 \\ 0 & I_k
\end{pmatrix}
+\begin{pmatrix}
0 & 0 \\ \ast & E(t)
\end{pmatrix}
$$
where $\vert E\vert <\varepsilon$, with $E=0$ if $\gamma(0)=x_0$.

Let $w\in\ker dF(\eta(t_1))$ and $v\in\ker dF(\eta(t_2))$ be unit vectors in the kernel of $dF$ for $t_1<t_2 $ close to $0$. 
It follows that both $v$ and $w$ are spanned by the last $k$ coordinates.
We then find:
$$
0 = w \cdot dF(\eta(t_2))\cdot v-v \cdot dF(\eta(t_1))\cdot w = 
(t_2-t_1)w\cdot v
+ w(E(t_2)-E(t_1))v
$$
and it follows (for some $t_1<t^\ast<t_2 $):
$$
(t_2-t_1)w\cdot v < 2 \varepsilon |w| |v| (t_2-t_1)
$$
and so:
\begin{equation}\label{almost orthogonal basis} 
 w\cdot v < 2\varepsilon
\end{equation} 
This also shows that the set of $t$'s such that $dF(\eta(t))$ is singular is discrete.

Say the point $x_0=(z_0,t_0)$ is the $j$-th conjugate point along the integral curve of $r$ through $x_0$ from $z_0$, and recall that it is of order $k$ as conjugate point. As $z$ moves towards $z_0$, all functions $\lambda_j(z),\dots,\lambda_{j+k}(z)$ converge to $t_0 $. Let $z_i$ be a sequence of points converging to $z_0$ such that the integral curve through $z_i$ meets its $k$ conjugate points near $z_0$ at $M$ linear subspaces (e.g. $\lambda_j(z_i)=\dots=\lambda_{j+k_1}(z_i)$; $\lambda_{j+k_1+1}(z_i)=\dots=\lambda_{j+k_2}(z_i)$; ...; $\lambda_{j+k_{M-1}+1}(z_i)=\dots=\lambda_{j+k_M}(z_i)$). we get the following result (see also lemma 1.1 in \cite{Itoh Tanaka 00}):
\begin{lem}
The subspaces $\ker d_{(\lambda_{j+k_l}(z_i),z_i)}F$ for $l=1,\dots, M$ converge to orthogonal subspaces of $\ker d_{(\lambda_{j}(z_0),z_0)}F $, for the standard inner product in the special coordinates at the point $(\lambda_{j}(z_0),z_0) $.
\end{lem}

\subsection{A useful lemma}

\begin{lem}\label{the set and the cone}
Let $U$ be an open set in $\Rn$, $A\subset U$ a proper open set,
$C^+\subset\RR^n$ an open cone, $V\subset U$ an arbitrary open set and
$\varepsilon>0$ such that at any point $q\in \partial A\cap V$, we have
$(q+C^+)\cap (q+B_{\varepsilon}) \subset A$.

Then $\partial A\cap V$ is a Lipschitz hypersurface. Moreover, for any vector
$X\in C^+$, take coordinates so that $X=\frac{\partial}{\partial x_{1}}$. Then
$\partial A\cap V$ is a graph $S=\lbrace
(h(x_{2},..,x_{n}),x_{2},$ $..,x_{n})\rbrace$ for a Lipschitz function $h$.
\end{lem}

\begin{proof}
Choose the vector $X\in C^+$ and coordinate system in the statement. Assume $X$ has norm $1$, so that $q+t X\in q+B_t$ for small positive $t$. Take any point $p\in \partial A\cap V$. We claim that all points $p+t\frac{\partial}{\partial x_{1}}$ for $0<t<\varepsilon$ belong to $A$, and all points $p+t\frac{\partial}{\partial x_{1}}$ for $-\varepsilon<t<0$ belongs to $U\setminus A$.
Indeed, there cannot be a point $p+t\frac{\partial}{\partial x_{1}}\in A$ for $-\varepsilon<t<0$ because the set $(p+ t\frac{\partial}{\partial x_{1}})+ (C^+\cap B_{\varepsilon})$ would contain an open neighborhood of $p$, which contains points not in $A$.
In particular, there is at most one point of $\partial A\cap V$ in each line with direction vector $\frac{\partial}{\partial x_{1}}$.

Take two points $q_{1},q_{2}\in \Rnuno$ sufficiently close and consider the lines $L_{1}=\lbrace (t,q_{1}), t\in \RR \rbrace$ and $L_{2}=\lbrace (t,q_{2}), t\in \RR\rbrace$. Assume there is a $t_{1}$ such that $(t_{1},q_{1})$ belongs to $\partial A$.
If there is no point of $\partial A$ in $L_{2}$ then either all points of $L_{2}$ belong to $A$ or they belong to $U\setminus A$.
Both of these options lead to a contradiction if $((t_{1},q_{1})+C^+)\cap ((t_{1},q_{1})+ B_{\varepsilon})\cap L_{2} \neq \emptyset$ (this condition is equivalent to $K\vert q_{1}-q_{2}\vert <\varepsilon$ for a constant $K$ that depends on $C^+$ and the choice of $X\in C^+$ and the coordinate system).

Thus there is a point $(t_{2},q_{2})\in \partial A$. 
For the constant $K$ above and $t\geq t_{1}+K\vert q_{1}-q_{2}\vert$, the point $(t,q_{2})$ lies in the set $(t_{1},q_{1})+C^+$,
so we have
$$
t_{2}<t_{1}+K\vert q_{1}-q_{2}\vert
$$
The points $q_{1}$ and $q_{2}$ are arbitrary, and the lemma follows.
\end{proof}

\begin{remark}
 We are working in a paper about some limitations of the technique of 3d printing known as fused desposition modeling. The above lemma is used to prove that all current 3d printers using this technique will print pieces that are specially fragile in some directions.
\end{remark}

\subsection{Some generalities on HJ equations.}

\begin{lem}
 For fixed $ M$ and $H$, two functions $g,g':\partial M\rightarrow \RR$ have the same characteristic vector field in $\partial M$ iff $g'$ can be obtained from $g$ by addition of a constant at each connected component of $\partial M$.
\end{lem}
\begin{proof} It follows from 
\eqref{equation for the characteristic vector field} that $g$ and $g'$ have the
same characteristic vector field at all points if and only $d g=d g'$ at all
points.
\end{proof}

For our next definition, observe that given $ M$, $H$ and $g$, we can define a map $\tilde{u}: V\rightarrow\RR $ by 
$\tilde{u} (t,z)=t+g(z)$.
\begin{dfn}\label{made of characteristics} 
We say that a function $u:  M\rightarrow\RR$ is \emph{made from characteristics} iff $u\vert_{\partial M}=g $ and $u$ can be written as $u(p)=\tilde{u}\circ s$ for a (not necessarily continuous) section $s$ of $F:V\rightarrow  M $.

\end{dfn}

\paragraph{Remark.} In the paper \cite{Mennucci}, the same idea is expressed in different terms: all characteristics are used to build a multi-valued solution, and then some criterion is used to select a one-valued solution. The criterion used there is to select the characteristic with the minimum value of $\tilde{u}$.

\begin{lem}\label{una unica solucion continua por caracteristicas}
The viscosity solution to \eqref{HJequation} and \eqref{HJboundarydata} is the
unique continuous function that is made from characteristics.
\end{lem}
\begin{proof} 
Let $h$ be a function made from characteristics, and $u$ be the viscosity
solution given by formula 
\eqref{Lax-Oleinik}. Let $Sing$ be the closure of the singular set of $u$.

Take a point $z\in \partial  M$. 
Define:
$$t^\ast_z=\sup
\left\{  
t\geq 0:\;
h(F(\tau,z))=u(F(\tau,z))\;
\forall 0\leq \tau<t
\right\}
$$

Let $p=F(t^\ast_z,z)$. Assume for simplicity that $h(p)=u(p)$.

\strong{Claim}: $t^\ast_z< \rho_{Sing}(z)$ implies $h$ is discontinuous at $F(t^\ast_z,z)$.

\strong{Proof of the claim}: 
Assume that $t^\ast_z< \rho_{Sing}(z)$ and $h$ is continuous at $F(t^\ast_z,z)$ for some $z\in \partial M $.

As $t^\ast_z< \rho_{Sing}(z) \leq\lambda_1(z) $, there is an open neighborhood $O$ of $(t^\ast_z,z) $ such that $F|_{O}$ is a diffeomorphism onto a neighborhood of $p=F(t^\ast_z,z)$.

By hypothesis, there is a sequence $t_n\rightarrow t^\ast_z$ and $p_n=F(t_n,z)$ such that $h(p_n)\neq u(p_n)$.
As $h$ is built from characteristics using a section $s$, we have $h(p_n)=\tilde{u}(s(p_n))=\tilde{u}((s_n,y_n))=s_n+g(y_n)$, for $(s_n,y_n)\neq (t_n,z)$.

For $n$ big enough, the point $(s_n,y_n) $ does not belong to $O$, as $ (t_n,z)$ is the only preimage of $p_n$ in $O$.
As $h(p_n)\rightarrow h(p) $, and $\partial M$ is compact, we deduce the $s_n$ are bounded. 
We can take a subsequence of $(s_n,y_n) $ converging to $(s_\infty,y_\infty)\not\in O $. So we have $p=F(t^\ast_z,z)=F(s_\infty,y_\infty)$.
If $p\not\in Sing $, we deduce that $\lim_{n\rightarrow \infty} h(p_n)=\tilde{u}(s_\infty,y_\infty) > h(p)=u(p)=\tilde{u}(t^\ast_z,z)$, so $h$ is discontinuous at $p$.

Using the claim, we conclude the proof:
if $h$ is continuous, then $ \rho_{Sing}(z)\leq t^\ast_z $ for all $z\in \partial M $, and $u=h$, as any point in $ M$ can be expressed as $F(t,z)$ for some $z$, and some $t\leq\rho_{Sing}(z) $.

\end{proof}

\pagebreak[3]

We will need later the following version of the same principle:
\begin{lem}\label{una unica solucion por caracteristicas continua en un conjunto grande}
Let $S$ be a split locus, and $h$ be the function associated to $S$ as in definition \ref{u associated to S}.
If $\rho_S$ is continuous, and $h$ can be extended to $ M$ so that it is continuous except for a set of null $\mathcal{H}^{n-1}$ measure, then $S=Sing$.
\end{lem}
\begin{proof} 
Define
$$Y_0=\left\{
z\in\partial M:\; h(F(t,z))\neq u(F(t,z))\;\text{ for some } t\in[0,\rho_{Sing}(z))
\right\} $$
By the claim in the previous lemma, $Y_0$ is contained in:
$$Y=\left\{
z\in\partial M:\; h  \text{ discontinuous at } F(t,z)\; \text{ for some }  t\in[0,\rho_{Sing}(z))
\right\} $$

Let $A=A(Sing)$ be the set in definition \ref{splits}. 
The map $F$ restricts to a diffeomorphism from $A$ onto $ M\setminus Sing $.
The set $Y$ can be expressed as:
$$
Y=\pi_2\circ (F\vert_A)^{-1}\left(\{p\in M\setminus Sing:\; h \text{ discontinuous at } p \})\right)
$$
and thus by the hypothesis has null $\Hnuno$ measure.
Therefore, $\partial M\setminus Y_0 $ is dense in $\partial M$.

We claim now that $S\subset Sing $.
To see this, let $p\in S\setminus Sing $. Then $p=F(t^\ast,z^\ast)$ for a unique $(t^\ast,z^\ast) \in A $.
It follows $\rho_S(z^\ast)\leq t^\ast<\rho_{Sing}(z^\ast)$.
As $\rho_S$ is continuous, $\rho_S(z)<\rho_{Sing}(z) $ holds for all $z$ in a neighborhood of $z^\ast $ in $\partial M $ and, in particular, for some $z\in \partial M\setminus Y_0$.
This is a contradiction because, for $\rho_{S}(z)<t<\rho_{Sing}(z) $, $h(F(t,z))=\tilde{u}(t',z') $ for $(t',z')\neq (t,z) $, and $t<\rho_{Sing}(z) $ implies 
$h(F(t,z))=\tilde{u}(t',z')> \tilde{u}(t,z)=u(F(t,z))$, forcing $z\in  Y_0$.

We deduce $S=Sing$ using lemma \ref{characterization of split locus} and the fact that $Sing$ is a split locus. 
\end{proof}

\section{$\rho_S$ is Lipschitz}\label{section:rho is Lipschitz} 

In this section we study the functions $\rho_S$ and $\lambda_j$ defined earlier.
The fact that $\rho_S$ is Lipschitz will be of great importance later.
The definitions and the general approach in this section follow \cite{Itoh
Tanaka 00}, but our proofs are shorter, provide no precise quantitative bounds,
use no constructions from Riemannian or Finsler geometry, 
and work for Finsler manifolds, thus providing a new and shorter proof for the
main result in \cite{Li Nirenberg}. The proof that $\lambda_j$ are Lipschitz
functions was new for Finsler manifolds when we published the first version of
the preprint of this paper. Since then, the paper \cite{Castelpietra Rifford}
has appeared which shows that $\lambda_1$ is actually semi-concave.

\begin{proof}[Proof of \ref{landa es Lipschitz}]
It is immediate to see that the functions $\lambda_j $ are continuous, since
this is 
property (R3) of Warner (see  \cite[pp. 577-578 and Theorem 4.5]{Warner}).

Near a conjugate point $x^0$ of order $k$, we can take special coordinates as in section \ref{subsection: special coordinates}:
$$
F(x_1,\dots,x_n)=(x_1,\dots,x_{n-k},F_{n-k+1},\dots,F_n)
$$
Conjugate points near $x$ are the solutions of
$$
d(x_1,\dots,x_n)=det(dF)=\sum_{\sigma}(-1)^\sigma \frac{\partial F_{\sigma(n-k+1)}}{\partial x_{n-k+1}}\dots \frac{\partial F_{\sigma(n)}}{\partial x_n}=0
$$
From the properties of the special coordinates, we deduce that:
\begin{equation}\label{d alpha d = 0 for some alphas} 
D^\alpha d(0)=0\qquad
\forall \vert \alpha\vert<k
\end{equation} 
and 
$$
\frac{\partial^k}{\partial x_1^k}d=1
$$

We can use the preparation theorem of Malgrange (see \cite{Golubitsky Guillemin}) to find real valued functions $q$ and $l_i$ in an open neighborhood $U$ of $x$ such that $d(x)\neq 0$ and:
$$
q(x_1,\dots,x_n)d(x_1,\dots,x_n)=x_1^k+x_1^{k-1}l_1(x_2,\dots,x_n)+\dots+l_k(x_2,\dots,x_n)
$$
and we deduce from \eqref{d alpha d = 0 for some alphas} that
\begin{equation}\label{the lower derivatives of l_i vanish}
 D^\alpha l_i(0)=0\qquad
\forall \vert \alpha\vert<i
\end{equation} 
which implies
\begin{equation}\label{bounds for l_i}
|l_i(x_2,\dots, x_n)| < \bar{C} \max\{\vert x_2\vert, \dots, \vert x_n\vert\}^{i}
\end{equation} 

At any conjugate point $(x_1,\dots,x_n)$, we have $q(x)=0$, so:

\begin{equation*}
 -x_1^k=x_1^{k-1}l_1(x_2,\dots,x_n)+\dots+l_k(x_2,\dots,x_n)\\ 
\end{equation*}
\noindent and therefore
\begin{equation*}
|x_1|^k<|x_1|^{k-1}|l_1|+\dots+|l_k|
\end{equation*}

Combining this and \eqref{bounds for l_i}, we get an inequality for $\vert
x_1\vert $ at any conjugate point $(x_1,\dots,x_n)$, where the constant $C$
ultimately depends on bounds for the first few derivatives of $F$:
\begin{equation}
 \vert x_1\vert^k<C\max\{\vert x_1\vert, \dots, \vert x_n\vert\}^{k-1}\max\{\vert x_2\vert, \dots, \vert x_n\vert\}
\end{equation}

If
$\vert x_1\vert>\max\{\vert x_2\vert, \dots, \vert x_n\vert\}$,
then
$\vert x_1\vert^k<C\vert x_1\vert^{k-1}\max\{\vert x_2\vert, \dots, \vert x_n\vert\}$.
If the opposite holds, then
$\vert x_1\vert^k<C\max\{\vert x_2\vert, \dots, \vert x_n\vert\}^k$.
So we get:
$$
\vert x_1\vert<\max\{C,1\}\max\{\vert x_2\vert, \dots, \vert x_n\vert\}
$$
This is the statement that all conjugate points near $x$ lie in a cone of fixed width containing the hyperplane $x_1=0$. Thus all functions $\lambda_j$ to $\lambda_{j+k}$ are Lipschitz at $(x_2\dots,x_n)$ with a constant independent of $x$.
\end{proof}

\paragraph{Remark.} A proof of lemma \ref{landa es Lipschitz} in the language of section \ref{subsection: lagrangian} seems possible: let $\Lambda( M)$ be the bundle of Lagrangian submanifolds of the symplectic linear spaces $T^{\ast}_p  M $ and let $\Sigma( M)$ be the union of the Maslov cycles within each $\Lambda_p( M)$.
Define $\lambda:V\rightarrow \Lambda( M)$ where $\lambda(x)$ is the tangent to $\Theta $ at $D(\Phi(x)) $ (recall \eqref{definition of Theta}).
The graphs of the functions $\lambda_k$ are the preimage of the Maslov cycle $\Sigma( M)$. The geodesic vector field (transported to $T^{\ast}  M $), is transversal to the Maslov cycle.
Showing that the angle (in an arbitrary metric) between this vector field and the Maslov cycle at points of intersection can be bounded from below is equivalent to showing that the $\lambda_k$ are Lipschitz.

\begin{lem}\label{rho es menor que landa1} 
 For any split locus $S$ and point $y\in \partial M $, there are no conjugate points in the curve $t\rightarrow \exp(ty) $ for $t<\rho_{S}(y) $.
In other words, $\rho_S\leq \lambda_1 $.
\end{lem}
\begin{proof}
Assume there is $x$ with $\rho_S(x)-\varepsilon>\lambda_1(x) $.
By \cite[3.4]{Warner}, the map $F$ is not injective in any neighborhood of $(x,\lambda_1(x))$. There are points $F(x_n,t_n)$ of $S$ with $x_n\rightarrow x$ and $t_n<\rho_S(x)-\varepsilon$ (otherwise $S$ does not split $ M $). Taking limits, we see $F(x,t)$ is in $S$ for some $t<\rho_S(x)-\varepsilon$, which contradicts the definition of $\rho_S(x)$.
\end{proof}

From now on and for the rest of the paper, $S$ will always be a balanced split locus:
\begin{lem}\label{rho es Lipschitz en rho lejos de lambda} 
Let $E\subset \partial M $ be an open subset whose closure is compact and has a neighborhood where $\rho <\lambda_1 $.
Then $\rho_S $ is Lipschitz in $E$.
\end{lem}

\begin{proof}
The map $G(x)$ defined in \ref{various definitions related to the exponential map}, written as $(F(x),dF_x(r)) $ is an embedding of $V$ into $TM$. There is a constant $c$ such that for $x,y\in V$:
\begin{equation}\label{(F,dF(r)) is an embedding} 
\vert F(x)-F(y)\vert+\vert dF_x(r)-dF_y(r)\vert \geq
c\min\{\vert x-y\vert,1\}
\end{equation} 

Recall the exponential map is a local diffeomorphism before the first conjugate point.
Points $p=F((z,\rho(z))) $ for $z\in E$ have a set $R_p$ consisting of the vector $dF_{(z,\rho(z))}(r) $, and vectors coming from $V\setminus E$. 
Choose one such point $p$, and a neighborhood $U$ of $p$.
The above inequality shows that there is a constant $m$ such that:
$$
\vert dF_x(r)-dF_y(r)\vert \geq m
$$
for $x=(z,\rho(z))$ with $z\in E$ and $y =(w,\rho(w))\in Q_p$ with $w\in V\setminus E$. 
By the balanced condition \ref{balanced}, any unit vector $v$ tangent to $S$ satisfies $\widehat{dF_x(r)}(v)= \widehat{dF_y(r)}(v)$ for some such $y$ and so:
$$
\widehat{dF_x(r)}(v)<1-\varepsilon
$$
Thus for any vector $w$ tangent to $E$ both vectors $(w,d\rho_-(w))$ and
$(w,d\rho_+(w))$ lie in a cone of fixed amplitude around the kernel of $
\widehat{dF_x(r)} $ (the hyperplane tangent to the indicatrix at $x$).
Application of lemma \ref{the set and the cone} shows that $\rho$ is Lipschitz.
\end{proof}

\begin{lem}\label{rho es Lipschitz en rho < lambda pero cerca de un conjugado} 
Let $z_0\in \partial M $ be a point such that $\rho(z_0)=\lambda_1(z_0) $.
Then there is a neighborhood $E$ of $z_0$ and a constant $C$ such that for all $z$ in $E$ with $\rho(z)<\lambda_1(z)$, $\rho$ is Lipschitz near $z$ with Lipschitz constant $C$.

\end{lem}
\begin{proof}
 Let $O$ be a compact neighborhood of $(z_0,\lambda_1(z_0)) $ where special coordinates apply. 
Let $x=(z,\rho(z))\in O$ be such that $\rho(z)<\lambda_1(z) $. In particular, $d_x F$ is non-singular.
We can apply the previous lemma and find $\rho $ is Lipschitz near $z$. We just need to estimate the Lipschitz constant uniformly. Vectors in $R_{F(x)}$ that are of the form $dF_{y}(r)$ for $y\in V\setminus O$, are separated from $dF_x(r)$ as in the previous lemma and pose no trouble, but now there might be other vectors $dF_{y}(r)$ for $y\in O$.

Fix the metric $\langle\cdot \rangle$ in $O$ whose matrix in special coordinates is the identity.
Any tangent vector to $S$ satisfies $\widehat{dF_x(r)}(v)=\widehat{dF_y(r)}(v)$, for some $y\in Q_{F(x)} $.
A uniform Lipschitz constant for $\rho$ is found if we bound from below the angle in the metric $\langle\cdot \rangle$ between $r$ and 
$d_x F^{-1}(v)$ for any vector $v$ with this property. 
This is easy to do if $y\not\in O$, so fix a point $y\in O$ with $F(x)=F(y)$, and let $X=dF_{x}(r)$, $Y=dF_{y}(r)$ and $\alpha=\widehat{X}- \widehat{Y} $. We need to bound from below the angle between $r_x$ and the hyperplane $\ker F^\ast\alpha $.

This is equivalent to proving that there is $\varepsilon_1>0 $ independent of $x$ such that:
$$
\dfrac{F_{x}^\ast \alpha (r)}{\Vert F_{x}^\ast \alpha\Vert }>\varepsilon_1
$$
which is equivalent to:
$$
\widehat{X}(X)-\widehat{Y}(X)<\varepsilon_1 \Vert F_{x}^\ast \alpha\Vert
$$

$$
\widehat{Y}(X)<1-\varepsilon_1 \Vert F_{x}^\ast \alpha\Vert
$$
in the norm $\Vert\cdot \Vert$ associated to $\langle\cdot \rangle$.

Notice first that $X$ and $Y$ belong to the indicatrix at $F(x)=F(y)$, which is
strictly convex. 
By this and \eqref{(F,dF(r)) is an embedding}, we see that for some
$\varepsilon_2>0$:
$$
\widehat{Y}(X)<1-\varepsilon_2 \Vert X-Y\Vert ^2<1-c\varepsilon_2 \Vert x-y\Vert ^2
$$
So it is sufficient to show that for some $C_1$ independent of $x$:
\begin{equation}\label{F ast alpha: goal}
\Vert F_x^\ast \alpha\Vert <
C_1\Vert x-y\Vert ^2
\end{equation}

Using a Taylor expansion of $\frac{\partial \varphi}{\partial x_j} $ in the second entry, we see the form $ F_x^\ast \alpha $ can be written in coordinates (with implicit summation over repeated indices):
\begin{equation}\label{expasion of F ast alpha}
 \begin{array}{rcl}
  F_x^\ast \alpha&=&
  \left(
  \frac{\partial \varphi}{\partial x_j}(p,X)-\frac{\partial \varphi}{\partial x_j}(p,Y)
  \right)
  \frac{\partial F_j}{\partial x_l}\\
  &=&
  \frac{\partial^2 \varphi}{\partial x_i x_j}(p,X)
  \left(
  X_i - Y_i
  \right)
  \frac{\partial F_j}{\partial x_l}
  +O(\Vert X-Y\Vert )^2\\
  &=&
  \frac{\partial^2 \varphi}{\partial x_i x_j}(p,X)
  \left(
  X_i - Y_i
  \right)
  \frac{\partial F_j}{\partial x_l}
  +O(\Vert x-y\Vert )^2
 \end{array}
\end{equation}
Define the bilinear map $g(p,X)$ with coordinates $\frac{\partial^ 2\varphi}{\partial x_i \partial x_j}(p,X)$.
It is sufficient to prove that for some $C_2$ independent of $x$:

$$
\Vert
g_{i,j}(p,X)
  \left(
  X_i - Y_i
  \right)
  \frac{\partial F_j}{\partial x_l}
 \Vert\leq C_2\Vert x-y\Vert ^2
$$

This is equivalent to showing that for every vector $v\in T V$:
$$
\Vert
g_{i,j}(p,X)
  \left(
  X_i - Y_i
  \right)
  \frac{\partial F_j}{\partial x_l}v_l
 \Vert =
 \Vert g(p,X)(X-Y,dF(v))\Vert
 \leq C_2\Vert x-y\Vert ^2\Vert v\Vert
$$
We can of course restrict to vectors $v$ of norm $1$.
The maximum norm is achieved when $dF(v)$ is proportional to $X-Y$.
The map $d_x F$ is invertible, so for the vector $v_0=\frac{dF^{-1}(X-Y)}{\Vert dF^{-1}(X-Y) \Vert}$, we have:
$$
\sup_{\Vert v\Vert = 1} \Vert g(p,X) (X-Y, dF(v))\Vert = \Vert g(p,X) (X-Y, dF(v_0))\Vert
$$

Thus by \eqref{expasion of F ast alpha} and the convexity of $\varphi$ we have:
\begin{align*} 
\Vert F_{(z,\rho(z))}^\ast \alpha\Vert \,& <
C_3 \frac{\Vert X-Y\Vert^2}{\Vert dF^{-1}(X-Y) \Vert}+O(\Vert x-y\Vert)^2
\\
& <  
C_4 \frac{\Vert x-y\Vert^2}{\Vert dF^{-1}(X-Y) \Vert}+O(\Vert x-y\Vert)^2
\end{align*}
for constants $C_3$ and $C_4$, and it is enough to show there is $\varepsilon_3$ independent of $x$ and $y$ such that:
\begin{equation}
\Vert dF^{-1}(X-Y) \Vert>\varepsilon_3
\end{equation}
Let $G(x)= d_xF(r) $. We have:
$$
X-Y=G(x)-G(y)=dG_x(x-y)+O(\Vert x-y\Vert)
$$
so in order to prove \eqref{F ast alpha: goal} it is enough to show the following:
$$
\Vert dF^{-1}dG_x(x-y)\Vert >\varepsilon_4
$$
for $\varepsilon_4 $ independent of $x $ and $y$.

Assume that $(\rho(z_0),z_0)$ is conjugate of order $k$, so that $\rho(z_0)=\lambda_1(z_0)=\dots=\lambda_k(z_0)$. Thanks to Lemma \ref{landa es Lipschitz} and reducing to a smaller $O$, we can assume that $a_1=(\lambda_1(z),z)$ to $a_k=(\lambda_k(z),z)$ all lie within $O$ (some of them may coincide). Let $d_i = \lambda_i(z)-\rho(z)$ be the distance from $x$ to the $a_i$. At each of the $a_i$ there is a vector $w_i\in \ker d_{a_i}F$ such that all the $w_i$ span a $k$-dimensional subspace. Recall from section \ref{subsection: lagrangian} that we can choose $w_i$ forming an almost orthonormal subset for the above metric, in the sense that $\left\langle w_i,w_j\right\rangle=\delta_{i,j}+\varepsilon_{i,j} $ for $\varepsilon_{i,j}<<1$.

The kernel of $d_y F$ is contained in $K=\langle \dfrac{\partial}{\partial x_{n-k+1}},\dots,\dfrac{\partial}{\partial x_{n}}\rangle $ for all $y\in O$, and thus $K=\langle w_1,\dots,w_k\rangle$.
Write $w_i=\sum_{j\geq n-k+1} w_i^j \dfrac{\partial }{\partial x^j}$.
Then we have $\dfrac{\partial }{\partial x^1}\dfrac{\partial }{\partial w_i}F(a)=z_i+R_i(a)$, for $z_i=\sum w_i^k \dfrac{\partial }{\partial y^k}$, $\|R_i(a)\|<\varepsilon$ and $a\in O$.
We deduce $\dfrac{\partial }{\partial w_i}F(x)=\dfrac{\partial }{\partial w_i}F(a_i)+d_i (z_i+v_i)=d_i (z_i+v_i)$ for $\|v_i\|<\varepsilon$.

By the form of the special coordinates, $x-y\in K$. Let $x-y=\sum b_iw_i$. 
Since $|w_i|$ is almost $1$, there is an index $i_0$ such that $|b_{i_0}|>\frac{1}{2n}\|x-y\|$. 
We have the identity:
$$
0=F(y)-F(x)=d_x F(y-x)+O(\|x-y\|^2)
= \sum b_i d_i (z_i+ v_i ) + O(\|x-y\|^2)
$$
Multiplying the above by $\pm z_j$,
we deduce $d_j| b_j |= -\sum |b_i| d_j(\varepsilon_{i,j}+v_i z_j) + O(\|x-y\|^2)$, which leads to 
\begin{equation}\label{desigualdad cerca de un punto conjugado} 
 \sum |b_i| d_i<C_4\|x-y\|^2
\end{equation} 

At the point $x$, the image by $d_{x}F $ of the unit ball $B_x V$ in $T_x V$ is contained in a neighborhood of $Im (d_{a_{i}}F)$ of radius $2d_{i}$. 
We use the identity
$$
\Vert dF^{-1}dG_x(\frac{ x-y}{\Vert x-y\Vert} )\Vert^{-1} =
\sup\{ t: tdG_x(\frac{ x-y}{\Vert x-y\Vert} ) \in d_x F(B_x V) \}
$$
We can assume the distance between the vectors $dG_{x}(\frac{x-y}{\Vert x-y\Vert})$ and $\sum \frac{ b_i}{\Vert x-y\Vert}z_i$ is smaller than $\frac{1}{4n} $. 
In particular, looking at the $i_0$ coordinate chosen above,
we see that the vector $dG_{x}(\frac{ x-y}{\Vert x-y\Vert})$ needs to be rescaled by a number no bigger than $8nd_{i_0} $ in order to fit within the image of the unit ball.
In other words, the sup above is smaller than $8nd_{i_0} $.

$$
\Vert dF^{-1}dG_x(\frac{ x-y}{\Vert x-y\Vert} )\Vert 
>\frac{1}{8n d_{i_0}}
>\frac{ |b_{i_0}|}{8nC_4\Vert x-y\Vert^2 }
>\frac{\varepsilon_4}{\Vert x-y\Vert }
$$
for $\varepsilon_4=\frac{1}{16n^2C_4}>0$, which is the desired inequality.
\end{proof}

\begin{proof}[Proof of Lemma \ref{rho is lipschitz}]
We prove that $\rho$ is Lipschitz close to a point $z^0$. Let $E$ be a neighborhood of $z^0$ such that $\lambda_1 $ has Lipschitz constant $L$, and $\rho $ has Lipschitz constant $K $ 
for all $z\in E$  such that $\rho(z)<\lambda(z)$. 
Let $z^1,z^2\in E$ be such that $\rho(z^1)<\rho(z^2)$.

If $\rho (z^1)=\lambda_1(z^1)$ we can compute
$$
\vert \rho (z^2)- \rho (z^1)\vert=\rho (z^2)- \rho (z^1)<\lambda(z^2)-\lambda(z^1)<L\vert z^2-z^1\vert
$$
where $L$ is a Lipschitz constant $L$ for $\lambda $ in $U$.

Otherwise take a linear path with unit speed $\xi:[0,t]\rightarrow \partial  M$ from $z^1$ to $z^2$ and let $a$ be the supremum of all $s$ such that $\rho(\xi(s))<\lambda(\xi(s)) $. Then
$$
\vert \rho (z^2)- \rho (z^1)\vert<
\vert \rho (z^2)- \rho (\xi(a))\vert+\vert \rho (\xi(a))- \rho (z^1)\vert
$$

\noindent The second term can be bound:
$$
\vert \rho (\xi(a))- \rho (z^1)\vert<K a
$$

\noindent If $\rho (z^2)\geq \rho (\xi(a)) $, we can bound the first term as
$$
\vert \rho (z^2)- \rho (\xi(a))\vert=\rho (z^2)- \rho (\xi(a))<
\lambda(z_2)-\lambda(\xi(a))<L\vert t-a\vert
$$
while if  $\rho (z^2)< \rho (\xi(a)) $, we have
$$
\vert \rho (z^2)- \rho (z^1)\vert<\vert \rho (\xi(a))- \rho (z^1)\vert
$$
so in all cases, the following holds:
$$
\vert \rho (z^2)- \rho (z^1)\vert<\max\{L,K\}t<\max\{L,K\}\vert z^2 - z^1\vert
$$
 
\end{proof}

\section{Proof of the main theorems.}\label{section:proof of the main theorems}
Take the function $h$ associated to $S$ as in definition \ref{u associated to S}. At a cleave point $x$ there are two geodesics arriving from $\partial M$; each one yields a value of $h$ by evaluation of $\tilde{u}$.
The \emph{balanced} condition implies that $\widehat{X}_1(v)=\widehat{X}_2(v)$ for the speed vectors $X_1$ and $X_2$ of the characteristics reaching $x$ and any vector $v$ tangent to $S$.
But $\widehat{X}$ is $dh$, so the difference of the values of $h$ from either side is constant in every connected component of the cleave locus.

We define an $(n-1)$-current $T$ in this way: Fix an orientation $\mathcal{O} $ in $ M$. For every smooth $(n-1)$ differential form $\phi$, restrict it to the set of cleave points $\mathcal{C}$ (including degenerate cleave points). In every component $\mathcal{C}_{j}$ of $\mathcal{C}$ compute the following integrals

\begin{equation}\label{}
\int_{\mathcal{C}_{j,i}}h_{i}\phi \qquad i=1,2
\end{equation} 
where $\mathcal{C}_{j,i}$ is the component $\mathcal{C}_{j}$ with the orientation induced by $\mathcal{O} $ and the incoming vector $V_{i}$, and $h_{i}$ for $i=1,2$ are 
the limit values of $h$ from each side of $\mathcal{C}_{j}$.

We define the current $T(\phi)$ to be the sum:

\begin{equation}\label{definition of T}
T(\phi)=\sum_{j}\int_{\mathcal{C}_{j,1}}h_{1}\phi
+\int_{\mathcal{C}_{j,2}}h_{2}\phi
=\sum_{j}\int_{\mathcal{C}_{j,1}}(h_{1}-h_{2})\phi
\end{equation} 

The function $h$ is bounded and the $\Hnuno$ measure of $\mathcal{C} $ is finite (thanks to lemma \ref{rho is lipschitz}) so that $T$ is a real flat current that represents integrals of test functions against the difference between the values of $h$ from both sides.

If $T=0$, we can apply lemma \ref{una unica solucion por caracteristicas continua en un conjunto grande} and find $u=h$.

We will prove later that the boundary of $T$ as a current is zero.
Assume for the moment that $\partial T=0 $. It defines an element of the homology space $H_{n-1}( M)$ of dimension $n-1$ with real coefficients. We can study this space using the long exact sequence of homology with real coefficients for the pair $ ( M,\partial M)$:
\begin{eqnarray}\label{long exact sequence} 
0\rightarrow H_{n}( M)\rightarrow H_{n}( M,\partial M)\rightarrow\notag\\
H_{n-1}(\partial M)\rightarrow H_{n-1}( M)\rightarrow H_{n-1}( M,\partial M)\rightarrow \dots
\end{eqnarray}

\subsection{Proof of Theorem \ref{maintheorem0}.}
We prove that under the hypothesis of \ref{maintheorem0}, the space $H_{n-1}( M)$ is zero, and then we deduce that $T=0$.

As $ M$ is open, $H_{n}( M)\approx 0$. As $ M$ is simply connected, it is orientable, so we can apply Lefschetz duality with real coefficients (\cite[3.43]{Hatcher}) which implies:
$$H_{n}( M,\partial M)\approx H^{0}( M)$$
and  
$$H_{n-1}( M,\partial M)\approx H^{1}( M)=0$$

As $\partial M$ is connected, we deduce $H_{n-1}( M)$ has rank $0$, and $T=\partial P$ for some $n$-dimensional flat current $P$. The flat top-dimensional current $P$ can be represented by a density $f\in L^{n}( M)$ (see \cite[p 376, 4.1.18]{Federer}):

\begin{equation}
P( M)=\int _{ M}f M,,\qquad  M\in \Lambda^{n}( M)
\end{equation} 

We deduce from \eqref{definition of T} that the restriction of $P$ to any open
set disjoint with $S$ is closed, so $f$ is a constant in such open set.
It follows that the constant is zero because the boundary of $P$ for a constant non-zero function is a current supported on $\partial M$.

\subsection{Proof of Theorem \ref{maintheorem1}.}
Assume now that $\partial M$ has $k$ connected components $\Gamma_{i}$. We look
at \eqref{long exact sequence}, and recall the map $H_{n-1}(\partial
M)\rightarrow H_{n-1}( M)$ is induced by inclusion.
We know by Poincar\'e duality that $H_{n-1}(\partial M)$ is isomorphic to the linear combinations of the fundamental classes of the connected components of $\partial M$ with real coefficients.
We deduce that $H_{n-1}( M)$ is generated by 
the fundamental classes of the connected components of $\partial M$, and that it
is isomorphic to the quotient of all linear combinations by the subspace of
those linear combinations with equal coefficients.
Let
$$
R=\sum a_{i}\left[ \Gamma_{i}\right] 
$$
be the cycle to which $T$ is homologous (the orientation of $\Gamma_{i} $ is such that, together with the inwards pointing vector, yields the ambient orientation).

If we define $a(x)=a_{i}$, $\forall x\in \Gamma_{i}$, solve the HJ equations with boundary data $g-a$ and compute the current $\widehat{T} $ corresponding to that data, we see that $\widehat{T} = T - j_{\sharp}R$, where $j$ is the retraction $j$ of $ M$ onto $S$ that fixes points of $S$ and follows characteristics otherwise. The homology class of $\widehat{T}$ is zero, and we can prove $\widehat{T}=0 $ as before. It follows that $S$ is the singular set to the solution of the Hamilton-Jacobi equations with boundary data $g-a $.

\subsection{Proof of Theorem \ref{maintheorem2}.}
For this result we cannot simply use the sequence 
\eqref{long exact sequence}. We first give a procedure for obtaining balanced
split loci in $ M$ other than the cut locus. 

A function $a:[\partial \widetilde{ M}]\rightarrow \RR $ that assigns a real number to each connected component of $\partial \widetilde{ M} $ is \emph{equivariant} iff for any automorphism of the cover $\varphi$ there is a real number $c(\varphi)$ such that $a\circ \varphi =a+c(\varphi)$.

A function $a:[\partial \widetilde{ M}]\rightarrow \RR $ is \emph{compatible}
iff $\widetilde{g}-a$ satisfies the compatibility condition 
(\eqref{compatibility condition}.

An equivariant function $a$ yields a group homomorphism from $\pi_1( M, \partial M)$ into $\RR$ in this way:
\begin{equation}\label{from a function a to an element of Hom(pi,R)} 
\sigma\rightarrow a(\widetilde{\sigma}(1))-a(\widetilde{\sigma}(0)) 
\end{equation} 
where $\sigma:[0,1]\rightarrow  M $ is a path with endpoints in $\partial M $  and  $\widetilde{\sigma}$ is any lift to $\widetilde{ M} $ . The result is independent of the lift because $a$ is equivariant. 
On the other hand, choosing an arbitrary component $[\Gamma_0]$ of $\partial M$ and a constant $a_0=a([\Gamma])$, the formula:
\begin{equation}\label{from an element of Hom(pi,R) to a function} 
[\Gamma]\rightarrow a([\Gamma_0])+l(\pi\circ\tilde{\sigma}),\text{ for any path }\tilde{\sigma}\text{ with }\tilde{\sigma}(0)\in \Gamma_0,\sigma(1)\in \Gamma
\end{equation} 
assigns an equivariant function $a$ to an element $l$ of $Hom(\pi_1( M, \partial M), \RR)\sim H^{1}( M, \partial M)  $.

Up to addition of a global constant, these two maps are inverse of one another, so there is a one-to-one correspondence between elements of 
$H^{1}( M, \partial M) $
and equivariant functions $a$ (with $a+c$ identified with $a$ for any constant $c$).
The compatible equivariant functions up to addition of a global constant can be identified with an open subset of $H^{1}( M, \partial M) $ that contains the zero cohomology class.

Let $\widetilde{ M} $ be the universal cover of $ M $.
We can lift the Hamiltonian $H$ to a function $\widetilde{H} $ defined on $T^\ast\widetilde{ M}$ and the function $g$ to a function $\tilde{g} $ defined on $\partial \widetilde{ M}$.
The preimage of a balanced split locus for $ M $, $H$ and $g$ is a balanced split locus for $\widetilde{ M} $, $\widetilde{H} $ and $\tilde{g} $ that is invariant by the automorphism group of the cover, and conversely, a balanced split locus $\widetilde{S}$ in $\widetilde{ M} $ that is invariant by the automorphism group of the cover descends to a balanced split locus on $ M $.

Any function $a$ that is both equivariant and compatible can be used to solve the Hamilton-Jacobi problem $\widetilde{H}(p,du(p))=1$ in $\widetilde{ M} $ and $u(p)=\widetilde{g}(p)-a(p)$. 
If $\pi_1( M)$ is not finite, $\widetilde{ M}$ will not be compact, but this is not a problem (see remark 5.5 in page 125 of \cite{Lions}).
The singular set is a balanced split locus that is invariant under the action of $\pi_1( M) $ and hence it yields a balanced split locus in $ M $. 
We write $S[a]$ for this set.
It is not hard to see that the map $a\rightarrow S[a]$ is injective.

Conversely, a balanced split locus in $ M $ lifts to a balanced split locus
$\widetilde{S}$ in~$\widetilde{M}$. 
The reader may check that the current $T_{\widetilde{S}}$ is the lift of $T_S$,
and in particular it is closed. 
As in the proof of Theorem \ref{maintheorem1}, we have $H^1(\widetilde{ M})=0$, and we deduce
$$
T_{\widetilde{S}} = \sum_j a_j[\Lambda_j] +\partial P
$$
where $\Lambda_j$ are the connected components of $\partial \widetilde{ M}$.

This class is the lift of the class of $T\in H_{n-1}( M)$ and thus it is
invariant under the action of the group of automorphisms of the cover. 
Equivalently, the map defined in 
\eqref{from a function a to an element of Hom(pi,R)} is a homomorphism.
Thus $a$ is equivariant.
Similar arguments as before show that $S=S[a]$.

Thus the map $a\rightarrow S[a] $ is also surjective, which completes the proof that there is a bijection between equivariant compatible functions $a:[\partial \widetilde{ M}]\rightarrow \RR $ and balanced split loci.

\section{Proof that $\partial T=0$}\label{section: proof that partial T=0}
It is enough to show that $\partial T=0$ at all points of $ M$ except for a set of zero $(n-2)$-dimensional Hausdorff measure. This is clear for points not in $S$.
Due to the structure result \ref{complete description}, we need to show the same at cleave points (including degenerate ones), edge points and crossing points.
Along the proof, we will learn more about the structure of $S$ near those kinds of points.

Throughout this section, we assume $n=dim( M)>2$. This is only to simplify notation, but the case $n=2$ is covered too. We shall comment on the necessary changes to cover the case $n=2$, but do not bother with the trivial case $n=1$.

\subsection{Conjugate points of order $1$.}

We now take a closer look at points of $A(S)$ that are also conjugate points of order $1$. Fortunately, because of \ref{complete description} we do not need to deal with higher order conjugate points.
In a neighborhood $O$ of a point $x^{0}$ of order $1$, in the special coordinates of section \ref{subsection: special coordinates}, we have $x^0=0$ and $F$ looks like:
\begin{eqnarray}\label{local form for order 1}
F(x_{1},x_{2},\dots,x_{n}) =
(x_{1},x_{2},\dots,F_n(x_{1},\dots,x_{n}))
\end{eqnarray}

Let $\widetilde{S}$ be the boundary of $A(S)$, but without the points $(0,z)$ for $z\in\partial M $.
It follows from \ref{rho is lipschitz} that $\widetilde{S}$ is a Lipschitz graph on coordinates given by the vector field $r$ and $n-1$ transversal coordinates. 
It is not hard to see
that it is also a Lipschitz graph $x_{1}=\tilde{t}(x_{2},\dots,x_{n})$ in the above coordinates $x_{i}$, possibly after restricting to a smaller open set.

Because of  Lemma \ref{rho es menor que landa1}, we know $x^0$ is a first conjugate point, so we can assume that $O$ is a coordinate cube $\prod (-\varepsilon_i,\varepsilon_i)$, and that $F$ is a diffeomorphism when restricted to $\{x_1=s\} $ for $s<-\varepsilon_1/2$.

\begin{dfn}\label{univocal} 
 A set $O\subset V$ is \emph{univocal} iff for any $ p\in  M$ and $x^{1},x^{2}\in Q_{p}\cap O$ we have $\tilde{u}(x^{1})=\tilde{u}(x^{2})$.
\end{dfn}

\begin{remark}
The most simple case of univocal set is a set $O$ such that $F\vert O$ is injective. 
\end{remark}

\begin{lem}\label{uniqueness near order 1 points} 
 Let $x^{0}\in V$ be a conjugate point of order 1. Then $x^{0}$ has an univocal neighborhood.
\end{lem}
\begin{proof}
Let $O_1$ and $U_1$ be neighborhoods of $x^0$ and $F(x^0)$ where the special
coordinates 
\eqref{local form for order 1} hold;
let $x_{i}$ be the coordinates in $O_1$ and $y_{i}$ be those in $U_1$.

Choose smaller $U\subset U_1$ and $O\subset F^{-1}(U)\cap O_1$ so that
we can assume that if a point $x'\in V\setminus O_1$ maps to a point in $U$, then for the vector $Z=dF_{x'}(r)$ we have
\begin{equation}\label{separation between O and V minus O_1} 
\hat{Z}(\frac{\partial }{\partial y_{1}}) <
\hat{X}(\frac{\partial }{\partial y_{1}})
\end{equation} 
for any $X=dF_{x}(r)$ with $x\in O$ and also
\begin{equation}\label{vectors from O_1 are mostly vertical} 
\hat{Y}(\frac{\partial }{\partial y_{1}}) > 1-k
\end{equation} 
for some $k>0$ sufficiently small and all $Y=dF_{y'}(r)$ for $y'\in O_1$.

Take $x^{1},x^{2}\in Q_{q}\cap O$ for $q\in U$.
The hypothesis $x^{1},x^{2}\in Q_{q}$ implies $q=F(x^{1})=F(x^{2})$, and so $x^{1}_{j}=x^{2}_{j} $ follows for all $j<n$. Let us write $a_{j}=x^{1}_{j}=x^{2}_{j} $ for  $j<n$, $s^{1}=x^{1}_{n}$ and $s^{2}=x^{2}_{n}$.
Fix $a_{2},\dots,a_{n-1}$ and consider the set
$$
H_{a}=
\left\lbrace 
x\in O: x_{i}=a_{i};i=2,\dots, n-1
\right\rbrace
$$
Its image by $F$ is a subset of a plane in the $y_i$ coordinates: 
$$
L_{a}=\left\lbrace y\in U: y_{i}=a_{i}, i=2,\dots, n-1\right\rbrace
$$
Points of $O_1$ not in $H_{a}$ map to other planes.
If $n=2$, we keep the same notation, but the meaning is that $H_a=O$ and $L_a=V$.

There is $\varepsilon>0$ such that for $t<-\varepsilon/2$, the line $\{x_1=t\}\cap H_a $ maps diffeomorphically to $\{y_1=t\}\cap L_a$.

Due to the comments at the beginning of this section, $\widetilde{S}$ is given as a Lipschitz graph $x_{1}=\tilde{t}(x_{2},\dots,x_{n})$.
The identity $a_{1}=\tilde{t}(a_{2},\dots,a_{n-1},s^{i})$ holds for $i=1,2$ because $x^{1},x^{2}\in Q_{q}$. We define a curve $\sigma:[s^{1},s^{2}]\rightarrow \widetilde{S}$ by $\sigma(s)=(\tilde{t}(a_{2},\dots,a_{n-1},s),a_{2},\dots,a_{n-1},s) $. The image of $\sigma$ by $F$ stays in $S$, describing a closed loop based at $q$; we will establish the lemma by examining the variation of $\tilde{u}$ along $\sigma$.

For $i=1,2$, let $\eta^i:(-\varepsilon_i,a_1]\rightarrow H_a$ given by $\eta^i(t)=(t,a_2,\dots,a_{n-1},s^i)$ be the segments parallel to the $x_{1}$ direction that end at $x^{i}$, defined from the first point in the segment that is in $O$.
We can assume that the intersection of $O$ with any line parallel to $\frac{\partial}{\partial x_1}$ is connected, and that the intersection of $U$ with any line parallel to $\frac{\partial}{\partial y_1}$ is connected too.
We can also assume $\varepsilon_i<\varepsilon$.

Let $D$ be the closed subset of $H_{a}$ delimited by the Lipschitz curves
$\eta^1$, $\eta^2$ and $\sigma$, and let $E$ be the closed subset of $L_{a}$
delimited by the image of $\eta^1$ and~$\eta^2$. 

We claim $D$ is mapped onto $E$.
First, no point in $int(D) $ can map to the image of the two lines, cause this
contradicts either $\rho\leq \lambda_{1}$, or the fact that
$\rho(a_{2},\dots,a_{n-1},s^{i})$ is the first time that the line parallel to
the $x_{1}$ direction hits~$\widetilde{S}$, for either $i=1$ or $i=2$.
We deduce $D$ is mapped into $E$.

Now assume $G=E\setminus F(D)$ is nonempty, and contains a point
$p\!=\!(p_1,...,p_n)$. 
If $Q_p$ contains a point $x\in O_1\setminus F(D)$, following the curve
$t\rightarrow (t,x_2,\dots,x_n)$ backwards from $x=(x_1,x_2,\dots,x_n)$, we must
hit either a point in the image of $\eta^i\vert_{(-\varepsilon_1,a_1)}$ (which
is a contradiction with the fact that both $(t,\dots,x_n) $ for $t<x_1$ and
$(t,a_2,\dots,a_{n-1},s^i) $ for $t<a_1$ are in $A(S)$; see definition
\ref{splits}), or the point $q$ (which contradicts  
\eqref{vectors from O_1 are mostly vertical}). Thus for any point $p\in G$, we
have $Q_p\subset V\setminus O_1$.  

Now take a point $p\in \partial G$, and
pick up a sequence approaching it from within $G$ and contained in a line with speed vector $\frac{\partial }{\partial y_{1}} $. By the above, the set $Q$ for points in this sequence is contained in $V\setminus O_1$. We can take a subsequence carrying a convergent sequence of vectors, 
 and thus $R_{p}$ has a vector of the form $dF_{x^{\ast}}(r) $ for $x^{\ast}\in D\subset O$.
This violates the balanced condition, because of 
\eqref{separation between O and V minus O_1}.
This implies $\partial G=\emptyset$, thus $G=\emptyset$ because $E$ is connected
and $F(D)\neq\emptyset $.

Finally, we claim there are no vectors coming from $V\setminus O_1$ in $R_p$ for
$p\in int(E)$. 
The argument is as above, but we now approach a point with a vector from
$V\setminus O_1$ in $R_p$ within $E=F(D)$ and with speed $-\frac{\partial
}{\partial y_{1}}$. The approaching sequence may be chosen so that it carries a
convergent sequence of vectors from $F(D)$, and again 
\eqref{separation between O and V minus O_1} gives a contradiction with the
balanced condition.

We now compute:
\begin{equation}\label{u dere- u zurda es una integral} 
 \tilde{u}(x^{1})-\tilde{u}(x^{2})
=\int_{s_{1}}^{s_{2}} \frac{d(\tilde{u}\circ \sigma)}{ds}
=\int_{\sigma} d\tilde{u}
\end{equation}

The curve $F\circ\sigma$ runs through points of $S$. If $F(\sigma(s))$ is a cleave point, then $F\circ\sigma$ is a smooth curve near $s$. We show that cleave points are the only contributors to the above integral. 
If a point is not cleave, either it is the image of a conjugate vector, or has more than $2$ incoming geodesics. As $F\circ\sigma$ maps into $int(E)$, all vectors in $R_{F(\sigma(s))}$ come from $O$.

Let $N$ be the set of $s$ such that $\sigma(s)$ is conjugate. We notice that $\sigma(s)$ is not an A2 point for $s\in N $. This is proposition \ref{no A2 in S}, and is a standard result for cut loci in Riemannian manifolds. This means that at those points the kernel of $dF$ is contained in the tangent to $\widetilde{S}$. The intersection of $\widetilde{S}$ with the plane $H_a$ is the image of the curve $\sigma $. Thus, for $s\in N $ the tangent to the curve $\lambda_1$ is the kernel of $d_{\sigma(s)}F$.
If $\sigma$ is differentiable at a point $s$ we deduce, thanks to \ref{rho es menor que landa1}, that the tangent to the curve $\lambda_1$ is the kernel of $d_{\sigma(s)}F$.

We now use a \emph{variation of length} argument to get a variant of the Finsler Gauss lemma. Let $c=(l,w) $ be a tangent vector to $V\subset \RR\times\partial M $ at the point $x=(t,z)$, and assume $d_xF(c)=0$.
We show that this implies $d\tilde{u}(c)=0$.
Let $\gamma_s$ be a variation through geodesics with initial point in $z(s)\in \partial M $ and the characteristic vector field at $z(s)$ as the initial speed vector, such that $\frac{\partial}{\partial s}z(s)=w$, and with total length $t+sl$.
By the first variation formula and the equation for the characteristic vector field at $\partial M$, the variation of the length of the curve $\gamma_s$ is $\frac{\partial 	\varphi}{\partial v}(p,d_x F(r_x)) \cdot d_x F(c) - \frac{\partial \varphi}{\partial v}(p,d_z F(r)) \cdot w = -dg(w) $, and by the definition of $\gamma_s$, it is also $l$. 
We deduce $l= -dg(w)$, and thus $d\tilde{u}(c)=l+dg(w)=0$. 

It follows that $d_{\sigma(s)}F(\sigma'(s))=0$ at points $s\in N$ where $\sigma$ is differentiable.
As $\sigma$ is Lipschitz, the set of $s$ where it is not differentiable has measure $0$, and we deduce:
$$
\int_{N}  d\tilde{u}(\sigma') = 0
$$

$N$ is contained in the set of points where $d(F\circ \sigma)$ vanishes. Thus, by the Sard-Federer theorem, the image of $N$ has Hausdorff dimension $0$.

Let $\Sigma_{2}$ be the set of points in $L_{a}$ with more than $2$ incoming geodesics. From the proof of \ref{main theorem 4}, we see that the tangent to $\Sigma_{2}$ has dimension $0$ and thus $\Sigma_{2}$ has Hausdorff dimension $0$.

As $F$ is non-singular at points in $[s_1,s_2]\setminus N$, the set of $s$ in $[s_1,s_2]\setminus N$ mapping to a point in $\Sigma_{2}\cup N$ has measure zero.

Altogether, we see that the integral 
\eqref{u dere- u zurda es una integral} can be restricted to the set $C$ of $s$
mapping to a cleave point. $C$ is an open set and thus can be expressed as the
disjoint union of a countable amount of intervals. Let $A_{1}$ be one of those
intervals. It is mapped by $F\circ\sigma$ diffeomorphically onto a smooth curve
$c_0$ of cleave points contained in $L_{a}$. Points of the form
$(t,a_{2},\dots,a_{n-1},s) $ for $t<\tilde{t}(a_{2},\dots,a_{n-1},s)$ map
through $F$ to a half open ball in $E$. There must be points of $D$ mapping to
the other side of $c_0$.
Because of all the above, $c_0$ is also the image of other points in
$[s_{1},s_{2}]$. As $c$ is made of cleave points, it must be the image of
another component of $C$, which we call $B_{1}$, also mapping diffeomorphically
onto $c_0$. Choose a new component $A_{2}$, which is matched to another
component $B_{2}$, different from the above, and so on, till the $A_{i}$ and
$B_{i}$ are all the components of $C$.

We can write the integral on $B_{i}$ as an integral on $A_{i}$ (we add a minus
sign, because the curve is traversed in opposite directions):
$$
\int_{A_{i}} d\tilde{u}(\sigma')+
\int_{B_{i}} d\tilde{u}(\sigma')=
\int_{A_{i}} du_{l}((F\circ\sigma)') - du_{r}((F\circ\sigma)')
$$
where $u_{l} $ and $u_{r} $ are the values of $u$ computed from both sides, evaluated at points in $U$. 
The balanced condition implies
$(F\circ\sigma)'\in \ker(d\tilde{u}_{l} - d\tilde{u}_{r}) $,
and thus the above integral vanishes.
The integral \eqref{u dere- u zurda es una integral} is absolutely convergent
by Lemma \ref{rho is lipschitz}, and the proof follows.

\end{proof}

\begin{remark}
The above proof took some inspiration from \cite[5.2]{Hebda}. The reader may be interested in James Hebda's \emph{tree-like curves}. 
\end{remark}

\subsection{Structure of S near cleave and crossing points}

In this section we prove some more results about the structure of a balanced split locus near degenerate cleave and crossing points. Besides their importance for proving that $\partial T=0$, we believe they are interesting in their own sake.

\begin{lem}\label{structure of cleave points} 
Let $p\in S$ be a (possibly degenerate) \emph{cleave} point, and let $Q_p=\{x^1,x^2\}$.

There are disjoint univocal neighborhoods $O_1$ and $O_2$ of $x^1$ and $x^2$, and a neighborhood $U$ of $p$ such that for any $q\in U$, $Q_q$ is contained in $O_1\cup O_2$.

Furthermore, if we define:
$$
A_i=\left\lbrace q\in U \text{ such that } Q_q\cap O_i\neq \emptyset  \right\rbrace 
$$
for $i=1,2$, then $A_1\cap A_2$ is the graph of a Lipschitz function, for adequate coordinates in $U$.
\end{lem}
\begin{proof}
The points $x^1$ and $x^2$ are at most of first order, so we can take univocal neighborhoods $O_{1}$ and $O_{2}$ of $x^1$ and $x^2$.
By definition of $Q_p$ and the compactness of $ M$, we can achieve the first property, reducing $U$ if necessary.

We know $\widehat{d_{x^1}F(r)}$ is different from $\widehat{d_{x^2}F(r)}$.
For fixed arbitrary coordinates in $U$, we can assume that $\{\widehat{d_{x}F(r)} \text{ for } x \in O_1\}$ can be separated by a hyperplane from $\{\widehat{d_{x}F(r)} \text{ for } x \in O_2\}$, after reducing $U$, $O_1$, $O_2$ if necessary.
Therefore, there is a vector $Z_0\in T_p M$ and a number $\delta>0$ such that
\begin{equation}\label{separation between A_1 and A_2} 
\widehat{d_{x}F(r)}(Z) < \widehat{d_{x'}F(r)}(Z)+\delta\qquad \forall\, x\in O_1, \, x'\in O_2
\end{equation} 
for any unit vector $Z$ in a neighborhood $G$ of $Z_0$. Let $C^+=\{tZ:t>0,Z\in G\} $ be a one-sided cone containing $Z$. We write $q+C^+$ for the cone displaced to have a vertex in $q$.

Choose $q\in A_1\cap A_2$, and $Z\in G $. Let $\mathcal{R}=\{q'\in U:q'=q+tZ,t>0\} $ be a ray contained in $(q+C^+)\cap U$.
We claim $\mathcal{R}\subset A_1 \setminus A_2 $.

For two points $q_1=q+t_1 Z,q_2=q+t_2 Z\in \mathcal{R} $, we say $q_1<q_2$ if and only $t_1<t_2$.
If $\mathcal{R}\cap A_2\neq\emptyset $, let $q_0$ be the infimum of all points $p>0$ in $\mathcal{R}\cap A_2$, for the above order in $\mathcal{R}$.
If $q_0 \in A_1$ (whether $q_0=q$ or not), we can approach $q_0$ with a sequence
of points $q_n=F(x_n)>q_0$ carrying vectors $d_{x_n}F(r)$ with $x_n\in O_2$. The
limit point of this sequence is $q_0$, and the limit vector is $d_{x}F(r)$ for
some $x\in O_2$, but the incoming vector is in $-G$, which contradicts the
balanced condition by 
\eqref{separation between A_1 and A_2}.

If $q_0\in A_2\setminus A_1$, then approaching $q_0$ with points $q<q_n=F(x_n)<q_0$, we get a new contradiction with the balanced property. The only possibility is  $\mathcal{R}\subset A_1\setminus A_2$. As the vector $Z$ is arbitrary, we have indeed $(q+C^+)\cap U\subset A_1\setminus A_2$.

Fix coordinates in $U$, and let $\varepsilon=\frac{1}{2}dist(p,\partial U)$.
Let $B_\varepsilon$ be the ball of radius $\varepsilon $ centered at $p$.
By the above, the hypothesis of lemma \ref{the set and the cone} are satisfied, for $A=A_1\setminus A_2$, the cone $C^+$, the number $\varepsilon$, and $V=B_\varepsilon$.
Thus, we learn from lemma \ref{the set and the cone} that  $A_1\cap A_2\cap B_\varepsilon$ is the graph of a Lipschitz function along the direction $Z_0$ from any hyperplane transversal to $Z_0$.

\end{proof}

The following three lemmas contain more detailed information about the structure of a balanced split locus near a crossing point. The following is stated for the case $n>2$, but it holds too if $n=2$, though then $L$ reduces to a single point $\{a\}$.

\begin{dfn}
 The \emph{normal} to a subset $X\subset T_p^{\ast}   M$ is the set of vectors $Z$ in $T_p  M$ such that $ M(Z)$ is the same number for all $ M \in X$.
\end{dfn}

\begin{lem}\label{structure of crossing points1} 
Let $p\in S$ be a crossing point. Let $B\subset T_p^{\ast}   M$ be the affine plane spanned by $R_p^\ast $. Let $L$ be the normal to $B$, which by hypothesis is a linear space of dimension $n-2$, and let $C$ be a (double-sided) cone of small amplitude around $L$.

There are disjoint univocal open sets $O_1,\dots,O_N\subset V$ and an open neighborhood $U$ of $p$
such that $Q_q\subset \cup_i O_i$ for all $q$ in $U$.

Furthermore, define sets $A_i$ as in lemma \ref{structure of cleave points}, and call $\mathcal{S}=\cup_{i,j}A_i\cap A_j$ the \emph{essential part} of $S$.
Define $\Sigma= \cup_{i,j,k}A_i\cap A_j\cap A_k$ and let $\mathcal{C}= \mathcal{S}\setminus \Sigma$.
\begin{enumerate}
\item[(1)] At every $q\in \Sigma$, there is $\varepsilon>0$ such that $\Sigma\cap (q+ B_\varepsilon)\subset q+C$.
\item[(2)] $\Sigma$ itself is contained in $p+C$.
\end{enumerate}
\end{lem}

The next lemma describes the intersection of $\mathcal{S}$ with $2$-planes transversal to $L$.

\begin{lem}\label{structure of crossing points2} 
Let $p\in S$ be a crossing point as above. Let $P\subset  T_p M$ be a $2$-plane intersecting $C$ only at the origin, and let $P_a=P+a$ be a $2$-plane parallel to $P$ for $a\in L $.
\begin{enumerate}
\item If $\vert a\vert<\varepsilon_1$, the intersection of $\mathcal{S} $, the plane $P_a$, and $U$ is a connected Lipschitz tree.
\item The intersection of $\mathcal{S} $, the plane $P_a$, and the annulus of inner radius $c\cdot\vert a\vert $ and outer radius $\varepsilon_2 $:
$$
A(c\:\vert a\vert,\varepsilon_2)=\{q\in U: c\:\vert a\vert<\vert q\vert<\varepsilon_2\}
$$
is the union of $N$ Lipschitz arcs separating the sets $A_i$.
\end{enumerate}
\end{lem}

\begin{figure}[ht]
 \centering
 \includegraphics[width=.9\textwidth]
 {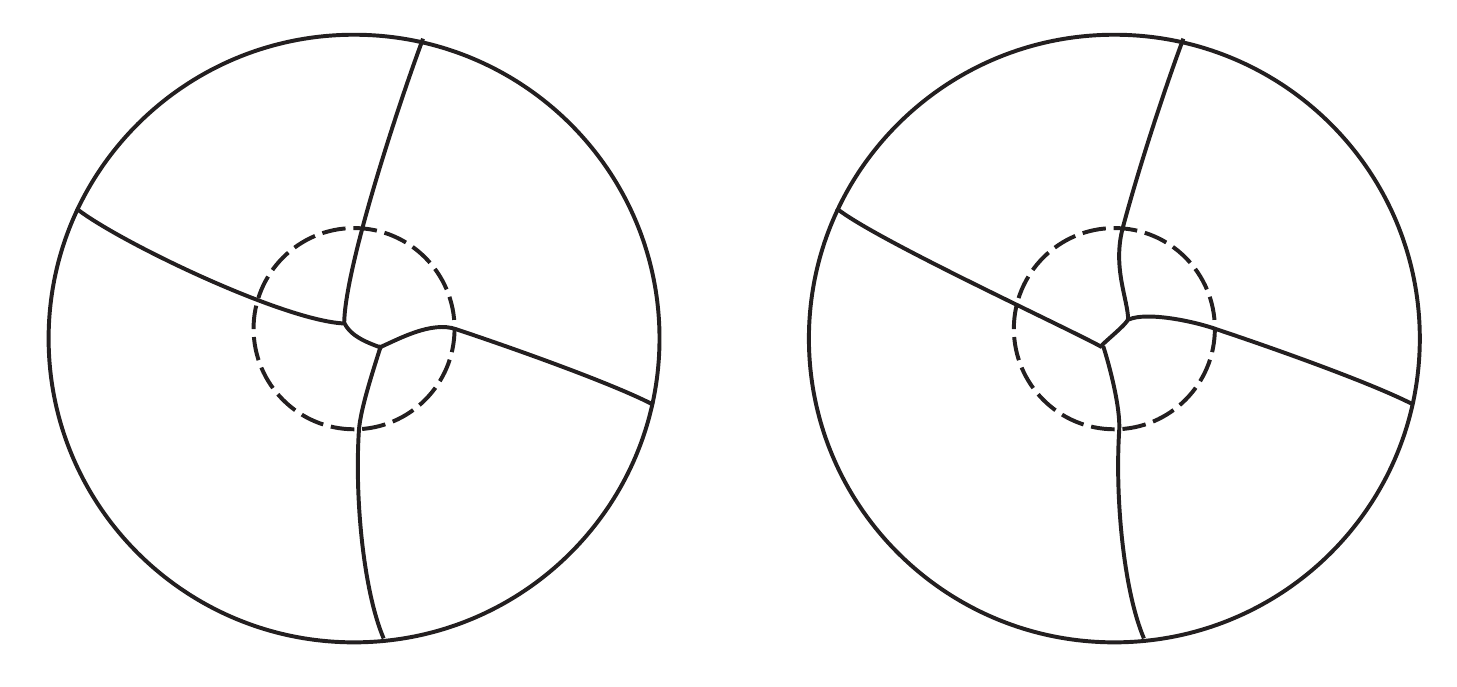}
 
 \caption{Two possible intersections of a plane $P_a$ with $\mathcal{S}$}
 \label{fig:intersections of plane with S near crossing}
\end{figure}

\begin{remark}
We cannot say much about what happens inside $P_a\cap B(P,c\:\vert a\vert) $. The segments in $ P_a\cap A(c\:\vert a\vert,\varepsilon_2)$ must meet together, yielding a connected tree, but this can happen in several different ways (see figure \ref{fig:intersections of plane with S near crossing}). 
\end{remark}

Finally, we can describe the connected components of $\mathcal{C}=\mathcal{S}\setminus \Sigma$ within $U$:
\begin{lem}\label{structure of crossing points3} 
Under the same hypothesis, for every $i=1,\dots,N$ there is a coordinate system in $U$ such that:
\begin{itemize}
\item The set $\partial A_i$ is the graph of a Lipschitz function $h_i$, its domain delimited by two Lipschitz functions $f_l$ and $f_r$, for $L^\ast \subset L$:
$$
\partial A_i = \left\lbrace (a,t,h_i(t)), a\in L^\ast, f_{l}(a)<t<f_r(a) \right\rbrace
$$
\item A connected component $\mathcal{C}_0$ of $\mathcal{C}$ contained in $\partial A_i$ admits the following expression, for Lipschitz functions $f_1$ and $f_2$, for $L_0 \subset L$:
$$
\mathcal{C}_0 = \left\lbrace (a,t,h_i(t)), a\in L_0, f_1(a)<t<f_2(a) \right\rbrace
$$
\end{itemize}
\end{lem}
\begin{cor}
$\mathcal{H}^{n-2}(\Sigma)<\infty$.
\end{cor}
\begin{proof}[Proof of corollary]
We apply the \emph{general area-coarea formula} (see \cite[3.2.22]{Federer}), with $W=\Sigma$, $Z=L$, and $f$ the projection from $U$ onto $L$ parallel to $P$, and $m=\mu=\nu=n-2$, to learn:
$$
\int_{\Sigma}ap\: Jf d\mathcal{H}^{n-2}
=\int_L \mathcal{H}^0(f^{-1}(\{z\})) d\mathcal{H}^{n-2}(z)
=\int_L \mathcal{H}^0(\Sigma\cap P_a) d\mathcal{H}^{n-2}(a)
$$
$ap\: Jf|_{\Sigma}$ is bounded from below, so
if we can bound $\mathcal{H}^0(\Sigma\cap P_a) $ uniformly, we get a bound for $\mathcal{H}^{n-2}(\Sigma) $.

The set $\mathcal{C} \cap P_a \cap U$ is a simplicial complex of dimension $1$, and a standard result in homology theory states that the number of edges minus the number of vertices is the same as the difference between the homology numbers of the complex: $h^1-h^0 $. The graph is connected and simply connected, so this last number is $-1$. The vertices of $\mathcal{C} \cap P_a \cap U $ consist of $N$ vertices of degree $1$ lying at $\partial U$ and the interior vertices having degree at least $3$. 
The \emph{handshaking lemma} states that the sum of the degrees of the vertices of a graph is twice the number of edges, so we get the inequality $2e\geq N+3 \bar{v} $ for the number $e$ of edges and the number $\bar{v}$ of interior vertices. Adding this to the previous equality $e-(N+\bar{v})=-1 $, we get $\bar{v}\leq N-2 $.
We have thus bounded $\bar{v}= \mathcal{H}^0(\Sigma\cap P_a)$ with a bound valid for all $a$.

\end{proof}

\begin{proof}[Proof of \ref{structure of crossing points1}]
This lemma can be proven in a way similar to \ref{structure of cleave points}, but we will take some extra steps to help us with the proof of the other lemmas.

First, recall the map $\Delta$ defined in 
\eqref{definition of the embedding of V into TastOmega}.
Each point $x$ in $\Delta^{-1}(R_p^{\ast})$ has a univocal neighborhood $\mathcal{O}_x $. 
Recall $R_p^{\ast}$ consists only of covectors of norm $1$. Let $\gamma $ be the curve obtained as intersection of $B$ and the covectors of norm $1$.
Instead of taking the neighborhoods $\mathcal{O}_x $ right away, which would be sufficient for this lemma, we cover $R_p^{\ast}$ with open sets of the form $\Delta(\mathcal{O}_x) \cap \gamma$.

By standard results in topology, we can extract a finite refinement of the covering of $R_p^{\ast}\subset\gamma$ by the sets $\Delta(\mathcal{O}_x)\cap \gamma$ consisting of disjoint non-empty intervals $I_1,\dots,I_N$.
Let $\tilde{I}_i$ be the set of points $tx$ for $t\in (1-\varepsilon_1,1+\varepsilon_1)$ and $x\in I_i$, and choose a linear space $M_0$ of dimension $n-2$ transversal to $B$. Define the sets of our covering:
$$
O_i = \Delta^{-1}(\tilde{I}_i + B(M_0,\varepsilon_2))
$$
for the ball of radius $\varepsilon_2$ in $M_0$ ($\varepsilon_1 $ and $\varepsilon_2 $ are arbitrary, and small).

We can assume that $Q_q\subset \cup_i O_i$ for all $q$ in $U$ by reducing $U$ and the $O_i$ further if necessary, hence we only need to prove the two extra properties to conclude the theorem. 

The approximate tangent to $\Sigma$ at a point $q\in \Sigma\cap U$ is contained in the normal to $R_q^{\ast}$ (recall the definition \ref{approximate tangent cone} of approximate tangent cone, and use proposition \ref{main theorem 4}, or merely use the balanced property).
If $R_q^{\ast}$ is contained in a sufficiently small neighborhood of $\gamma$ and contains points from at least three different $I_i$, its normal must be close to $L$.
Thus if we chose $\varepsilon_1$ and $\varepsilon_2$ small enough, the approximate tangent to $\Sigma\cap U$ at a point $q\in \Sigma$ is contained in $C$.
If property (1) did not hold for any $\varepsilon$ at a point $q$, we could find a sequence of points converging to $q$ whose directions from $q$ would remain outside $C$,  violating the above property.

Finally, the second property holds if we replace $U$ by $U\cap B_\varepsilon$, for the number $\varepsilon$ that appears when we apply property (1) to $p$.
\end{proof}

\begin{figure}[ht]
 \centering
 \includegraphics[width=.9\textwidth]{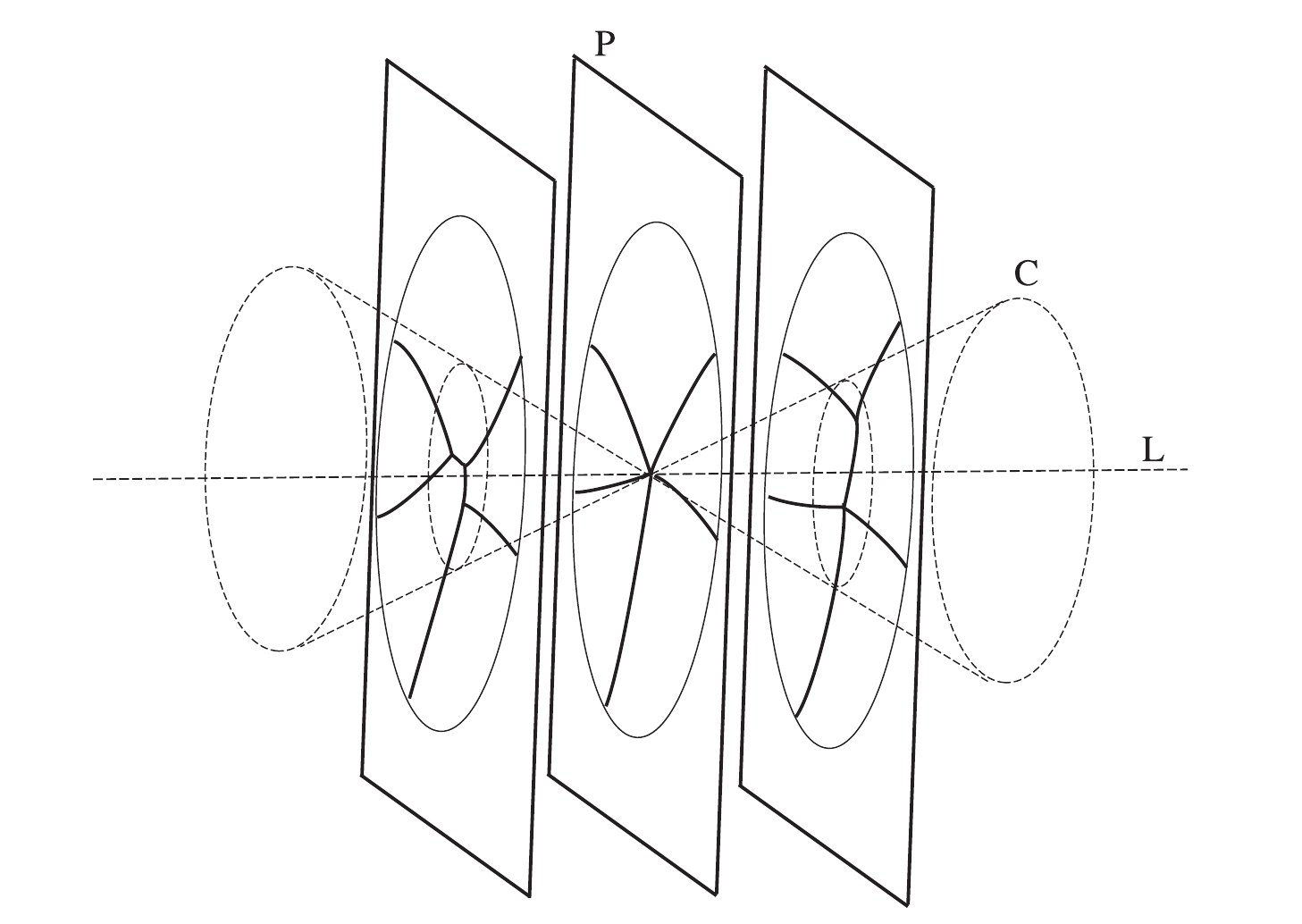}
 
 \caption{$\mathcal{S}$ near a crossing point} 
\label{fig: S near a crossing point}
\end{figure}

\begin{proof}[Proof of \ref{structure of crossing points2}]
Just like in \ref{structure of cleave points}, we can assume that each set
$\{\widehat{d_{x}F(r)}$ for $x \in O_i\}$ can be separated from the
others by a hyperplane (e.g., a direction~$Z_i$), such that:
\begin{equation}\label{separation between A_i and the others} 
\widehat{d_{x}F(r)}(Z) < \widehat{d_{x'}F(r)}(Z)+\delta\qquad \forall \, x\in O_i,\, x'\in O_j,\, i\neq j
\end{equation} 
for some $\delta>0$ and any unit vector $Z$ in a neighborhood $G_i$ of $Z_i$.
Thanks to the care we took in the proof of the previous lemma, we can assume all $Z_i$ belong to the plane $P$ in the statement of this lemma: indeed the intervals $I_i$ can be separated by vectors in any plane transversal to $L$, and the sets $\Delta(O_i)$ are contained in neighborhoods of the $I_i$.

Define the one-sided cones $C^+_i=\{tZ:t>0,Z\in G_i\}$.
The above implies that the intersection of each $C_i^+$ with $P$ is a nontrivial cone in $P$ that consists of rays from the vertex. 

By the same arguments in \ref{structure of cleave points}, we can be sure that whenever $q\in A_i $, then $(q+C^+_i )\cap U\subset A_i$. 
This implies that $\partial A_i$ is the graph of a Lipschitz function along the direction $Z_i$ from any hyperplane transversal to $Z_i$. 
We notice $\partial A_i$ is (Lipschitz) transversal to $P$, so for any $a\in L$, $\partial A_i\cap P$ is a Lipschitz curve.
As the cone $C$ is transversal to $P$, and the tangent to $\Sigma$ is contained in $C$, we see $\Sigma\cap P_a$ consists of isolated points.

Thus $\mathcal{S} \cap P_a$ is a Lipschitz graph and $\Sigma\cap P_a$ is the set
of its vertices. 
If~it were not a tree, there would be a bounded open subset of $P_a \cap
U\setminus\mathcal{S}$ with boundary contained in $\mathcal{S}$. An interior
point $q$ belongs to some $A_k$. Then the cone $q+C^+_k$ is contained in $A_k$,
but on the other hand its intersection with $P_a$ contains a ray that must
necessarily intersect $\mathcal{S}$, which is a contradiction.

We notice $P_a\cap (p+C^+_i)\subset A_i$.
This set is a cone in $P_a$ (e.g. a circular sector) with vertex at most a distance $c_1\vert a\vert $ from $p+a$, where $c_1>0$ depends on the amplitude of the different $C_i$.

If $a=0$, the $N$ segments departing from $p$ with speeds $Z_i$ belong to each $A_i$ respectively.
Let us assume that the intervals $I_i$ appearing in the last proof are met in the usual order $I_1,I_2\dots,I_N$ when we run along $\gamma$ following a particular orientation, and call $P^i$ the region delimited by the rays from $p$ with speeds $Z_i$ and $Z_{i+1}$ (read $Z_1$ instead of $Z_{N+1}$).
If there is a point $q\in P^i\cap A_k\cap B(\varepsilon_2)$ for sufficiently small $\varepsilon_2$, then $(q+C^+_k)\cap U$ would intersect either $p+C^+_i$ or $p+C^+_{i+1}$, and yield a contradiction if $k$ is not $i$ or $i+1$.
Thus $P^i\subset A_i\cup A_{i+1}$. Clearly there must be some point $q$ in $P_i\cap A_i\cap A_{i+1}$, to which we can apply lemma \ref{structure of cleave points}. 
$A_i\cap A_{i+1}$ is a Lipschitz curve near $q$ transversal to $Z_i$ (and to $Z_{i+1}$), and it cannot turn back.
The curve does not meet $\Sigma$, and it cannot intersect the rays from $p$ with speeds $Z_i$ and $Z_{i+1}$, so it must continue up to $p$ itself.
For any $q\in A_i\cap A_{i+1}$, the cone $q+C^+_i$ is contained in $A_i$, and the cone $q+C^+_{i+1}$ is contained in $A_{i+1}$. This implies there cannot be any other branch of $A_i\cap A_{i+1}$ inside $P_i$.

This is all we need to describe $\mathcal{S}\cap P \cap B(\varepsilon_2)$: it consists of $N$ Lipschitz segments starting at $p$ and finishing in $P \cap \partial B(\varepsilon_2) $. The only multiple point is $p$.

For small positive $\vert a\vert$, we know by condition (2) of the previous lemma that $P_a\cap \Sigma\subset C\cap P_a= B(c_2\vert a\vert) \cap P_a$ for some $c_2>0$.
Similarly as above, define regions $P_a^i\subset P_a \cap A(c\vert a\vert,\varepsilon_2)$ delimited by the rays from $a$ with directions $Z_i$ and $Z_{i+1}$, and the boundary of the ring $A(c\vert a\vert,\varepsilon_2)$, for a constant $c>\max(c_1,c_2)$.
Take $c$ big enough so that  for any $q\in P_a^i $ and any $k\neq i,i+1$ , $q+C^+_k \cap U\cap P_a$ intersects either $p+C^+_i$ or $p+C^+_{i+1} $.
The same argument as above shows that $A_i\cap A_{i+1} \subset P_a^i \subset A_i\cup A_{i+1} $.
We conclude there must be a Lipschitz curve of points of $A_i\cap A_{i+1}$, which starts in the inner boundary of $A(c\vert a\vert,\varepsilon_2)$, and ends up in the outer boundary.

\end{proof}

\begin{figure}[ht]
 \centering
 \includegraphics[width=.9\textwidth]{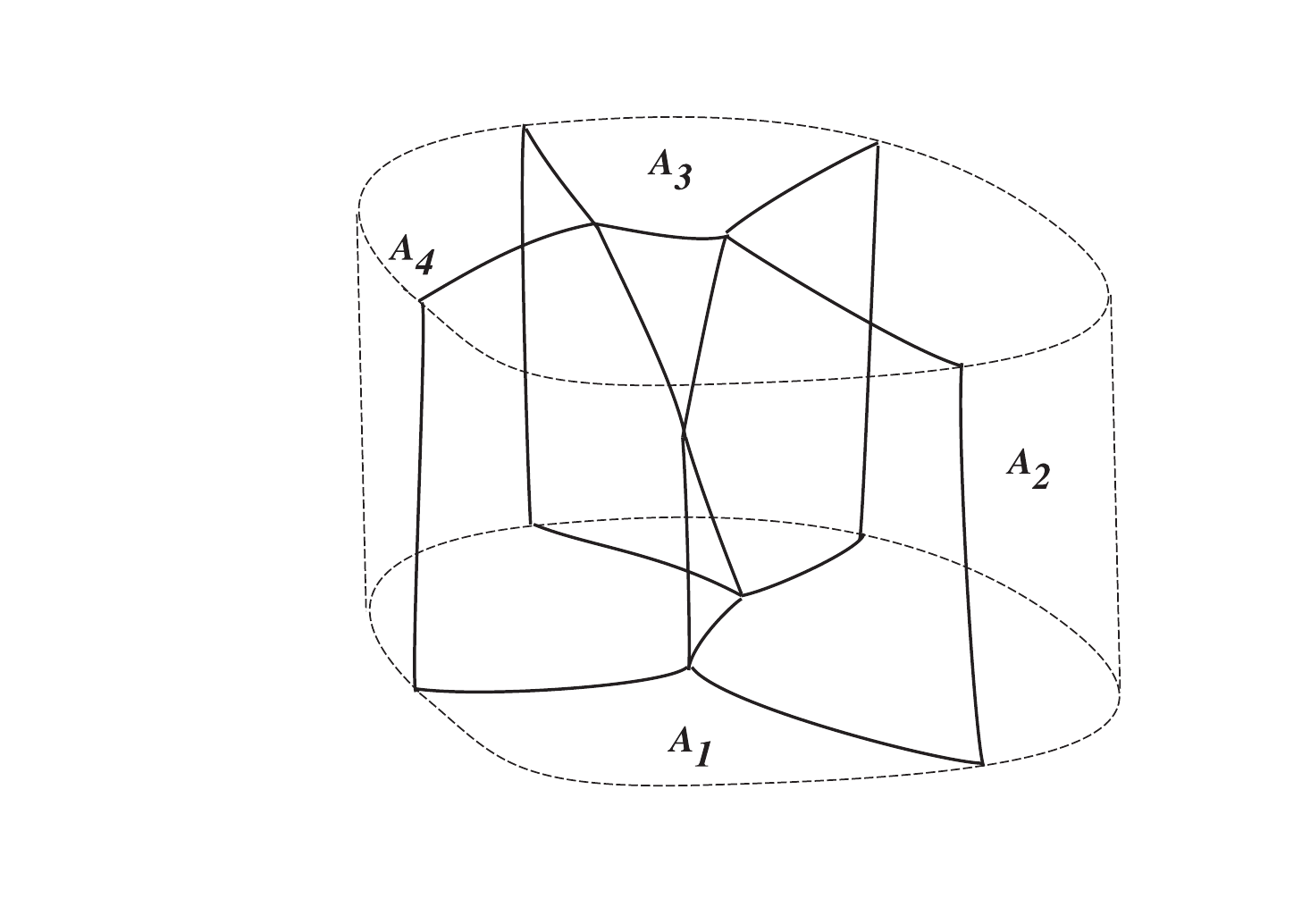}
 
 \caption{A neighborhood of a crossing point (this view is rotated with respect to figure \ref{fig: S near a crossing point})}
 \label{fig: neighborhood of a crossing point}
\end{figure}

\begin{proof}[Proof of \ref{structure of crossing points3}]
First we assume $U$ has a product form $U=L^\ast\times P^\ast$ for open discs $L^\ast\subset L$ and $P^\ast=B(P,\varepsilon_2)\subset P$.

Recall $\partial A_i$ is the graph of a Lipschitz function along the direction $Z_i$ from any hyperplane transversal to $Z_i$.
Let $H_i=L+W$ be one such hyperplane that contains the subspace $L$ and the vector line $W\subset P$, and construct coordinates $L\times W\times <Z_i>$.
It follows from the previous lemma that $\partial A_i\cap P^\ast_a$ is a connected Lipschitz curve.
In these coordinates $\partial A_i$ is the graph of a Lipschitz function $h_i$.
Its domain, for fixed $a$, is a connected interval, delimited by two functions $f_l:L^\ast\rightarrow W$ and $f_r:L^\ast\rightarrow W$. Condition (1) of lemma \ref{structure of crossing points1} assures they are Lipschitz.

A connected component $\mathcal{C}_0$ of $\mathcal{C}$ is contained in only one $A_i\cap A_j$.
We can express it in the coordinates defined above for $\partial A_i$.
The intersection of $\mathcal{C}_0$ with each plane $P_a$ is either empty or a connected Lipschitz curve. The second part follows as before.
\end{proof}

\subsection{Conclusion}

Using lemma \ref{uniqueness near order 1 points}, we show without much effort that $\partial T$ vanishes near edge points.
Using the structure results from the previous section, we show also that it vanishes at cleave points (including degenerate ones) and crossing points.

\begin{prop}
 Let $p\in S$ be an \emph{edge} point. 
Then the boundary of $T$ vanishes near $p$.
\end{prop}
\begin{proof}
Let $p$ be an edge point with $Q_p=\{x\}$.
Let $O$ be a univocal neighborhood of $x$.
It follows by a contradiction argument that there is an open neighborhood $U$ of $p$ such that $Q_q\subset O$ for all $q\in U $.
Recall the definition of $T$:
$$
T(\phi)=\sum_{j}\int_{\mathcal{C}_{j,1}}(h_{1}-h_{2})\phi
$$
For any cleave point $q\in U$ with $Q_q=\{x_1,x_2\}$, $h_i(q)=\tilde{u}(x_i)$.
By the above, both $x_1$ and $x_2$ are in $O$.
As $O$ is univocal, we see $h_1=h_2$ at $q$. The integrand of $T$ vanishes near $p$, and thus $\partial T =0$.
\end{proof}

\begin{prop}
Let $p\in S$ be a (possibly degenerate) \emph{cleave} point. 
Then $\partial T$ vanishes near $p$.
\end{prop}
\begin{proof}

Use the sets $U$, $A_1$ and $A_2$ of lemma \ref{structure of cleave points}.

Whenever $\phi$ is a $n-1$ differential form with support contained in $U$, we can compute:
$$
T(\phi)=\int_{A_1\cap A_2}(h_1-h_2)\phi
$$
The components of cleave points inside either $A_1$ or $A_2$ do not contribute to the integral, for the same reasons as in the previous lemma. Recall the definition of $\partial T$, for a differential $n-2$ form $\sigma$:
$$
\partial T(\sigma)=T(d\sigma)=
\int_{A_1\cap A_2}(h_1-h_2)d\sigma
$$
We can apply a version of Stokes theorem that allows for Lipschitz functions. We will provide references for this later:

$$
T(d\sigma)=\int_{A_1\cap A_2}d(h_1-h_2)\sigma
$$

The balanced condition imposes that for any vector $v$ tangent to $A_1\cap A_2$ at a non-degenerate cleave point $q$ with $Q_q=\{x_1,x_2\}$.
$$
\hat{X}^1(v)=\hat{X}^2(v)
$$
for the incoming vectors $X^i=d_{x_i}F(r)$. Recall that $\Hnuno$-almost all points are cleave, and $dh_i$ is dual to the incoming vector $X^i$, so $T(d\sigma)=0 $.
\end{proof}

\begin{prop}
Let $p\in S$ be a \emph{crossing} point. 
Then the boundary of the current $T$ (defined in \ref{definition of T}) vanishes near $p$.
\end{prop}

\begin{proof}

We use lemma \ref{structure of crossing points3} to describe the structure of connected components of $\mathcal{C}$ near $p$.
Let $\Sigma_T$, the set of \emph{higher order points}, be the set of those points such that $R_q^{\ast}$ spans an affine subspace of $T^{\ast}_q  M$ of dimension greater than $2$.

Take any connected component $\mathcal{C}_0$ of $\mathcal{C}$ contained in $\partial A_i$.
$\partial \mathcal{C}_0 $ decomposes into several parts: 

\begin{itemize}
 \item The regular boundary, consisting of two parts $D_1$ and $D_2$:
$$D_1\!=\!\{(a_1,\dots,a_{n-2},f_1(a),h_i(f_1(a))), \forall a\in L^\ast \text{
such that } f_l(a)<f_1(a)<f_2(a)\}
$$ 
$$
D_2\!=\!\{(a_1,\dots,a_{n-2},f_2(a),h_i(f_2(a))), \forall a\in L^\ast \text{
such that } f_1(a)<f_2(a)<f_r(a)\}
$$

\item The points of higher order, or $\partial \mathcal{C}_0 \cap \Sigma_T$.

\item The singular boundary, or those points $q=(a_1,\dots,a_{n-2},f_1(a),h_i(f_1(a))) $ where $f_1(a)=f_2(a)$ and $R_q$ is contained in an affine plane.

\item A subset of $\partial U$.

\end{itemize}

Using a version of Stokes theorem that allows for Lipschitz functions, we see that
$$
\int_{\mathcal{C}_0} v d\sigma= \int_{\mathcal{C}_0} d(v \sigma)-\int_{\mathcal{C}_0} (dv) \sigma=
\int_{D_1} v \sigma - \int_{D_2} v \sigma-\int_{\mathcal{C}_0} (dv) \sigma
$$
for any function $v$ and $n-2$ form $\sigma$ with compact support inside $U$.
Indeed, the last coordinate of the parametrization of $\mathcal{C}_0 $ is given by a Lipschitz function, so we can rewrite the integral as one over a subset of $L\times W$, and only Gauss-Green theorem is needed.
We can apply the version in \cite[4.5.5]{Federer}, whose only hypothesis is that the current $\Hnuno\lfloor \partial \mathcal{C}_0 $ must be representable by integration.
Using \cite[4.5.15]{Federer} we find that it is indeed, because its support is contained in a rectifiable set.
Here we are assuming that $D_1$ is oriented as the boundary of $\mathcal{C}_0 $, while $D_2$ is oriented in the opposite way, to match the orientation of $D_1$.

Notice we have discarded several parts of $\partial \mathcal{C}_0$:
\begin{itemize}
 \item A subset of $\partial \mathcal{C}_0$ inside $\partial U$ does not contribute to the integral because $supp(\sigma)\subset\subset U $.
 \item $\partial \mathcal{C}_0 \cap \Sigma_T$ does no contribute because it has Hausdorff dimension at most $n-3$.
 \item The singular boundary does not contribute either, because the normal to $\widetilde{\mathcal{C}_0}$ at a point of the singular boundary does not exist (see \cite[4.5.5]{Federer}).
\end{itemize}

We now prove that $\partial T=0$.

For a form $\sigma $ of dimension $n-2$ and compact support inside $U$:
\begin{align*}
T(d\sigma)& =\sum_{i}\int_{\mathcal{C}_i}(h_l-h_r)d\sigma\\
&=
\sum_{i}\int_{\mathcal{C}_i}d(h_l-h_r)\sigma+\sum_{i}\left(\int_{D_{i,1}}(h_l-h_r)\sigma-\int_{D_{i,2}}(h_l-h_r)\sigma\right)
\end{align*}
where $D_{i,1}$ and $D_{i,2}$ are the two parts of the regular boundary of $\mathcal{C}_i $.

The first summand is zero and the remaining terms can be reordered (the sum is
absolutely convergent because $h$ is bounded and $\Hndos(\Sigma)$ is finite):
$$
\sum_{i}\left(\int_{D_{i,1}}(h_l-h_r)\sigma-\int_{D_{i,2}}(h_l-h_r)\sigma\right)=
\int_{\Sigma\setminus\Sigma_T} \sum_{(i,j)\in I(q)}(h_{i,j,l}-h_{i,j,r})\sigma dq
$$
where every point $q\in \Sigma\setminus\Sigma_T$ has a set $I(q)$ consisting of
those $i$ and $j=1,2$ such that $q$ is in the boundary part $D_j$ of the
component $\mathcal{C}_i$.
The integrand at point $q$ is then:
$$
\sigma\sum_{(i,j)\in I(q)}(h_{i,j,l}-h_{i,j,r})
$$
where $h_{i,j,l}$ is the value of $\tilde{u}(x)$ coming from the side $l$ of component $\mathcal{C}_i$ and boundary part $D_j$.

By the structure lemma \ref{complete description}, we can restrict the integral to crossing points.
Let $O_1,\dots,O_N$ be the disjoint univocal sets that appear when we apply \ref{structure of crossing points1} to $p$.
For a crossing point $q$, $I(q)$ is in correspondence with the set of indices $k$ such that $O_k\cap Q_p \neq \emptyset $. Indeed, the intersection of $\mathcal{S} $ with the plane $P_a$ containing $q$ is a Lipschitz tree, and $q$ is a vertex, and belongs to the regular boundary of the components that intersect $P_a$ in an edge. 
The $h_{i,j,l}$ in the sum appear in pairs: one is the value from the left coming from one component $\mathcal{C}_i$ and the value from the right of another component $\mathcal{C}_{i'}$.
Each one comes from a different side, so they carry opposite signs, and they cancel.
The integrand at $q$ vanishes altogether, so $\partial T =0$.

\end{proof}

\chapter{A new proof of the Ambrose conjecture for generic $3$-manifolds}\label{chapter: ambrose}

We give a proof of the {\emph{Ambrose conjecture}}, a global version of the Cartan local lemma.
The proof is given only for a generic class of Riemannian manifolds of dimension $3$.
In 2010, J. Hebda gave a proof in \cite{Hebda10} of the Ambrose conjecture for a (different) generic class of Riemanian manifolds of any dimension.
His proof is also much shorter.
However, his proof does not extend to arbitrary metrics and we think that our proof might, even though we have been unable to do so to this day. Indeed, the proof presented here extends to some manifolds that are not covered by the result of J. Hebda, so this is truly a different approach.

Finally, some of the techniques presented here, such as the \emph{conjugate descending flow}, or the \emph{linking curves} might be useful for other problems, as commented in the chapter \ref{chapter: beyond}.

\section{Introduction}

\subsection{Cartan lemma}

Let ($M_{1} ,g_{1}$) and ($M_{2} ,g_{2}$) be two Riemannian manifolds
{\emph{of the same dimension}}, with selected points $p_{1} \in M_{1}$ and
$p_{2} \in M_{2}$. 
We will speak about the pointed manifolds ($M_{1} ,p_{1}$)
and ($M_{2} ,p_{2}$). Any linear map $L:T_{p_{1}} M_{1} \rightarrow T_{p_{2}}
M_{2}$ induces the map $\varphi = \exp_{2} \circ L \circ ( \exp_{1} |_{O_{1}}
)^{-1}$, defined in any domain $O_{1} \subset T_{p_{1}} M_{1}$ such that
$ e_{1}|_{O_{1}}$ is injective (tipically, $O_{1}$ is a normal neighborhood of
$p_{1}$).

A classical theorem of E. Cartan {\cite{Cartan51}} identifies a situation
where this map is an isometry. The following is both a reformulation and a
slight generalization:

\begin{dfn}
  Let $(M_{1}, p_{1} )$ and $(M_{2}, p_{2} )$ be complete Riemannian manifolds
  of the same dimension with base points, and $L:T_{p_{1}} M_{1} \rightarrow
  T_{p_{2}} M_{2}$ a linear map.
  
  Let $\gamma_{1}$ and $\gamma_{2}$ be the geodesics defined in the interval
  $[0,1]$, with $\gamma_{1}$ starting at $p_{1}$ with initial speed vectors $x
  \in T_{p_{1}} M_{1}$ and $\gamma_{2}$ starting at $p_{2}$ with initial speed
  $L(x)$.
  
  For any three vectors $v_{1} ,v_{2} , v_{3}$ in $T_{p_{1}} M_{1}$, define:
  \begin{itemizedot}
    \item $R_{1} (v_{1} ,v_{2} , v_{3} )$ is the vector of $\Tone$ obtained by
    performing parallel transport of $v_{1} ,v_{2} , v_{3}$ along
    $\gamma_{1}$, computing the Riemann curvature tensor at the point
    $\gamma_{1} (1) \in M_{1}$ acting on those vectors, and then performing
    parallel transport backwards into the point $p_{1}$.
    
    \item $R_{2} (v_{1} ,v_{2} , v_{3} )$ is the vector of $\Tone$ obtained by
    performing parallel transport of $L(v_{1} ),L(v_{2} ), L(v_{3} )$ along
    $\gamma_{2}$, computing the Riemann curvature tensor at the point
    $\gamma_{2} (1) \in M_{2}$ acting on those vectors, then performing
    parallel transport backwards into the point $p_{2}$, and finally applying
    $L^{-1}$ to get a vector in $T_{p_{1}} M_{1}$.
  \end{itemizedot}
  If $R_{1} (v_{1} ,v_{2} , v_{3} )=R_{2} (v_{1} ,v_{2} , v_{3} )      \forall
  v_{1} ,v_{2} , v_{3} \in T_{p_{1}} M_{1}$ for any two geodesics $\gamma_{1}$
  and $\gamma_{2}$ as above, we say that the curvature tensors of $(M_{1}
  ,p_{1} )$ and $(M_{2} ,p_{2} )$ are $L$-related.
\end{dfn}

The usual way to express that $M_{1}$ and $M_{2}$ are $L$-related is to say
that the {\emph{parallel traslation of curvature along geodesics}} on $M_{1}$
and $M_{2}$ coincides.
\begin{dfn}

  \label{L-related}We say $(M_{1} ,p_{1} )$ and $(M_{2} ,p_{2} )$ are
  $L$-related iff they have the same dimension and, whenever $\exp_{1}
  |_{O_{1}}$ is injective for some domain $O_{1} \subset T_{p_{1}} M_{1}$,
  then the map $\varphi = \exp_{2} \circ L \circ ( \exp_{1} |_{O_{1}} )^{-1}$
  is an isometric inmersion.
\end{dfn}

\begin{theorem}
  \label{Cartan's theorem}If the curvature tensors of $(M_{1} ,p_{1} )$ and
  $(M_{2} ,p_{2} )$ are $L$-related, then $(M_{1} ,p_{1} )$ and $(M_{2} ,p_{2}
  )$ are $L$-related.
\end{theorem}

\begin{proof}
  Lemma 1.35 of \cite{Cheeger Ebin}.
\end{proof}

In 1956 (see \cite{Ambrose}), W. Ambrose proved a global version of the above
theorem, but with stronger hypothesis: if the parallel traslation of curvature along \emph{broken geodesics} on $M_{1}$ and $M_{2}$ coincide,
then there is a global isometry $\varphi:M_1\rightarrow M_2$ whose differential at $p_1$ is $L$.
It is simple to prove that $\varphi$ can be constructed as above.
Ambrose himself showed that is enough if the hypothesis holds for broken geodesics with only one ``{\emph{elbow}}''.
The reader can find more details in the standard reference \cite{Cheeger Ebin}.

However, he conjectured that the same hypothesis should suffice, except for
the obvious counterexample of covering spaces:

\subsection{Ambrose Conjecture}

The {\strong{Ambrose conjecture}} states that if the curvature tensor of $(M_1,p_1)$ and $(M_2,p_2)$ are $L$-related, and if furthermore $M_{1}$ and $M_{2}$ are simply connected, there is an isometry $\psi :M_{1} \rightarrow M_{2}$ such that $\psi \circ \exp_{1} = \exp_{2} \circ L$.

\begin{dfn}
  A Riemannian covering is a local isometry that is also a covering map.
\end{dfn}

\begin{conjecture}[\textsc{Ambrose Conjecture}]\label{ambrose conjecture}
  Let ($M_{1} ,p_{1}$) and ($M_{2} ,p_{2}$) be two $L$-related pointed Riemannian manifolds.
  
  Then there is a Riemannian manifold ($M,p$) (the
  {\strong{synthesis}} of ($M_{1} ,p_{1}$) and $(M_{2} ,p_{2} )$), linear
  isometries $L_{i} :T_{p} M_{s} \rightarrow T_{p_{i}} M_{i}$, for $i=1,2$,
  and Riemannian coverings $\pi_{i} :M_{s} \rightarrow M_{i}$ for $i=1,2$ such
  that $\pi_{i} \circ \exp_{p} = \exp_{p_{i}} \circ L_{i}$ and $L \circ L_{1}
  =L_{2}$.
  
  $$
  \xymatrix{ & T_pM  \ar@//[dd]_{\e} \ar[dl]^{L_1} \ar[dr]_{L_2}\\
T_{p_1}M_1 \ar@//[dd]^{\exp_{p_1}} \ar[rr]^{\quad L}& & T_{p_2}M_2
\ar@//[dd]^{\exp_{p_2}} \\   & M \ar[dl]^{\pi_1} \ar[dr]_{\pi_2} & \\
  M_1 &  & M_2}
  $$

  In particular, if $M_{1}$ and $M_{2}$ are simply-connected, the maps
  $\pi_{i}$ are isometries, and $\pi_{2} \circ \pi_{1}^{-1} :(M_{1} ,p_{1} )
  \rightarrow (M_{2} ,p_{2} )$ is an isometry (``{\emph{the}}'' isometry)
  whose tangent at $p_{1}$ is $L$.
\end{conjecture}

The main result of this chapter is:

\begin{theorem}
  \label{main theorem ambrose}
  The Ambrose Conjecture \ref{ambrose conjecture} holds if the metric of $M_1$ belongs to the generic class of metrics $\mathcal{G}_{M}$, as defined in \ref{the set of generic metrics}.
\end{theorem}

\begin{remark}
  The {\emph{synthesis}} manifold that we build is a \emph{least common Riemannian covering} (see \ref{universal property of the synthesis}).
\end{remark}

\subsection{Existing results}

Ambrose was able to prove the conjecture if all the data is analytic.
In \cite{Hicks}, in 1959, the conjecture was generalized to parallel transport for affine connections, and in \cite{Blumenthal Hebda}, in 1987, to Cartan connections.
Also in 1987, in the paper \cite{Hebda}, James Hebda proved that the conjecture was true for surfaces that satisfy a certain regularity hypothesis, that he was able to prove true in 1994 in \cite{Hebda94}. J.I. Itoh also proved the regularity hypothesis independently in \cite{Itoh96}.
The latest advance came in 2010, after we had started our research on the Ambrose conjecture, when James Hebda proved in \cite{Hebda10} that the conjecture holds if $M_1$ is a \emph{heterogeneous manifold}. Such manifolds are generic.

\section{Notation and preliminaries}\label{Ambrose notation}

$M$ is an arbitrary Riemannian manifold, $p$ a point of $M$, $(M_{1} ,p_{1} )$
and $(M_{2} ,p_{2} )$ are two Riemannian manifolds that are $L$-related.

Throughout this chapter, $ e_{1}$ stands for $\exp_{p_{1}}$ and $ e_{2}$ for
$\exp_{p_{2}} \circ L$.

$T_{p} M$ has the Riemannian manifold structure induced by the scalar product
$g_{p}$.
We denote by $R(v)=|v|$ the norm in $\T$. Using this name will be useful when we use non-linear coordinates in $\T$.
The radial vector field at $v \in \T$ is the vector $\partial_{r}=\frac{\partial}{\partial r}=\frac{v}{|v|}$.
Finally, we also define:
$$
B_{R_{0}} =\{x \in \T :|x|<R_{0} \}
$$
$$
B_{R_{0}} (y)=\{x \in \T :|x-y|<R_{0} \}
$$

The proof of the Ambrose conjecture for surfaces given by James Hebda in
{\cite{Hebda}} relies on properties of $\tmop{Cut}_{p}$, the {\strong{cut
locus}} of $M$ with respect to $p$. Let us
define also the {\emph{injectivity set}} $O_{p} \subset \T$, consisting of
those vectors $x$ in $\T$ such that $d( \e (tx),p)=t$ for all 
$0 \leqslant t\leqslant 1$, and let $\tmop{TCut}_{p} = \partial O_{p}$ be the {\emph{tangent cut locus}}.
It is a well known fact that $\tmop{TCut}_{p}$ maps onto $\tmop{Cut}_{p}$ by $\e$.

In our proof, we will need to use a set bigger than the injectivity set, defined as follows. Recall the functions $\lambda_{k} :S_{p_{1}} M_{1} \rightarrow \mathbb{R}$ as the parameter $t_{\ast}$ for which $t \cdot x$ is the $k$-th conjugate point along $t \rightarrow tx$ (counting multiplicities. We proved in \ref{landa es Lipschitz} that these functions are Lipschitz.
We define $V_{1}$ as the set of tangent vectors such that $|x| \leqslant \lambda_{1} (x/|x|$), a set with Lipschitz boundary.
Indeed, in \cite{Castelpietra Rifford}, it was shown that $\lambda_1$ is semiconcave.
It is well known that $O_{p} \subset V_{1}$.

Let $\tmop{AC}_{p} (X)$ be the space of absolutely continuous curves in the
manifold $M$ starting at $p$, with the topology defined as in {\cite{Hebda}}.
We will also use the affine developement $\tmop{Dev}_{p} : \tmop{AC}_{p} (M)
\rightarrow \tmop{AC}_{0} (T_{p} M)$ defined in that reference, or in the standard reference \cite{Kobayashi Nomizu}.

Finally, we introduce tree-formed curves, following James Hebda (\cite{Hebda}). 
The model for a tree-formed curve $u:[0,1] \rightarrow M$ is an absolutely continuous curve that factors through a finite topological tree $\Gamma$.
In other words, $u= \bar{u} \circ T$ for the quotient map $T:[0,1]\rightarrow \Gamma$ with $T(0)=T(1)$. 
The concept is similar to the {\emph{tree-like paths}} of the theory of rough paths. 
J. Hebda uses a more general definition, allowing for an arbitrary quotient map $T:[0,1] \rightarrow \Gamma$, and an absolutely continuous curve $u$ such that:
\[ \int_{t_{1}}^{t_{2}} \varphi (s)(u' (s))ds=0 \]
for any continuous $1$-form $\varphi$ along $u$ $( \varphi (s) \in
T^{\ast}_{u(s)} M)$ that factors through $\Gamma$ $(T (s_{1} ) =T (s_{2} )$
implies $\varphi (s_{1} ) = \varphi (s_{2} ) )$, and $t_{1}$, $t_{2}$ such
that $T(t_{1} )=T(t_{2} )$. Thus if $\Gamma =[0,1]$ and $T$ is the identity,
the definition is empty, and we will rather use the definition saying
that a certain curve $u$ is tree-formed with respect to an identication map
with $T(t_{1} )=T(t_{2} )$ as a rigorous way to say that $u_{|[t_{1} ,t_{2}
]}$ is a tree-like path. In the most common case, $T(0)=T(1)$, and we say the
curve is {\emph{fully tree-formed}}.

\subsection{The approach of James Hebda using tree-formed
curves}\label{subsection: Hebda's approach}

In this section we give a sketch of the paper {\cite{Hebda}}. The reader can find more details in that paper.

Theorem \ref{Cartan's theorem} shows that $\varphi = \exp_{2} \circ L \circ (
\exp_{1} |_{U_{p_1}} )^{-1}$ is an isommetric immersion from $U_{p_{1}} =M_{1}
\setminus \tmop{Cut}_{p_{1}}$ into $M_{2}$. The starting idea is to prove that
whenever a point in $\tmop{Cut}_{p_{1}}$ is reached by two geodesics
$\gamma_{1}$ and $\gamma_{2}$, meaning that $ e_{1} ( \gamma_{1}' (0))=  e_{1} (
\gamma_{2}' (0))$, then $ e_{2} ( \gamma_{1}' (0))=  e_{2} (
\gamma_{2}' (0))$. Then the formula $\varphi (p)=e_{2} (x)$, for any $x \in
(O_{p} \cup \tmop{TCut}_{p} ) \cap e_{1}^{-1} (p)$ gives a well-defined map
$\varphi :M_{1} \rightarrow M_{2}$ that is an isometry at least on
$U_{p_{1}}$.

As we know from \ref{cleave points are a manifold},
the cut locus looks specially simple at the {\emph{cleave points}}, for which there are exactly
two minimizing geodesics from $p$, and both are non-conjugate. Near a cleave
point, the cut locus is a smooth hypersurface. The rest of the cut locus is
more complicated, but we know that $\Hnuno ( \tmop{Cut} \setminus
\tmop{Cleave} )=0$ and, indeed, that $\tmop{Cut} \setminus \tmop{Cleave}$ has
Hausdorff dimension at most $n-2$, for a smooth Riemannian manifold.

An isometric inmersion from $M_{1} \setminus A$ into a complete manifold, with
\linebreak[4]
$\Hnuno (A)=0$, can be extended to an isometric inmersion from $M_{1}$. Thus,
it only remains to show that, for a cleave point $\cleaveq$, we have $e_{2}(x_{1} )=  e_{2} (x_{2} )$.

The way to do this is to find for each cleave point $q$ as above, a sequence
$Y_{j}$ of curves in $\Tone$ such that $Y_{j} (t) \in \tmop{int} (O_{p_{1}} )$
for all $j$ and $t$, 
$Y_j(0)= e_{1}((1-1/j)x_1)$, $Y_j(1)= e_{1}((1-1/j)x_2)$,
and $Y_{j}$ converges to a curve $Y$ in 
$\tmop{TCut}_{p_{1}}$ (in the metric space $\tmop{AC} (M$) of absolutely
continuous curves) such that $Y(0)=x_{1}$, $Y(1)=x_{2}$, and $ e_{1} \circ
Y:[0,1] \rightarrow M_{1}$ is {\emph{fully tree-formed}}.

Consider the curve $u= \gamma_{x_{1}} \ast (  e_{1} \circ Y)$, the
concatenation of the geodesic with initial speed $x_{1}$ with the curve $ e_{1}
\circ Y$, defined in the interval $[0,l_{1} +l_{2} ]$, where $l_{i}$ is the
length of each of these two segments. If $Y$ is absolutely continuous, this is
an absolutely continuous curve in $\Tone$, and so admits an {\emph{affine
developement}} from $p_{1}$. Composing with $L$ we get a curve in $\Ttwo$, and
the inverse affine developement from $p_{2}$ yields a curve $v$ in $M_{2}$.

J. Hebda proves that the affine developement and the inverse affine
developement of a tree-formed curve that factors through $\Gamma$ is also
tree-formed and factors through $\Gamma$. From $ e_{1} (x_{1} )=  e_{1} (x_{2}
)$ we learn $u(1/2)=u(1)$, so that $u$ factors through some $\Gamma$ with
$T(1/2)=T(1)$, and this shows $v(1/2)=v(1)$.

We also know that $ e_{2} \circ (  e_{1} |_{O_{p_{1}}} )^{-1}$ is an isometric
immersion from $U_{p_{1}}$ into $M_{2}$, and thus the curves 
$\gamma_{(1-1/j)x_{1}} \ast \left(  e_{1} \circ Y_{j}\right)$
map isometrically to 
$\tilde{\gamma}_{(1-1/j)x_{1}} \ast \left(  e_{2} \circ Y_{j}\right)$,
where $\tilde{\gamma}_x$ is the geodesic in $M_2$ with initial speed $L(x)$.
The affine developement
conmutes with an isometry, and we learn that $v_{j} =  e_{2} \circ (  e_{1}
|_{O_{p_{1}}} )^{-1} \circ u_{j}$, so that $ e_{2} (x_{1} )= \lim_{j
\rightarrow \infty}  e_{2} (Y_{j} (0))= \lim_{j \rightarrow \infty} v_{j}
(1/2)=v(1/2)$ and similarly, $v(1)=  e_{2} (x_{2} )$.

The way to find the curves $Y_{j}$ works only in dimension $2$. Let $S_{p_{1}}
M_{1}$ be the set of unit vectors in $\Tone$ parametrized with a coordinate
$\theta$, and define $\rho :S_{p_{1}} M_{1} \rightarrow \mathbb{R}$ as the
first cut point along the ray $t \rightarrow tv$ for $t>0$ (and $\rho ( \theta
)= \infty$ if there is no cut point). Given a cleave point $\cleaveq$, with
$x_{i} =( \rho ( \theta_{i} ), \theta_{i}$), then $\rho$ is finite in at least
one the two arcs in $S_{p_{1}} M_{1}$ that join $\theta_{1}$ and
$\theta_{2}$, which we write $[ \theta_{1} , \theta_{2} ]$. Then the curve $Y(
\theta )=( \rho ( \theta ), \theta$) defined in $[ \theta_{1} , \theta_{2} ]$,
together with the curves $Y_{j} ( \theta )=((1- \frac{1}{j} ) \rho ( \theta ),
\theta )$, satisfies the previous hypothesis.

It is important that $Y$ be absolutely continuous, which follows once it is
proved that $\rho$ is. This was shown independently in {\cite{Hebda94}} and
{\cite{Itoh96}}, and later generalized to arbitrary dimension in
{\cite{Itoh Tanaka 00}}.

\subsection{Difficulties to extend the proof to dimension higher than $2$}

In dimension higher than $2$, there is no natural choice for such a curve $Y$.
Indeed, a manifold can be built for which this technique does not work, roughly
as follows:

Using the techniques in {\cite{MR0221434}}, we can build a three dimensional
manifold $M$ whose cut locus with respect to a point does not contain
conjugate points (in other words, any minimizing geodesic segment is
non-conjugate). Let $\cleaveq$ be a cleave point and $Y$ be a path joining
$x_{1}$ and $x_{2}$ within the tangent cut locus. Assume for simplicity that
the path consists only of cleave points and isolated non-cleave points (this
is {\emph{generic}} in a certain sense, as we will see later).

If $ e_{1} \circ Y$ is fully tree-formed, then it has one terminal vertex
$q_{0} =  e_{1} (x^{0} )$. We~can approach this vertex with a sequence of
cleave points $q^{j} =  e_{1} (x^{j}_{1} )=  e_{1} (x^{j}_{2} )$ such that
$x_{1}^{j} \rightarrow x^{0}$ and $x_{2}^{j} \rightarrow x^{0}$. But then
$x^{0}$ is conjugate and minimizing, contrary to the hypothesis.

\section{Generic exponential maps}\label{section: generic
metric}

A generic perturbation of a Riemannian metric greatly simplifies the types of
singularities that can be found on the exponential map
({\cite{Weinstein}},{\cite{Klok}}) or the cut locus with respect to any point
({\cite{Buchner Stability}}). In {\cite{Weinstein}}, A. Weinstein showed that for a
generic metric, the set of conjugate points in the tangent space near a
singularity of order $k$ is given by the equations:
\[ \text{$\left|\begin{array}{cccc}
     x_{1} & x_{2} & \ldots & x_{k}\\
     x_{2} & x_{k+1} & \ldots & x_{2k-1}\\
     \vdots &  &  & \vdots\\
     x_{k} & x_{2k-1} & \ldots & x_{\frac{k(k+1)}{2}}
   \end{array}\right| =0$} \]
where $x_{1} , \ldots x_{n}$ are coordinates in $\Tone$, and $k(k+1)/2 \leqslant n$.
This is called a conical singularity.

In {\cite{Buchner Stability}}, M. Buchner studied the energy functional on curves
starting at $p_{1}$ and the endpoint fixed at a different point of the
manifold, as a family of functions parametrized by the endpoint. He proved a
multitransversality statement about this family of functions that we will
comment on later, and then used this information to provide a description of
the cut locus of a generic metric.

It is well known that a exponential map only has lagrangian singularities. 
In~{\cite{Klok}}, Fopke Klok showed that the generic singularities of the
exponential maps are the generic singularities of lagrangian maps. These
singularities are, in turn, described by means of the generalized phase
functions of the singularities. This is the approach more useful to our purposes.

\subsection{Generalized phase functions}

A {\emph{generalized phase function}} is a map $F:U \times \RR^{k}
\rightarrow \RR$ such that $D_{q} F= \left( \frac{\partial F}{\partial
q_{1}} , \ldots , \frac{\partial F}{\partial q_{k}} \right) :U \times
\RR^{k} \rightarrow \RR^{k}$ is transverse to $\{0\} \in
\RR^{k}$. We will use a result that relates generalized phase
functions defined at $U \times \RR^{k}$ and Lagrangian subspaces of
$T^{\ast} U$:

\begin{prop}
  If $L \subset T^{\ast} U$ is a Lagrangian submanifold and $p \in L$, it is locally given as the graph of $\phi |_{C} :C \rightarrow T^{\ast} U$, where $C=(D_{ q} F)^{-1} (0)$ and $\phi (x,q)=(x,D_{x} F(x,q) )$, for some
  generalized phase function $F$.
  
  Furthermore, we can assume:
  \begin{itemize}
    \item $k= \tmop{corank} (L,p)$
    
    \item $F(0,0)=0$
    
    \item $0 \in \RR^{k}$ is a critical point of $F(0, \cdot
    ):\RR^{k} \rightarrow \RR$
    
    \item $\frac{\partial^2 F}{\partial q_{i} \partial q_{j}} =0$ for all $i$ and $j$ in $1,\dots,k$
  \end{itemize}
\end{prop}

\begin{proof}
  This is found in section 1 of \cite{Klok}, specifically in proposition 1.2.4 and the comments in page 320 after proposition 1.2.6.
\end{proof}

Given a germ of generalized phase function $F:\RR^{n} \times
\RR^{k} \rightarrow \RR$, the lagrangian map is built in this
way: $D_{q} F$ is transverse to $\{0\}$, and we can assume the last $k$
$x$-coordinates are such that the derivative of $D_{q} F$ in those coordinates
is an invertible matrix. Let us split the $x$ coordinates in $(y,z) \in
\RR^{n-k} \times \RR^{k}$. Our hypothesis is that $D_{qz} F$
is invertible.

The implicit equations $D_{q} F=0$ defines functions $f_{j} :\RR^{n-k}
\times \RR^{k} \rightarrow \RR$ such that, locally near $0$,
$D_{qz} F(y,f(y,q),q)=0$.

\begin{dfn}
 A \emph{Lagrangian map} $\lambda:L\rightarrow M$ is the composition of a Lagrangian immersion $i:L\rightarrow T^\ast M$ with the projection $\pi:T^\ast  M\rightarrow M$ (a Lagrangian immersion is an immersion such that the image of sufficiently small open sets are Lagrangian submanifolds).
\end{dfn}
\begin{dfn}
 Two Lagrangian maps $\lambda_j=:L_j\rightarrow M_j$, with corresponding immersions $i_j:L\rightarrow T^\ast M$, $j=1,2$, are \emph{Lagrangian equivalent} iff there are diffeomorphisms $\sigma:L_1\rightarrow L_2$, $\nu:M_1\rightarrow M_2$  and $\tau:T^\ast M_1\rightarrow T^\ast M_2$ such that the following diagram conmutes:
$$
\xymatrix{ L_1 \ar[d]^{\sigma} \ar[r]^{i_1} & T^\ast M_1 \ar[d]^{\tau} \ar[r]^{\pi_1} & M_1 \ar[d]^{\nu}\\
	   L_2 		       \ar[r]^{i_2} & T^\ast M_2 	       \ar[r]^{\pi_2} & M_2
}
$$
and $\tau$ preserves the symplectic structure.
\end{dfn}

Lagrangian equivalence corresponds to equivalence of generalized phase functions (this is proposition 1.2.6 in \cite{Klok}). Two generalized phase functions are equivalent iff we can get one from the other composing three operations:
\begin{enumerate}
 \item Add a function $g(x)$ to $F$. This has no effect on the functions
$f_{j}$.

\item Pick up a diffeomorphism $G:\RR^{n} \rightarrow \RR^{n}$,
and replace $F(x,q)$ by $F(G(x),q)$. If the map $G$ has the special form
$G(x)=G(y,z)=(g(y),h(z))$, the effect is to replace the map $(y,q) \rightarrow
(y,f(y,q))$ by $(y,q) \rightarrow (y,h^{-1} (f(g(y),q)))$.

\item Pick up a map $H:\RR^{n} \times \RR^{k} \rightarrow
\RR^{k}$ such that $D_{q} H$ is invertible, and replace $F(x,q)$ by
$F(x,H(x,q))$. If the map $H$ does not depend on the $z$ variables, the effect
is to replace the map $(y,q) \rightarrow (y,f(y,q))$ by $(y,q) \rightarrow
(y,f(y,H(y,q)))$
\end{enumerate}

\subsection{The singularities of a generic exponential map}

Using theorem 1.4.1 in {\cite{Klok}}, we get the following result: fix a smooth
manifold $M$, a point $p \in M$. For a residual set of metrics in $M$ the
exponential map $T_{p} M \rightarrow M$ is nonsingular except at a set
$\tmop{Sing}$, which is a smooth stratified manifold with the following strata
(we describe the different singularities in some detail below):
\begin{itemizedot}
  \item A stratum of codimension $1$ consisting of {\emph{folds}}, or
  lagrangian singularities of type $A_{2}$.
  
  \item A stratum of codimension $2$ consisting of {\emph{cusps}}, or
  lagrangian singularities of type $A_{3}$.
  
  \item Strata of codimension $3$ consisting of lagrangian singularities of
  types $A_{4}$ (swallowtail), $D_{4}^{-}$ (elliptical umbilic) and
  $D_{4}^{+}$ (hyperbolic umbilic).
  
  \item We do not need to worry about the rest, which consists of strata of
  codimension at least $4$.
\end{itemizedot}

\begin{dfn}
  We define the sets $\mathcal{A}_{2}$, $\mathcal{A}_{3}$, etc as the set of all points of $V_{1}$ that have a singularity of type $A_{2}$, $A_{3}$, etc. 
  We also define $\mathcal{C}$ as the set of conjugate (singular) points and $\NC$ as the set of non-conjugate (non-singular) points.
\end{dfn}

Thus, $\tmop{Sing}$ is a smooth hypersurface of $\T$ near a conjugate point of
order~$1$ (including $A_{2}$, $A_{3}$ and $A_{4}$ points), and is
diffeomorphic to the product of a cone in $\mathbb{R}^{3}$ with a cube near a
conjugate point of order $2$ (including $D_{4}^{\pm}$). The $A_{2}$ points are
characterized as those for which the kernel of the differential of the
exponential map is a vector line transversal to the tangent plane to
$\tmop{Sing}$.

Furthermore, the image by $\e$ of each stratum of canonical singularities is
also smooth. There might be strata of high codimension that are not uniform,
in the sense that the exponential map at some points in those strata may not
have the same type of singularity (in other words, the singularities are
{\emph{non-determinate}}).
This only happens in some strata of codimension at least $5$, and is not a problem for our arguments.

There are also other generic property that interests us: the image of the
different strata intersect ``transversally'':

Take two different points $x_{1} ,x_{2} \in \T$ mapping to the same point of
$M$, and assume $x_{1}$ and $x_{2}$ lie in $\mathcal{A}_{2} \cup
\mathcal{A}_{3} \cup \mathcal{A}_{4} \cup \mathcal{D}_{4}$. Then the points
$x_{1}$ and $x_{2}$ have neighborhoods $U_{1} ,U_{2}$ such that $\e (U_{1}
\cap \mathcal{C} )$ and $\e (U_{2} \cap \mathcal{C} )$ are transversal (each
pair of strata intersect transversally).

This follows from proposition 1 in page 215 of {\cite{Buchner Stability}}, with $p=2$,
so that $_{2} j_{2}^{k} H( \alpha )$ is transversal to the orbit in
$\mathbb{R}^{2} \times [J^{k}_{ 0} (n,1)]^{2}$ where the first jet is of type
$\mathcal{T}_{1}$ and the second one is of type $\mathcal{T}_{2}$. Even though
that proposition is stated for manifolds of dimension less or equal than $5$,
the proof covers our statement for any dimension, because we only need
transversality to a few particular orbits of low codimension.

For any singularity in the above list, we can choose coordinates near $x$ and
$\e (x)$ so that $\e$ is expressed by standard formulas. For example, the
formulas near an $A_{3}$ point are $(x_{1} , \ldots ,x_{n-1} ,x_{n} )
\rightarrow (x_{1}^{3} \pm x_{1} x_{2} ,x_{2} , \ldots ,x_{n} )$.

The coordinates that we will use are derived using generalized phase functions
(see {\cite{Klok}} for example). We list the generalized phase functions and
the corresponding coordinates for the exponential function that derives from
it for the singularities $A_{2}$, $A_{3}$, $A_{4}$ and $D_{4}^{\pm}$:
\addtolength{\leftmargini}{-10pt} 
\begin{itemizedot}\itemsep=10pt 
\small 
  \item $A_{2}$:\begin{tabular}{l}
    $F(x_{1} , \tilde{x}_{1} ,x_{2} ,x_{3} , \ldots ,x_{n} )= \frac{1}{3}
    x_{1}^{3} - \tilde{x}_{1} x_{1}$\\[5pt] 
    $\e : (x_{1} ,x_{2} ,x_{3} , \ldots ,x_{n} ) \rightarrow (x_{1}^{2} ,x_{2}
    ,x_{3} , \ldots ,x_{n} )$
  \end{tabular}
  
  \item $A_{3}$:\begin{tabular}{l}
    $F(x_{1} , \tilde{x}_{1} ,x_{2} ,x_{3} , \ldots ,x_{n} )= \frac{1}{4}
    x_{1}^{4} \pm \frac{1}{2} x_{2} x_{1}^{2} - \widetilde{x_{1}} x_{1}$
    \\[5pt] 
    $\e : (x_{1} ,x_{2} ,x_{3} , \ldots ,x_{n} ) \rightarrow (x_{1}^{3} \pm x_{1}
    x_{2} ,x_{2} ,x_{3} , \ldots ,x_{n} )$
  \end{tabular}

  \item $A_{4}$:\begin{tabular}{l}
    $F(x_{1} , \tilde{x}_{1} ,x_{2} ,x_{3} , \ldots ,x_{n} )= \frac{1}{5}
    x_{1}^{5} + \frac{1}{3} x_{2} x_{1}^{3} + \frac{1}{2} x_{3} x_{1}^{2} -
    \widetilde{x_{1}} x_{1}$\\[5pt] 
    $\e : (x_{1} ,x_{2} ,x_{3} , \ldots ,x_{n} ) \rightarrow (x_{1}^{4}
    +x^{2}_{1} x_{2} +x_{1} x_{3} ,x_{2} ,x_{3} , \ldots ,x_{n} )$
  \end{tabular}
   
  \item $D_{4}^{-}$:\begin{tabular}{l}
    $F(x_{1} ,x_{2} , \tilde{x}_{1} , \widetilde{x_{2}} ,x_{3} , \ldots ,x_{n}
    )= \frac{1}{6} x_{1}^{3} - \frac{1}{2} x_{1} x_{2}^{2} +x_{3} (
    \frac{1}{2} x_{1}^{2} + \frac{1}{2} x_{2}^{2} )- \widetilde{x_{1}} x_{1} -
    \widetilde{x_{2}} x_{2}$\\[5pt] 
    $\e\!:\!(x_{1} ,x_{2} ,x_{3}, \ldots ,x_{n} ) \rightarrow (
    \frac{1}{2} x_{1}^{2} - \frac{1}{2} x_{2}^{2} +x_{1} x_{3} ,-x_{1} x_{2}
    +x_{2} x_{3} ,x_{3}, \ldots ,x_{n} )$
  \end{tabular}
  
  \item $D_{4}^{+}$:\begin{tabular}{l}
    $F(x_{1} ,x_{2} , \tilde{x}_{1} , \widetilde{x_{2}} ,x_{3} , \ldots ,x_{n}
    )= \frac{1}{6} x_{1}^{3} + \frac{1}{6} x_{2}^{3} +x_{1} x_{2} x_{3} -
    \widetilde{x_{1}} x_{1} - \widetilde{x_{2}} x_{2}$\\[5pt] 
    $\e :(x_{1} ,x_{2} ,x_{3} ,x_{4} , \ldots ,x_{n} ) \rightarrow (
    \frac{1}{2} x_{1}^{2} +x_{2} x_{3} , \frac{1}{2} x_{2}^{2} +x_{1} x_{3}
    ,x_{3} ,x_{4} , \ldots ,x_{n} )$
  \end{tabular}
\end{itemizedot}
\addtolength{\leftmargini}{10pt} 

\begin{dfn}
 The above expression is the \strong{canonical form} of the exponential map at the singularity.
 The canonical form is only defined for the singularities in the above list.

We call \strong{adapted coordinates} any set of coordinates for which the
expression of the exponential map is canonical.
\end{dfn}

\begin{dfn}
 Let $U$ be a neighborhood of adapted coordinates near a conjugate point $x$.
 The \strong{lousy metric} on $U$ is the metric whose matrix in adapted coordinates is the identity.
\end{dfn}

\begin{remark}
  We call this metric lousy because it does not have any geometric meaning, and it depends on the particular choice of adapted coordinates.
 However, it is useful for doing analysis.
\end{remark}

However, while the adapted coordinates make the exponential map simple, radial
geodesics from $p$ are no longer straight lines, and the spheres of constant
radius in $\T$ are also distorted. We do not know of any result that gives an
explicit canonical formula for the exponential map and also keeps radial
geodesics in $\T$ simple. The results of section \ref{section: CDCs in
adapted coords} suggest that this might be possible to some extent, but the
classification that might derive from it must be finer than the one above. We
will find examples showing that the radial vector can be placed in different,
non-equivalent positions.

For example, near an $A_{3}$ point, $\mathcal{C}$ is given by $3x_{1}^{2}
=x_{2}$. The radial vector $r=(r_{1} , \ldots ,r_{n} )$ at $(0, \ldots ,0)$ is
transversal to $\mathcal{C}$, and thus must have $r_{2} \neq 0$. There are two
possibilities:
\begin{itemizedot}
  \item A point is $A_3(I)$ iff $r_{2} >0$.
  
  \item A point is $A_3(II)$ iff $r_{2} <0$.
\end{itemizedot}
Even though the exponential map has the same expression in both cases (for
adequate coordinates), they differ for example in the following:

Let $x \in \mathcal{A}_{3} \cap V_{1}$ (a first conjugate point), and let $U$
be a neighborhood of $x$ of adapted coordinates. Then $\e (V_{1} \cap U$) is
a neighborhood of $\e (x$) iff $x$ is $A_3(I)$. A proof for this fact will be
trivial after section \ref{section: A3 are inequivocal}.

In fact, the above can be used as a characterization (for points in
$\mathcal{A}_{3} \cap V_{1}$) that shows that the definition is independent of
the adapted coordinates chosen. We remark that in a neighborhood of an
$A_3(I)$ point, there are no $A_3(II)$ points, and viceversa.

We will get back to this distinction later, and we will also make a similar
distinction with $D_{4}^{+}$ points.

\begin{remark}
  Sometimes singularities of real functions of type $A_{3}$ are subdivided
  into $A_{3}^{+}$ and $A_{3}^{-}$ points. A canonical form for an
  $A_{3}^{\pm}$ singularity is 
  $$F^{\pm} (x_{1} , \tilde{x}_{1} ,x_{2} ,x_{3} ,
  \ldots ,x_{n} )= \pm \frac{1}{4} x_{1}^{4} - \frac{1}{2} x_{2} x_{1}^{2} -
  \widetilde{x_{1}} x_{1}$$
  When $F^{\pm}$ are generalized phase functions,
  each subtype gives equivalent singularities. However, in the work of
  Buchner, the same singularities appear, now as the energy function in a
  finite dimensional approximation to the space of paths with fixed endpoints.
  In this second context, it is not equivalent if a geodesic is a local
  minimum, or a maximum, of the energy functional, and it would make sense to
  use the distinction between $A_{3}^{+}$ and $A_{3}^{-}$, rather than the
  similar-but-not-the-same distinction between $A_3(I)$ and $A_3(II)$.
  
  This can also serve as an illustration that the classification of singularities of the exponential map by F. Klok and M. Buchner is not equivalent, even though the final result is indeed quite similar. In the classification of F. Klok, the $A_3$ singularities are not divided into the two subclasses $A_3^+$ and $A_3^-$.
\end{remark}

\begin{dfn}\label{the set of generic metrics}
We define $\mathcal{G}_{M}$ as the set of Riemannian metrics for the smooth manifold $M$ such that the singular set of $\e$ is stratifed by singularities of types $A_{2}$, $A_{3}$, $A_{4}$ and $D_{4}^{\pm}$ with the codimensions listed above, plus strata of different types with codimension at least $4$, and such that the images of any two strata intersect transversally as stated above. 
\end{dfn}

Thanks to the work of M. Buchner and F. Klok, we know that this set is open
and dense in the set of all Riemannian metrics for $M$.

\section{Proof of the conjecture for generic $3$-manifold}

\subsection{Main idea}

For any point $x\in V_1$, the Cartan lemma provides an isometry from a
neighborhood of $ e_{1} (x)$ to one of $ e_{2} (x)$. We cannot use this fact to
get an isommetric immersion into $M_{2}$ from a set much bigger than $M_{1}
\setminus \tmop{Cut}_{p_{1}}$, but we can try to collect local mappings to
build a covering space, as stated in the main theorem \ref{main theorem ambrose}.

If $ e_{1}$ has no singularities, we can pull the metric from $M_{1}$ onto
$T_{p_{1}} M_{1}$ and the desired Riemannian coverings are $ e_{1}$ and
$ e_{2}$. In the presence of singularities, the idea is to build the synthesis
as a quotient of a subset of $V_{1}$ that
identifies pairs of points with the same image by both $ e_{1}$ and $ e_{2}$.

As mentioned above, as well as NC points for $ e_{1}$, there are points of
$\Tone$ with singularities for $ e_{1}$ of types $A_{2}$, $A_{3}$, $A_{4}$,
$D_{4}^{+}$ and $D_{4}^{-}$. The $A_{3}$ points are further divided into
$A_3(I)$ and $A_3(II)$ points.

Our way to deal with a singularity $x$ of type $A_3(I)$ is to show that it is
{\emph{unequivocal}}, which means that it can play the same role in the
quotient as a non-singular point.

\begin{dfn}
  \label{definition:inequivocal}We say that an open set $O \subset \Tone$ is
  {\strong{unequivocal}} iff $ e_{1} (O \cap V_{1} )$ is open, $ e_{2} (O \cap
  V_{1} )$ is open and there is an isometry $\varphi :  e_{1} (O \cap V_{1} )
  \rightarrow  e_{2} (O \cap V_{1} )$ such that $\varphi \circ  e_{1} |_{O \cap
  V_{1}} \nobracket =  e_{2} |_{O \cap V_{1}} \nobracket$, for any pointed manifold $(M_2,p_2)$ that is $L$-related to $(M_1,p_1)$.
  
  We say $x \in V_{1}$ is {\strong{unequivocal}} if it has a neighbourhood
  base consisting of unequivocal sets.
\end{dfn}

Regarding a singularity $x$ of a different type, we will show that there is a
{\strong{linking curve}} between $x$ and an unequivocal point $y$ of smaller
radius. A linking curve between $x$ and $y$ is a curve $\alpha :[0,t_{0} ]
\rightarrow T_{p_{1}} M_{1}$ such that $\alpha (0)=x$, $\alpha (t_{0} )=y$,
$ e_{1} \circ \alpha$ is fully tree-formed and $\tmop{Im} ( \alpha )$ is
contained in $V_{1}$. It also satisfies some technical restrictions that we
will present later. Linking curves play the role of the curve $Y$ in the proof
of the conjecture for surfaces by J. Hebda: we will see that if there is a
linking curve between two points, they are {\emph{linked}}.

\diff{
\begin{dfn}
  Two points $x, y \in \Tone$ are {\strong{linked}} ($x\leftrightsquigarrow y$) iff either $x=y$, or:
  $$ e_{1} (x)=  e_{1}(y), e_{2} (x)=  e_{2} (y)$$
  and there are neighborhoods $U$ of $x$ and $V$ of $y$ such that
  $$\forall z\in U, w\in V:\quad  e_{1}(z)= e_{1}(w) \Rightarrow  e_{2}(z)= e_{2}(w)$$
  for any pointed manifold $(M_2,p_2)$ that is $L$-related to $(M_1,p_1)$.
\end{dfn}

\begin{remark}
Note that with the above definition of \emph{linked}, it may not be an equivalence relation (depending on $M_1$), but in fact, the relation is transitive under some conditions that hold in our setting:
\end{remark}
}\diff{
\begin{lem}
 Let $x,y,z,v \in\Tone$.
 \begin{enumerate}
  \item If $x\leftrightsquigarrow y$, $z\leftrightsquigarrow y$ and $y$ in unequivocal, then $x\leftrightsquigarrow z$.
  \item Assume that $M_1$ is a generic manifold.
  
  If $\tilde{x} \leftrightsquigarrow  x\leftrightsquigarrow y \leftrightsquigarrow \tilde{y}$ and $\tilde{x}$ and $\tilde{y}$ are unequivocal, then $\tilde{x}\leftrightsquigarrow \tilde{y}$.
 \end{enumerate}
\end{lem}
}\diff{
\begin{proof}
 Let $(M_2,p_2)$ be a pointed manifold that is $L$-related to $(M_1,p_1)$.
 
 The hypothesis of the first part imply that:
 \begin{itemize}
  \item $\exists U^x, V^y$: $\forall z\in U^x, w \in V^y,  e_{1}(z)= e_{1}(w) \Rightarrow  e_{2}(z)= e_{2}(w) $
  \item $\exists U^y, V^z$: $\forall z\in U^y, w \in V^z,  e_{1}(z)= e_{1}(w) \Rightarrow  e_{2}(z)= e_{2}(w) $
  \item $\exists W^y$: $ e_{1}(W^y)$ is an open neighborhood of $ e_{1}(x)= e_{1}(y)= e_{1}(z)$
 \end{itemize}
 
 Then we take open sets $A= e_{1}^{-1}( e_{1}(W^y\cap V^y \cap U^y ))\cap U^x$ of $x$ and $B= e_{1}^{-1}( e_{1}(W^y\cap V^y \cap U^y ))\cap V^z$ of $z$.
 
 Suppose there are $z\in A,w \in B$ such that $ e_{1}(z)= e_{1}(w)$. Then $ e_{1}(z)\in  e_{1}(W^y\cap V^y \cap U^y)$, so that there is some $v\in W^y \cap V^y \cap U^y$ such that $ e_{1}(v)= e_{1}(z)= e_{1}(w)$, and it follows that $ e_{2}(z)= e_{2}(w)$.
  
 The hypothesis for the second part, in turn, imply that $e_{1}(x) =   e_{1}(y) =   e_{1}( \tilde{x}) =   e_{1}( \tilde{y})$ (we call this point $q$), and:
\begin{itemize}
  \item $\exists U^{\tilde{x}}, V^x$: $\forall z \in U^{\tilde{x}}, w \in V^z,
   e_{1}(z) =   e_{1}(w) \Rightarrow   e_{2}(z) =   e_{2}(w)$
  
  \item $\exists U^x, V^y$: $\forall z \in U^x, w \in V^y,   e_{1}(z) =   e_{1}(w)
  \Rightarrow   e_{2}(z) =   e_{2}(w)$
  
  \item $\exists U^y, V^{\tilde{y}}$: $\forall z \in U^y, w \in V^{\tilde{y}},
   e_{1}(z) =   e_{1}(w) \Rightarrow   e_{2}(z) =   e_{2}(w)$
  
  \item $\exists W^{\tilde{x}}$: $ e_1(W^{\tilde{x}})$ is an open neighborhood of
  $q$ and there is an isometry
  $\varphi_{\tilde{x}} :   e_{1}( W^{\tilde{x}}) \rightarrow   e_{2}( W^{\tilde{x}})$ such
  that $\varphi \circ e_1 = e_2$ on $W^{\tilde{x}}$.
  
  \item $\exists W^{\tilde{y}}$: $ e_1(W^{\tilde{y}})$ is an open neighborhood of
  $q$ and there is an isometry
  $\varphi_{\tilde{y}} :   e_{1}( W^{\tilde{y}}) \rightarrow   e_{2}( W^{\tilde{y}})$ such
  that $\varphi \circ  e_{1} =  e_{2}$ on $W^{\tilde{y}}$.
\end{itemize}

The genericity hypothesis also imply that $ e_1(U^x) \cap  e_1(V^y) \cap  e_1(U^{\tilde{x}}) \cap e_1(V^x) \cap  e_1(U^y) \cap  e_1(V^{\tilde{y}}) \cap  e_1(W^{\tilde{x}}) \cap e_1(W^{\tilde{y}})$ is a set with non empty interior. 
Indeed, for a generic metric, the image by $e_1$ of a neighborhood of a singular point is a stratifed manifold with non-empty interior, and bounded by hypersurfaces. 
The image of two such neighborhoods are two transversal stratified manifolds with at least one point in common, and thus they must share some $0$-cell with $q$ in its boundary, so that the intersection of this cell with the image of the open sets $e_1(U^{\tilde{x}} \cap  W^{\tilde{x}})$ and $e_1(V^{\tilde{y}}\cap W^{\tilde{y}})$ also has non-empy interior.

Any point $q$ in this set can be expressed as $q =  e_1( x_0) =  e_1( x_1) =  e_1( y_1) =  e_1( y_0)$ for $x_0 \in
W^{\tilde{x}} \cap U^{\tilde{x}}$, $x_1 \in U^x\cap V^x$, $y_1 \in V^y\cap U^y$, $y_0 \in
W^{\tilde{y}} \cap U^{\tilde{y}}$, so it follows that $ e_2( x_0) =  e_2( x_1) =  e_2( y_1) =  e_2( y_0)$, but $ ( x_0) =
\varphi_{\tilde{x}} \left(  ( x_0) \right)$ and $ ( y_0) = \varphi_{\tilde{y}}
\left(  ( y_0) \right)$ so $\varphi_{\tilde{x}}$ and $\varphi_{\tilde{y}}$ are
isometries that agree on an open set, so they must agree at least in the connected component of
$e_1(W^{\tilde{x}}) \cap e_1(W^{\tilde{y}})$ that contains $q$.

Thus $\tilde{x}$ and $\tilde{y}$ are linked, as we can check by defining $U=\left( e_1 |_{W^{\tilde{x}}} \nobracket \right)^{-
1} ( W^{\tilde{x}} \cap W^{\tilde{y}})$ and $V = \left( e_1 |_{W^{\tilde{y}}}
\nobracket \right)^{- 1} ( W^{\tilde{x}} \cap W^{\tilde{y}})$.

\end{proof}

\begin{cor}
 Let $M_1$ be a generic manifold such that every point is linked to an unequivocal point.
 
 Then the linked relation is transitive.
\end{cor}
\begin{proof}
 Let $x,y,z\in\Tone$ be such that $x\leftrightsquigarrow y \leftrightsquigarrow z$. 
 
 Then there are unequivocal points $\tilde{x},\tilde{y},\tilde{z}\in\Tone$ such that $x\leftrightsquigarrow \tilde{x}$, $y\leftrightsquigarrow \tilde{y}$ and $z\leftrightsquigarrow \tilde{z}$.
 
 By the second part of the above proof, we learn that $\tilde{x}\leftrightsquigarrow\tilde{y}\leftrightsquigarrow\tilde{z}$.

 Then, by the first part of the above proof, we learn that $x\leftrightsquigarrow\tilde{y}$, then that $x\leftrightsquigarrow\tilde{z}$, and finally $x\leftrightsquigarrow z$.

  $$
  \xymatrix{  
  x \ar@{<~>}[d]\ar@{<~>}[r]   & y \ar@{<~>}[d]\ar@{<~>}[r]  & z \ar@{<~>}[d] \\
  \tilde{x} \ar@{<~>}[r]       & \tilde{y} \ar@{<~>}[r]      &\tilde{z}
  }
  $$
\end{proof}
}
In the next section, we build the synthesis manifold $M_{s}$ as a quotient
space of a subset of $V_{1} \subset \Tone$, identifying linked points. Define
$\mathcal{I} =( \NC \cup \mathcal{A}_{3}(I)  ) \cap V_{1}$ and
$\mathcal{J}=(\mathcal{A}_{2} \cup \mathcal{A}_{3}(II) \cup \mathcal{A}_{4}   \cup
\mathcal{D}_{4}^{\pm} ) \cap V_{1}$. The following claim is all we need to use
the results in the next section:

\begin{theorem}
  \label{claim: hypothesis of the synthesis theorem}Points in $\mathcal{I}$
  are unequivocal, and any point in $\mathcal{J}$ is linked to a point in
  $\mathcal{I}$.
\end{theorem}

\pagebreak[3]

We will actually prove the theorem in a simpler situation first:

\begin{dfn}\label{definition: easy manifold}
  A manifold $M$ is {\emph{easy from $p$}} iff the exponential map from $p$
  only has singularities of type $A_{2}$ and $A_{3}$.
\end{dfn}

\begin{theorem}
  \label{claim: hypothesis of the synthesis theorem for easy manifolds}
  In an easy manifold, points in $\mathcal{I} = \NC \cup \mathcal{A}_{3}(I) $ are unequivocal, and any point in $\mathcal{J}= \mathcal{A}_{2} \cup \mathcal{A}_{3}(II) $ is linked to a point in $\mathcal{I}$.
\end{theorem}

\subsection{Synthesis}\label{subsection: synthesis}
In this section, $A$ is an arbitrary topological space, $X_{1}$, $X_{2}$ are
Riemannian manifolds, and $e_{1} :A \rightarrow X_{1}$, $e_{2} :A \rightarrow
X_{2}$ are arbitrary continuous maps. The concepts of {\emph{unequivocal
point}} and {\emph{linked pair of points}} make sense in this slightly more general setting with the obvious changes.

\diff{
\begin{prop}
  \label{synthesis: statement}Let $A$ be a topological space, $X_{1}$, $X_{2}$
  Riemannian manifolds, $e_{1} :A \rightarrow X_{1}$, $e_{2} :A \rightarrow
  X_{2}$ be continuous maps such that $\leftrightsquigarrow$ is an \emph{equivalence relation} and 
  the following property holds:
  
  For every $x \in A$, there is some $y \in A$ such that:
  \begin{itemizedot}
    \item $x$ is linked to $y$
    
    \item $y$ in unequivocal.
  \end{itemizedot}
  Then there is a Riemannian manifold $X$ (the {\emph{synthesis}} of
  $X_{1}$ and $X_{2}$), a continuous map $e:A \rightarrow X$ and local
  isometries $\pi_{1} :X \rightarrow X_{1}$ and $\pi_{2} :X \rightarrow
  X_{2}$, such that $e_{i} = \pi_{i} \circ e$, for $i=1,2$.
  
  $$
  \xymatrix{   & A \ar@//[ddl]_{e_1} \ar@//[ddr]^{e_2} \ar@//[d]^{e} \\
  & X \ar[dl]^{\pi_1} \ar[dr]_{\pi_2} &  \\ X_1 &  & X_2}
  $$
\end{prop}
}\diff{
\begin{proof}
  Define $X$ as a quotient by the linked relation:
  \begin{equation*}
    X = A/ \leftrightsquigarrow
  \end{equation*}
  Let $e:A \rightarrow X$ be the projection map. We define maps $\pi_{ i} :X
  \rightarrow X_{i}$ by $\pi_{i} ([x])=e_{i} (x)$. Both maps are clearly well
  defined.
  \begin{itemize}
    \item Topology of $X$: A basis for the topology of $X$ is given by all
    $[W]=\{[x],x \in W\}$, for an unequivocal open set $W$.
    
    \item $e$ is continuous at every point $x\in A$:
    There is an unequivocal point $z\in [x]$, thus $\exists U^z, V^x$: $\forall v \in U^z, w \in V^x,
    e_{1}(v) = e_{1}(w) \Rightarrow e_{2}(v) = e_{2}(w)$.
    
    Let $U=[W]$ be a basis open neighborhood of $[x]$:
    
    $\exists W^z\subset U^z\cap e_1^{-1}(\pi_1(U))$: $ e_1(W^z)$ is an open neighborhood of $e_1(z)$ and there is an isometry $\varphi_z: e_{1}( W^z) \rightarrow e_{2}( W^z)$ such that $\varphi \circ e_1 = e_2$ on $W^z$.
    
    Then $O=U^x\cap e_1^{-1}(e_1(W^z))$ is an open neighborhood of $x$.
    We want to show that $O\subset e^{-1}(U)$.

    We first show $O\subset e^{-1}([W^z]) $: let $v\in O$.
    There is some $w\in W^z$ such that $e_1(w)=e_1(v)$ and this implies also that $e_2(w)=e_2(v)$.

    For the same reason, the sets $O$ and $W^z$ also satisfy the necessary property to show that $v$ is linked to $w$, thus $[v]\in [W^z]$.
    
    It remains to show that $[W^z]\subset [W]$.
    We can assume both $W$ and $W^z$ are connected.
    The unequivocal sets $W^z$ and $W$ have associated isometries $\varphi_z$ and $\varphi$, and they agree on $e_1(U^x)\cap e_1(V^z)\cap e_1(W^z)$, a set with non-empty interior, so they agree on $e_1(W^z)$.
    Finally, for any point $\tilde{z}\in W^z$ there is another $\tilde{x}\in W$ such that $e_1(\tilde{z})=e_1(\tilde{x})$, and $\varphi_z=\varphi$ implies $e_2(\tilde{z})=e_2(\tilde{x})$.
    So we conclude as before that $\tilde{z} \leftrightsquigarrow \tilde{x}$.

    \item For $i=1,2$, $\pi_{i} |_{[W]}$ is injective for any basis open set $[W]$: 
    WLOG, take $i=1$, and let $[x_{1} ],[x_{2} ] \in [W]$ be such that $\pi_{1} ([x_{1} ])= \pi_{1} ([x_{2} ])$. 
    We can assume $x_{1} ,x_{2} \in W$. 
    By the property of $W$, $e_{1} (x_{1} )=e_{1} (x_{2} )$ implies $e_{2}(x_{1} )=e_{2} (x_{2} )$, and taking $U^{x_1}=V^{x_2}=W$ does the rest of the job of proving that $x_{1} \leftrightsquigarrow x_{2}$.

    \item For $i=1,2$, $\pi_{i}$ is continuous. WLOG, take $i=1$. We show that
    $\pi_{1} |_{[W]}$ is continuous, for a basis set $[W]$: let $U$ be an open
    subset of $\pi_{1} ([W])=e_1(W)$. Then $( \pi_{1} |_{[W]} )^{-1} (U)=[W] \cap
    \pi_{1}^{-1} (U)=[W \cap e_{1}^{-1} (U)]$ is open, because $W \cap
    e_{1}^{-1} (U) \subset W$, and $e_{1} (W \cap e_{1}^{-1} (U))=U$, and thus
    $W \cap e_{1}^{-1} (U)$ is also unequivocal.
    
    \item For a basis open set $[W]$, $\pi_{i} ([W])$ is open by definition.
    Hence, $\pi_{i}$ is open for $i=1,2$. Thus, $\pi_{i} |_{[W]}$ is an
    homeomorphism onto its image.
    
    \item Hence, $\pi_{1}$ and $\pi_{2}$ are local homeomorphisms. We can use
    $\pi_{1}$ to give $X$ the structure of a Riemannian manifold, which
    trivially makes $\pi_{1}$ a local isometry. For an unequivocal set $W$,
    with $e_{2} |_{W} = \varphi \circ e_{1} |_{W}$, then $\pi_{2} \circ (
    \pi_{1} |_{[W]} )^{-1} = \varphi$ is an isometry from $\pi_{1} ([W])=e_{1}
    (W)$ to $\pi_{2} ([W])=e_{2} (W)$, so $\pi_{2}$ is also a local isometry.
  \end{itemize}
\end{proof}
}
Let us mention that the synthesis that we constructed satisfies an universal
property, and thus is unique up to global isometry:
\diff{
\begin{lem}\label{universal property of the synthesis}
  Under the same hypothesis of \ref{synthesis: statement}, the synthesis
  manifold $X$ constructed in the proof satisfies the following universal
  property:
  
  For any Riemannian manifold $X'$, continuous surjective map $e' :A
  \rightarrow X'$ and local isometries $\pi'_{1} :X' \rightarrow X_{1}$ and
  $\pi'_{2} :X' \rightarrow X_{2}$, such that $e_{i} = \pi'_{i} \circ e'$, for
  $i=1,2$, there is a local isometry $\pi :X' \rightarrow X$ such that $\pi'_{i}=\pi_i\circ\pi$ and $\pi\circ e' = e$:
   \begin{center}
\begin{tabular}{ll}
  $$
  \xymatrix{
  & A \ar@//[ddl]_{e_1} \ar@//[ddr]^{e_2} \ar@//[d]^{e'} \\
  & X' \ar@//[dl]^{\pi_1'} \ar@//[dr]_{\pi_2'} \\
  X_1 &  & X_2}
  $$
 & 
  $$
  \xymatrix{
  & X' \ar@//[ddl]_{\pi_1'} \ar@//[ddr]^{\pi_2'} \ar@//[d]^{\pi} \\
  & X \ar[dl]^{\pi_1} \ar[dr]_{\pi_2} &  \\ X_1 &  & X_2}
  $$
   \end{tabular}
   \end{center}
  
\end{lem}
}
\begin{proof}
\diff{
  Define $\pi (q)=[x]$ for any $x \in A$ such that $e' (x)=q$. 
  For any other $y$ such that $e' (y)=q$, we have $e_{i} (y)= \pi_{i}' (q)=e_{i} (x)$.
  We can also take open neighborhoods $U^x,V^y$ of $x$ and $y$ contained on $e'^{-1}(A)$, for an open neighborhood $A$ of $q$ such that $\pi|_{A}$ is an homeomorphism.
  Then, if $e_1(z)=e_1(w)$ for $z\in U$ and $w\in V$, it follows from $\pi'_1(e'(z))=\pi'_1(e'(w))$ that $e'(z)=e'(w)$ and thus $e_2(z)=e_2(w)$.
  It follows that $x \leftrightsquigarrow y$ and $\pi$ is well defined. 
  
  We also check that $\pi_{i} ( \pi
  (q))=e_{i} (x)= \pi_{i}' (e' (x))= \pi_{i}' (q)$. Any $q \in X'$ has a
  neighborhood $U' \subset X'$ such that $\pi_{i}' |_{U'}$ is an isometry.
  There is also $U \subset X$ \ such that $\pi_{i} |_{U}$ is an isometry. Let
  $V' =( \pi'_{1} )^{-1} ( \pi'_{1} (U' ) \cap \pi_{1} (U))$. Then $\pi |_{V'}
  =( \pi_{i}^{-1} \circ \pi_{i}' )|_{V'}$, and thus $\pi$ is a local isometry.
}
\end{proof}

\begin{remark}
  We have not proved that $\pi_{1}$ and $\pi_{2}$ are coverings maps. It would
  be enough to show that $X$ is complete, but this is not true in such
  generality, as the following example shows:
  
  Let $A \subsetneq M$ be an open subset of a connected Riemannian manifold,
  and let $i:A \rightarrow M$ be the inclussion. Take $X_{1} =X_{2} =M$ and
  $e_{1} =e_{2} =i$. Then $e=i$, $\pi_{1} = \pi_{2} = \tmop{id}_{M}$ and $X=A$
  satisfy the thesis of the theorem, and $e_{1}$ and $e_{2}$ are isometries,
  but not covering maps.
  
  We will prove in section \ref{subsection: from local homeo to covering} that
  $X$ is complete when $M_{1}$ are $M_{2}$ are complete Riemannian manifolds
  with a generic metric, $A$ is $V_{1}$, $e_{1}$ is $\exp_{p_{1}}$ and $e_{2}$
  is $\exp_{p_{2}} \circ L$.
\end{remark}

\subsection{Proof that $A_3(I)$ first conjugate points are
unequivocal}\label{section: A3 are inequivocal}

Consider an $A_3(I)$ point $x$ in the manifold $(M_{1} ,p_{1} )$ that is $L$-related to $(M_{2} ,p_{2} )$, and use adapted coordinates near $x=(0,0,0)$, in an arbitrarily small neighborhood $O$:
\begin{itemizedot}
  \item Define $\gamma (x_{1} ,x_{3} )=x_{1}^{2}$.
  
  \item Let $A$ be the subset of $O$ given by $x_{2} < \gamma (x_{1} ,x_{3}
  )$. $ e_{1}$ maps difeomorphically $A$ onto a big subset of $ e_{1} (O)$. Only
  the points with $x_{1} =0,x_{2} \geqslant 0$ are missing. $x$ is $A_3(I)$, so
  $\bar{A} \subset V_{1}$, and $ e_{1} (O \cap V_{1} )$ is open.
  
  \item For any $(x_{1} ,x_{3} )$, the pair of points $(x_{1} ,x_{1}^{2}
  ,x_{3} )$ and $(-x_{1} ,x_{1}^{2} ,x_{3} )$ map to the same point by
  $ e_{1}$, the curve $t \rightarrow (t,t^{2} ,x_{3} )$, $t \in [-x_{1} ,x_{1}
  ]$ maps to a tree-formed curve. This shows that the two points map to the
  same point by $ e_{2}$ as well. The details go exactly like in two dimensions.
  
  \item Define a map $\varphi :  e_{1} (O) \rightarrow  e_{2} (O)$ by $\varphi
  (p)=  e_{2} (a)$, for any $a \in \bar{A}$ such that $p=  e_{1} (a)$. By the
  above, this is unambigous.
  
  \item The rest of the proof proceeds as in lemma 2.1 in {\cite{Hebda}}: for
  a pair of linked points $x=(x_{1} ,x_{1}^{2} ,x_{3} , \ldots ,x_{n} )$ and
  $\bar{x} =(-x_{1} ,x_{1}^{2} ,x_{3} , \ldots ,x_{n} )$, we have two
  different local isometries from a neighborhood of $p=  e_{1} (x)=  e_{1} (
  \bar{x} )$ into $M_{2}$, given by $ e_{2} \circ (  e_{1} |_{O_{i}} )^{-1}$,
  for neighborhoods $O_{i}$ of $x$ and $\bar{x}$ such that $ e_{1} (O_{1} )=
   e_{1} (O_{2} )$ and we need to show that they agree. They both send $p$ to
  the same point, and we only need to check that their differential is the
  same. These are linear isometries, and they agree on the hyperplane $x_{1}
  =0$ (tangent to the image of $\partial A$: $x_{1} =0,x_{2} \geqslant 0$). It
  is easy to see that they both preserve orientation (for example: there is
  continuous curve of local isometries joining them), so they coincide.
  
  \item We know that $\varphi \circ  e_{1} (x)=  e_{2} (x)$, for $x \in
  \bar{A}$. Let $y \in O \setminus \bar{A}$. There is a unique point $x$
  in the radial line through $y$ in $\partial A$. We know $\varphi \circ  e_{1}
  (x)=  e_{2} (x)$, and the radial segment from $x$ to $y$ map by both $\varphi
  \circ  e_{1}$ and $ e_{2}$ to a geodesic segment with the same length,
  starting point and initial vector. We conclude $\varphi \circ  e_{1} (y)=
   e_{2} (y)$.
\end{itemizedot}
\begin{remark}
  The only place where we used that the point is $A_3(I)$ is when we assumed
  that $A \subset V_{1}$.
\end{remark}

\subsection{Conjugate flow}

We now introduce the main ingredient in the construction of the linking
curves. The idea in the definition of conjugate flow was used in
{\cite{Hebda82}} to prove lemma 2.2, but the idea for that proof is attributed
to an anonymous referee{\footnote{James Hebda said ``I wish to thank the
referee for the simple proof of lemma 2.2''.}}, and we cannot track the origin
of the idea any further.

Near a conjugate point of order 1, the set $C$ of conjugate points is a smooth
hypersurface. Furthermore, we know $\ker  d F$ does not contain $r$ by Gauss'
lemma. Thus we can define a one dimensional distribution $D$ within the set of
points of order $1$ by the rule:
\begin{equation}
  \label{1d distribution at conjugate points} D= ( \ker  d F \oplus <r> ) \cap
  T  C
\end{equation}
\begin{dfn}
  \label{definition: conjugate flow}A {\strong{conjugate descending curve}}
  (CDC) is a smooth curve, consisting only of $A_{2}$ points, except possibly
  at the endpoints, and such that the speed vector to the curve is in $D$ and
  has negative scalar product with the radial vector $r$. Therefore, the
  radius is decreasing along a descending flow line of conjugate points.
  
  The {\strong{canonical parametrization}} of a CDC $\gamma$ is the one that
  makes $\e \circ \gamma'$ a unit vector. By Gauss lemma, it is also the one
  that makes \mbox{$dR( \gamma' )\!=\!1$}.
\end{dfn}

\pagebreak[2]

\begin{dfn}
  \label{definition: retort}Let $\alpha :[0,t_{1} ] \rightarrow T_{p} M$ be a
  smooth curve, and $x \in \T$ be a point such that $\e (x)= \e ( \alpha
  (t_{1} ))$. A curve $\beta :[0,t_{1} ] \rightarrow T_{p} M$ is a
  {\strong{retort}} of $\alpha$ starting at $x$ iff $\beta (t_{1} -t) \neq
  \alpha (t)$ for any $t \in [0,t_{1} )$, but $\e ( \alpha (t))= \e ( \beta
  (t_{1} -t))$ for any $t \in [0,t_{1} ]$, and $\beta (t)$ is NC for any $t
  \in (0,t_{1} )$. Whenever $\beta$ is a retort of $\alpha$, we say that
  $\beta$ {\strong{replies}} to $\alpha$. A {\strong{partial retort}} of
  $\alpha$ is a retort of the restriction of $\alpha$ to a subinterval $[t_{0}
  ,t_{1} ]$, for $0<t_{0} <t_{1}$.
\end{dfn}

We have seen that near an $A_{2}$ point $x$, there are coordinates near $x$
and $\e (x)$ such that $\e$ reads $(x_{1} ,x_{2} , \ldots ,x_{n} ) \rightarrow
(x_{1}^{2} ,x_{2} , \ldots ,x_{n} )$. The $A_{2}$ points are given by $x_{1}
=0$, and no other point $y \neq x$ maps to $\e (x)$. Thus, there is a
neighborhood of any CDC such that any CDC has no retorts.

\begin{lem}
  \label{unbeatable lemma}Let $x$ be an $A_{2}$ point. Then there is a
  $C^{\infty}$ CDC $\alpha :[0,t_{0} ) \rightarrow \T$ with $\alpha (0)=x$.
  The CDC is unique, up to reparametrization. Furthermore:
  \begin{itemizedot}
    \item $| \alpha (0)|-| \alpha (t_{0} )|= \tmop{length} ( \e \circ \alpha
    )$
    
    \item If $\beta$ is a non-trivial retort of $\alpha$, then of course,
    
    $$\tmop{length} ( \e \circ \alpha )= \tmop{length} ( \e \circ \beta ),\,
    \text{but}\, | \beta (t_{0} )|-| \beta (0)|< \tmop{length} ( \e \circ
\beta ).$$ 
We say
    that segments of descending conjugate flow are {\strong{unbeatable}}.
  \end{itemizedot}
\end{lem}

\begin{proof}
  Both $\mathcal{A}_{2}$ and the distribution $D$ are smooth near $x$, so the
  first part is standard.
  
  We also compute:
  \[ \tmop{length} ( \e \circ \alpha ) = \int |( \e \circ \alpha )' |= \int |d
     \e ( \alpha' )| \]
  By definition of $D$, $\alpha' =ar+v$ is a linear combination of a multiple
  of the radial vector and a vector $v \in \ker (d \e )$. By the Gauss lemma,
  $|d \e ( \alpha' )|=a$. On the other hand, $v$ is tangent to the spheres of
  constant radius, so:
  \[ | \alpha (0)|-| \alpha (t_{0} )|= \int \frac{d}{dt} | \alpha |= \int a=
     \tmop{length} ( \e \circ \alpha ) \]
  For a retort $\beta :[0,t_{1} ] \rightarrow \T$, we also have $\beta' =b r+v$
  for a function $b:[0,t_{1} ] \rightarrow \mathbb{R}$ and a vector $v(t) \in
  T_{\beta (t)} ( \T )$ that is always tangent to the spheres of constant
  radius, and $v(t)$ is not identically zero because $ e_{1} \circ \beta$ is
  not a geodesic. However, $\beta (s)$ is non-conjugate, so $|d \e ( \beta'
  )|= \sqrt{b^{2} +|d \e (v)|^{2}} >b$. The result follows.
\end{proof}

\begin{remark}
 We recall that the plan is to build linking curves, whose composition with the exponential is tree formed.
 If a linking curve contains a CDC, it must also contain a retort for that CDC.
 The ``unbeatable'' property of CDCs is interesting, because the radius decreases along a CDC and along the retort it never increases as much as it decreased in the first place.
 This way, our prospective linking curve will stay within a sphere of finite radius.
\end{remark}

\subsection{CDCs in adapted coordinates near $A_{3}$
points\label{subsection: CDCs in adapted coords near A3}}

As we mentioned in section \ref{section: generic metric}, the radial vector
field, and the spheres of constant radius of $\T$, that have very simple
expressions in standard linear coordinates in $\T$, are distorted in canonical
coordinates. Thus, the distribution $D$ and the CDCs do not always have the
same expression in adapted coordinates. In this section, we see what we can
say about these curves near an $A_{3}$ point. We will use the name $R: \T
\rightarrow \mathbb{R}$ for the radius function, and $r$ for the radial
vector field, and we assume that our conjugate point is a first conjugate
point (it lies in $\partial V_{1}$).

In a neighborhood $O$ of special coordinates of an $A_{3}$ point,
$\mathcal{C}$ is given by $3x_{1}^{2} =x_{2}$. At each $A_{3}$ point, the
kernel is spanned by $\frac{\partial}{\partial x_{1}}$. At points in
$\mathcal{C}$, we can define a 2D distribution $D_{2}$, spanned by $r$ and
$\text{$\frac{\partial}{\partial x_{1}}$}$. We extend this distribution to all
of $O$ in the following way:

\begin{dfn}
For any point $x \in O$, there are $y \in   \mathcal{C}$ and $t_{0}$ such that $x= \phi_{t_{0}} ( y )$, where $\phi_{t}$ is the radial flow, and $y$ and $t$ are unique. Define $D_{2} ( x )$ as $( \phi_{t_{0}} )_{\ast} ( D_{2} ( y ) )$.
\end{dfn}

Let $P$ be the integral manifold of $D_{2}$ through $x_{0} =(0, 0 ,0)$.
The integral curve $C$ of $D$ through $x_{0}$ is contained in $P$, and $C
\setminus \{ x_{0} \}$ consists of two CDCs. We claim that if the point is
$A_3(I)$, the two CDCs descend into $x_{0}$, but if the point is $A_3(II)$, they
start at $x_{0}$ and flow out of $O$.
$P$ is also obtained by flowing the CDC with the radial vector field.

We can assume that $r$ is close to $r(x_{0} )$ in $O$. The tangent $T_{x}$ to
the sphere of constant radius $\{y:R(y)=R(x)\}$ must contain
$\frac{\partial}{\partial x_{1}}$ (the kernel of $d \e$) if $x \in
\mathcal{C}$, by Gauss lemma, and we can assume that the angle between $T_{x}$
and $\frac{\partial}{\partial x_{1}}$ is small if $x \nin \mathcal{C}$.

$\mathcal{A}_{3}$ is transversal to $D_{2}$, so $\{x_{0} \}= \mathcal{A}_{3}
\cap P$. CDCs have non-zero speed, so we only need to show that the two CDCs
have greater radius than $x_{0}$. Otherwise, for some $x \in \mathcal{C} \cap
P$ close to $x_{0}$, the curve $\{y:R(y)=R(x)\} \cap P$ is forced to make a
sharp turn and become ``vertical'' (parallel to $r$), as it cannot intersect
$\{y:R(y)=R(x_{0} )\} \cap P$ (the dashed line in figure \ref{figure: near an
A3 point} below). But its tangent is close to $\frac{\partial}{\partial
x_{2}}$, which is a contradiction.

\begin{figure}[ht]
 \centering
 \includegraphics[width=0.6\textwidth]{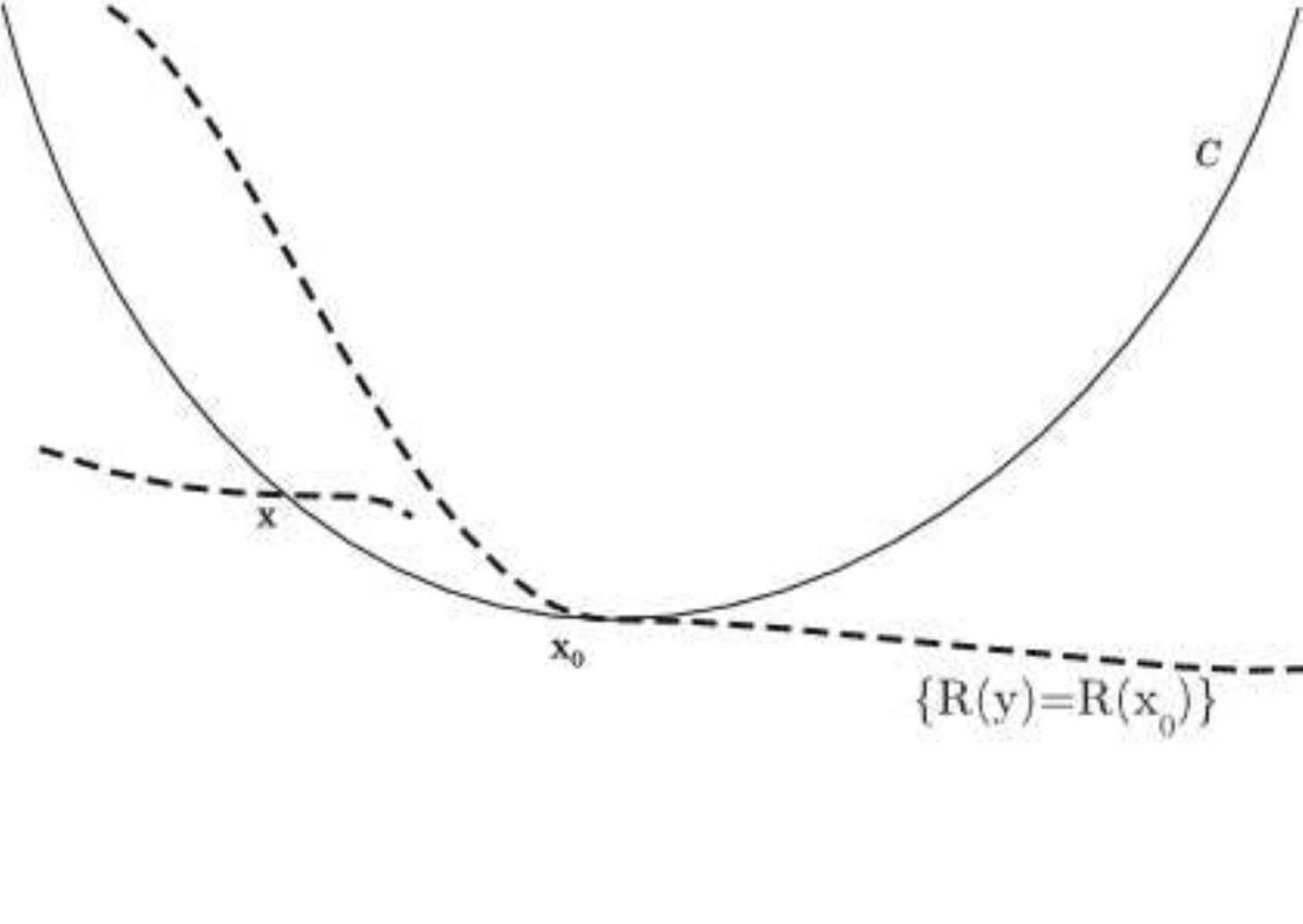}
 \caption{A neighborhood of $(0, \ldots ,0)$ in the plane $P$,
{\emph{if}} the curve $\{y:R(y)=R(x_{0} )\} \cap P$ did not lie below
$\mathcal{C} \cap P$.}
 \label{figure: near an A3 point}
\end{figure}

If $r_{2} >0$ ($A_3(I)$ points), then $R(x)\geq R(x_{0}$) for any $x \in
\mathcal{C}$, while $r_{2} <0$ ($A_3(II)$ points), implies $R(x)\leq R(x_{0}$) for
any $x \in \mathcal{C}$, as required.

Thus, $A_3(I)$ points are {\emph{terminal}} for the conjugate flow, but $A_3(II)$
points are not. This is fortunate, because $A_3(II)$ points are not unequivocal
and thus we hope to link them to an unequivocal point. We have just learned
that we can at least start a CDC at those point.

\subsection{$A_{3}$ joins\label{section: A3 joins}}

We can continue a CDC as long as it stays within a stratum of $A_{2}$ points.
As we have seen, a CDC may enter a different singularity. The most important
situation is when the CDC reaches an $A_{3}$ point, because then we can start
a non-trivial retort right after the CDC. We may not be able to continue the
retort for the whole CDC curve, but we will deal with that problem later.

The set of conjugate points is a graph over the $x_{1} ,x_{3}$ plane: $x_{2} =
\alpha (x_{1} ,x_{3} )=3x_{1}^{2}$. A CDC is written $t \rightarrow (t,3t^{2}
,x_{3} (t))$, for $t \in [t_{0} ,0]$, finishing at an $A_{3}$ point
$(0,0,x_{3} (0))$, but it cannot be continued further. Fortunately, we can
start a retort for this segment of CDC right from the $A_{3}$ point. The
retort for this CDC is given explicitely by $t \rightarrow (-2t,3t^{2} ,x_{3}
(t))$.

In figure \ref{figure: A3 point} below, we can see a CDC (in solid red, coming
from right to left), and reaching the $A_{3}$ point, and a retort for this
curve (in green). On the right hand side, we can see the image by the
exponential of the concatenation of both curves (a tree formed curve: the
image of each curve is the same but run in opposite directions). The picture
is within the plane $P$ of section \ref{subsection: CDCs in adapted coords
near A3} and the blue lines are no more than vertical lines with their respective images. They are not geodesics, but are included to help interpret the picture.

\begin{figure}[ht]
 \centering
 \includegraphics[width=0.8\textwidth]{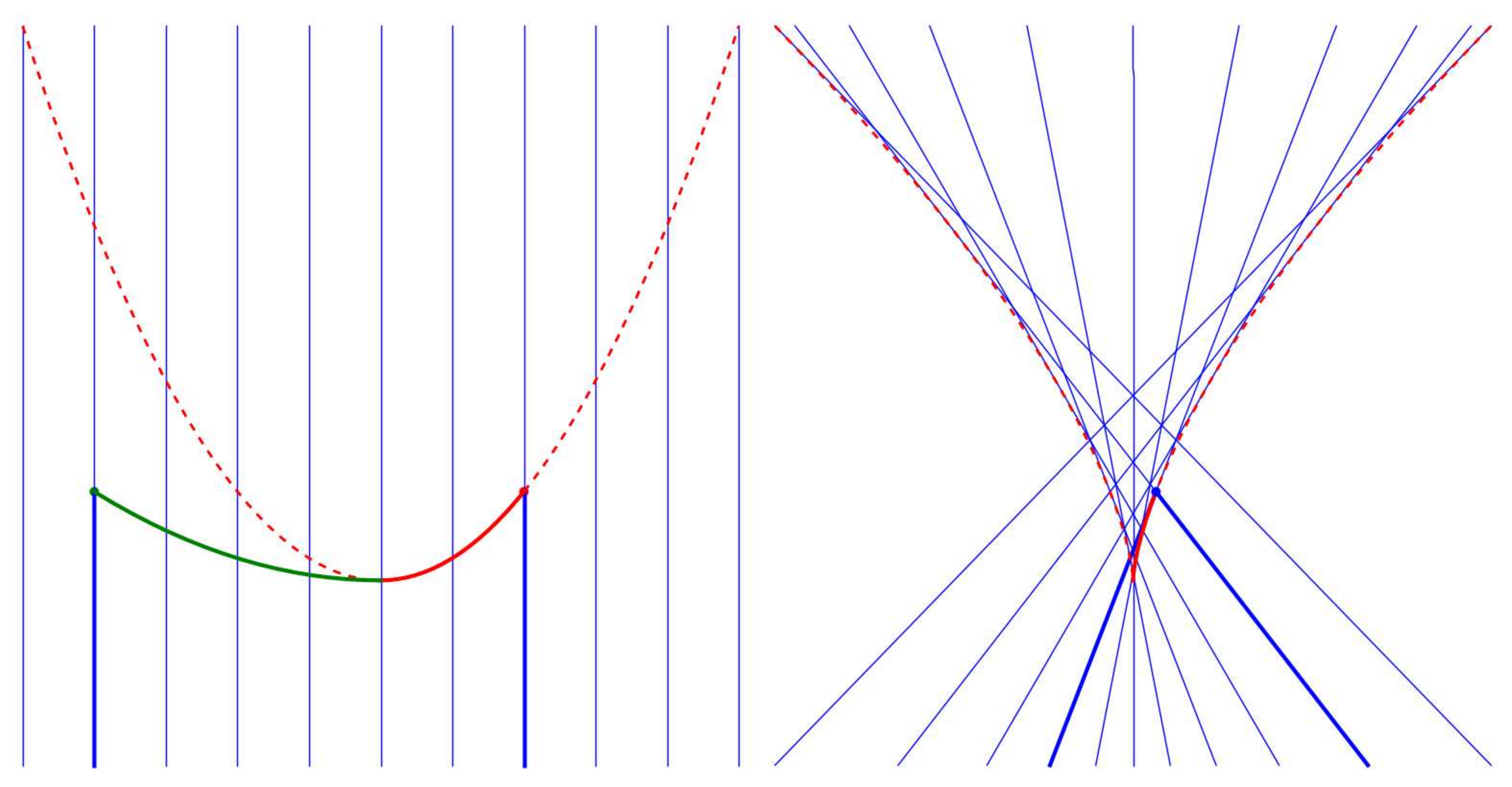}
 \caption{near an $A_3$ point}
 \label{figure:
A3 point}
\end{figure}

These curves, composed of a segment of CDC plus the corresponding retort, map
to a fully tree-formed map that shows that the point $(t,3t^{2} ,x_{3} )$ is
linked to $(-2t,3t^{2} ,x_{3} )$. We say that the CDC and the retort given
above are joined with an {\strong{$A_{3}$ join}}.

\subsection{Avoiding some obstacles}\label{section:avoiding some obstacles}

In order to build linking curves, it is simpler to replace CDCs with curves that are close to CDC curves, but avoid certain ``obstacles''. The following remark helps in that respect:

{\emph{A curve that is sufficiently $C^{1}$-close to a
CDC is also unbeatable}}. Actually, we can say more: the greater the angle
between $r_{x}$ and $\ker d_{x}  e_{1}$, the more we can depart from the CDC.

\begin{dfn}
  The \strong{slack} $A_{x}$ at a first order conjugate point $x$ is the absolute value of the sine of the angle between $D_{x}$ and $\ker  (d_{x}  e_{1} )$.
\end{dfn}

\begin{remark}
  The slack is positive iff the point is $A_{2}$
\end{remark}

\begin{lem}
  \label{slack lemma}For any positive numbers $R>0$ and $a>0$ there are constants $c>0$ and $\varepsilon>0$ depending
  on $M$, $R$ and $a$ such that the following holds:

  Any curve $\alpha :[t_{0}, t_1 ] \rightarrow \Tone$ of $A_{2}$ points such that:
  \begin{enumerate}
   \item $| \alpha (0)| \leqslant R$
   \item $\langle \alpha', r\rangle <0$
   \item $\alpha' (t)$ is within a cone around $D$ of amplitude $cA_{\alpha (t)}^{3}$
   \item the slack at all points of $\alpha$ is bounded below by a constant
$a>0$
  \end{enumerate}
  has the following properties:
  \begin{enumerate}
   \item $\alpha$ is unbeatable: any retort $\beta$ satisfies  $| \beta (t_{1} )|-| \beta (0)|<| \alpha (0)|-| \alpha (t_{1} )|$.
  \item $| \beta (t_{1} )|-| \beta (0)|<| \alpha(0)|-| \alpha (t_{1} )|- \varepsilon$.
   \end{enumerate}
\end{lem}

\begin{proof}
  We only need to prove the second statement, as any curve of $A_2$ points defined in a compact interval has slack bounded from below.

  Fix a neighborhood $U$ of adapted $A_2$ coordinates that contains the image of $\alpha$.
  We can assume that one such $U$ contains all of the image of $\alpha$, otherwise we can split $\alpha$ into parts.
  
  Let $v(t)$ be the vector at $\alpha(t)$ such that $d(t) = -r(\alpha(t)) + v(t)$ belongs to $D_x$.
  Then the slack $A_{\alpha (t)}$ is $\frac{|r(\alpha (t))|}{|-r(\alpha (t))+v(t)|} = \frac{1}{|d(t)|}$.
  
  We reparametrize $\alpha$ so that $\alpha'(t)= d(t)+p(t)$, with $p(t)$ is orthogonal to $d$ (this is the canonical parametrization). Then $|p(t)|<c|\alpha'(t)| $, and for $c<<1$, this implies that $|p(t)|<(c/2) |d(t)|=\frac{c}{2\: A_{\alpha(t)}}<\frac{c}{2\: a}$.

   We compute 
  \[ | \alpha (0)|-| \alpha (t_{1} )|= \int_{0}^{t_{1}} \frac{d}{dt} | \alpha
     |=  \int_{0}^{t_{1}} \langle r, \alpha' \rangle  \geq t_1(1-\frac{c}{2a}) \]
  
  An $A_2$ point only has one preimage in $U$, so any retort $\beta$ of $\alpha$ lies outside of $U$. 
  As $\e ( \alpha (t))= \e ( \beta (t_{1} -t))$, we have:
  $$|d \e (r(\alpha(t))) - d\e(r(\beta(t_1-t)))|>\varepsilon_1$$
  for some $\varepsilon_1$ depending on $U$.
  $U$, in turn, contains a ball around $x$ of radius at least $r_0$, a number which depends on $a$ and $R$: the differential of the slack is bounded, so if $A_x=a>0$, it cannot drop to $0$ in a ball of sufficiently small radius.
  Thus, we can switch to a smaller $\varepsilon_1>0$ that depends only on $a$
and~$R$.
  
  Write $\beta'(t)=b(t)r(\beta (t)) +w(t)$, where $w$ is a vector orthogonal to $r(\beta ( t ))$.
  It follows from the above that $|b(t)|<1-\varepsilon_2 $ for some
$\varepsilon_2 $ depending on $M$ and~$\varepsilon_1$.
  
  We compute:
  \[ | \beta (0)|-| \beta (t_{1} )|= \int_{0}^{t_{1}} \frac{d}{dt} | \beta
     |=  \int_{0}^{t_{1}} b(t)  \leq t_{1}(1-\varepsilon_2) \]

  and the result follows.
  
\end{proof}

With this lemma, we can perturb a CDC slightly to avoid
some points:

\begin{dfn}
  An {\emph{approximately conjugate descending curve}} (ACDC) is a $C^1$ curve $\alpha$
  of $A_{2}$ points such that $\alpha' (t$) is within a cone around $D$ of
  amplitude $cA_{\alpha (t)}^{3}$, where $c$ is the constant in the previous
  lemma for $R= \alpha (0$).
\end{dfn}

\subsection{Linking curves}

Using the results so far, it is not hard to prove (see lemma \ref{existence of
GACDC for easy manifolds} and \ref{existence of GACDC}) that there is always
a ACDC $\alpha :[0,t_{0} ] \rightarrow \T$
starting at any point $x \in \mathcal{J}$, whose interior consists only of
$A_{2}$ points and ending up in an $A_{3}$ point. We also know that we can
start a retort $\widetilde{\alpha_{1}}$ at the $A_{3}$ point.

We can continue the retort while it remains in the interior of $V_{1}$, where
$ e_{1}$ is a local diffeomorphism and we can lift any curve. However, we might
be unable to continue the retort up to $x$ if the returning curve hits the set
of conjugate points.

If we hit an $A_{2}$ point $y= \widetilde{\alpha_{1}} (t_{1} )$, we can take
a ACDC $\beta :[0,t_{2} ] \rightarrow V_{1}$ starting at this point and ending
in an $A_{3}$ point. If $\beta$ has a retort $\tilde{\beta} :[0,t_{2} ]
\rightarrow \T$ that ends up in a non-conjugate point $\tilde{\beta} (t_{2}
)$, we can continue with the retort $\widetilde{\alpha_{2}}$ of $\alpha
|_{[0,t_{0} -t_{1} ]}$ starting at $\tilde{\beta} (t_{2} )$. If
$\widetilde{\alpha_{2}}$ can be continued up to $x= \alpha (0)$, the concatenation of $\widetilde{\alpha_{1}}$, $\beta$, $\tilde{\beta}$ and $\widetilde{\alpha_{2}}$ can play the same as the retort of $\alpha$ (see figure \ref{figure: standard T}).

There are a few things that may go wrong with the above argument: the retort
$\widetilde{\alpha_{1}}$ may meet $\mathcal{J} \setminus \mathcal{A}_{2}$, or
$\beta$ may not admit a full retort starting at $\beta (t_{2} )$, or $\alpha
|_{[0,t_{0} -t_{1} ]}$ may not admit a full retort starting at $\tilde{\beta}
(t_{2} )$. The first problem can be avoided if the ACDCs are built to dodge some small sets, as we will see later.
Then, if we assume that a retort never meets $\mathcal{J} \setminus \mathcal{A}_{2}$, we can iterate the above argument whenever a retort is interrupted upon reaching an $\mathcal{A}_{2}$ point.
We will prove later that the argument only needs to be applied a finite number of times.

This is the motivation for the definition of linking curve:

\begin{dfn}\label{definition: linking curve}
  A {\strong{linking curve}} is a continuous curve $\alpha :[0,t_{0} ]
  \rightarrow T_{p} M$ that is the concatenation $\alpha = \alpha_{1} \ast
  \ldots \ast \alpha_{ n}$ of ACDCs and non-trivial retorts of those ACDCs,
  all of them {\emph{of finite length}}, such that:
  \begin{itemizedot}
    \item Starting with the tuple $( \alpha_{1} , \ldots , \alpha_{n} )$
    consisting of the curves that $\alpha$ is made of, in the same order, we
    can reach the empty tuple by iteration of the following rule:
    
    Cancel an ACDC $\alpha_{j}$ together with a retort $\alpha_{j+1}$ of
    $\alpha_{j}$ that follows inmediately:

    $( \alpha_{1} , \ldots , \alpha_{j} , \alpha_{j+1} , \ldots \alpha_{n} )
    \rightarrow ( \alpha_{1} , \ldots , \alpha_{j-1} , \alpha_{j+2} , \ldots
    \alpha_{n} )$, if $\alpha_{j+1}$ is a retort of $\alpha_{j}$.
    
    \item The extremal points of the $\alpha_{i}$ are called the
    {\strong{vertices}} of $\alpha$. The vertices of $\alpha$ fall into one
    of the following categories:
    \begin{itemizeminus}
      \item starting point (first point of $\alpha_{1}$): a point in
      $\mathcal{J}$.
      
      \item end point (last point of $\alpha_{n}$): a point in $\mathcal{I}$.
      
      \item $A_{3}$ join, as explained in section \ref{section: A3 joins}.
      
      \item a {\strong{splitter}}: a vertex that joins two ACDCs whose
      concatenation is also a ACDC.
      
      \item a {\strong{hit}}: a vertex that joins a retort that reaches
      $\mathcal{A}_{3}(I) $ transversally, and an ACDC starting at the intersection point.
      
      \item a {\strong{reprise}}: a vertex that joins a retort that
      completes its task of replying to a ACDC $\alpha_{j}$, and the retort
      for a different ACDC $\alpha_{i}$ (it follows from the first condition
      that $i<j$).
    \end{itemizeminus}
    \item The preimage of a point of $M$ by $ e_{1} \circ \alpha$ falls into one of the following categories:
    \begin{itemizeminus}
      \item it can be empty.
      
      \item it can have one point that is an $A_{3}$ join.
      
      \item it can have two points, one $A_{2}$ point in the interior of an
      ACDC and an NC point in the retort of that ACDC.
      
      \item it can have two points, the first and the last points of
      $\alpha$.
      
      \item it can consist of three vertices: a splitter, a hit and a reprise,
      such that the six curves $\alpha_{i}$ contiguous to any of these three
      points map to a $T$-shaped curve, with two curves mapping into each
      segment of the T. See figure \ref{figure: standard T}. We call this
      combination of three vertices a {\strong{standard T}}.
    \end{itemizeminus}
  \end{itemizedot}
\end{dfn}

\begin{dfn}
  A {\strong{standard T}} consists of three vertices: a splitter, a hit and
  a reprise, such that the six curves $\alpha_{i}$ contiguous to any of these
  three points map to a $T$-shaped curve, with two curves mapping into each
  segment of the T. See figure \ref{figure: standard T}.
\end{dfn}

\begin{figure}[H]
  \includegraphics[width=0.8\textwidth]{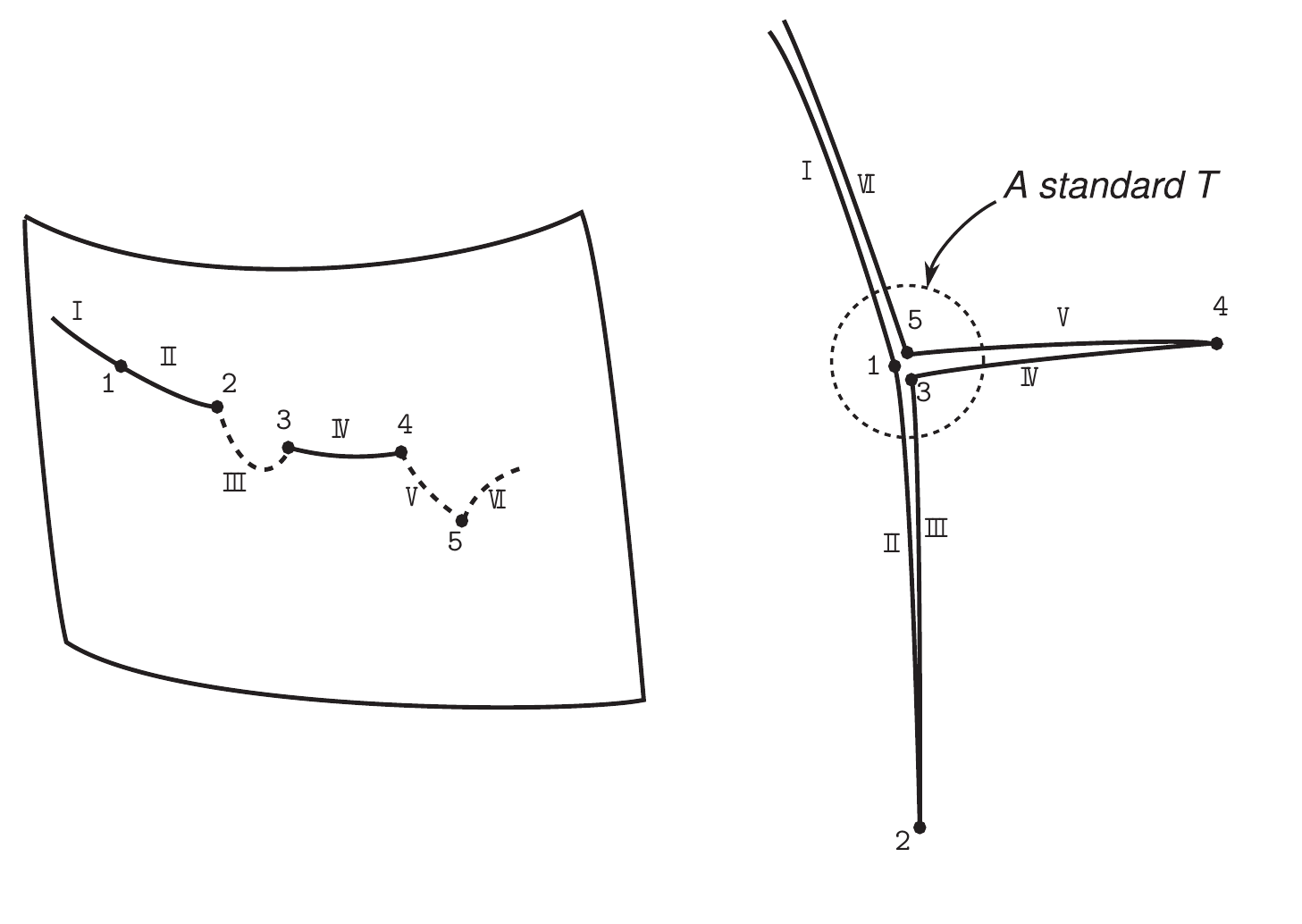}
  \caption{A {\emph{Standard T}}: The left hand side
  displays a curve $\alpha$ in $\T$, while the right hand side displays $\e
  \circ \alpha$.
  I, II and IV are ACDCs, III is the retort of II, V is the retort of IV, and
  VI is the retort of I. Vertices 2 and 4 are $A_{3}$ joins, vertex 1 is a
  splitter, vertex 3 is a hit and vertex 5 is a reprise. There can be more
  than two segments between a splitter and its matching hit, and between a hit
  and its matching reprise.}
  \label{figure: standard T}
\end{figure}

\begin{remark}
  A linking curve is {\strong{non-trivial}} if it contains at least one ACDC.
\end{remark}

\pagebreak[2]

\begin{lem}
 Let $\alpha = \alpha_{1} \ast \ldots \ast \alpha_{ n}$ be a non-trivial linking curve:
 \begin{itemize}
  \item $\alpha_{1}$ is an ACDC and $\alpha_{n}$ is its retort.
  \item Whenever $\alpha_k$ is the retort of $\alpha_j$, for $1<j<k<n$, then $\alpha_{j} \ast \ldots \ast \alpha_{ k}$ is a linking curve.
 \end{itemize}
\end{lem}
\begin{proof}
 The proof is simple and is left to the reader.
\end{proof}

\begin{prop}\label{main properties of linking curves}
  Let $\alpha$ be a linking curve between $x,y \in \Tone$, and $M_{1}$ and
  $M_{2}$ two $L$-related Riemannian manifolds. Then:
  \begin{itemizedot}
    \item $|x|>|y|$
    
    \item $\e \circ \alpha$ is fully tree-formed (in particular, it is continuous)
    
\diff{
    \item $e_1(x)=e_1(y)$ and $e_2(x)=e_2(y)$
}
  \end{itemizedot}
\end{prop}

\begin{proof}
  The first part follows trivially from lemma \ref{unbeatable lemma} and its
  generalization, lemma \ref{slack lemma}. Each pair of a ACDC and its retort
  adds a negative amount to the radius of $x$.
  
  For the second part, we reparametrize $\alpha$ to the unit interval $[0,1]$.
  Let $T:[0,1] \rightarrow \tmop{Im} ( \exp \circ \alpha )$ be the
  identification given by $\exp \circ \alpha$. Let us show that $u= \exp \circ
  \alpha$ is tree-formed with respect to $T$: let $t_{1}$, $t_{2}$ such that
  $u(t_{1} )=u(t_{2} )$, and $\varphi$ a continuous $1$-form along $u$ $(
  \varphi (s) \in T^{\ast}_{u(s)} M)$ that factors through $\Gamma$. Then we
  claim that:
  \begin{equation}
    \int_{t_{1}}^{t_{2}} \varphi (s)(u' (s))ds \label{integral that must be 0
    for the curve to be tree formed}
  \end{equation}
  splits as a sum of integrals over the image by $\exp$ of an ACDC and the
  image of its matching retort. The curves in each such pair have the
  same image, and the integrals cancel out, as the integral of a $1$-form is
  independent of the parametrization, and only differs by sign.
  
  The claim follows if $u^{-1} (u(t_{1} ))$ consists of two points, because
  $t_{1}$ is in the domain of an ACDC $\alpha_{i}$ and $t_{2}$ lies in the
  retort $\alpha_{j}$ of $\alpha_{i}$. We recall it is possible to reach an
  empty tuple by cancelling adjacent pairs of an ACDC and its retort. Thus, in
  order to cancel $\alpha_{i}$ and $\alpha_{j}$, it must be possible to cancel
  all the curves $\alpha_{k}$ with $i<k<j$. These curves can be matched in
  pairs $\{ ( \alpha_{n} , \alpha_{m} ) \}_{( n,m ) \in \mathcal{P}}$ of ACDC
  and retort, with $i<n<m<j$ for each pair $( n,m ) \in \mathcal{P}$. Then we
  have:
\begin{align*}
    \int_{t_{1}}^{t_{2}} \varphi (s)(u' (s))ds  =\, & 
    \int_{t_{1}}^{t_{2}^{i}}
      \varphi (s)(( \exp \circ \alpha_{i} )' (s))ds+
    \\
    & \sum_{( n,m ) \in \mathcal{P}} \Big( \int_{t^{n}_{1}}^{t^{n}_{2}}
    \varphi (s)(( \exp \circ \alpha_{n} )' (s))ds\:+\\
    & \hskip46pt\int_{t^{m}_{1}}^{t^{m}_{2}} \varphi (s)(( \exp \circ \alpha_{m}
)' (s))ds \:
    \Big) +\\
    & \int_{t_{1}^{j}}^{t_{2}} \varphi (s)(( \exp \circ \alpha_{j} )' (s))ds
\end{align*}

  The remaining two integrals also cancel out, proving the claim.
  
  If $t_{1}$ and $t_{2}$ are two of the three points of a standard T, we can
  take points $t^{\ast}_{1}$ and $t^{\ast}_{2}$ as close to $t_{1}$ and
  $t_{2}$ as we want, but in an ACDC and its retort, respectively, and such that $u(t_1^\ast)=u(t_2^\ast)$.
  The result follows because the integral \ref{integral that must be 0 for the curve to
  be tree formed} depends continuously on $t_{1}$ and $t_{2}$.
  
  The last part is similar to lemma 4.1 in {\cite{Hebda}}. In the hypothesis,
  we are assuming that the curve has a specific structure, which makes the proof
  simpler, but we do not ask for the sequence of curves converging to the
  linking curve in the hypothesis of that lemma, so we will have to build it
  ourselves.
  
  Let $\alpha = \alpha_{1} \ast \ldots \ast \alpha_{ n}$ be a linking curve
  between $x \in V_{1} \subset \Tone$ and $y \in V_{1}$, with each
  $\alpha_{i}$ either a ACDC, or the retort of one of the previous ACDCs. We
  write $j= \bar{i}$ whenever $\alpha_{j}$ is a retort of $\alpha_{i}$.
  
  We want to find an open set $O^{0} \subset V_{1}$ such that $\e|_{O^{0}}$ is
  injective, $\tmop{Im} ( \alpha ) \subset \overline{O^{0}}$, and a sequence
  of curves $\alpha^{k}$ converging to $\alpha$ in $\tmop{AC} (M_{1} )$ such
  that $\tmop{Im} ( \alpha^{k} ) \subset O^{0}$.
  
  We first construct a set $O$ as the union of neighborhoods $O_{k}$ of the
  vertices of $\alpha$, neighborhoods $U_{i}$ of the ACDCs in $\alpha$ and
  neighborhoods $W_{j}$ of the retorts of those ACDCs.
  
  First, we take disjoint neighborhoods $O_{k}$ of the vertices. We can assume
  that they are disjoint with the preimages of the images of the other
  $O_{k}$, except for the neighborhoods of the three vertices of the same
  standard T.
  
  Next, we take neighborhoods $U_{i} \subset \Tone$ of (the interior of the
  image of) each ACDC $\alpha_{i}$ in $\alpha$, such that there is no
  non-trivial retort of $\alpha_{i}$ in $U_{i}$.
  By the third property in the definition \ref{definition: linking curve}, we can assume that
  $\overline{U_{i_{1}}} \cap \e^{-1} \left( \e \left( \overline{U_{i_{2}}}
  \right) \right)$ \ is empty unless $\alpha_{i_{1}}$ and $\alpha_{i_{2}}$ are
  consecutive ACDCs joined by a ``splitter'' vertex, in which case the
  intersection is only the splitter. Also, $U_{i}$ should only intersect
  $\e^{-1} \left( \e ( \overline{O_{k}} ) \right)$ when $O_{k}$ is a
  neighborhood of one of the two endpoints of $\alpha_{i}$. It follows that no
  retort of any part of $\alpha_{i_{0}}$ passes through $\cup
  \overline{U_{i}}$.
  
  And last, the neighborhood $W_{j_{0}}$ of (the interior of the image of) the
  retort $\alpha_{j_{0}}$ of $\alpha_{i_{0}}$ has to be chosen so that:
  \begin{itemize}
    \item $\overline{W_{j_{0}}}$ is disjoint with $\cup_{j \neq j_{0}} \e^{-1}
    \left( \e ( \overline{W_{j}} ) \right)$ and $\cup _{i \neq i_{0}} \e^{-1}
    \left( \e ( \overline{U_{i}} ) \right)$, or consists of just one vertex if
    the curves are consecutive.
    
    \item $W_{j_{0}}$ should only intersect $\e^{-1} \left( \e (
    \overline{O_{k}} ) \right)$ when $O_{k}$ is a neighborhood of one of the
    two endpoints of $\alpha_{j_{0}}$.
  \end{itemize}

  We still have to build the set $O^{0}$. The neighborhood $U_{i}$ of an ACDC
  $\alpha_{i}$ maps 2:1 to a half ball by $\e$. Its intersection with $V_{1}$
  is a set $U^{0}_{i}$ that maps 1:1 onto the same image. The neighborhood
  $W_{j}$ of its retort $\alpha_{j}$ maps 1:1 to a tubular neighborhood of
  $\tmop{Im} \left( \e \circ \alpha_{i} \right) = \tmop{Im} \left( \e \circ
  \alpha_{j} \right)$. We take $W_{j}^{0} =W_{j} \setminus
  \overline{ e_{1}^{-1} ( \e (U_{i} ))}$ as the neighborhood of
   $\alpha_{j}$.

  We next describe how to build neighborhoods for each type of vertex, so that
  they are compatible with the neighborhoods $U_{i}^{0}$ and $W_j^0$ for the curves.
  \begin{description}
    \item[An $A_{3}$ join] We take the neighborhood $O^{0}_{k}$ defined (in
    special coordinates) by $\{x_{2} <3x_{1}^{2} ; x_{1} \leq 0\} \cup \{x_{2} <
    \frac{3}{4} x_{1}^{2} ;x_{1} >0\}$ (we assume that the CDC in the join
    comes from the ``left'' side $x_{1} <0$). As shown in section
    \ref{section: A3 joins}, the boundary of $O^{0}_{k}$ consists of a half
    surface foliated by CDCs and another half surface foliated by the retorts
    of those curves.
    
    \item[A splitter (in a standard T)] For this type of point we proceed as
    if the two ACDCs that join at the split point were one only ACDC. So we
    take the neighborhood $O_{k}$ of the splitter point, and intersect it with
    $V_{1}$: $O_{k}^{0} =O_{k} \cap V_{1}$.
    
    \item[A hit (in the same standard T)] For this point (which is $A_{2}$) we
    intersect its neighborhood $O_{l}$ with $V_{1}$, and also remove the
    preimage of the image of the neighborhood $O_{k}$ of the accompanying
    splitter: $O_{l}^{0} =(O_{l} \cap V_{1} ) \setminus \overline{ e_{1}^{-1} (
     e_{1} (O^{0}_{k} ))}$.
    
    \item[A reprise (in the same standard T)] This is a non-conjugate point,
    and we remove from its neighborhood $O_{m}$ the preimage of the images of
    the neighborhoods $O^{0}_{k}$ and $O_{l}^{0}$ of the accompanying splitter
    and hit:
    \[ \text{$O_{m}^{0} =O_{m} \setminus \overline{(  e_{1}^{-1} (  e_{1}
       (O^{0}_{l} )) \cup  e_{1}^{-1} (  e_{1} (O^{0}_{k} )))}$.} \]
  \end{description}

  The reader can check that for $O^{0} = \bigcup_{k} O_{k}^{0} \cup
  \bigcup_{i} O_{i}^{0} \cup \bigcup_{j} O^{0}_{j}$, $\e|_{O^{0}}$ is
  injective and $\tmop{Im} ( \alpha ) \subset \overline{O^{0}}$.
  
  Let us build the $k$-th approximation to $\alpha$. This will be a curve
  $\alpha^{k}$ with $\tmop{Im} ( \alpha^{k} ) \subset O^{0}$, consisting of
  $2n+1$ parts: one for each curve $\alpha_{i}$ in $\alpha$ and one for each
  vertex $v_{k}$. The part corresponding to the curve $\alpha_{i}$ is a curve
  in a $1/k$ neighborhood of $\tmop{Im} ( \alpha_{i} )$. The part
  corresponding to the vertex $v_{k}$ will be a $C^{1}$ curve that has length
  bounded by $C k$ for some universal constant and joins the
  approximations to the curves $\alpha_{i}$ adjacent to $v_{k}$. The argument
  proceeds now as sketched in section \ref{subsection: Hebda's approach}, or
  proved in detail in lemma 4.1 of {\cite{Hebda}}:

  Let $Y$  be the concatenation of the radial line in $\Tone$ that ends up in $\alpha(0)$ with $\alpha$, with both curves rescaled so that $Y(1/2)=\alpha(0)$, and $Y(1)=\alpha(1)$.
  Let $u= e_{1}\circ Y$, and let $v$ be the curve obtained by affine developement of $u$ followed by inverse affine developement onto $\Ttwo$.
  The curve $u$ is tree-formed with respect to an identification $T$ with $T(1/2)=T(1)$, and as we have seen it follows that $v$ also has that property.
  In particular, $v(1/2)=v(1)$

  Let $Y_k$ be the concatenation of the radial line in $\Tone$ that ends up in $\alpha_k(0)$ with $\alpha_k$, with both curves rescaled so that $Y_k(1/2)=\alpha_k(0)$, and $Y_k(1)=\alpha_k(1)$.
  Let $u_k= e_{1}\circ Y_k$, and let $v_k$ be the curve obtained by affine developement of $u_k$ followed by inverse affine developement onto $\Ttwo$.
  The curves $u_k$ are contained in $O^0$, where $ e_{2}\circ e_{1}^{-1}$ is an isometry, so $v_k= e_{2}\circ e_{1}^{-1}\circ u_k$.

  It follows at last that $v(1/2)=\lim_k v_k(1/2)=\lim_k v_k(1/2)= e_{2} \lim_k Y_k(0)= e_{2}(x)$ and $v(1/2)=\lim_k v_k(1)=\lim_k v_k(1)= e_{2} \lim_k Y_k(1)= e_{2}(y)$, and thus $ e_{2}(x)= e_{2}(y)$.

\end{proof}

\diff{
As we promised, the following is also true:

\begin{prop}\label{main properties of linking curves 2}
  Let $\alpha$ be a linking curve between $x,y \in \Tone$, and $M_{1}$ and $M_{2}$ two $L$-related Riemannian manifolds. 
  
  Then $x$ and $y$ are linked.
\end{prop}

but we defer the proof until \ref{proof that linking curves produce linked points}.
}

\subsection{Existence of linking curves for easy manifolds from a point}

We now begin the proof of theorem \ref{claim: hypothesis of the synthesis theorem for easy manifolds}. The only places where we assume that $M_1$ is easy from $p$ is in lemma \ref{existence of GACDC for easy manifolds} and theorem \ref{typical sings have transient neighborhoods for easy manifolds}.

The goal of this section is to prove the existence of a linking curve starting at an arbitrary point $x \in \mathcal{J}$.
The set $\{y:|y|<|x|, \e (y)= \e
(x))\}=\{y_{j} \}$ is finite. This follows because $\{y:|y| \leqslant |x|\}$
can be covered with a finite amount of neighborhoods of adapted coordinates,
and in any of them the preimage of any point is a finite set. At least one
$y_{j}$ realizes the minimum distance from $p$ to $q= \e (x)$, and must be
either $A_3(I)$ or NC (in other words, $y \in \mathcal{I}$). We will show that
there is a linking curve joining $x$ and one $y_{j} \in \mathcal{I}$, though
it may not be the one with minimal radius.

\begin{theorem}
  \label{theorem: existence of linking curves for easy manifolds}For any $x \in \mathcal{J}$,
  there is a linking curve that joins $x$ to some $y \in \mathcal{I}$.
\end{theorem}

We start with a generalization of a linking curve that we can describe
informally as a linking curve {\emph{under construction}}:

\begin{dfn}
  An {\strong{aspirant curve}} is a continuous curve $\alpha :[0,t_{0} ]
  \rightarrow T_{p} M$ that is the concatenation $\alpha = \alpha_{1} \ast
  \ldots \ast \alpha_{ n}$ of ACDCs and non-trivial retorts of those ACDCs,
  such that:
  \begin{itemizedot}
    \item Starting with the tuple $( \alpha_{1} , \ldots , \alpha_{n} )$
    consisting of the curves that $\alpha$ is made of in order, we can reach
    a tuple \strong{with no retorts}, by iteration of the
    following rule:
    
    {\emph{Cancel an ACDC together with a retort of that ACDC that follows
    right after it: $( \alpha_{1} , \ldots , \alpha_{j-1} , \alpha_{j} ,
    \alpha_{j+1} , \alpha_{j+2} , \ldots \alpha_{n} ) \rightarrow ( \alpha_{1}
    , \ldots , \alpha_{j-1} , \alpha_{j+2} , \ldots \alpha_{n} )$, if
    $\alpha_{j+1}$ is a retort of $\alpha_{j}$. }}
    
    \item In all other regards, an aspirant curve satisfies the same
    conditions as a linking curve.
  \end{itemizedot}
  The {\strong{loose}} ACDCs in $\alpha = \alpha_{1} \ast \ldots \ast
  \alpha_{k}$ are the ACDC curves $\alpha_{j}$ for which there is no retort in
  $\alpha$.

  The \strong{tip} of $alpha$ is its endpoint $\alpha(t_0)$.
\end{dfn}

\begin{dfn}
  We define some important sets:
  \[ S_{R} =B_{R} \cap \e^{-1} ( \e ( \mathcal{A}_{2} \cap B_{R} )) \]
  \[ V_{1}^{0} =\{x \in V_{1} : \e^{-1} ( \e (x)) \cap B_{|x|} \subset \NC
     \cup \mathcal{A}_{2} \} \]
  \[ \mathcal{SA}_{2} =\{x \in \mathcal{A}_{2} :    \exists y \in
     \mathcal{A}_{2} ,   \e (y)= \e (x), |y|<|x|\} \]
  
  In other words, $V_{1}^{0}$ consists of those points $x \in V_{1}$ such that
  all preimages of $\e (x$) with radius smaller than $|x|$ are $\NC$ {\emph{or}}
  $\mathcal{A}_{2}$.
\end{dfn}

\begin{dfn}
  Let $F \subset \mathcal{C}$ be a finite set. A GACDC with respect to $F$, or
  GACDC when $F$ is implicit (G is for {\emph{generic}}) is an ACDC $\alpha$
  such that
  \begin{itemizedot}
    \item $\tmop{Im} ( \alpha )$ is contained in $( \mathcal{C} \cap V_{1}^{0}
    ) \setminus F$.
    
    \item for $y \in B_{| \alpha (t_{0} )|} \cap \mathcal{A}_{2}$ such that
    $\e ( \alpha (t_{0} ))= \e (y)$, $\e \circ \alpha$ is transversal to $\e
    (\mathcal{A}_{2} \cap B_{\varepsilon} (y))$ at $t_{0}$, for some
    $\varepsilon >0$.
  \end{itemizedot}
\end{dfn}

The motivation for the definition of GACDC is to find curves starting at points $x \in \mathcal{I}$ so that any possible retort avoids all singularities that are not $A_{2}$.
The GACDC will also ``avoid itself'': this is indeed the finite set F that it must avoid, as we will see later.

\begin{lem}
  \label{existence of GACDC for easy manifolds}For any $R>0$ there is $L>0$ such that any GACDC
  starting at $x \in A_{2} \cap B_{R}$ has length at most $L$, and can be
  extended until it reaches an $A_{3}$ point ($F$ can be any finite set).
\end{lem}

\begin{proof}
  First, we prove local existence (and thus, continuation) of GACDC. Let $x
  \in \mathcal{A}_{2}$.
  
  Let $y \in \mathcal{C} \cap \e^{-1} (x)$. Let $U \subset M$ be a
  neighborhood of $q= \e (x)= \e (y)$, $W_{1}$ and $W_{2}$ disjoint
  neighborhoods of $x$ and $y$ mapping into $U$. $\mathcal{C} \cap W_{1}$ is a
  smooth hypersurface containing $x$. $\mathcal{C} \cap W_{2}$ may have
  conical singularities, but is a stratified manifold in any case. The
  transversality result mentioned at the end of section \ref{section:
  generic metric} implies that $\e (\mathcal{C} \cap W_{1} )$ is transversal
  to each stratum of $\e (\mathcal{C} \cap W_{2} )$.
  
  The CDCs foliate $\mathcal{C} \cap W_{1}$, so the set of points of
  $\mathcal{C}$ whose CDC sinks into a stratum of $\e^{-1} ( \e (\mathcal{C}
  \cap W_{2} ))$ with singularities other than $A_{2}$ has positive
  codimension in $\mathcal{C} \cap W_{1}$. Replacing one small subcurve of the
  CDC with an ACDC we can move from one CDC to a neighbouring one, thus
  avoiding those singularities. We might not be able to avoid that our ACDC
  meets $\e^{-1} ( \e ( \mathcal{A}_{2} ))$, but we can take our ACDC so that
  it intersects that set transversally.

  We can continue the ACDC within a patch $U$ of adapted coordinates. There
  is some $L>0$ such that any ACDC within $U$ can be extended by a curve of
  length at most $L$ that may end up in an $A_{3}$ point, or reach the
  boundary of $U$.
 
There is a smaller neighborhood $V \subset U$ such that any GACDC starting at
$x \in V$ is continued within $U$ up to an $A_{3}$ point, or up to a point in
$\partial U$ with smaller radius that any point in $V$. Thus $V$ is transient,
in the sense that an ACDC that passes through $V$ will either finish or leave
$U$ and never return to~$V$. It is simple to choose such a set $V$; it will be
clear how to do it after we prove claim 
\ref{typical sings have transient neighborhoods for easy manifolds}.

  We have shown that there is an GACDC with bounded length that exists~$V$, but
indeed, the length of any ACDC in $V$ is also bounded, because in the
plane~$\mathcal{A}_{2}$, any ACDC is a $C^1$ graph over any CDC.
  
  The radius decreases along an ACDC, and thus an ACDC starting at $x$ never
  leaves $\{v \in \T :|v|<|x|\}$. Take a finite cover of this set by
transient~sets.
  An ACDC that starts at $x$ will run through a finite amount of transient sets.
  Each transient set only contributes a finite length to the total
  length of the GACDC that started at $x$.
\end{proof}

Diagram \ref{figure: algorithm for linking curves} shows the algorithm that we follow in order to find the
linking curves, starting with the trivial aspirant curve $\{x\}$.

\begin{figure}[ht]
 \includegraphics[width=0.9\textwidth]{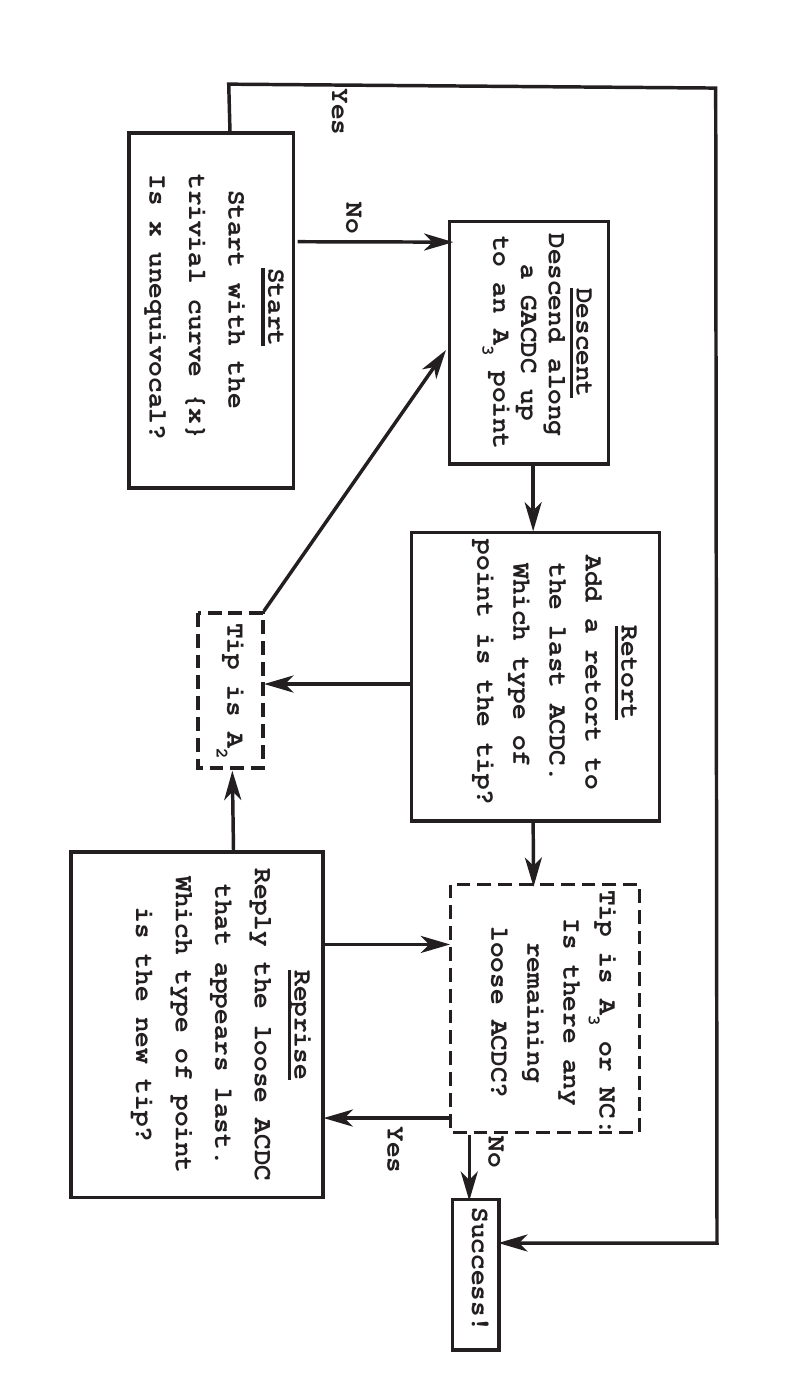}
 \caption{Flow diagram for building linking curves}
 \label{figure: algorithm for linking curves}
\end{figure}

The linking curve is built step by step, starting with the trivial
{\emph{curve}} $\alpha =\{x\}$, and adding segments to the aspirant curve
$\alpha = \alpha_{1} \ast \ldots \ast \alpha_{k}$ following these rules:
\begin{description}
  \item[Descent] If the end of $\alpha_{k}$ is a point in $\mathcal{J}$, let
  $\gamma$ be a GACDC contained in $V_{1}^{0}$ that starts at $x$. 
  The curve $\gamma$ must also avoid the finite set $F= \e^{-1} ( \tmop{Im} ( \e \circ \alpha ) \cap B_{| \alpha ( 0
  ) |} ) \cap \mathcal{C}$.
  We also know that $\gamma$ intersects $\SAtwo$ in a finite set and, for convenience, we split $\gamma$ into $r$ GACDCs $\alpha_{k+1},\dots,\alpha_{k+r}$ such that each of these curves intersects $\SAtwo$ only at its extrema.
  The new curve $\alpha \ast \alpha_{k+1} \ast \dots \ast \alpha_{k+r}$ ends up in an $A_{3}$ point. The next step is a retort.
  
  \item[Retort] 
  If $\alpha_{k} :[0,T] \rightarrow V_{1}$ is a ACDC ending up in an $A_{3}$ point, add the retort $\alpha_{k+1}$ of $\alpha_{k}$ that starts at the $A_{3}$ join.
  This is always possible, since $\alpha_{k}$ does not intersect $\SAtwo$.
  The new tip of $\alpha \ast \alpha_{k+1}$ will be $NC$, $A_2$ or $A_3$, but the latter can only happen if $\alpha \ast \alpha_{k+1}$ is a linking curve.

  \item[Reprise] If the tip of $\alpha$ is $NC$ and $\alpha$ is not a linking curve, let $\alpha_j$ be the latest loose curve in $\alpha$.
  We add the retort $\alpha_{k+1}$ of $\alpha_{j}$ starting at the tip of $\alpha$.
  This is always possible, since $\alpha_{j}$ does not intersect $\SAtwo$.
  The new tip of $\alpha \ast \alpha_{k+1}$ will be $NC$, $A_2$ or $A_3$, but the latter can only happen if $\alpha \ast \alpha_{k+1}$ is a linking curve.

  \item[Success!] If $\alpha$ is a linking curve, we report success and stop the algorithm.
  For completeness, the algorithm also reports success if $\alpha =\{x\}$, for
  $x \in \mathcal{I}$.
\end{description}

The algorithm can also be presented in a \emph{recursive} fashion. We start with some definitions:

\begin{itemize}
 \item $Tip(\alpha)=\alpha(T)$, for any curve $\alpha$ defined in an interval $[0,\alpha]$.
 \item $Ret(\alpha, y)$ is the retort of $\alpha$ starting at $y$, for any curve $\alpha$ contained in $V^1_0\setminus\SAtwo$, and a point $y\in V_1$ such that $\e(y)=\e(Tip(\alpha))$.
\end{itemize}

Then for any $x\in V_1$, we define an aspirant curve $L(x)$ by the following rules:
\begin{itemize}
 \item If $x\in\mathcal{J}$, then $L(x)=\{x\}$
 \item If $x\in \mathcal{I}$, then compute the GACDC curve $\gamma=\gamma_{1} \ast \dots \ast \gamma_{r}$, as above. Then $L(x)=\gamma_1\ast L(Tip(\gamma_1)) \ast Ret(\gamma_1,Tip(L(Tip(\gamma_1))))$
\end{itemize}

\begin{remark}
 The reader have probably noticed that $\gamma_{2}$ to $\gamma_{r}$ are discarded, and only $\gamma_1$ is kept (the ACDC up to the first $A_2$ point).
 This causes a small problem with the recursive definition because of the non-deterministic descent step.
 We have shown that there is a GACDC starting at any point in $\mathcal{I}$, and this curve intersects $\SAtwo$ in finitely may points, but if we only keep the first segment of the GACDC up to the first intersection with $\SAtwo$ and repeat the process, we have not shown that an $A_3$ point will be reached in finitely many steps.
 This can be solved in one blow by an application of the axiom of choice.
 We might also come back to \ref{existence of GACDC for easy manifolds} and refine it as needed.
 But the easiest solution is to use the iterative version of the algorithm.
\end{remark}

  In order to satisfy the last technical condition in the definition of
  linking curve, we have added to the ``Descent'' section the condition that
  $\e \circ \alpha_{k+1}$ does not intersect the image of $\e \circ \alpha$.
  
  We recall that we can ask that $\alpha_{k+1}$ avoids a finite set $F \subset
  \mathcal{C}$. The image of $\e \circ \alpha$ is the same as the image by
  $\e$ of $\tmop{Im} ( \alpha ) \cap \mathcal{A}_{2}$, or in other words, the
  image by $\e$ of only the ACDCs in $\alpha$. Each ACDC $\alpha_{j}$ in
  $\alpha$ was built so that it did not intersect $\e^{-1} ( \e (S))$, for any
  strata $S$ of singularities with smaller radius than $\alpha$, except for
  strata of $A_{2}$ singularities, which it would intersect transversally. The
  ACDC $\alpha_{k+1}$ is contained in a strata $S_{0}$ of $A_{2}$ points, and
  thus $S_{0} \cap \e^{-1} ( \alpha_{j} )$ is a finite set.

  Thus, lemma \ref{existence of GACDC for easy manifolds} guarantees that we can always perform
the ``{\emph{descent}}'' step in the diagram. We have already shown why the other steps can always be performed.

We conclude that it is always possible to perform one more step of the algorithm, if it hasn't reported ``success!'' yet.
However, the algorithm may get hooked up in an infinite sequence of GACDC, retorts and reprises.
We devote the rest of the section to prove that this is not the case, for a generic metric.

\begin{dfn}
  A pair ($S,O$) of open subsets of $\T$ with $\bar{S} \subset O$, is
  {\strong{transient}} iff for any point $x$ in $S \cap \mathcal{J}$, a
  finite number of iterations of the algorithm starting at $\{x\}$ gives an
  aspirant curve that extends outside of $O$ (or reports success!), and then any
  curve obtained by any number of iterations of the algorithm never has its
  endpoint in $S$.
  
  The {\strong{gain}} of a transient pair $(S,O)$ is the infimum of all
  $|x|-|y|$, for all $x \in S$, $y \in V_{1} \setminus O$ such that there is
  an aspirant curve starting at $x$ and ending at $y$.

  A transient pair is {\strong{positive}} if it has positive gain.
\end{dfn}

\begin{theorem}
  \label{typical sings have transient neighborhoods for easy manifolds}For any point $x$ of type
  NC, $A_{2}$ or $A_{3}$ 
  there is a positive transient pair $(S,O)$, with $x \in S$.
\end{theorem}

It follows from this theorem that there is a linking curve starting at any point.

Define:
\[ R_{0} = \sup \left\{ R: \forall x \in B_{R} ,  \begin{array}{l}
     \text{the algorithm starting at $x$ reports }\\
     \text{sucess! after a finite amount of iterations}
   \end{array} \right\} \]

We will assume that $R_{0}$ is finite and derive a contradiction, thus showing the existence of linking curves for all points in $\mathcal{A}_{2}$.
Take a covering of $\overline{B_{R_{0}}}$ by a finite number of neighborhoods $\{S_{i} \}_{i=1}^{N}$, where $(S_{i} ,O_{i} )$ are transient pairs.
Then $B_{R_{0} + \varepsilon}$ is also covered by $\cup S_{i}$ for some $\varepsilon >0$.
Let $\varepsilon_{0}$ be the minimum of $\varepsilon$, and all the gains of the $N$ pairs.

Take a point $x \in B_{R_{0} + \varepsilon_{0}}$ and assume $x \in S_{1}$.
Iterate the algorithm until it reports success! or builds an aspirant curve $\alpha$
with endpoint $y$ outside of $O_{1}$.

Thanks to the way we have chosen $\varepsilon_0$, we can assume $|y|<R_{0}$, and by hypothesis there is a linking curve that joins $y$ to some point $z$.
Append that linking curve to $\alpha$ to achieve an aspirant curve starting at $x$ and ending at $z$. For this aspirant curve to become a linking curve, it remains to reply to all the loose ACDCs in $\alpha$. 
Each of them, except possibly its endpoint, is contained in $V^1_0\setminus\SAtwo$.
If, after replying to one of them, we hit an $A_2$ point $y_{0}$, then $y_{0} \in B_{R}$, and thus we can append a linking curve that joins $y_{0}$ to some $z_{0} \in \mathcal{N C} \cap B_{|y_{0} |}$. Then we can continue to reply to the remaining loose ACDCs, and the process
finishes in a finite number of steps. This is the desired contradiction. It only remains to prove theorem \ref{typical sings have transient
neighborhoods for easy manifolds}.

\subsection{Existence of positive transient pairs in easy manifolds}
Let $x \in V_{1}$ be a point and $O$ be a cubical neighborhood
of adapted coordinates around it. $S$ will be a ``small enough'' subset of
$O$:
\begin{description}
  \item[$NC$] The algorithm reports success! in one step for any non-conjugate point, so any $S
  \subset O$, such that $O$ has no conjugate points, satisfies the claim.
  The gain is the infimum of the empty set, $+\infty$, so the pair is positive.
  
  \item[$A_{2}$] The CDC $\alpha_{0}$ starting at $x_{0}$ that reaches
  $\partial O$ has a length $\varepsilon >0$. For $x$ in a sufficiently small
  neighborhood $S$ of $x_{0}$, there is a GACDC $\alpha$ that reaches $y \in
  \partial O$ and has length at least $\varepsilon /2$ (for any finite set
  $F$).
  
  If there is an aspirant curve that starts with $\alpha$, and later has a
  retort of $\alpha$, starting at a point $z$, then $|z|<|y|$, because the
  restriction of the curve from $y$ to $z$ is a linking curve.
  
  Further, $\alpha_{0}$ is unbeatable, so that any non-trivial retort of this
  short curve will increase the radius at most $|x|-|y|- \delta$ for some
  $\delta >0$. The inequality still holds with $\delta /2$ if instead of
  $\alpha_{0}$ we have a GACDC starting at some $x$ in a small enough
  neighborhood $V$ of $x_{0}$.
  
  So if we take $S$ as the intersection of $V$ and a ball of radius $\delta
  /2$, then $( S,O )$ is transient, and the gain is at least $\delta/2$.
  
  \item[$A_{3}$] We recall that the set of singular points $\mathcal{C}$ near
  an $A_{3}$ point is an hypersurface, and the stratum of $A_{3}$ points is a
  smooth curve. An ACDC starting at any $A_{2}$ point will flow either into
  the stratum of $A_{3}$ points transversally (within $\mathcal{C}$), or into
  the boundary of $O$.
  
  For points in a smaller neighborhood $V \subset O$, one of the following
  things happen:
  \begin{itemize}
    \item If an ACDC starting at $x \in V \cap \mathcal{A}_{2}$ flows into an
    $A_{3}$ point, then it can be replied in one step, and the algorithm
    stops. The algorithm also stops if $x \in \mathcal{A}_{3}$.
    
    \item If the ACDC starting at $x \in V \cap \mathcal{A}_{2}$ flows into $y
    \in \partial O$, the argument is the same as that for an $A_{2}$ point.
  \end{itemize}
\end{description}
This concludes the proof of claim \ref{claim: hypothesis of the synthesis
theorem}, and thus we can apply proposition \ref{synthesis: statement} to
build the synthesis manifold, for easy manifolds.

We have chosen to defer the proof for the existence of linking curves for generic manifolds to section \ref{section: proof for a 3-manifold with a generic metric}. The next section does not require the easy hypothesis, so the reader is presented with a full argument that works for some manifolds for which the Ambrose conjecture was yet unknown.

\subsection{Proof that $\pi_{1}$ and $\pi_{2}$ are covering maps}\label{subsection: from local homeo to covering}

We still have to prove that the synthesis manifold $M$ given by theorem
\ref{synthesis: statement} is a covering space of $M_{1}$ and $M_{2}$. We
start with a general lemma:

\begin{lem}
  \label{norm dominated by exp}Let $\e : \T \rightarrow M$ be the exponential
  map from a point $p$ in a Riemannian manifold $M$. Then for any absolutely continuous
  path $x:[0,t_0] \rightarrow \T$, the total variation of $t \rightarrow |x(t)|$
  is no longer than the length of $t \rightarrow \e (x(t))$. In particular:
  \[ |x(t_0)|-|x(0)|< \tmop{length} ( \e \circ x) \]
\end{lem}

\begin{proof}
 For an absolutely continuous path $x$:

  \[ \tmop{length} ( \e \circ x ) = \int |( \e \circ x )' |= \int |d \e ( x' )| \]
  The speed vector $x' =ar+v$ is a linear combination of a multiple
  of the radial vector and a vector $v$ perpendicular to the radial direction.
  By the Gauss lemma, $|d \e ( x' )|  = \sqrt{a^{2} +|d \e (v)|^{2}} \geq |a|$.
  On the other hand, $v$ is tangent to the spheres of constant radius, so:
  \[ V_0^{t_{0}} (|x|)= \int \left|\frac{d}{dt} |x|\right|= \int |a| \leq \tmop{length} ( \e \circ x) \]

\end{proof}

Define $d:M \rightarrow \mathbb{R}$ by:
\[ d(q)= \inf_{x \in e^{-1} (q)} \{|x|\} \]
If we could prove that $e$ is the exponential map of the Riemannian manifold
$M$ at the point $p=e ( 0 )$, it would follow that $d$ is the distance to $p$,
and the following proposition would be trivial.

\begin{prop}
  \label{d is distance-decreasing}$d$ is distance-decreasing. In other words:
  \[ | d(q_{2} )-d(q_{1} ) | \leq d_{M} (q_{1} ,q_{2} ) \]
\end{prop}

\begin{proof}
  We can assume that $q_{1}$ and $q_{2}$ both lie in the same basic open set
  $[O]$. Otherwise, take a smooth path joining $q_{1}$ and $q_{2}$ of length
  at most $d_{M} ( q_{1} ,q_{2} ) + \varepsilon$ and place enough intermediate
  points $q_{i}$. If we prove that $| d ( q_{i} ) -d ( q_{i+1} ) | <d_{M} (
  q_{i} ,q_{i+1} )$, it follows that $| d(q_{2} )-d(q_{1} ) | < \sum d_{M} (
  q_{i} ,q_{i+1} )$, a number that we can assume is less than $d_{M} ( q_{1}
  ,q_{2} ) +2 \varepsilon$. Then we repeat the argument for a sequence of
  $\varepsilon \rightarrow 0$.
  
  Fix a smooth generic path $\alpha :[0,L] \rightarrow [ O ]$ of constant unit
  speed connecting $q_{1}$ and $q_{2}$ in $[ O ] \subset M$, of total distance
  $L<d_{M} (q_{1} ,q_{2} )+ \varepsilon$, and let $x \in \mathcal{I}$ such
  that $x \in e^{-1} (q_{1} )$ and $|x|=d(q_{1} )$ (we can assume that $e^{-1}
  (q_{1} ) \cap B_{R}$ is finite for any $R>0$). We can assume also that $x
  \in O$.
  
  The image of $\alpha$ by $\pi_{1}$ is also generic, and we can assume it
  only intersects $e_{1} ( \mathcal{C} \cap B_{|x|} )$ transversally in a
  finite set of $A_{2}$ points, except possibly at the endpoints, which may be
  $A_{3}$ points. We claim that we can lift $\alpha$ to a curve $\beta :[0,L]
  \rightarrow V_{1}$ (not necessarily continuous) with $\beta (0)=x$.
  
  To begin with, we can lift $\alpha$ to a continuous curve $\beta$ in any
  subinterval ($t_{1} ,t_{2} ) \subset [0,L]$ such that $e_{1}^{-1} ( \pi_{1}
  \circ \alpha ((t_{1} ,t_{2} )) \cap B_{| x |} ) \subset \NC$ (which also
  implies $e^{-1} \alpha ((t_{1} ,t_{2} )) \subset \NC$). In each such
  subinterval, we can apply lemma \ref{norm dominated by exp}, and learn that
  $| \beta (t_{2} )|-| \beta (t_{1} )|\leq t_{2} -t_{1}$. At an $A_{2}$ point
  $\beta (t_{0}^{-}$), we can make a discrete jump to a point $\beta (t_{0}^{+} ) \in
  \NC$ that is linked to $\beta (t_{0}^{-})$) and such that $| \beta(t_{0}^{+} )|<| \beta (t_{0}^{-} )|$. 
  Thus, finally, we obtain a point
  $\beta (L) \in e^{-1} ( \alpha ( L ) ) \cap \mathcal{I}=e^{-1} ( q_{2} )
  \cap \mathcal{I}$ such that $| \beta (L)|\leq | \beta (0)|+L< | x | +d_{M} (
  q_{1} ,q_{2} ) + \varepsilon$. This implies $d(q_{2} )<| \beta (L)|<d(q_{1}
  )+d_{M} ( q_{1} ,q_{2} ) + \varepsilon$. As $\varepsilon$ is arbitrary, the
  proof is completed.
\end{proof}

It follows from the above result that $M$ is complete: let $q_{n}$ be a Cauchy
sequence in $M$. Then there is $R>0$ such that $d ( q_{n} ,q_{1} ) <R$. Thanks
to the above result, we can find $x_{n} \in e^{-1} ( q_{n} ) \cap B_{| q_{1} |
+R}$. As $x_{n}$ is bounded, it has a subsequence that converges to some
$x_{0}$, and then $q_{n} \rightarrow e ( x_{0} )$.

\subsection{Proof of \ref{main properties of linking curves 2}}
\label{proof that linking curves produce linked points}

\diff{
  Let $x$ and $y$ be two points joined by a linking curve.
  We already know that $e_1(x)=e_1(y)$ and $e_2(x)=e_2(y)$, and we need to find neighborhoods $U^x$ and $V^y$ as in the definition of linked.
  We take $U^x$ and $V^y$ to be disjoint neighborhoods of adapted coordinates for $x$ and $y$
  
  Assume $e_1(z)=e_1(w)=q$ for $z\in U^x\cap V_1, w\in V^y\cap V_1$.
  We take a generic path $\rho$ joining $z$ to $x$, then append the linking curve between $x$ and $y$ and then append a ``lift'' of $\alpha=e_1\circ\rho$ as in the previous section, including a linking curve between $\beta (t_{0}^{+} )$ and $\beta (t_{0}^{-})$ whenever there is a jump.
  
  This closes up a curve that satisfies all the properties of a linking curve except for the fact that the first segment is not a descent. However, the proof that $e_1(x)=e_1(y)$ and $e_2(x)=e_2(y)$ still applies.
}

\subsection{Proof for a $3$-manifold with a generic metric}
\label{section: proof for a 3-manifold with a generic metric}

Next, we {\emph{assume that the metric of $M_{1}$ is in
$\mathcal{G}_{M_{1}}$}}, and its {\emph{dimension is $3$}}.

This time, there are points of $\Tone$ with singularities for $ e_{1}$ of types $A_{2}$, $A_{3}$, $A_{4}$, $D_{4}^{+}$ and $D_{4}^{-}$, with the $A_{3}$ points further divided into $A_3(I)$ and $A_3(II)$ points.

Define $\mathcal{I} =( \NC \cup \mathcal{A}_{3}(I)  ) \cap V_{1}$ and $\mathcal{J}=(\mathcal{A}_{2} \cup \mathcal{A}_{3}(II)  \cup \mathcal{A}_{4}   \cup \mathcal{D}_{4}^{\pm} ) \cap V_{1}$. We need to prove theorem \ref{claim: hypothesis of the synthesis theorem}, which in turn reduces to proving the following two results:

\begin{lem} \label{existence of GACDC}
  Let $M$ be a manifold with a Riemannian metric in $\mathcal{G}_{M_{1}}$.

  For any $R>0$ there is $L>0$ such that any GACDC starting at $x \in \mathcal{J} \cap B_{R}$ has length at most $L$, and can be extended until it reaches an $A_{3}$ point ($F$ can be any finite set).
\end{lem}

\begin{theorem}
  \label{typical sings have transient neighborhoods}
  For any point $x$ of type NC, $A_{2}$, $A_{3}$, $A_{4}$, $D_{4}^{+}$ or $D_{4}^{-}$,
  there is a positive transient pair $(S,O)$, with $x \in S$.
\end{theorem}

The rest of the proof for easy manifolds work verbatim, so we devote the next sections to proving these two results.

\subsection{CDCs in adapted coordinates\label{section: CDCs in adapted coords}}

As we mentioned in section \ref{section: generic metric}, the radial vector
field, and the spheres of constant radius of $\T$, which have very simple
expressions in standard linear coordinates in $\T$, are distorted in canonical
coordinates. Thus, the distribution $D$ and the CDCs do not always have the
same expression in adapted coordinates. In this section, we study them qualitatively.
We will use the name $R: \T \rightarrow \mathbb{R}$ for the
radius function, and $r$ for the radial vector field, and we assume that our
conjugate point is a first conjugate point (it lies in $\partial V_{1}$).

\subsubsection{$A_{4}$ points}

In a neighborhood $O$ of an $A_{4}$ point, $\Tone$ can be stratifed as an
isolated $A_{4}$ point, inside a stratum of dimension $1$ of $A_{3}$ points, inside a smooth surface consisting otherwise on $A_{2}$ points.
The conjugate points are given by $4x_{1}^{3} +2x_{1} x_{2}
+x_{3} =0$, and the $A_{3}$ points are given by the additional equation $12x_{
1}^{2} +2x_{2} =0$. The kernel is generated by the vector
$\frac{\partial}{\partial x_{1}}$ at any conjugate point and we can
assume that $D$ is close to $\frac{\partial}{\partial x_{1}}$ in $O \cap
\mathcal{C}$.

We do not know precisely where the radial vector is, but the distribution $D$
is a smooth line distribution and its integral curves are smooth. Thus, the
$A_{4}$ point belongs to exactly one integral curve of $D$.

\begin{figure}[ht]
 \centering
 \includegraphics[width=0.8\textwidth]{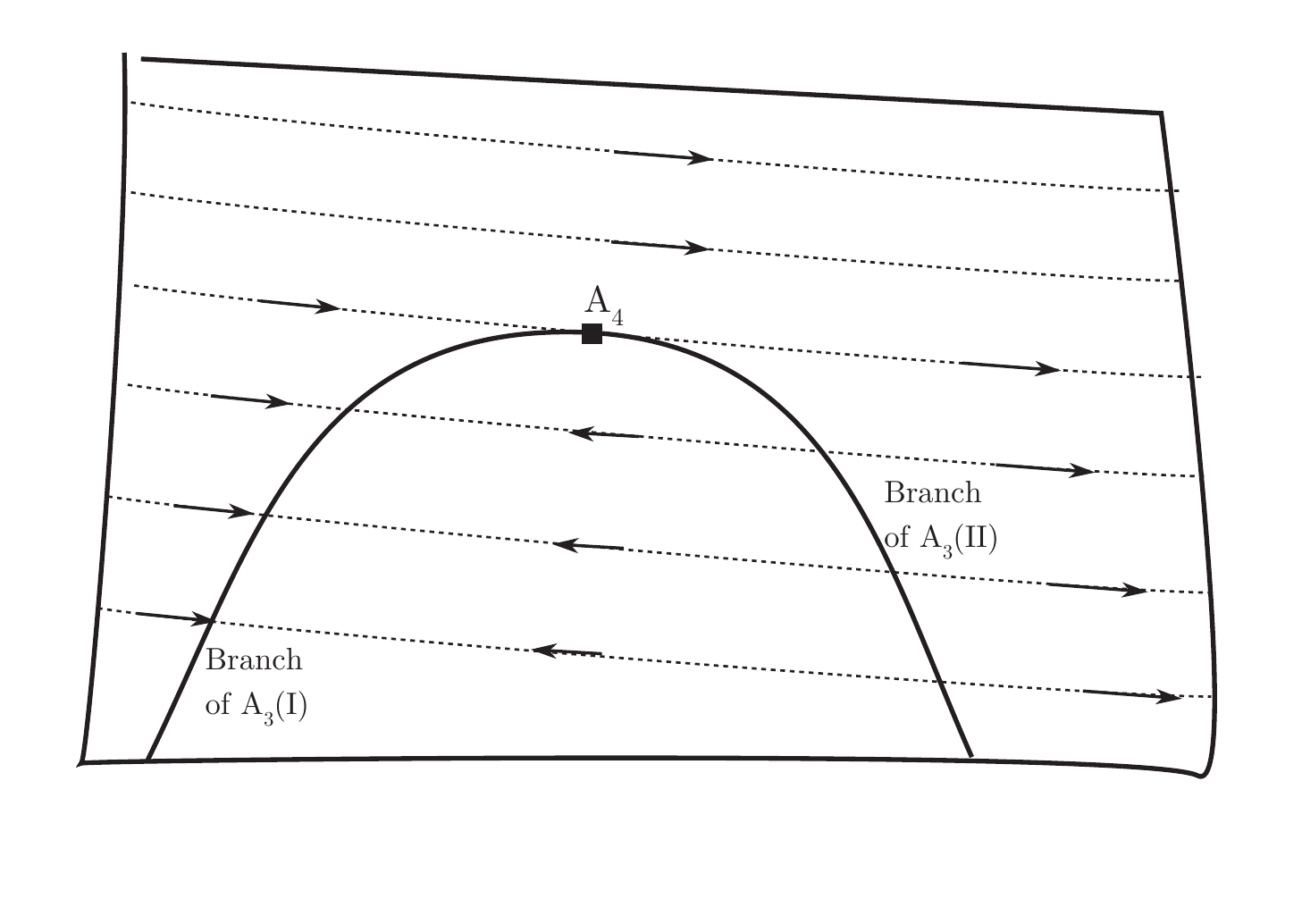}
 \caption{The
distribution $D$ and the CDCs at the conjugate points near an $A_{4}$
point.}
 \label{figure: near an A4 point}
\end{figure}

As we saw, $A_3(I)$ (resp $A_3(II)$) points have neighborhoods without $A_3(II)$
(resp $A_3(I)$) points. The $A_{4}$ point splits $\mathcal{A}_{3}$ into two
branches, and it can be shown easily that they must be of different types.
Composing with the coordinate change ($x_{1} ,x_{2} ,x_{3} ) \rightarrow
(-x_{1} ,x_{2} ,x_{3}$) if necessary, we can assume that the CDCs travel in
the directions shown in figure \ref{figure: near an A4 point}.

\subsubsection{$D_{4}^{-}$ points}

In a neighborhood $O$ of adapted coordinates near a $D_{4}^{-}$ point,
$\mathcal{C}$ is a cone given by the equations $0=-x_{1}^{2} -x_{2}^{2}
+x_{3}^{2}$. The kernel of $d  e_{1}$ at the origin is the plane $x_{3} =0$,
which intersects this cone only at ($0, \ldots ,0$). Three generatrices of the
cone consist of $A_{3}$ points (they are given by the equations $x_{2} =0$,
$x_{1} -x_{3} =0$ and $2x_{1} +x_{3}=0$, plus the equation of the cone), and the
rest of the points are $A_{2}$.

The radial vector field ($r_{1} ,r_{2} ,r_{3}$) at the origin must lie within
the solid cone $-r_{1}^{2} -r_{2}^{2} +r_{3}^{2} >0$, because the number of
conjugate points (counting multiplicities) in a radial line through a point
close to ($0,0,0$), must be $2$. In particular, $|r_{3} |>0$. Composing with
the coordinate change ($x_{1} ,x_{2} ,x_{3} ) \rightarrow (-x_{1} ,-x_{2}
,-x_{3}$) to the left and ($x_{1} ,x_{2} ,x_{3} ) \rightarrow (x_{1} ,x_{2}
,-x_{3}$) to the right, if necessary, we can assume that $r_{3} >0$.

The kernel at the origin is contained in the tangent to the hypersurface
$T_{0} = \{ R ( y ) =R ( 0 ) \}$, and the radius always decreases along a CDC.
Thus a CDC starting at a first conjugate point moves away from the origin and
may either hit an $A_{3}$ point, or leave the neighborhood. Thus these points
are not sinks of CDCs starting at points in $V_{1}$.

We now claim that there are three CDCs that start at any $D_{4}^{-}$ point and
flow out of $O$, and three CDCs that flow into any $D_{4}^{-}$ point, but the
latter ones are contained in the set of second conjugate points.

Recall that the $D_{4}^{-}$ point is the origin. We write the radial vector as
its value at the origin plus a first order perturbation:
\[ r=r^{0} +P(x) \]
with $|P(x)|<C|x|$ for some constant $C$.

We will consider angles and norms in $O$ measured in the adapted coordinates
in order to derive some qualitative behaviour, even though these quantities do
not have any intrinsic meaning.

We can measure the angle between a generatrix $G$ and $D$ by the determinant
of a vector in the direction of $G$, the radial vector $r$ and the kernel $k$
of $ e_{1}$: the determinant is zero if and only if the angle is zero. The
angle between $k$ and $r$ in this coordinate system is bounded from below, and
the norm of $r$ is bounded close to $1$. Thus if we use unit vectors that span
$G$ and $k$, we get a number $d(x)$ that is comparable to the sine of the
angle between $G$ and the plane spanned by $r$ and $k$. Thus $c|d(x)|$ is a
bound from below to $| \sin ( \alpha )|$, where $\alpha$ is the angle between
$G$ and $D$, for some $c>0$.

The kernel is spanned by $(-x_{1} +x_{3} ,x_{2} ,0)$ if $-x_{1} +x_{3} \neq
0$. The generatrix of $C$ at a point $(x_{1} ,x_{2} ,x_{3} ) \in C$ is the
line through $(x_{1} ,x_{2} ,x_{3} )$ and the origin. So $d$ is computed as
follows:
\[ d(x)= \frac{1}{x_{1}^{2} +x_{2}^{2} +x_{3}^{2}} \left|\begin{array}{ccc}
     x_{1} & x_{2} & x_{3}\\
     -x_{1} +x_{3} & x_{2} & 0\\
     r_{1} & r_{2} & r_{3}
   \end{array}\right| \]

Let us look for the roots of the lower order ($0$-th order) approximation:
\[ d_{0} (x)= \frac{1}{x_{1}^{2} +x_{2}^{2} +x_{3}^{2}}
   \left|\begin{array}{ccc}
     x_{1} & x_{2} & x_{3}\\
     -x_{1} +x_{3} & x_{2} & 0\\
     r^{0}_{1} & r^{0}_{2} & r^{0}_{3}
   \end{array}\right| \]

where $(r^{0}_{1} ,r^{0}_{3} ,r^{0}_{3} )$ are the coordinates of $r^{0}$.

The equation $\left|\begin{array}{ccc}
  x_{1} & x_{2} & x_{3}\\
  -x_{1} +x_{3} & x_{2} & 0\\
  r^{0}_{1} & r^{0}_{2} & r^{0}_{3}
\end{array}\right| =0$ is homogeneous in the variables $x_{1}$, $x_{2}$ and
$x_{3}$, so we can make the substitution $-x_{1} +x_{3} =1$ in order to study
its solutions. We only miss the direction $\lambda (1,0,1)$, where $D$ is not
aligned with $G$ because it consists of $A_{3}$ points.

Points in $C$ now satisfy $1+2x_{1} -x_{2}^{2} =0$, and $d_{0} (x)$ becomes
$p(x_{2} )=- \frac{1}{2}   \hspace{0.25em} ( a - 1 )  x_{2}^{3}  + 
\frac{1}{2}   \hspace{0.25em} b x_{2}^{2}  -  \frac{1}{2}   \hspace{0.25em} (
a + 3 )  x_{2}  +  \frac{1}{2}   \hspace{0.25em} b$, for $a=
\frac{r^{0}_{1}}{r^{0}_{3}}$ and $b= \frac{r^{0}_{2}}{r^{0}_{3}}$ (recall
$r^{0}_{3} >0$). The lines of $A_{3}$ points correspond to $x_{2} =
\frac{-1}{\sqrt{3}}$, $x_{2} = \frac{1}{\sqrt{3}}$, and the third line lies at
$\infty$. We prove that $p$ has three different roots, one in each interval:
($- \infty , \frac{-1}{\sqrt{3}}$), ($\frac{-1}{\sqrt{3}} ,
\frac{1}{\sqrt{3}}$), ($\frac{1}{\sqrt{3}} , \infty$). This follows
inmediately if we prove $\lim_{x_{2} \rightarrow - \infty} p(x_{2} )=-
\infty$, $p( \frac{-1}{\sqrt{3}} )>0$, $p( \frac{1}{\sqrt{3}} )<0$ and
$\lim_{x_{2} \rightarrow \infty} p(x_{2} )= \infty$ for all $a$ and $b$ such
that $a^{2} +b^{2} <1$. The first and last one are obvious, so let us look at
the second one. The minimum of
\[ p( \frac{-1}{\sqrt{3}} )= \frac{2 \sqrt{3}}{9} a+ \frac{2}{3} b+ \frac{4
   \sqrt{3}}{9} \]
in the circle $a^{2} +b^{2} \leqslant 1$ can be found using Lagrange
multipliers: it is exactly $0$ and is attained only at the boundary $a^{2}
+b^{2} =1$. The third inequality is analogous.

Thus, there is exactly one direction where $D$ is aligned with $G$ en each
sector between two lines of $A_{3}$ points. Take polar coordinates ($\phi ,r$)
in $C \cap V_{1}$. The roots of $d_{0}$ are transversal, and thus if
$\phi_{0}$ corresponds to a root of $_{} d_{0}$, then at a line in direction
$\phi$ close to $\phi_{0}$, the angle between $D$ and $G$ is at least $c( \phi
- \phi_{0} )+ \eta ( \phi ,r$), for $c>0$ and $\eta ( \phi ,r)=o(r$). If, at a
point in the line with angle $\phi$, and sufficiently small $r>0$, we move
upwards in the direction of $D$ (in the direction of increasing radius), we
hit the line of $A_{3}$ points, not the center. There are two CDCs starting at
each side of every $A_{3}$ point. A continuity argument shows that there must
be one CDC in each sector that starts at the origin (see figure \ref{figure:
near elliptic umbilic}).

\begin{figure} 
\includegraphics[width=0.8\textwidth]{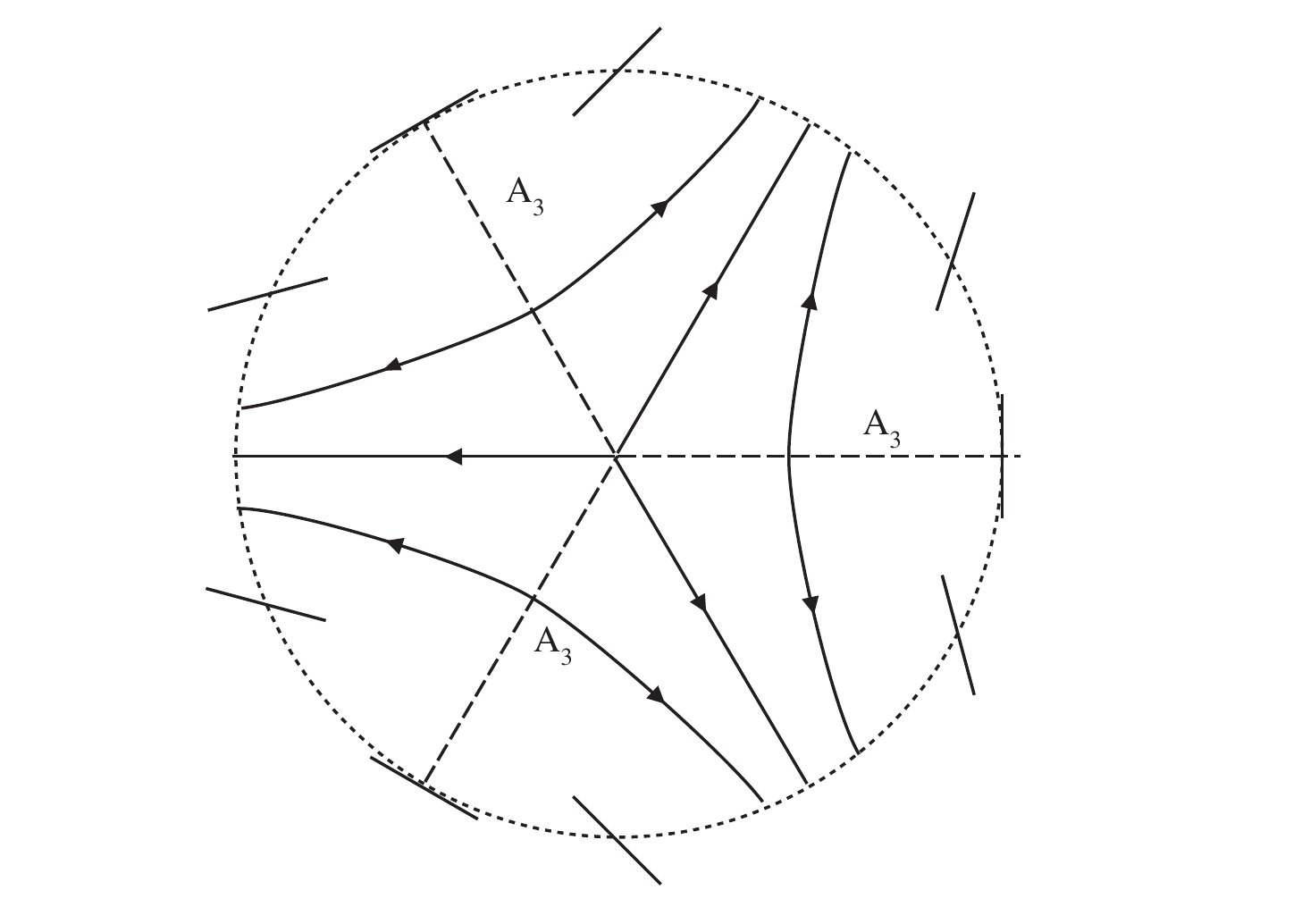}
 \caption{CDCs in the half-cone of first conjugate points near an
elliptic umbilic point, using the chart
$(x_{1} ,x_{2} ) \rightarrow (x_{1} ,x_{2} ,- \sqrt{x_{1}^{2} +x_{2}^{2}} )$,
for $r_{0} =(0,0,1)$.
The distribution $D$ makes half turn as we make a full turn around $x_{1}^{2}
+x_{2}^{2} =1$, spinning in
the opposite direction.}
 \label{figure: near elliptic umbilic}
\end{figure}

Reversing the argument, we see that there are three CDCs that descend into the
elliptic umbilic point, one in each sector, all contained in the the set of
second conjugate points.

\subsubsection{$D_{4}^{+}$ points}

The conjugate points in a neighborhood of adapted coordinates lie in the
cone $C$ given by $0=x_{1} x_{2} -x_{3}^{2} = \frac{1}{4} (x_{1} +x_{2} )^{2}
- \frac{1}{4} (x_{1} -x_{2} )^{2} -x_{3}^{2}$. This time, the kernel of $d \e$
at the origin intersects this cone in two lines through the origin, and the
inside of the cone $x_{1} x_{2} -x_{3}^{2} >0$ is split into two parts. There
is one line of $A_{3}$ points, the generatrix of the cone with parametric
equations: $t \rightarrow (t,t,t)$.

The radial vector at $r=(r_{1} ,r_{2} ,r_{3} )$ must lie within the solid cone
$r_{1} r_{2} -r_{3}^{2} >0$, for the same reason as above. Composing with the
coordinate change $(x_{1} ,x_{2} ,x_{3} ) \rightarrow (-x_{1} ,-x_{2} ,-x_{3}
)$ to the left and $(x_{1} ,x_{2} ,x_{3} ) \rightarrow (x_{1} ,x_{2} ,-x_{3}
)$ to the right, if necessary, we can assume that $r_{1} >0$ and $r_{2} >0$.

We write the radial vector as its value at the origin plus a first order
perturbation:
\[ r=r^{0} +P(x) \]
with $|P(x)|<C|x|$ for some constant $C$.

As before, the radius decreases along a CDC, but this time, a CDC starting at
a first conjugate point might end up at the origin. Let $F$ be the half cone
of first conjugate points (given by the equations $x_{1} x_{2} =x_{3}^2$ and
$\frac{1}{2} (x_{1} +x_{2} ) <0$). Let $F_{+}$ be the points of $F$ with
radius greater than the origin. Its tangent cone at the origin is $F \cap
\{x_{3} <0\}$ or $F \cap \{x_{3} >0\}$, depending on the sign of the third
coordinate of $r_{0}$.

As in the previous case, we can measure the angle between a generatrix $G$ and
$D$ by the determinant of a vector in the direction of $G$, the radial vector
$r$ and the kernel $k$ of $ e_{1}$. This time, the kernel is spanned by
$(-x_{3} ,x_{1} ,0)$ in the chart $x_{1} \neq 0$.
$$ d(x)= \frac{1}{x_{1}^{2} +x_{2}^{2} +x_{3}^{2}} \left|\begin{array}{ccc}
     x_{1} & x_{2} & x_{3}\\
     -x_{3} & x_{1} & 0\\
     r_{1} & r_{2} & r_{3}
   \end{array}\right|
$$

Again, we look for the roots of the lower order ($0$-th order) approximation,
which is equivalent to looking for the zeros of:
\[ \tilde{d} (x)= \left|\begin{array}{ccc}
     x_{1} & x_{2} & x_{3}\\
     -x_{3} & x_{1} & 0\\
     a & b & 1
   \end{array}\right| \]
in the cone $C$, for $a= \frac{r^{0}_{1}}{r^{0}_{3}}   \tmop{and}   b=
\frac{r^{0}_{2}}{r^{0}_{3}}$. We can make the substitution $x_{1} =-1$ in
order to study the zeros of the polynomial (we choose $x_{1} <0$ because we
are interested in the half cone of first conjugate points). This implies
$x_{2} =-x_{3}^{2}$ for a point in $C$, and we are left with $p(x_{3}
)=-x_{3}^{3} -bx_{3}^{2} +ax_{3} +1=0$. If $b^{2} +3a>0$, $p$ has two critical
points $\frac{-b \pm \sqrt{b^{2} +3a}}{3}$, otherwise it is monotone
decreasing. But even when $p$ has two critical points, the local maximum may
be negative, or the local minimum positive, with one real root.

The vector $r^{0}$ must satisfy $r_{3}^{0} \neq 0$ and $x_{1} x_{2} -x_{3}^{2}
>0$, or $ab>1$. There are two {\emph{chambers}} for $r^{0}$: $r_{3}^{0} >0$
and $r_{3}^{0} <0$. We will say that a $D_{4}^{+}$ point such that $r_{3}^{0}
>0$ (resp, $r_{3}^{0} <0$) is of type I (resp, type II).

If $r_{3}^{0} >0$ (or $a,b>0$), then $r^{0}$ and $L \cap F$ lie at opposite
sides of the kernel of $d  e_{1}$ at the origin. The cubic polynomial $p$ has
limit $\mp \infty$ at $\pm \infty$, \ and $p(0)>0$. The line of $A_{3}$ points
intersects $x_{1} =-1$ at $x_{3} =-1$. We check that $p(x_{3} =-1)=2-a-b$ is
always negative in the region $a>0$, $b>0$, $ab>1$. Thus there is exactly one
positive root, and two negative ones, one at each side of the line of $A_{3}$
points. This correponds to the top right picture in figure \ref{figure:
hyperbolic umbilic}, where the $x_{3}$ axis is vertical, and the CDCs descend,
because $r_{3}^{0} >0$.

\begin{figure}
\resizebox{\textwidth}{!}{
  \begin{tabular}{ll}
    \includegraphics{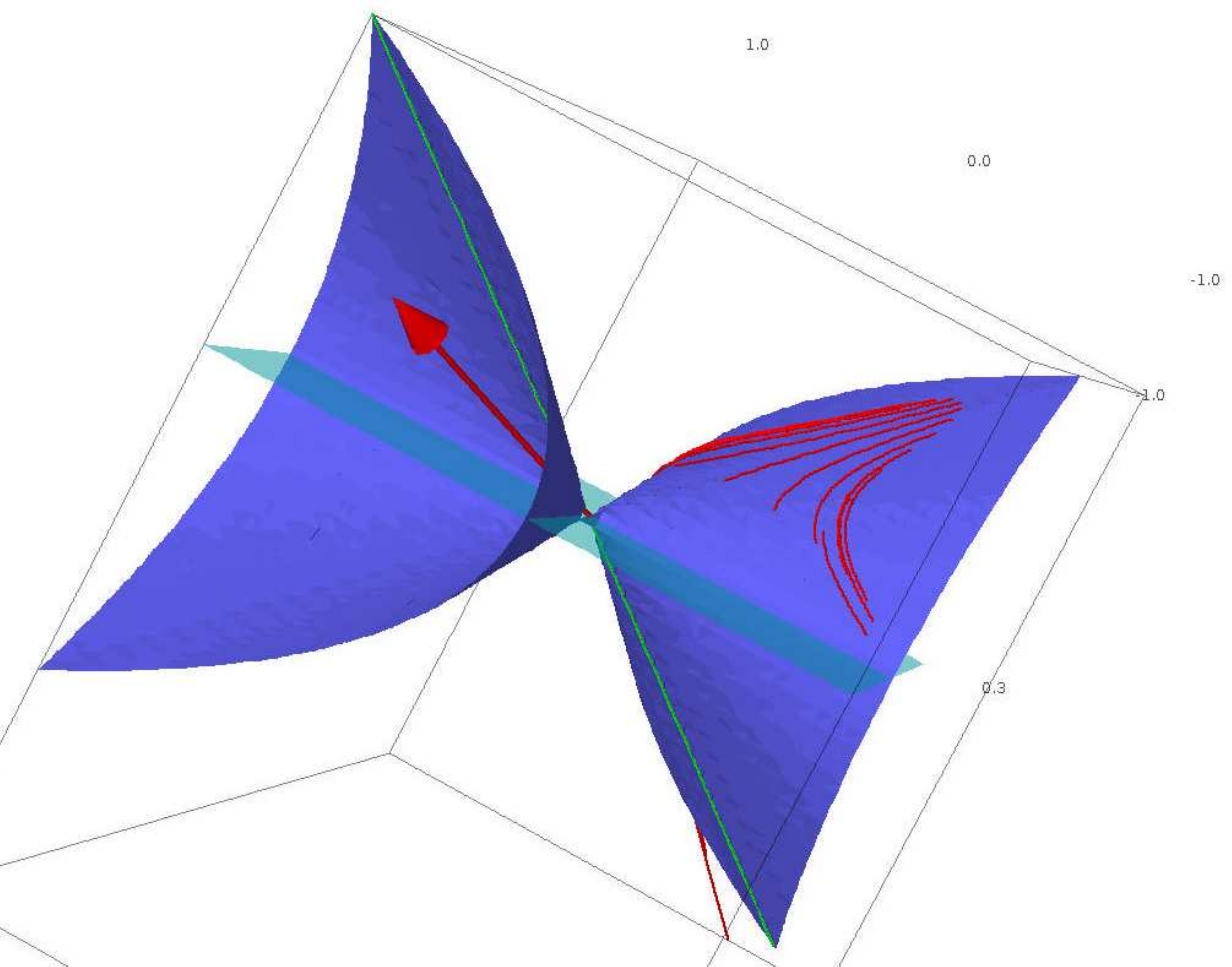} &
    \includegraphics{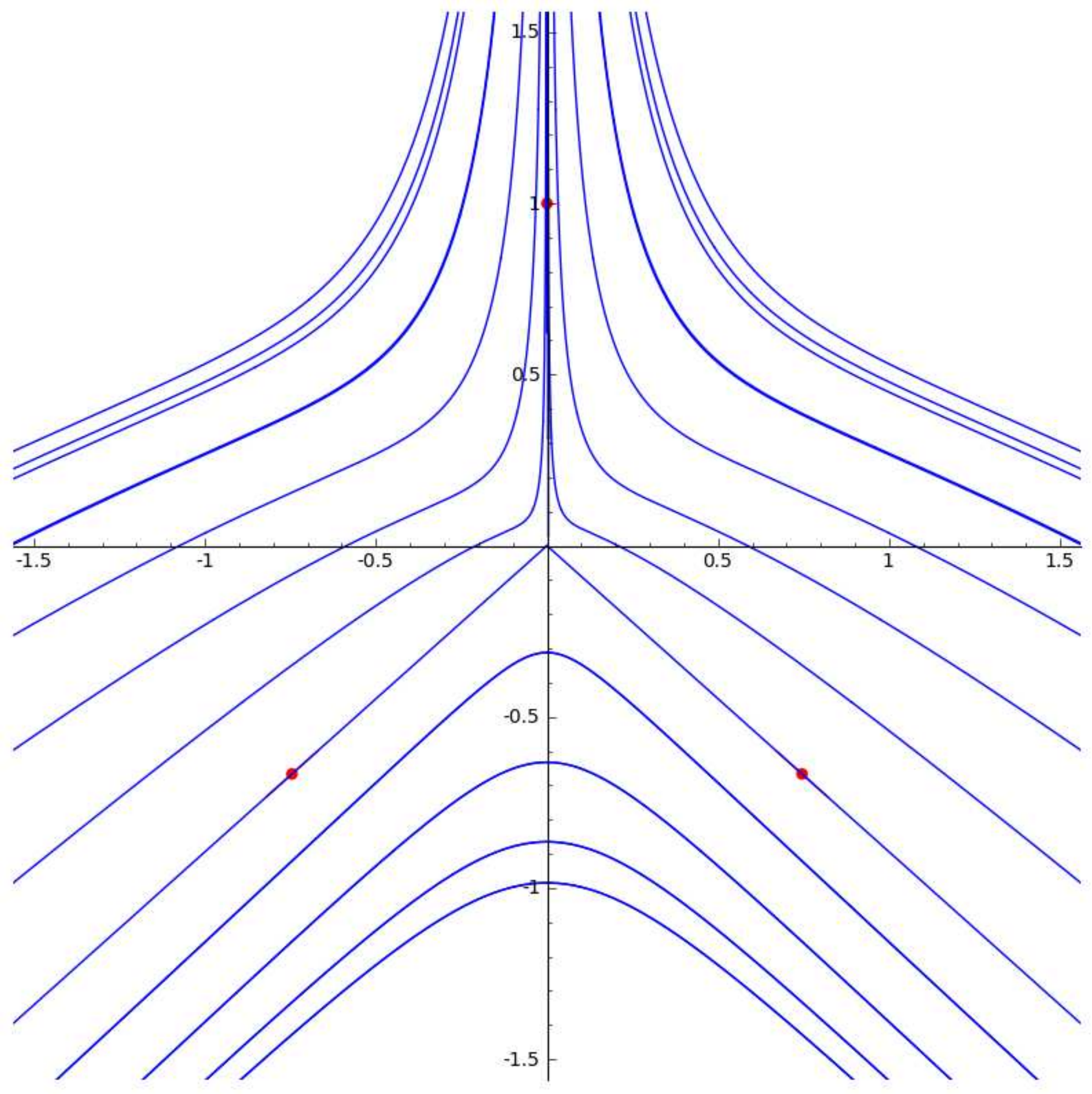}\\
    & \\
    \includegraphics{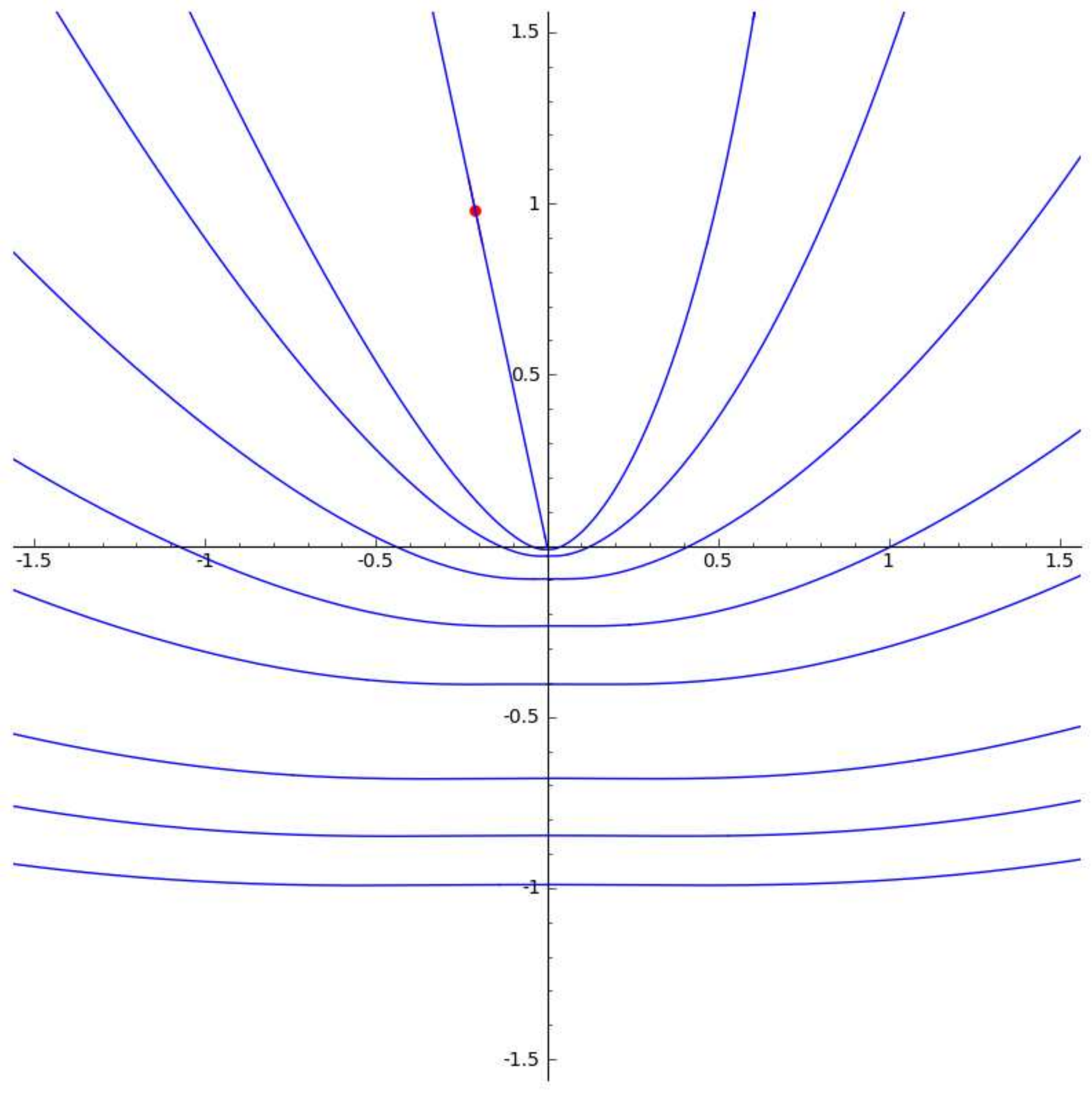} &
    \includegraphics{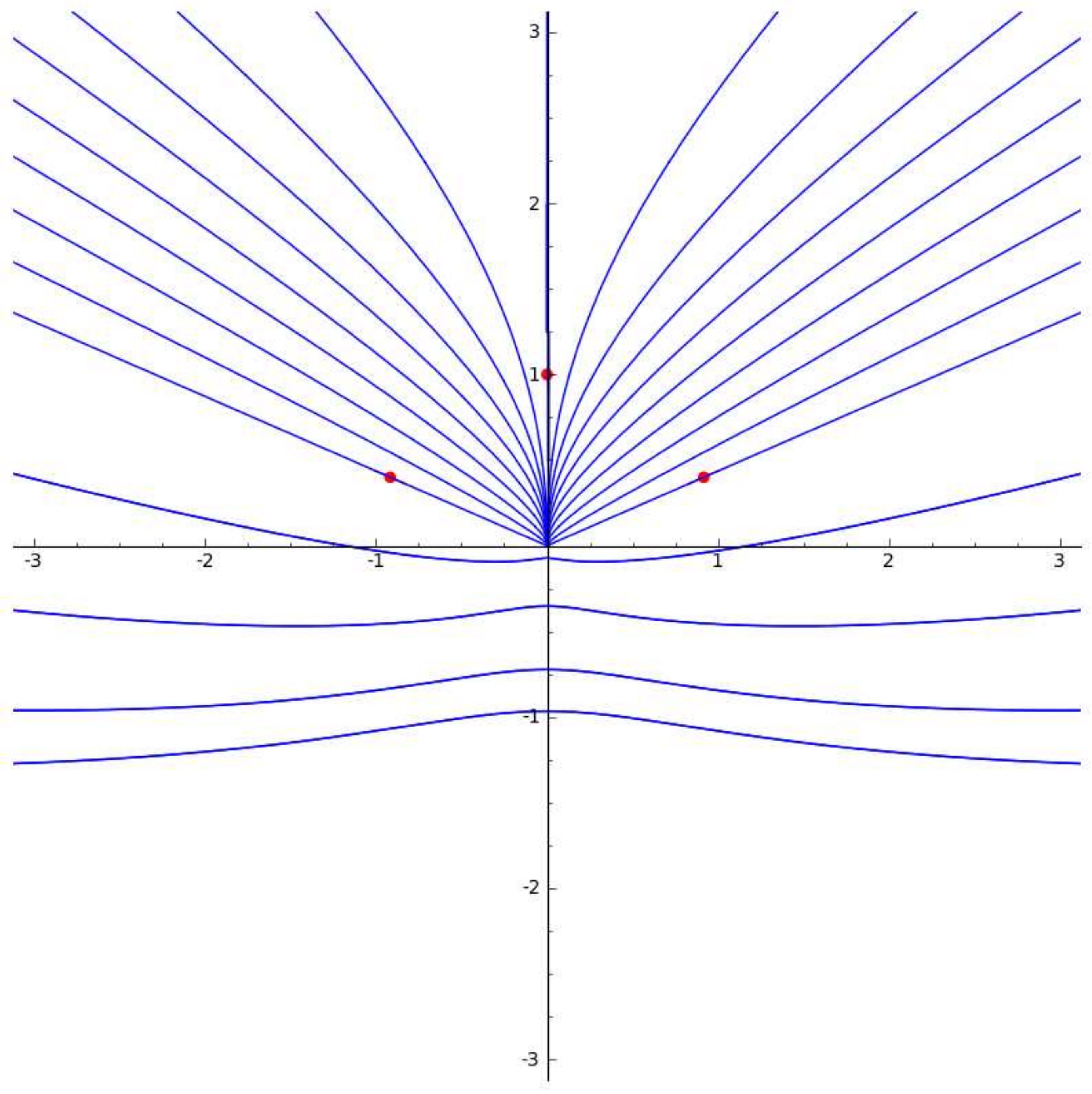}
  \end{tabular}}
  \caption{An hyperbolic umbilic point.}
  \label{figure: hyperbolic umbilic}
\begin{minipage}[b]{0.9\linewidth}  
\mbox{}\\[-.5\baselineskip] 
\paragraph{Explanation of figure \ref{figure: hyperbolic umbilic}}
 In the \strong{TopLeft} corner, the cone $C$ appears in blue, the line of
$A_{3}$ points in green,
  the radial vector at the origin in red, and the CDCs in red.

  The other pictures show the CDCs in the parametrization of the half cone of
  first conjugate points, obtained by projecting onto the plane spanned by
  $(1,-1,0)$ and $(0,0,1)$. The red dots indicate the directions where $D$ is
  parallel to the generatrix of the cone. The $A_{3}$ points lie in the half
  vertical line with $x_{3} <0$.
  \begin{description}
   \item[TopRight] $a>0$, $b>0$.
   \item[BottomLeft]$a<0$, $b<0$, $p$ has only one real root.
   \item[BottomRight]$a<0$, $b<0$, $p$ has three distinct real roots.
  \end{description}
\end{minipage}  
\end{figure}

The positive root gives a direction that is tangent to a CDC that enters into
the $D^{+}_{4}$ point, but moving to a nearby point we find CDCs that miss the
origin, and approach either of the two CDCs that depart from the origin,
corresponding to the negative roots of $p$.

However, if $r_{3}^{0} <0$ (type II), $p$ may have one or three roots.
We revert the direction of the CDC taking $p(z)=p(-x_3)$.
We note that $p(0)=1>0$, and $p' (z)>0$ for $z>0$, $a<0$ and $b<0$,
so there cannot be any positive root. A CDC starting at a point in $F$ flows
away from the stratum of $A_{3}$ points and out of the neighborhood (see the
bottom pictures at figure \ref{figure: hyperbolic umbilic}). It can be checked
by example that both possiblities do occur.

We want to remark that if there are three roots, the $D_{4}^{+}$ point is the
endpoint of the CDCs starting at any point in a set of positive
$\mathcal{\mathcal{H}}^{2}$ measure. Fortunately, all these points are second
conjugate points. This is the main reason why we build the synthesis as a
quotient of $V_{1}$ rather than all of $\T$ but more important: this is a hint
of the kind of complications we might find in arbitrary dimension, or for an
arbitrary metric, where we cannot list the normal forms and study each
possible singularity separately.

\begin{remark}
  In order to find out the number of real roots of $p$, for any value of $a$ and $b$, we used Sturm's method. 
  However, once we found out the results, we found alternative proofs and did not need to mention Sturm's method in the proof. 
  The precise boundary between the sets of $a,b$ such that $p$ has one or three real roots is found by Sturm method. 
  It is given by:
  
  $p_{3} =-9  \hspace{0.25em} a^{2}  b^{2}  - 36  \hspace{0.25em} a^{3}  - 36 
  \hspace{0.25em} b^{3}  - 162  \hspace{0.25em} a b + 243=0$
\end{remark}

\paragraph{\strong{A CDC starting at $\mathcal{J}$ singularities}}
We have shown that there is a CDC starting at an $A_{4}$ point and a $D_{4}^{+}$ of type II, while there are two CDCs starting at a $D_{4}^{+}$ point of type I, three CDCs starting at a $D_{4}^{-}$ point, and at least one starting at a
$D_{4}^{+}$ point of type II.
This will imply, ultimately, that points of those kinds are also linked to a point in $\mathcal{I}$.

\subsection{Proof of theorem \ref{typical sings have transient neighborhoods}}

Let $x \in V_{1}$ be a point and $O$ be a cubical neighborhood
of adapted coordinates around it. $S$ will be a ``small enough'' subset of
$O$:
\begin{description}
  \item[$A_{4}$] Near an $A_{4}$ point, $\mathcal{C}$ is a smooth hypersurface
  and $\mathcal{A}_{3}$ is a smooth curve sitting inside $\mathcal{C}$. The
  $A_{4}$ point is isolated and splits the curve $\mathcal{A}_{3}$ into two
  parts. One of them, which we call Branch I, consists of $A_{3} ( I )$
  points, and the other branch consists of $A_3(II)$ points. The conjugate
  distribution $D$ coincides with the kernel of $\e$ at the $A_{4}$ point, and
  is contained in the tangent to the manifold of $A_{3}$ points.
  
  As we saw before, a CDC that ends up in the $A_{4}$ point can be perturbed
  so that it either hits an $A_{3}$ point, or leaves the neighborhood.
  
  Let $H$ be the set of points such that the CDC starting at that point flows
  into the $A_{4}$ point. $H$ is a smooth curve, and splits $U$ into two
  parts. One of them, $U_{1}$, contains only $A_{2}$ points, while the other,
  $U_{2}$, contains all the $A_{3}$ points.
  
  \begin{figure}[ht]
    \includegraphics[width=0.8\textwidth]{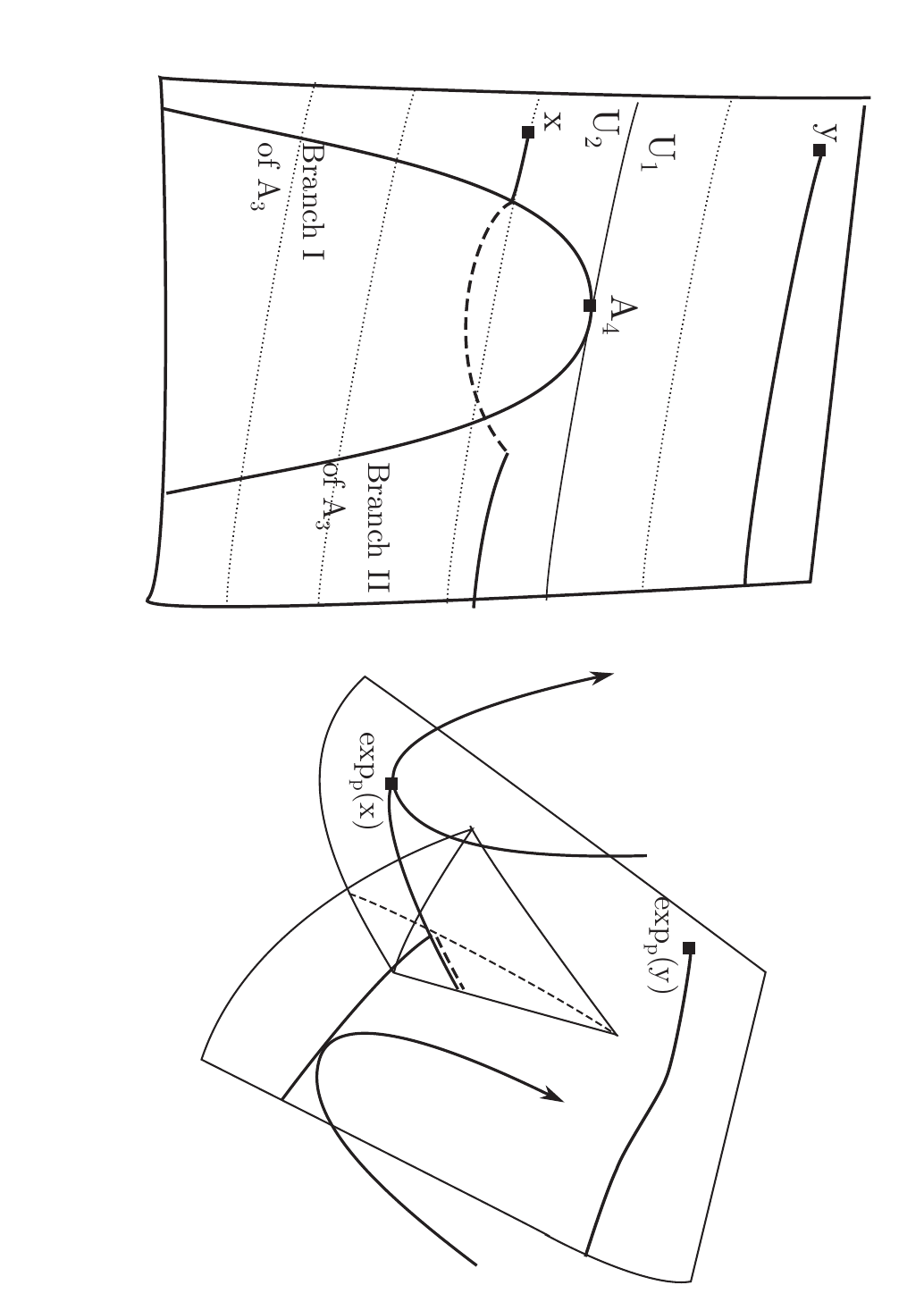}
    \caption{\label{figure: linking curves near A4 points}This picture shows a
    neighborhood of an $A_{4}$ point in $\T$, together with the linking curves
    that start at $x$ and $y$ (to the left) and the image of the whole sketch
    by $\e$ (to the right).}
  \end{figure}
  
  Look at figure \ref{figure: linking curves near A4 points}: a CDC starting
  at a point $y \in U_{1}$ flows into the boundary of $U$ without meeting any
  obstacle. A CDC $\alpha$ starting at a point $x \in U_{2}$, however, flows
  into the branch I of $\mathcal{A}_{3}$. We can start a retort $\beta$ at
  that point, but it will get interrupted when $\exp \circ \beta$ reaches the
  stratum of the queue d'aronde that is the image of two strata of $A_{2}$
  points meeting transversally. The retort cannot go any further because only
  the points ``above $\exp ( \mathcal{C} )$'' (the side of ) have a preimage,
  and points in the main sheet of $\exp ( \mathcal{C} )$ have only one
  preimage, that is $A_{2}$. When he hit the stratum of $A_{2}$ points, we
  follow a CDC to get a curve that leaves the neighborhood in a similar way as
  the curve starting at $y$ did.
  
  \item[$D_{4}$] Any CDC starting at any point in a neighborhood of a
  $D^{-}_{4}$, or $D_{4}^{+}$ of type I point leaves the neighborhood without
  meeting other singularities. A nearby GACDC will also do. We only have to
  worry about the one CDC that flows into the $D_{4}^{+}$ of type II, but we
  always take a nearby GACDC that avoids the center.
\end{description}

This concludes the proof of claim \ref{claim: hypothesis of the synthesis theorem} for generic metric, and thus we can apply proposition \ref{synthesis: statement} to build the synthesis manifold.
The results in section \ref{subsection: from local homeo to covering} can be applied without changes, and thus the theorem is proved for generic manifolds.

\addtocontents{toc}{\protect\mbox{}\protect\newline}
\addtocontents{toc}{\protect\pagebreak[3]}
\chapter{Further questions}\label{chapter: beyond}
In this chapter we collect open questions that are suggested by the previous work.
At some points, we comment on our attempts to prove these conjectures.
The author did spend quite some time working in some of them, specially the last one, but very little time in some of the others, and thinks that some of them are suitable for a student, specially the conjectures about magician's hats.

\section{Order $k$ conjugate cut points}
In the exponential map of a Finsler manifold (recall we only need to consider the exponential from the boundary), the image of the focal points of order $k$ can have Hausdorff dimension $n-k$.

However, theorem \ref{main theorem 3} and the classical result \ref{theorem: conjugate of order 1} suggest the following conjecture: 

\begin{conjecture}\label{conjecture: Hausdorff dimension of the points with a minimizing conjugate geodesic}
Let $M$ be a Finsler manifold with boundary of dimension $n$.
The set of points $p$ in $M$ such that there is a \emph{minimizing} geodesic of order $k$ from $\partial M$ to $p$ has Hausdorff dimension at most $n-k-1$.
\end{conjecture}

Inspection of the proofs of those two results actually suggests to divide the set $\mathcal{C}_k$ conjugate points of order $k$ into two subsets:

\begin{description}
 \item[$\mathcal{C}_k^1$] Points such that the kernel of the exponential is contained in the tangent cone to the set of conjugate points.
 \item[$\mathcal{C}_k^2$] The complementary set in $\mathcal{C}_k$.
\end{description}

We can venture a further conjecture that would imply the previous one with little effort. Recall the definition of CDCs in \ref{definition: conjugate flow}.

\begin{conjecture}
 If, at a conjugate point $x\in V$ of order $k$, the kernel of the exponential is \emph{not}
 contained in the tangent cone to the set of conjugate points, then there is a CDC starting at $x$.
\end{conjecture}

The reason why this conjecture is enough to prove \ref{conjecture: Hausdorff dimension of the points with a minimizing conjugate geodesic} is that the image under the exponential of a CDC is never a minimizing curve, as shown in lemma 2.2 of \cite{Hebda82}.

While thinking about this conjecture, we came out with a similar one, that does not follow from or implies the previous one in a direct way, but we believe is interesting in its own right:

\begin{conjecture}
 Let $\gamma:(-\varepsilon,0]\rightarrow\mathcal{C}$ be a $C^1$ curve of conjugate points.
 Then $\gamma$ can be extended to a $C^1$ curve of conjugate points defined in $(-\varepsilon,\varepsilon_2)$ for some $\varepsilon_2>0$.
\end{conjecture}

This conjecture states that the set of conjugate points, which can be a complicated subset of the tangent space (or $V$ in the general setting \ref{section: intro to exponential maps}), does not have edges or pointed tips. 
The conjecture holds for generic manifolds, as follows trivially from the first structure result for generic manifolds, \cite{Weinstein}.
The conjecture can probably be proved by approximating the metric with generic ones.

\section{HJBVP and balanced split loci}
The techniques in chapter \ref{Chapter: balanced} could in principle be applied to other first order PDEs, or systems of PDEs.
The idea of using the balanced property was actually inspired by the Rankine-Hugoniot conditions for the shock in solutions to equations of conservation laws.
We think that the main appeal of the balanced condition is that, unlike the Rankine-Hugoniot conditions, it does not assume any \emph{a priori} regularity on the shock.
So we think it is interesting to study whether the balanced property, and at least some of the structure results, can be carried over to other equations.

In particular, we believe our proofs of \ref{landa es Lipschitz} and \ref{rho is lipschitz} are more easily extensible to other settings than the previous ones in the literature.
This may simplify the task of proving that the singular locus for other PDEs have locally finite $n-1$ Hausdorff measure.

However, this may not be possible for conservation laws, as the Rankine-Hugoniot conditions are really very different to the balanced property, in that they only prescribe one tangent vector that must be contained in the tangent plane to the shock, while the balanced condition prescribes all the tangent directions.
Hamilton-Jacobi Cauchy problems are also different in nature, and the results presented here may not apply.
We would like to mention that L. C. Evans has recently made big improvements in the understanding of the HJ Cauchy problems for non-convex hamiltonian (see \cite{Evans nonconvex 1} and \cite{Evans nonconvex 2}).
It would be very interesting if non-convex hamiltonians were also better understood for BVPs.
HJBVP problems with a Hamiltonian dependent on $u$ seems a more feasible target.

Finally, Philippe Delanoë and others have suggested that it would be interesting to try this approach in sub-riemannian geometry.

\section{Poincaré conjecture}\label{section: proving poincare conjecture, or not}
The Ambrose conjecture is related to the Poincaré conjecture in several ways. We mention one link between them that haunted the author for some time:

The proof of the Ambrose conjecture for surfaces by James Hebda works in two dimensions because the cut locus is a tree. In \cite{Hebda}, he mentions that it works also in those manifolds whose cut locus is triangulable and \emph{descends simplicially to a point}. In a compact, simply connected manifold, the cut locus is homotopic to a point, but even if the cut locus is triangulable, it may not be possible to collapse simplices one by one until the whole cut locus becomes trivial. For example, the $3$-sphere admits a metric such that the cut locus with respect to a certain point is the \emph{house with two rooms} (see figure \ref{fig: house with two rooms}).

\begin{figure}[ht]
 \centering
 \includegraphics[width=.8\textwidth]{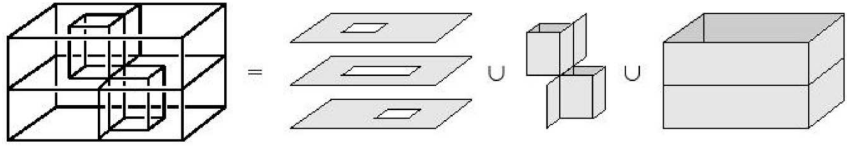}
 \caption{The ``house with two rooms''}
 \label{fig: house with two rooms}
\end{figure}

The first motivation for studying split loci was that they may help overcome this difficulty, if only we could find a split loci that does collapse simplicially to a point. 
But if that strategy worked for any manifold, it would also provide a proof of the Poincaré conjecture.
The reason is that in a compact simply-connected manifold with such a collapsible split locus from a point $p$ we can find a vector field with only one source $p$, and only one sink, using the deformation retract that collapses the split locus onto one point.
It is well known that this proves that the original manifold is homeomorphic to the sphere.
Thus the following conjecture is stronger than the Poincaré conjecture:

\begin{conjecture}\label{existence of collapsible split locus}
 Any compact simply-connected $3$-manifold admits a metric and a split locus that collapses simplicially to a point.
\end{conjecture}

Let us explain the motivation behind the conjecture:

Let $M$ be a compact simply-connected $3$-manifold.
Any manifold admits a metric whose geodesic flow in $T M$ is ergodic.

If the manifold did not have any pair of conjugate points, the exponential map from any point would be a covering map, but this is incompatible with the hypothesis. So we can assume that there is one point $x\in T M$ such that the geodesic $\gamma\subset T M$ starting at $x$ meets a point $y=\gamma(T)$ such that $y$ is conjugate to $x$ along $\gamma$. The continuity of $\lambda_1$ shows that the same happens for any point in a neighborhood $U\subset T M$ of $x$.

The geodesic flow being ergodic, any geodesic eventually enters $U$, and develops a conjugate point. This means that $\lambda_1$ is finite.

Any metric close to the ergodic one will also have finite $\lambda_1$. We can choose one such metric such that the exponential map is generic in the sense of Klok, Buchner or both.

The compact set $\left\{x:\lambda_1\left(\frac{x}{|x|}\right)\leq |x| \right\}\subset T M$ maps many-to-one to $M$, so that each point has several preimages. The genericity hypothesis allows to decompose $M$ into finitely many chambers. Every point in each chamber has the same number of preimages. That decomposition induces another one in $T M$, in which the chambers map diffeomorphically into chambers of $M$, but different chambers of $T M$ may map into the same chamber of $M$. The goal is then to select only one of the preimages of each chamber in $M$, in such a way that the union of all the chambers in $T M$ is star-shaped with respect to the origin. This is equivalent to selecting a split locus that is composed of images of conjugate points and geodesics that cut the radial geodesics so that each point in $M$ has only one preimage before the split locus is reached.

There is no guarantee that any of these split loci collapses to a point.
In fact, some of these selections of split loci might be quite similar to the cut locus.
Our idea was to use linking curves to make some of the required choices, in order to make a partial selection of chambers in a way such that the corresponding split locus is tree-like.

Faced with the multiplicity of chambers, we can pick up one generic $A_2$ first conjugate point $x$ and find a linking curve starting at that point.
This linking curve can also provide a linking curve for any point in the same $2$-cell to which $x$ belongs.
In fact, the linking curve will pass through several $2$-cells, and this could help us select  the right chambers, whose boundary would map to a $2$-dimensional complex full of tree-formed curves.
The main problem is that in three dimensions, a generic linking curve can intersect twice the same radial line, so this procedure does not help to select chambers.

Thus we conjecture that in three dimensions, the conjecture \ref{existence of collapsible split locus} is false, but in higher dimensions, we do not know.
Indeed, in higher dimensions, even if the conjecture is false, there is still hope that this argument can be useful in some way...

\section{Magician's hat}

We arrived at the following definition when we were working on \ref{subsection: from local homeo to covering}.
We wanted to be able to lift paths in $M_1$ to the synthesis manifold $M$, and a \emph{pre-compactness} result for the exponential map just seemed natural.

\begin{dfn}
 A \strong{magician's hat} with respect to $p\in M$, is an open set $U$ in a Riemannian manifold $M$ such that its preimage by the exponential from $p$ has an unbounded connected component.
\end{dfn}

\begin{conjecture} \label{conjecture on magician's hats} \tmdummy
  \begin{enumerate}
  \item Let $p$ be a point in a Riemannian manifold $M$. Any point $q\in M$ has a ``sufficiently small'' neighborhood that is not a magician's hat with respect to $p$.
  \item For any ``bound on curvature'' $K$, there is a ``diameter'' $d$ such that on any manifold $M$ with curvature bounded by $K$, any set of diameter less than or equal to $d$ is not a magician's hat with respect to any point. We intentionally leave open the question of what are the appropriate notions of ``bound on curvature''.
  \end{enumerate}
\end{conjecture}

The reason for choosing such a name is that a curve contained in the preimage of a magician's hat $U$ that goes to infinity corresponds to a family of geodesics with starting point $p$ and endpoint in $U$ that get longer and longer.
This, to the author, would be similar to a magician pulling a long handkerchief out of his hat.

\section{Proof for a 3-manifold with an arbitrary metric}\label{section: proof for a 3-manifold with arbitrary metric}

The Ambrose conjecture involves both topological and analytical challenges. 
We already mentioned in section \ref{section: proving poincare conjecture, or not} one topological difficulty.
The Poincaré conjecture was an open, and very hot conjecture, for many years.
Many reknowned topologists and geometers failed at finding a proof using an arsenal of algebraic topology, knot theory, hyperbolic geometry, and what not.
If a mathematician working in the Ambrose conjecture does not feel like giving a topological proof of the Poincaré conjecture in the way, some strategies are not very promising.
However, we thought that the idea of building a synthesis as in \ref{subsection: synthesis} would allow to cast away the topological difficulties, allowing us to prove Ambrose conjecture, but not Poincaré's.
Tree formed curves, now refurbished into linking curves, would help make the necessary identifications.

Using the cut locus for a synthesis would not work, because there is no canonical way to find linking curves.
An arbitrary split locus is no better, of course, if it does not have neither simpler topology, nor a canonical way to find linking curves within it.
The conjugate descending curves seemed like the only sensible choice.

But these curves are very tricky. 
The first obvious problem is that the singularities of the exponential map are complicated.
A first step towards dealing with that problem is to use our theorem \ref{main theorem 3}, building a synthesis of \emph{most} points of the manifold, and then extending the construction to a bona-fide synthesis by completion of the metric.

\subsection{Main idea}

In an arbitrary metric, the possible singularities of the exponential map no longer belong to a finite family of canonical forms. 
In order to prove the Ambrose conjecture, we will need to find a wider category of conjugate points that are unequivocal, and a wider category of conjugate points that are linked to the unequivocal points. 
It is also convenient to work with a {\emph{remainder}} of conjugate points about which we know very little, but such that the set of such points has sufficiently small Hausdorff dimension so that we can ignore them in our arguments.

\begin{dfn}
  \label{younger cousins}The {\emph{cousins}} of $x \in V_1$ are the
  preimages of its image by the exponential map.

  $C ( x ) =e^{-1} ( e(x) )$

  The {\emph{younger cousins}} of $x \in V_1$ are the cousins of smaller radius:

  $\tmop{YC} ( x ) =e^{-1} ( e(x) ) \cap B ( 0, | x | )$
\end{dfn}

We defer the definitions of terminal points of order 1 and the types of conjugate points of order $2$ for later sections.

\begin{dfn}
  We define some categories of points in $T_{p} M$ (recall $\NC$ are the non-conjugate points):
  \begin{itemizeminus}   
   \item $\mathcal{\mathcal{R}} = \mathcal{\NC} \cup \{ \text{terminal points of order $1$}  \}$
   \item $\mathcal{S} =\begin{array}[t]{ll}\left\{ \text{non terminal
points of order $1$} \right\} \\ \quad\cup \left\{ \text{conjugate points of
order $2$ and type I}
  \right\}\end{array}$
   \item $\mathcal{T} = \begin{array}[t]{ll}\left\{ \text{conjugate points of
order $2$ and type II}
  \right\} \\ \quad\cup \left\{ \text{conjugate points of order $3$} \right\}
  \end{array}$
   \item $\mathcal{I}=\mathcal{R}$
   \item $\mathcal{J} = \{ x \in \mathcal{S} : \tmop{YC} ( x ) \subset \mathcal{R}
  \cup \mathcal{S} \}$
   \item $\mathcal{K} = \mathcal{T}\cup \{ x \in \mathcal{S} : \tmop{YC} ( x ) 
   \nsubseteq \mathcal{R}\cup \mathcal{S} \}$
\end{itemizeminus}
\end{dfn}

We call the following the \strong{IJK conjecture}:

\begin{conjecture}\label{conjecture IJK}
  \label{points I, J, K}{\tmdummy}

  \begin{itemize}
    \item Points in $\mathcal{I}$ are unequivocal.

    \item Points in $\mathcal{J}$ are linked to a point in $\mathcal{I}$

    \item $e ( \mathcal{K} )$ has Haussdorf dimension at most $n-3$
  \end{itemize}
 
\end{conjecture}

\begin{remark}
  The last part follows directly from \ref{main theorem 3}.
\end{remark}

\subsection{Proof of the Ambrose conjecture modulo the IJK conjecture}

Using \ref{conjecture IJK}, we can build a synthesis $\mathcal{U}$ of $\mathcal{U}_{1}
=M_{1} \backslash e_{1} ( \mathcal{K} ) \subset M_{1}$ and $\mathcal{U}_{2}
=M_{2} \backslash  e_{2} ( \mathcal{K} ) \subset M_{2}$ using theorem
\ref{synthesis: statement}.

\begin{prop}
  The maps $\pi_{1}$ and $\pi_{2}$ in the synthesis are covering maps.
\end{prop}

\begin{proof}
  Proposition \ref{d is distance-decreasing} follows as in section
  \ref{subsection: from local homeo to covering}, once we notice that for any $R$:
  \begin{enumerate}
   \item $X_{R} =  \{x\in  \mathcal{A}_{2} \cap B ( 0,R ), YC(x)\subset \NC\cup \mathcal{A}_2 \}$
  is $( n-1 )$ -rectifiable, with finite $\mathcal{H}^{n-1}$ measure  
   \item $\mathcal{H}^{n-1} \left(  e_{1} \left( V_1 \setminus \left( \NC \cup \mathcal{A}_{2} \right) \right)  \right)=0$
  \end{enumerate}

  It follows that a generic path of finite length intersects $ e_{1} \left( T_{p} M \setminus \NC \right)$ only at a finite number of points in $X_{R}$
{\footnote{Project a set $S \subset \mathbb{R}^{n}$ such that
  $\mathcal{H}^{n-1} ( S ) < \infty$ onto the orthogonal hyperplane to the
  segment $[ q_{1} ,q_{2} ]$. Using the co-area formula, we see that almost
  sure, a line parallel to $[ q_{1} ,q_{2} ]$ intersects $S$ at a finite
  number of points.}}.
  However, $\mathcal{U}_{1}$ and $\mathcal{U}_{2}$ are not
  complete, so we will prove directly that $\pi_{1}$, for example, has the
  path lifting property.

  Let $p: [0,b] \rightarrow \mathcal{U}_{1}$ be a smooth, unit speed path,
  and $q_{0} \in \pi_{1}^{-1} (p(0))$. As $\pi_{1}$ is a local homeomorphism,
  a path can always be lifted for a short time. Let $q: [ 0,t^{\ast} )
  \rightarrow \mathcal{U}$ be a lift of $p$ starting at $q_{0}$ for a maximal time
  $t^{\ast}$. The Lipschitz property of $d$ shows that $d ( q ( t ) ) <d (
  q_{0} ) +t$ for any $t<t^{\ast}$.

  First we prove that $q$ can be extended to the compact interval $[
  0,t^{\ast} ]$: For all $t<t^{\ast}$, $\exists x ( t )$ such that $e ( x ( t
  ) ) =q ( t )$ and $| x ( t ) | \leqslant | x ( 0 ) | +t$. There is a
  sequence $t_{n} \nearrow t^{\ast}$ such that $x (t_{n} )$ converges to some
  $x^{\ast} \in T_{p} M$ such that $| x^{\ast} | \leqslant | x ( 0 ) |
  +t^{\ast}$. Also, $ e_{1} ( x^{\ast} ) =p ( t^{\ast} )$, so $x^{\ast} \in
  \mathcal{I} \cup \mathcal{J}$ because $p ( t^{\ast} ) \in \mathcal{U}_{1}$.
  This way we extend $q$ by setting $q ( t^{\ast} ) =e ( x^{\ast} )$, and it
  holds that $q ( t ) \rightarrow q ( t^{\ast} )$.

  Finally, assume $t^{\ast} <b$. As we have mentioned already, we can extend
  $q$ to a path defined up to time $t^{\ast} + \varepsilon$. This completes
  the proof that $\pi_{1}$ is a covering map.

\end{proof}

The subsets $\mathcal{U}_{1}$ of $M_{1}$ and $\mathcal{U}_{2}$ of $M_{2}$ are
big enough so that the construction extends to provide a synthesis of $M_{1}$
and $M_{2}$:

\begin{prop}
  \label{extending riemannian coverings}Let $\pi :Y \rightarrow X$ be a
  Riemannian covering of Riemannian manifolds. Assume $X \subset \bar{X}$
  where $\bar{X}$ is a complete Riemannian manifold and $\mathcal{H}^{n-2} (
  \bar{X} \setminus X ) =0$.

  Then there is a unique Riemannian covering from the completion $\bar{Y}$ of
  $Y$ into $\bar{X}$ that restricts to $\pi$. In particular, $\bar{Y}$ is a
  Riemannian manifold.
\end{prop}

The proposition follows from the following general topology lemma:

\begin{lem}\label{lem: extension of a covering map}
  Let $\pi :Y \rightarrow X$ be a covering map of locally simply-connected spaces. 
  Assume $X \subset \bar{X}$ where $\bar{X}$ is a locally simply-connected space such that the   intersection of any non-empty simply-connected open set $V \subset \bar{X}$ with $X$ is non-empty and simply-connected.

  Then there is a locally simply-connected space $\bar{Y}$ and maps $i:Y
  \rightarrow \bar{Y}$, $\bar{\pi} : \bar{Y} \rightarrow \bar{X}$ such that:
  \begin{itemize}
    \item $\bar{\pi}$ is a covering map

    \item $\bar{\pi} \circ i= \pi$

    \item For any simply-connected non-empty set $O \subset \bar{Y}$, $Y \cap
    i^{-1} ( O )$ is non-empty and simply-connected
  \end{itemize}
  Moreover, the space $\bar{Y}$ that we construct has the following universal property:

  Let $\tilde{Y} \nocomma , \tilde{\pi} , \tilde{i}$ satisfy the above
  properties. Then $\tilde{Y}$ is a covering space of $\bar{Y}$, with a
  covering map $\rho : \tilde{Y} \rightarrow \bar{Y}$ such that $\bar{\pi}
  \circ \rho = \tilde{\pi}$ and $\rho \circ \tilde{i} =i$.

  Thus $\bar{Y}$ is characterized by the above properties up to isomorphism.
\end{lem}

\begin{proof}
  The space is built as equivalence classes of pairs $( x \nocomma ,V,A )$,
  where $x \in \bar{X}$, $V \subset \bar{X}$ is a simply-connected
  neighborhood of $x$, $A \subset Y$ and $\pi |_{A} \nobracket :A \rightarrow
  U=V \cap X$ is a homeomorphism.

  Two pairs $( x \nocomma ,V,A )$ and $( x' \nocomma ,V' ,A' )$ are equivalent
  iff $x=x'$ and there is an open simply connected set $V''\subset V\cap V'$ such that $A \cap A' \cap \pi^{-1}(V'') \neq \varnothing$.

  The basis open sets of $Y$ are the sets $O_{V,A} = \{ [ ( x,V,A ) ] \nocomma
  ,x \in V \}$, for any $V \subset \bar{X}$ open simply-connected and $A
  \subset Y$ open simply-connected such that $\pi ( A ) =V \cap X$. It follows
  that $A$ is one of the connected components of $\pi^{-1} ( V \cap X )$.

  The map $i$ is defined by $i ( y ) = [ ( \pi ( y ) , \pi ( A ) ,A ) ]$,
  where $A \subset Y$ is any simply-connected open neighborhood of $y$.

  The map $\bar{\pi}$ is given by $\bar{\pi} ( [ ( x,V,A ) ] ) =x$.

  Let $V \subset \bar{X}$ be an open simply-connected set. 
  Its preimage by $\pi$ consists of all the classes $[ ( x,V,A_{i} ) ]$, where $x \in V$ and $A_{i}$
  is one of the connected components of $\pi^{-1} ( V \cap X )$ (each of which
  is homeomorphic to $V \cap X$, because it is simply-connected). 
  There are no more classes: let $[ ( x' ,V' ,A' ) ]$ be a class with $x' \in V$. 
  Then $V\cap V'$ is a neighborhood of $x'$ which contains a simply connected neighborhood $V''$ of $x'$.
  
  As $\pi |_{A'} $ is a homeomorphism, there is a point $y \in A' \cap \pi^{-1} ( V'' ) \neq
  \varnothing$. 
  As $\pi ( y ) \in V$, $y$ belongs to one of the $A_{i}$ and thus $[ ( x' ,V' ,A' ) ] = [ ( x' ,V,A_{i} ) ]$.

  The sets $O_{V,A_{i}} = \{ [ ( x,V,A_{i} ) ] \nocomma \nocomma x \in V \}$, where $V$ is fixed and $A_{i}$ are the connected components of $\pi^{-1} ( V )$, therefore they are open and disjoint. 
  The map $\bar{\pi}$ restricts to an homeomorphism from each $O_{V,A \nocomma_{i}}$ onto $V$ (an open set $O$ contained in $O_{V,A \nocomma_{i}}$ is of the form $O_{V' ,A_{i} \cap \pi^{-1} ( V' \cap X )}$
  for $V'=\pi(O) \subset V$).
  In particular, each $O_{V,A \nocomma_{i}}$ is connected, and thus $\pi^{-1} ( V ) = \sqcup O_{V,A_{i}}$ satisfies the stack property.

  The third property follows because $\bar{\pi}$ is an homeomorphism when restricted to a simply-connected set.

  In order to prove the universal property, let $\tilde{Y} \nocomma ,
  \tilde{\pi} , \tilde{i}$ satisfy the stated properties. 
  For a point $y \in \tilde{Y}$, we define $\rho ( y ) = [ ( \tilde{\pi} ( y ) , \tilde{\pi} (\tilde{V}) ,
  \tilde{i}^{-1} ( \tilde{V}) ) ]$ , where $\tilde{V}$ is a simply-connected neighborhood of $y$.

  We check that $\rho \circ \tilde{i} =i$: $i ( y ) = [ ( \pi ( y ) , \pi ( A
  ) ,A ) ]$ for a simply connected neighborhood $A$ of $y$, and $\rho (
  \tilde{i} ( y ) ) = [ ( \tilde{\pi} ( \tilde{i} ( y ) ) , \tilde{\pi}(\tilde{V}) ,
  \tilde{i}^{-1} ( \tilde{V} ) ) ]$ for a simply connected neighborhood
  $\tilde{V}$ of $\tilde{i} ( y )$. 
  The two points are the same because $y \in A \cap \tilde{i}^{-1} ( \tilde{V} )$, which a non-empty open set which contains a simply-connected neighborhood of $y$.
  It is trivial to check that $\bar{\pi} \circ \rho = \tilde{\pi}$.

  Let $\tilde{V} \subset \tilde{Y}$ be simply-connected. 
  It follows from $\bar{\pi} \circ \rho = \tilde{\pi}$ that $\rho$ is an homeomorphism when
  restricted to $\tilde{V}$.
\end{proof}

Proof of \ref{extending riemannian coverings}: The topological spaces $X$ and
$Y$ satisfy the hypothesis of the lemma by standard results of dimension theory (see {\cite{Hurewicz_Wallman_Dimension_Theory}}).

\begin{remark}
 We asked for suggestions about the general topology lemma \ref{lem: extension of a covering map} on the algebraic topology list ALGTOP-L.
 
 Ben Wieland suggested that the lemma is also true, and more natural, if the hypothesis is that $\bar{X}$ admits a basis $\mathcal{B}$ of simply-connected open sets such that the intersection of any basis set with $X$ is non-empty and simply-connected.
\end{remark}

\subsection{Terminal points for the conjugate descending flow}

So it remains to prove the IJK conjecture \ref{points I, J, K}.
The first task, of course, is to define \emph{terminal points} precisely:

\begin{dfn}
 A point $x$ is terminal if there is no CDC starting at that point.
\end{dfn}

We list some conjectures related to conjecture \ref{points I, J, K}:

\begin{conjecture}
 All terminal points of order $1$ are unequivocal.
\end{conjecture}

\begin{conjecture}
 All terminal points are unequivocal.
\end{conjecture}

\begin{conjecture}\label{conjecture: terminal points are open}
A point $x$ of order $1$ is terminal iff it has a neighborhood $U$ of special coordinates such that $ e_{1} ( U )$ is a neighborhood of $ e_{1} ( x )$.
\end{conjecture}

\begin{conjecture}
A point $x$ is terminal iff it has a neighborhood $U$ of special coordinates such that $ e_{1} ( U )$ is a neighborhood of $ e_{1} ( x )$.
\end{conjecture}

\begin{conjecture}
 The image by the exponential of all the conjugate terminal points of order $k$ has Hausdorff dimension at most $n-k-1$
\end{conjecture}

\begin{conjecture}
 All non-terminal points of order $1$ are linked to a point of smaller radius.
\end{conjecture}

\begin{conjecture}
 All non-terminal points are linked to a point of smaller radius.
\end{conjecture}

If we plan to use linking curves, its definition should be appropriately generalized, otherwise it is clear that linking curves, with finitely many segments, will not exist in arbitrary Riemannian manifolds.
We propose the following definition:
\begin{dfn}
 A \strong{linking curve} between $x$ and $y$ is a curve $\alpha:[0,T]\rightarrow T_p M$ such that $[0,T]$ is split into two subsets $A$ and $B$, so that:
 \begin{itemize}
  \item $B$ is closed, and the Hausdorff dimension of $\e(B)$ is $0$.
  \item $A$ is open, a countable union of open intervals $I_n, n\in \mathcal{I}$ and $J_m, m\in \mathcal{I}$, so that $I_n$ is an ACDC curve and $J_n$ is a retort.
 \end{itemize}
\end{dfn}

The image by the exponential of such a curve would be fully tree-formed, and the same proof we used in \ref{main properties of linking curves} would do, but it would be more technically challenging to prove that its extremae are linked without using new ideas.

\backmatter

\chapter{Conclusiones}

En el capítulo \ref{Chapter: IntroBalanced} introdujimos algunos resultados útiles.
El teorema \ref{reduce HJ to boundary value 0}, por ejemplo, da mucha
más potencia a los resultados de estructura de \cite{Li Nirenberg}, pues permite
eliminar la restricción, importante para las aplicaciones a problemas de
frontera, de que el dato de frontera sea nulo.

Los resultados del capítulo \ref{Chapter: structure} hacen parecer razonable la conjetura \ref{conjecture: Hausdorff dimension of the points with a minimizing conjugate geodesic}.
Además, hemos mostrado aplicaciones concretas para el resultado de estructura \ref{main theorem 3}, y sugerido otras, como las posibles aplicaciones al movimiento browniano en variedades al final de la sección \ref{section: introduction about structure of the cut locus}.

Los resultados del capítulo \ref{Chapter: balanced} son a nuestro entender bastante completos: trabajamos con hipótesis bastante habituales, como la convexidad del Hamiltoniano, sin las cuales la misma definición de solución de viscosidad no están claras.
Ni siquiera es habitual rebajar las condiciones de regularidad: en el paper \cite{Li Nirenberg} se trabaja con abiertos $C^{2,1}$, pero el resto de datos son $C^{\infty}$.
No pensamos que rebajar la regularidad hubiera producido resultados cualitativamente distintos, en este contexto.
Sin embargo, nos parece muy interesante haber evitado limitarnos a abiertos del plano simplemente conexos, pues ésto nos hubiera cerrado los ojos al bello resultado \ref{maintheorem2}, que dice mucho sobre la naturaleza de los balanced split loci.

Respecto a la conjetura de Ambrose, hemos dado una demostración para métricas genéricas susceptible de ser generalizada a métricas arbitrarias. Este último paso es técnicamente muy complicado, ya que supone un mejor entendimiento de las singularidades de la aplicación exponencial que permitan seguir un campo de vectores sobre una superficie singular que luego debe
ser ``respondido" con curvas que deben en lo posible mantenerse alejadas de las singularidades.
Sin embargo, creemos que nuestro enfoque es original, que introduce
algunas ideas nuevas e interesantes, como el enunciado, a nuestro entender muy
natural, del lema \ref{norm dominated by exp}, o las conjeturas \ref{conjecture on magician's hats}, y que usa de formas nuevas ideas poco conocidas, como las curvas de flujo conjugado descendiente, implícitas en el trabajo \cite{Hebda82}, o la síntesis de dos variedades, que aparece ya en \cite{O'Neill68}.

En definitiva, además, creemos que el problema no era fácil, como comentamos en
las secciones \ref{section: proving poincare conjecture, or not} o 
\ref{section: proof for a 3-manifold with arbitrary metric}.

\end{document}